\documentclass[a4paper,11pt,twoside]{article}
\usepackage{mathrsfs}
\usepackage{amssymb}
\usepackage{amsmath}
\usepackage{mathrsfs}
\usepackage{amsthm}
\usepackage{amsfonts}
\usepackage{color}
\usepackage{latexsym}
\usepackage{indentfirst}

\usepackage{amscd}
\usepackage{bbm}
\usepackage{graphicx}
\usepackage[all]{xy}
\usepackage[mathscr]{eucal}
\usepackage{subfigure}

\newcommand{\K}{{\mathbb K}}

\newcommand{\N}{{\mathbb N}}
\newcommand{\C}{{\mathbb C}}
\newcommand{\R}{{\mathbb R}}

\allowdisplaybreaks
\newtheorem{theorem}{Theorem}[section]

\newtheorem{corollary}[theorem]{Corollary}

\newtheorem{lemma}[theorem]{Lemma}

\newtheorem{proposition}[theorem]{Proposition}
\newtheorem{claim}[theorem]{Claim}

\newtheorem{assumption}[theorem]{Assumption}

\numberwithin{equation}{section}

\begin{document}

\title{\bf\Large
	Bifurcations of periodic and antiperiodic orbits near an equilibrium in
 autonomous differential delay systems\\
  with one or two delays
	\footnotetext{\hspace{-0.35cm} 2020
{\it Mathematics Subject Classification}.
Primary 34K18, 34K20.
\endgraf
{\it Key words and phrases.}
Delay or advanced differential equations,
bifurcation, periodic or antiperiodic solutions, Maslov index.
}}
\date{}
\author{Guangcun Lu\footnote{
E-mail: \texttt{gclu@bnu.edu.cn}/{July 15, 2026(first), July 21, 2026(revised)}.}}

\maketitle

\vspace{-0.7cm}

\begin{center}
\begin{minipage}{13cm}
{\small {\bf Abstract}\quad
We investigate bifurcations of periodic and antiperiodic solutions near equilibria for four classes of potential differential-delay equations involving one or two delays. By reformulating each system as a (generalized) Hamiltonian system,
and applying the Hamiltonian bifurcation theory recently developed by the author
we obtain  bifurcation results of both Fadell--Rabinowitz and Rabinowitz type. The analysis relies on the computation of Maslov--type indices for fundamental solutions of the associated linear Hamiltonian systems.
}
\end{minipage}
\end{center}

\vspace{0.2cm}
	
\tableofcontents

\vspace{0.2cm}

\section{Introduction}\label{sec:Intro0}
\setcounter{equation}{0}
Following the pioneering work of Nussbaum \cite{Nu75} and
Kaplan-Yorke \cite{KaYo74}, bifurcations of periodic solutions in differential delay equations (DDEs) have been extensively studied by various methods; see, e.g.,
\cite{ArHb90, Ben24, BiHanWu03,
ChowM78,ChowWa88, DiVLW, Do89, ElSM05, Er09, Fa06, Guo26, Hale, HaV93, HaW04, HanM05, HanBi04, HbQe06, Lop, LSBK23, ThWal05, Wal78, Wal83, Wal06, Wal}
 and the references therein.

However, bifurcation results of the Fadell--Rabinowitz and Rabinowitz type remain scarce. As a continuation of our variational bifurcation research program started in \cite{Lu8}, this paper aims to fill that gap by investigating these types of bifurcations near the trivial equilibrium for the following four types of potential differential delay systems with one or two delays.

The first class consists of the following differential delay systems with
a parameter $\lambda$ in a real interval $\Lambda$  and
 a fixed delay $\tau>0$:
\begin{equation}\label{e:Delay6Auto13}
\dot{x}(t) = \nabla_2 V(\lambda, x(t - \tau))\quad\text{and}\quad
x(t + 2\tau) = -x(t)\quad \forall t,
\end{equation}
and
\begin{equation}\label{e:2-Delay6Auto13}
\begin{cases}
\dot{x}(t) = \nabla_2 V(\lambda, x(t - \tau)) + \nabla_2 V(\lambda, x(t - 2\tau)), \\
x(t + 3\tau) = -x(t)\quad \forall t.
\end{cases}
\end{equation}
 Regarding the potential $V$, we make the following assumption.
\begin{assumption}\label{ass:BasiAss1Delay1}
{\rm	Assume $V: \Lambda \times \mathbb{R}^n \to \mathbb{R}$ is continuous and satisfies:
	\begin{itemize}
		\item[\rm (a)] For each $\lambda \in \Lambda$, the map $V(\lambda, \cdot): \mathbb{R}^n \to \mathbb{R}$ is even and of class $C^2$. Consequently, ${0} \in \mathbb{R}^n$ is an equilibrium of \eqref{e:Delay6Auto13}.
		\item[\rm (b)] The gradient $\nabla_2 V(\lambda, x)$ and the Hessian $\nabla^2_2 V(\lambda, x)$ with respect to $x$ depend continuously on $(\lambda, x)$.
	\end{itemize}}
\end{assumption}

The second class is the following differential delay system:
\begin{equation}\label{e:Delay2Auto}
\dot{x}( t ) = J_n\nabla_2 G (\lambda,  x ( t - \tau ) )\quad\text{and}\quad
		x( t + 2\tau ) =  x( t ), \quad \forall t,
\end{equation}
where $\lambda$ is a parameter varying in a real interval
$\Lambda$, and $\tau>0$. The function $G$
is assumed to satisfy the following.

\begin{assumption}
	\label{ass:bif-per2Delay2Auto}
Let  $G \in C^0(\Lambda \times \mathbb{R}^{2n}, \mathbb{R})$
 be such that for every $\lambda \in \Lambda$, the function $G(\lambda, \cdot): \mathbb{R}^{2n} \to \mathbb{R}$ is of class $C^2$,
 and the gradient $\nabla_2G(\lambda, x)$ and the Hessian $\nabla^2_2G(\lambda, x)$ of
$G(\lambda, \cdot)$ at $x\in\mathbb{R}^{2n}$ depend continuously on $(\lambda, x)$ and satisfy
\begin{equation}\label{e:V-function1G}
\nabla_2G(\lambda, {0})=0\quad\hbox{and}\quad
\nabla^2_2G\left(\lambda, {0}\right)=
\left( \begin{array} { c c }
 	D(\lambda) & -E(\lambda) \\
 	E(\lambda) &D(\lambda)
 \end{array} \right),
\end{equation}
where for each $\lambda\in\Lambda$, $D(\lambda)$ is an $n\times n$
symmetric matrix and $E(\lambda)$ is an $n\times n$ skew-symmetric
matrix.
\end{assumption}

As noted just before Claim~\ref{cl:crm1},
the second condition in \eqref{e:V-function1G}
follows from the first one and from the assumption that
$G(\lambda,x_1,x_2)=G(\lambda,-x_2,x_1)$
for all $(\lambda,x_1,x_2)\in\Lambda\times\mathbb R^{n}\times\mathbb R^{n}$.


The third class consists of the following differential delay systems with a parameter $\lambda$ in a real interval $\Lambda$ and a fixed $\tau>0$:
\begin{equation}\label{e:crm1}
\dot{x}(t)=-\nabla_3 H(\lambda, x(t),x(t-\tau))\quad\text{and}\quad
x(t+2\tau)=-x(t)\quad\forall t,
\end{equation}
and the advanced counterpart
\begin{equation}\label{e:crm1Ad}
\dot{x}(t)=-\nabla_3 H(\lambda, x(t),x(t+\tau))\quad\text{and}\quad
x(t+2\tau)=-x(t)\quad\forall t.
\end{equation}
We impose the following assumption on the function $H$.

\begin{assumption}\label{ass:Crm1}
	Assume that $H: \Lambda \times \mathbb{R}^{n} \times \mathbb{R}^{n} \to \mathbb{R}$, $(\lambda,x,y) \mapsto H(\lambda,x,y)$, is continuous and satisfies:
	\begin{itemize}
		\item[\rm (i)] For each fixed $\lambda \in \Lambda$, the map $H(\lambda, \cdot): \mathbb{R}^n \times \mathbb{R}^n \to \mathbb{R}$ is of class $C^2$,
and the Euclidean gradient and the Hessian of  $H(\lambda, z)$ with respect to
$z\in{\R}^{n}\times \mathbb{R}^n$, $\nabla_2H(\lambda, z)$ and $\nabla^2_3H(\lambda, z)$,
are continuous in  $(\lambda, z)\in\Lambda\times \R\times\mathbb{R}^{n}\times \mathbb{R}^n$.
		\item[\rm (ii)] $H(\lambda, x, y) = H(\lambda, y, -x)$ for all $(\lambda, x, y)$.
	\end{itemize}
\end{assumption}

The fourth class is given by the following distributed delay differential
system with a parameter $\lambda$ in a real interval $\Lambda$ and a fixed $\tau>0$:
\begin{equation}\label{e:Bifu-distributedDelay1}
\left\{\begin{array}{ll}
\dot{x}(t) = -\nabla_2 F\left(\lambda,  \displaystyle\int_0^\tau
\nabla_2 G(\lambda, x(t-s)) \,\mathrm{d}s\right), & x(t) \in \mathbb{R}^n,
\\
x(t+\tau) = -x(t) \quad \forall t,
\end{array}\right.
\end{equation}
where $F: \Lambda  \times \mathbb{R}^n \to \mathbb{R}$ and
$G: \Lambda \times \mathbb{R}^{n} \to \mathbb{R}$ are continuous functions satisfying the following:
\begin{assumption}\label{ass:Distri1Delay1}
\begin{itemize}
\item[\rm (i)] For each fixed $\lambda \in \Lambda$, the function
$F(\lambda,\cdot): \mathbb{R}^n \to \mathbb{R}$ is  $C^2$,
and the Euclidean gradient and the Hessian of  $F(\lambda, x)$ with respect to
$x\in{\R}^{n}$, $\nabla_2F(\lambda, x)$ and $\nabla^2_2F(\lambda, x)$,
are continuous in  $(\lambda, x)\in\Lambda\times\mathbb{R}^{n}$.

\item[\rm (ii)] For each fixed $\lambda\in \Lambda$, the function
$G(\lambda,\cdot): \mathbb{R}^n \to \mathbb{R}$ is even and  $C^2$,
and the Euclidean gradient and the Hessian of  $G(\lambda, x)$ with respect to
$x\in{\R}^{n}$, $\nabla_2G(\lambda, x)$ and $\nabla^2_2G(\lambda, x)$,
are continuous in  $(\lambda, x)\in\Lambda\times\mathbb{R}^{n}$.

\item[\rm (iii)] For each fixed $\lambda\in \Lambda$, $\nabla_2F(\lambda, 0)=\nabla_2G(\lambda, 0)=0$.
 $\nabla^2_2F(\lambda, 0)=\nabla^2_2G(\lambda, 0)=0$.
\end{itemize}
\end{assumption}

Two solutions $x$ and $y$ of
(\ref{e:Delay6Auto13}), (\ref{e:2-Delay6Auto13}),
 (\ref{e:Delay2Auto}),  (\ref{e:crm1}),  (\ref{e:crm1Ad}), or (\ref{e:Bifu-distributedDelay1}) with a fixed $\lambda$
are called \textsf{geometrically distinct} (or
\textsf{$\mathbb{R}$-distinct})
if and only if there is no $s\in\mathbb{R}$ such that $x= s\cdot y$,
where the  $\mathbb{R}$-action
 is given by $s\cdot y(t)=y(t+s)$ for all $t\in\mathbb{R}$.
  The orbit of $x$ under this action is denoted by
 $\mathbb{R}\cdot x:=\{s\cdot x\,|\, s\in\mathbb{R}\}$.

Let (1.$\ast$) denote one of (\ref{e:Delay6Auto13}), (\ref{e:2-Delay6Auto13}),
 (\ref{e:Delay2Auto}),  (\ref{e:crm1}),  (\ref{e:crm1Ad}), or (\ref{e:Bifu-distributedDelay1}), and let $0$ denote
 the trivial solution of (1.$\ast$).
  For a $\mu\in\Lambda$,  we call $(\mu, 0)$ a \textsf{bifurcation point} of
  (1.$\ast$)   if there exists a sequence $\{\lambda_k\}^\infty_{k=1}$
  in $\Lambda$ converging to $\mu$ and a nonzero solution $x^k$ of
  (1.$\ast$) with $\lambda=\lambda_k$ for each $k$, such that
  $\|x^k|_{[0,3\tau]}\|_{C^1}\to 0$ as $k\to \infty$.

 Taking $\tau=1$,
${F}(\lambda, x)=\lambda\cos x$ and
${G}(\lambda, x)\equiv\frac{1}{2}x^2$ in
(\ref{e:Bifu-distributedDelay1}) yields
\begin{equation}\label{e:distributedDelay8}
\left\{\begin{array}{ll}
\dot{x}(t) = \lambda\sin\left(\displaystyle\int_0^1
x(t-s) \,\mathrm{d}s\right), & x(t) \in \mathbb{R},
\\
x(t+1) = -x(t) \quad \forall t.
\end{array}\right.
\end{equation}
As an application of our Theorem~\ref{th:DistribifDelay1},
we obtain:
(i) if $(\mu,0)$ is a bifurcation point of
(\ref{e:distributedDelay8}), then $\mu\in\{\frac{(2k-1)^2\pi^2}{{2}}\,|\,
 k\in\mathbb{N}\}$;
(ii) conversely,
for any $\mu\in\{\frac{(2k-1)^2\pi^2}{{2}}\,|\,
 k\in\mathbb{N}\}$,  either the distributed-delay differential equation
(\ref{e:distributedDelay8})
with $\lambda=\mu$ admits a sequence of  distinct nonzero solutions
	$\{x^l\}_{l=1}^{\infty}$  such that $x^l\to 0$ in $C^1_{\rm loc}(\mathbb{R},\mathbb{R})$
	as $l\to\infty$, or
there exist left and right  neighborhoods $\Lambda^-$ and $\Lambda^+$ of $\mu$ in $\mathbb{R}$
	and nonnegative integers $n^+$ and $n^-$ satisfying $n^++n^-\ge 1$, such that for $\lambda\in\Lambda^-\setminus\{\mu\}$ (resp. $\lambda\in\Lambda^+\setminus\{\mu\}$),
	(\ref{e:distributedDelay8}) with parameter  $\lambda$  has at least $n^-$ (resp. $n^+$) $\mathbb{R}$-distinct  nontrivial solutions
	$x_\lambda^i$ ($i=1,\dots,n^-$ (resp. $n^+$))
	which converge to zero in $C^1_{\rm loc}(\mathbb{R}, \mathbb{R})$
	as $\lambda\to\mu$.

   Nakata \cite{Na21}  studied the existence of solutions
  to (\ref{e:distributedDelay8}). In \cite[Theorem~3.1]{Na21}
  it is shown that for each $\lambda>\frac{\pi^2}{2}$,
  (\ref{e:distributedDelay8}) has a unique solution $x_\lambda$ such that
  $x_\lambda(t)\ne 0$ for $0\le t<\frac{1}{2}$ and $x_\lambda(\frac{1}{2})=0$;
  moreover, $x_\lambda$ is expressed in terms of
Jacobi elliptic functions. Clearly,
$x_\mu$ cannot be our solutions $x_\mu^l$ for large $l$;
and similarly, $x_\lambda$ cannot be our solutions $x_\lambda^i$
for $\lambda$ near $\mu$ in general.\\

\noindent{\textbf{Methods}}.
This work is methodologically based on the observation by Kaplan and Yorke  (1974) \cite{KaYo74} that the problem of finding periodic solutions to the differential delay equations (\ref{e:Delay6Auto13})
 and (\ref{e:2-Delay6Auto13})
 can be mapped to the problem of finding periodic solutions with symmetric structures for an associated Hamiltonian system.
A similar transformation for problem (\ref{e:Delay2Auto})
[resp. (\ref{e:Bifu-distributedDelay1})] is discussed in \cite{Liu12}
(resp. \cite{Na20, Na21, LiuWZ25}).
Alternative bifurcation results of Fadell--Rabinowitz and Rabinowitz type for such systems were developed by the author in \cite{Lu11}. By applying these results---more precisely, Theorems~1.5, 1.8 and~1.14 of \cite{Lu11} and the corollaries thereof---we prove the main theorems, namely, Theorems~\ref{th:bif-per3Delay}, \ref{th:bif-per3DelayII}, \ref{th:bif-per3+3}, \ref{th:bif-per3DelayCrm}, \ref{th:DistribifDelay1}, and \ref{th:DistribifDelay2}.\\

\noindent{\textbf{Organisation of the paper}}.
Section~\ref{sec:delay1} investigates problem~\eqref{e:Delay6Auto13}. The main results are Theorems~\ref{th:bif-per3Delay} and~\ref{th:bif-per2Delay++}; Corollaries~\ref{cor:bif-per4Delay} and~\ref{cor:bif-per3DelayC} follow from Theorem~\ref{th:bif-per3Delay}. The proof of Theorem~\ref{th:bif-per3Delay} relies on Theorem~\ref{th:PIndex1}.

Section~\ref{sec:delay2} studies problem~\eqref{e:2-Delay6Auto13}. The main results are Theorems~\ref{th:bif-per3DelayII} and~\ref{th:2-bif-per2Delay++}, with Corollaries~\ref{cor:bif-per4DelayIIB} and~\ref{cor:bif-per3DelayIIC} derived from Theorem~\ref{th:bif-per3DelayII}. Its proof makes essential use of Theorem~\ref{th:PIndex3}.

Section~\ref{sec:delay3} concerns problem~\eqref{e:Delay2Auto}. Its main results are Theorems~\ref{th:bif-per3+3} and~\ref{th:bif-per2Delay++G}; Corollaries~\ref{cor:bif-per3+3} and~\ref{cor:bif-per3+4} are consequences of the former. A key tool in the proof of Theorem~\ref{th:bif-per3+3} is Theorem~\ref{th:PIndex4}.

Section~\ref{sec:delay4} addresses problems~\eqref{e:crm1} and~\eqref{e:crm1Ad}. The principal results for~\eqref{e:crm1} are Theorems~\ref{th:bif-per3DelayCrm} and~\ref{cor:crm5}; Corollaries~\ref{cor:bif-per3+3Crm} and~\ref{cor:bif-per3+4Cri} follow from Theorem~\ref{th:bif-per3DelayCrm}. Theorem~\ref{th:PIndex2} plays a central role in proving Theorem~\ref{th:bif-per3+3}. Immediately after Claim~\ref{cl:crm2}, we explain how to adapt the arguments so that the conclusions of Theorem~\ref{th:bif-per3DelayCrm} also hold for problem~\eqref{e:crm1Ad}.

Section~\ref{sec:delay5} studies problem~(\ref{e:Bifu-distributedDelay1}).
The main results are Theorems~\ref{th:DistribifDelay1},
\ref{th:DistribifDelay2},
and their consequence, Corollary~\ref{cor:DistribifDelay3}.

 Appendix~\ref{app:Maslov} provides the proofs of Theorems~\ref{th:PIndex1}, \ref{th:PIndex2}, \ref{th:PIndex3}, and~\ref{th:PIndex4}. These theorems compute Maslov-type indices for fundamental matrix solutions of several classes of linear Hamiltonian systems; their proofs are technical.\\

\noindent{\textbf{Notation and conventions}}. Let  $[a]$  denote the greatest integer less than or equal to $a\in\mathbb{R}$, and let
$\lceil a \rceil$ (the ceiling of $a$) be the least integer greater than or equal to $a$, i.e.,
$\lceil a \rceil=\min\{k\in\mathbb{Z}\,|\,k\ge a\}$. All vectors in $\R^m$ will be understood as column vectors.
Let $\mathbb{K}=\mathbb{R}$ or $\mathbb{C}$, and let $\mathbb{K}^{m\times m}$ denote the set of
all $m\times m$ matrices with entries in the field $\mathbb{K}$.
For $M\in\K^{m\times m}$ let $\sigma(M)$ be the set of all eigenvalues of $M$,
let $M^\top$ be the transpose of $M$, ${\rm Ker}(M)=\{x\in\K^m\,|\,Mx=0\}$ and $\exp M=e^M=\sum^\infty_{l=0}\frac{1}{l!}M^l$.
We denote by $(\cdot,\cdot)_{\mathbb{R}^m}$ the standard Euclidean inner product on
$\mathbb{R}^m$, and by $|\cdot|$ the corresponding norm.
Let $\mathcal{L}_s(\mathbb{R}^{m})$ be the set of all real symmetric
matrixes of order $m$.  $J_n$ denotes the standard complex structureon $\mathbb{R}^{2n}$ is given by (\ref{e:standcompl}).

%

\section{Bifurcation near the trivial equilibrium of system (\ref{e:Delay6Auto13})}\label{sec:delay1}



The following result furnishes a clear bifurcation diagram for system (\ref{e:Delay6Auto13}).

\begin{theorem}\label{th:bif-per3Delay}
Under Assumption~\ref{ass:BasiAss1Delay1},
let  $e_1(\lambda)\le e_2(\lambda)\le\cdots\le e_n(\lambda)$ be all eigenvalues of  $\nabla^2_2V\left(\lambda, 0\right)$.
Given a delay parameter $\tau>0$, define intervals of fixed length $2\pi/\tau$ as
$$
I_j := \left( \frac{(4j - 1)\pi}{2\tau}, \frac{(4j + 3)\pi}{2\tau} \right], \quad j \in \mathbb{Z},
$$
which are contiguous and non-overlapping, each of length $\frac{2\pi}{\tau}$.
Consequently, for each $e_k(\lambda)$, there exists a unique integer $j_k(\lambda) \in \mathbb{Z}$ such that
 \begin{equation}\label{e:uniqueInteger}
e_k(\lambda) \in I_{j_k(\lambda)} = \left( \frac{(4j_k(\lambda) - 1)\pi}{2\tau}, \frac{(4j_k(\lambda) + 3)\pi}{2\tau} \right].
\end{equation}
\begin{enumerate}
\item[\rm (I)]{\rm (\textsf{Necessary condition}):}
If $(\mu,0)$ with $\mu\in\Lambda$ is a bifurcation point of (\ref{e:Delay6Auto13}),
then $e_l(\mu)\in\frac{(4\mathbb{Z}+3)\pi}{2\tau}$ for some $l$.

\item[\rm (II)]{\rm (\textsf{Sufficient condition}):}
Let $\mu$ be an interior point of $\Lambda$. Then
$(\mu,0)$ is a bifurcation point of (\ref{e:Delay6Auto13}),
 provided that  there exists a partition
  $\{i_1,\dots,i_k\}\cup\{i_{k+1},\dots,i_n\}$  of $\{1,\dots,n\}$
  with $1 \le k \le n$, such that
 \begin{itemize}
\item[\rm (A)] $e_{i_l}(\mu)=\frac{(4j_{i_l}(\mu)+3)\pi}{2\tau}$ for  $l=1,\dots,k$,
and $e_{i_l}(\mu)\notin\frac{(4\mathbb{Z}+3)\pi}{2\tau}$ for $l=k+1,\dots,n$;

\item[\rm (B)] there exist two sequences $(\lambda_m^\pm)_m\subset
\Lambda$  converging to $\mu$ such that
$e_{i_l}(\lambda_m^-)< e_{i_l}(\mu)<e_{i_l}(\lambda_m^+)$
for all $m\in\mathbb{N}$ and all $l=1,\dots,k$.
\end{itemize}
%

\item[\rm (III)]{\rm (\textsf{Alternative bifurcations of Fadell-Rabinowitz type and of Rabinowitz type}):}
Let the assumptions of (II) hold, with condition (B) strengthened as follows:
\begin{itemize}
\item[\rm (SB)] For some $\epsilon>0$, one of the following two conditions holds:
\begin{itemize}
\item[\rm (SB.1)] $e_{i_l}(\lambda)< e_{i_l}(\mu)<e_{i_l}(\lambda')$
for all $(\lambda, \lambda')\in (\mu-\epsilon, \mu)\times (\mu, \mu+\epsilon)$ and
 $l=1,\dots,k$.
\item[\rm (SB.2)]
$e_{i_l}(\lambda)> e_{i_l}(\mu)>e_{i_l}(\lambda')$
for all $(\lambda, \lambda')\in (\mu-\epsilon, \mu)\times (\mu, \mu+\epsilon)$ and
 $l=1,\dots,k$.
 \end{itemize}
 \end{itemize}
   Then for the problem (\ref{e:Delay6Auto13}),
     there exist two possible alternatives:
\begin{itemize}
\item[\rm (i)] The problem (\ref{e:Delay6Auto13})  with $\lambda=\mu$ admits a sequence of
nonzero, $\R$-distinct solutions
$\{x^l\}^{\infty}_{l=1}$  such that $x^l\to 0$ in $C^1_{\rm loc}(\mathbb{R},\mathbb{R}^n)$
as $l\to\infty$.
\item[\rm (ii)] There exist left and right  neighborhoods $\Lambda^-$ and $\Lambda^+$ of $\mu$ in $\Lambda$
and nonnegative integers $n^+$ and $n^-\ge 0$ satisfying $n^++n^-\ge k$, such that
for $\lambda\in\Lambda^-\setminus\{\mu\}$ (resp. $\lambda\in\Lambda^+\setminus\{\mu\}$),
(\ref{e:Delay6Auto13}) with parameter  $\lambda$  has at least $n^-$ (resp. $n^+$) $\R$-distinct
 solutions $x_\lambda^i\ne 0$ ($i=1,\dots,n^-$ (resp. $n^+$))
which converge to zero in $C^1_{\rm loc}(\mathbb{R}, \mathbb{R}^n)$
as $\lambda\to\mu$.
\end{itemize}
Moreover, if $n>1$ and $k>1$, then at least one of (i), (iii), and (iv)  holds,
where (i) is as stated previously, and
\begin{itemize}
\item[\rm (iii)]  For every $\lambda\in\Lambda\setminus\{\mu\}$ near $\mu$, there is a
 nonzero solution ${x}_\lambda$  of (\ref{e:Delay6Auto13})
 with parameter $\lambda$, such that  $x_\lambda\to 0$ in $C^1_{\rm loc}(\mathbb{R},\mathbb{R}^n)$
as $\lambda\to\mu$.

\item[\rm (iv)] For a given $\varepsilon > 0$, there exists a one-sided neighbourhood $\Lambda^0$ of $\mu$ in $\Lambda$ such that, for any $\lambda \in \Lambda^0 \setminus \{\mu\}$, problem
    (\ref{e:Delay6Auto13})  with parameter $\lambda$ has either
    \begin{itemize}
    	\item[$\bullet$]
     infinitely many $\mathbb{R}$-distinct nonzero solutions $\{x_\lambda^l\}_{l=1}^\infty$ such that $\|x_\lambda^l|_{[0,2\tau]}\|_{C^1} < \varepsilon$ for all $l \in \mathbb{N}$,
     or

   \item[$\bullet$]  at least two $\mathbb{R}$-distinct nonzero solutions $\hat{x}_\lambda^1$ and $\hat{x}_\lambda^2$ satisfying the following inequalities: $\|\hat{x}_\lambda^i|_{[0,2\tau]}\|_{C^1} < \varepsilon$ for $i=1,2$, and
\begin{align*}
        & \bigl(\hat{x}_\lambda^1(0),  \hat{x}_\lambda^1(\tau) \bigr)_{\mathbb{R}^{n}}
          - \int_0^{\tau} \bigl( \dot{\hat{x}}_\lambda^1(t),  \hat{x}_\lambda^1(t-\tau) \bigr)_{\mathbb{R}^{n}} \, dt
          + \int_{0}^{2\tau} V\bigl(\lambda, \hat{x}_\lambda^1(t)\bigr) \, dt \\
        \ne \; &
        \bigl(\hat{x}_\lambda^2(0),  \hat{x}_\lambda^2(\tau) \bigr)_{\mathbb{R}^{n}}
          - \int_0^{\tau} \bigl( \dot{\hat{x}}_\lambda^2(t),  \hat{x}_\lambda^2(t-\tau) \bigr)_{\mathbb{R}^{n}} \, dt
          + \int_{0}^{2\tau} V\bigl(\lambda, \hat{x}_\lambda^2(t)\bigr) \, dt.
    \end{align*}
    \end{itemize}
\end{itemize}
\end{enumerate}
\end{theorem}




If $n=1$, since
$V(\lambda,x)=\int^x_0\nabla_2V(\lambda, y)dy$,
Theorem~\ref{th:bif-per3Delay} yields: 

\begin{corollary}\label{cor:bif-per4Delay}
Let $\Lambda \subset \mathbb{R}$ be an interval, and $V \in C(\Lambda \times \mathbb{R}, \mathbb{R})$ satisfy
 the following conditions:
  \begin{enumerate}
\item[\rm (C)] For each $\lambda \in \Lambda$,  $V(\lambda, \cdot): \mathbb{R}\to \mathbb{R}$ is even and $C^2$,
which implies $\frac{\partial V}{\partial x}(\lambda, 0) = 0$.
\item[\rm (D)] $\frac{\partial V}{\partial x}(\lambda, x)$
and $\frac{\partial^2 V}{\partial x^2}(\lambda, x)$
are continuous in $(\lambda, x)$.
    \end{enumerate}
Given a delay parameter $\tau>0$, consider the problem
\begin{equation}\label{e:Delay6Auto17*}
  \dot{x}( t ) = \frac{\partial V}{\partial x}(\lambda, x ( t - \tau) )\quad\hbox{and}\quad
  x( t + 2\tau) = - x( t )\;\forall t\in\mathbb{R}.
   \end{equation}
If $(\mu,0)$ with $\mu\in\Lambda$ is a bifurcation point of
(\ref{e:Delay6Auto17*}), then $\mu$ satisfies
  \begin{enumerate}
\item[\rm (E)]
 $\frac{\partial^2 V}{\partial x^2}(\mu, 0)=\frac{(4j+3)\pi}{2\tau}$ for some $j\in\mathbb{Z}$.
\end{enumerate}
Conversely, let $\mu$ be an interior point of $\Lambda$ satisfying (E).
Then $(\mu,0)$ is a bifurcation point of (\ref{e:Delay6Auto17*}) provided that
 there exist two sequences $(\lambda_m^\pm)_m\subset\Lambda$ converging to $\mu$,
 such that
$$
\frac{\partial^2 V}{\partial x^2}(\lambda_m^-, 0)<
 \frac{\partial^2 V}{\partial x^2}(\mu, 0)<\frac{\partial^2 V}{\partial x^2}(\lambda_m^+, 0)\quad\text{for all $m\in\mathbb{N}$}.
 $$
Moreover,  suppose that
an interior point $\mu$ of $\Lambda$ satisfies (E) and
  \begin{enumerate}
\item[\rm (F)]  In a punctured open neighborhood of $\mu$,
 $\frac{\partial^2 V}{\partial x^2}(\lambda, 0)-
 \frac{\partial^2 V}{\partial x^2}(\mu, 0)\ne 0$ and changes sign
 as $\lambda$ across $\mu$.
 (This can be satisfied if $\frac{\partial^2 V}{\partial x^2}(\lambda, 0)$ has a nonzero derivative at $\lambda=\mu$.)
\end{enumerate}
Then for the problem (\ref{e:Delay6Auto17*})
there are  the following alternatives:
 \begin{enumerate}
\item[\rm (i)] The problem (\ref{e:Delay6Auto17*})  with $\lambda=\mu$ has a sequence of nonzero,
 $\R$-distinct solutions
$\{x^k\}^{\infty}_{k=1}$ such that
$x^k\to 0$ in $C^1_{\rm loc}(\mathbb{R})$
as $k\to\infty$.
\item[\rm (ii)] There exist left and right  neighborhoods $\Lambda^-$ and $\Lambda^+$ of $\mu$ in $\mathbb{R}$
and nonnegative integers $n^+$ and $n^-$ satisfying $n^++n^-\ge 1$,
 such that for $\lambda\in\Lambda^-\setminus\{\mu\}$ (resp. $\lambda\in\Lambda^+\setminus\{\mu\}$),
(\ref{e:Delay6Auto17*}) with parameter  $\lambda$  has at least $n^-$ (resp. $n^+$) $\R$-distinct
 solutions $x_\lambda^i\ne 0$ ($i=1,\dots,n^-$ (resp. $n^+$))
which converge to the zero function in $C^1_{\rm loc}(\mathbb{R})$ as $\lambda\to\mu$.
\end{enumerate}
\end{corollary}
Obviously, $V(\lambda,x)=-\lambda x^2-x^4$ and
$V(\lambda,x)=-\lambda(x^2+ x^4)$ satisfy Corollary~\ref{cor:bif-per4Delay}.

Furthermore, we obtain from Theorem~\ref{th:bif-per3Delay} the following
bifurcation result concerning the delay parameters.

\begin{corollary}\label{cor:bif-per3DelayC}
Let $W:{\mathbb{R}}^n\to\mathbb{R}$ be even and $C^2$.
Denote by $e_1\le e_2\le\cdots\le e_n$  all eigenvalues of  $\nabla^2W\left({0}\right)$.
For each $\lambda\in\mathbb{R}$, there exists a unique integer $j_k(\lambda) \in \mathbb{Z}$ such that
 \begin{equation}\label{e:uniqueIntegerC}
\lambda e_k \in I_{j_k(\lambda)}: = \left(\frac{(4j_k(\lambda) - 1)\pi}{2}, \frac{(4j_k(\lambda) + 3)\pi}{2} \right].
\end{equation}
If $(\mu,0)$ with $\mu\in\mathbb{R}\setminus\{0\}$ is a bifurcation point of the problem
\begin{equation}\label{e:Delay6Auto13C-}
  \dot{x}( t ) = \nabla W (x ( t - \lambda) )\quad\hbox{and}\quad
  x( t + 2\lambda) = - x( t )\;\forall t\in\mathbb{R},
    \end{equation}
then $\mu e_l\in\frac{(4\mathbb{Z}+3)\pi}{2\tau}$ for some $l$.
Conversely, for some $\mu\in\mathbb{R}$, suppose there exists a partition
$\{i_1,\cdots,i_k\}\cup\{i_{k+1},\cdots,i_n\}$ of $\{1,\cdots,n\}$
   with $1 \le k \le n$, such that
   the following  condition is satisfied:
    \begin{enumerate}
\item[\rm (G)] $\mu e_{i_l}=\frac{(4j_{i_l}(\mu)+3)\pi}{2}$ for  $l=1,\dots,k$ (which implies that
$\mu\ne 0$ and $e_{i_l}\ne 0$ for $l=1,\dots,k$),
$e_{i_l}$($l=1,\dots,k$) have the same plus-minus sign, and
 $\mu e_{i_l}\notin\frac{\pi(4\mathbb{Z}+3)}{2}$ for any $l=k+1,\dots,n$.
 \end{enumerate}
 Then for the problem (\ref{e:Delay6Auto13C-})
    there are two possible alternatives:
 \begin{enumerate}
\item[\rm (i)] The problem (\ref{e:Delay6Auto13C-})  with $\lambda=\mu$
has a sequence of nonzero,  $\R$-distinct solutions $\{x^k\}^{\infty}_{k=1}$ such that
$x^k\to 0$ in $C^1_{\rm loc}(\mathbb{R}, \mathbb{R}^n)$
as $k\to\infty$.

\item[\rm (ii)] There exist left and right  neighborhoods $\Lambda^-$ and $\Lambda^+$ of $\mu$ in $\Lambda$
and nonnegative integers $n^+$ and $n^-\ge 0$ satisfying $n^++n^-\ge k$, such that
for $\lambda\in\Lambda^-\setminus\{\mu\}$ (resp. $\lambda\in\Lambda^+\setminus\{\mu\}$),
(\ref{e:Delay6Auto13C-}) with parameter value $\lambda$  has at least $n^-$ (resp. $n^+$) $\R$-distinct
 solutions $x_\lambda^i\ne 0$ ($i=1,\dots,n^-$ (resp. $n^+$))
which converge to zero in $C^1_{\rm loc}(\mathbb{R}, \mathbb{R}^n)$  as $\lambda\to\mu$.
\end{enumerate}
Moreover, if $n>1$ and $k>1$, then at least one of (i), (iii), and (iv) holds,
where (i) is as stated previously,
and the assertions of (iii) and (iv) are as follows:
\begin{enumerate}
\item[\rm (iii)]  For every $\lambda\in\Lambda\setminus\{\mu\}$ near $\mu$, there is a
 nonzero solution ${x}_\lambda$  of (\ref{e:Delay6Auto13C-})
 with parameter  $\lambda$, such that $x_\lambda$ converges to the zero function
 in $C^1_{\rm loc}(\mathbb{R},  \mathbb{R}^n)$  as $\lambda\to\mu$.

\item[\rm (iv)] For a given $\varepsilon>0$,  there is a one-sided  neighborhood $\Lambda^0$ of $\mu$ in $\Lambda$ such that,
for any $\lambda\in\Lambda^0\setminus\{\mu\}$, problem (\ref{e:Delay6Auto13C-}) with parameter $\lambda$
 has either
 \begin{itemize}
 	\item [$\bullet$]
  infinitely many $\mathbb{R}$-distinct nonzero solutions $\{x_\lambda^l\}_{l=1}^\infty$ such that
$$
\|{x}_\lambda^l\|_{C^0([0, 2|\lambda|])}+|\lambda|
\|\dot{x}_\lambda^l\|_{C^0([0, 2|\lambda|])}<\varepsilon
$$
for $l=1,2,\dots$, or
\item[$\bullet$]  at least two $\mathbb{R}$-distinct nonzero solutions
 $\hat{x}_\lambda^1$ and $\hat{x}_\lambda^2$ satisfying the following inequalities:
$$
\|\hat{x}_\lambda^i\|_{C^0([0, 2|\lambda|])}+|\lambda|
\|\dot{\hat{x}}_\lambda^i\|_{C^0([0, 2|\lambda|])}<\varepsilon, \;i=1,2,
$$
and
\begin{align*}
&(\hat{x}^1_\lambda(0), \hat{x}^1_\lambda(\lambda))_{\mathbb{R}^n} - \int^{\lambda}_0(\dot{\hat{x}}^1_\lambda(t), \hat{x}^1_\lambda(t-\lambda))_{\mathbb{R}^n}dt + \int^{2\lambda}_0 W(\hat{x}^1_\lambda(t))dt\\
&\ne (\hat{x}^2_\lambda(0), \hat{x}^2_\lambda(\lambda))_{\mathbb{R}^n} - \int^{\lambda}_0(\dot{\hat{x}}^2_\lambda(t), \hat{x}^2_\lambda(t-\lambda))_{\mathbb{R}^n}dt + \int^{2\lambda}_0 W(\hat{x}^2_\lambda(t))dt.
\end{align*}
\end{itemize}
\end{enumerate}
\end{corollary}

When $\lambda<0$, the equation (\ref{e:Delay6Auto13C-}) is termed a \textsf{differential equation with advanced arguments} (cf. \cite{ElN73,Hale}).

 \begin{enumerate}
\item[\rm ($\mathcal{F}1$)] If $f \in C^1(\mathbb{R})$ is an odd function with $f'(0) \ne 0$.
 \end{enumerate}
 Under this assumption, for $\lambda\in\mathbb{R}\setminus\{0\}$,
 consider  the problem
\begin{equation}\label{e:Delay6Auto17**}
  \dot{x}( t ) = \lambda f(x ( t - 1) )\quad\hbox{and}\quad
  x( t + 2) = - x( t )\;\forall t\in\mathbb{R}.
   \end{equation}
Let $F$ be any even primitive of $f$. Applying either Corollary~\ref{cor:bif-per3DelayC} with $n=1$ to $W=F$
 or Corollary~\ref{cor:bif-per4Delay}  to the function $V(\lambda, x) = \lambda F(x)$, we obtain:

  \begin{claim}\label{cl:Compare2.1}
Under assumption ($\mathcal{F}1$),
$(\mu,0)$ with $\mu\in\Lambda$ is a bifurcation point of the problem
(\ref{e:Delay6Auto17**}) if and only if
 $\mu\in \frac{\pi(4\mathbb{Z}+3)}{2}$. Moreover, when
 the latter condition holds,
the conclusions of Corollary~\ref{cor:bif-per4Delay}
remain valid with the modification that we replace (\ref{e:Delay6Auto17*}) by (\ref{e:Delay6Auto17**}) and $\tau$ by $1$.
  \end{claim}

 \begin{enumerate}
\item[\rm ($\mathcal{F}2$)](\cite[H~2.1]{Nu75})
 $f\in C(\mathbb{R})$
is bounded below, 
 has a derivative $f'(0)=1$ at $0$,  and satisfies $xf(x) > 0$ for every $x\in\mathbb{R}\setminus\{0\}$, which implies $f(0)=0$.
 \end{enumerate}
 Under this assumption,
Nussbaum \cite{Nu75} investigated the structure of the set of (suitably normalized) periodic solutions to the delay differential equation $\dot{x}(t) = -\alpha f(x(t-1))$ with $\alpha>0$, and  showed that there is a continuum of nontrivial periodic solutions of the latter equation which bifurcates from the trivial solution at $\alpha=\pi/2$.
Walther \cite[Theorem~1]{Wal78} restated \cite[Theorem~2.1]{Nu75} in a slightly different way as follows:

\begin{theorem}[Nussbaum] \label{th:Compare2.2}
Under assumption ($\mathcal{F}2$),
 there exists a closed connected set $P \subset C([-1,0])\times \mathbb{R}^+$ with the following properties:
\begin{itemize}
    \item[\rm (i)] $(0, \pi/2] \in P$, and for every $\alpha > \pi/2$ there is a function $\varphi \in C([-1,0])$ with $(\varphi, \alpha) \in P$,
    \item[\rm (ii)] $\varphi \neq 0$, if $(\varphi, \alpha) \in P$ and $\alpha \neq \pi/2$,
    \item[\rm (iii)] if $(\varphi, \alpha) \in P$ and $\varphi \neq 0$, then $\varphi$ increases on $[-1, 0]$, $\varphi(-1) = 0$, and $\varphi$ is the restriction of a slowly oscillating periodic function $x$ which satisfies $x'(t) = -\alpha f(x(t-1))$ on $\mathbb{R}$.
\end{itemize}
\end{theorem}

%

 \begin{enumerate}
\item[\rm ($\mathcal{F}3$)]  $f:\mathbb{R}\to\mathbb{R}$
is odd and $C^k$, $k\in\mathbb{N}$, $f'(0) = -1$, and
$xf(x) < 0$ for every $x\in\mathbb{R}\setminus\{0\}$.
 \end{enumerate}
Under this assumption,  Kaplan and Yorke  \cite{KaYo74} proved
the following result (formulated in \cite[Theorem~1.1]{Do89}):

 \begin{claim}\label{cl:Compare2.3}
 There are $C^k$-maps $x: \mathbb{R} \times \mathbb{R} \to \mathbb{R}^2$ and $\alpha: \mathbb{R} \to \mathbb{R}$, such that for all $z \in \mathbb{R}$ the following is true:
\begin{itemize}
\item[\rm (1)] ${x}(\cdot, z)$ is odd and satisfies (\ref{e:Delay6Auto17**})
 with $\lambda=\alpha(z)$.
\item[\rm (2)] $x(0, z) = 0$, $x(1, z) = z$;
and $x(\cdot, z)$ is strictly increasing on $[0, 1]$ if $z>0$.
\item[\rm (3)] $\alpha(0)=-\frac{\pi}{2}f'(0)=\frac{\pi}{2}$ and $\alpha'(0)=0$.
\end{itemize}
 \end{claim}
Hence this family of the Kaplan--Yorke solutions (parameterized by the amplitude in a differentiable way), which is called the \textsf{Kaplan--Yorke branch} or \textsf{primary branch} of $f$ in \cite{Do89}, forms a branch of symmetric periodic solutions which bifurcates at $\lambda=\alpha(0)$
from the trivial solution $x\equiv 0$ and has a vertical tangent there.
Walther \cite{Wal78} proved that the direction of the
branch (i.e., the sign of $\alpha''(0)$)
is equal to the sign of $f'''(0)$, and  Dormayer \cite{Do89} obtained
\begin{equation*}
\alpha''(0) = \frac{\pi}{8} \frac{f'''(0)}{[f'(0)]^2}.
\end{equation*}
Walther \cite{Wal83} showed that
there is a point on the primary branch off such that every neighbourhood
of this point contains a periodic solution, which is not symmetric
with the more general symmetry (S) $y(\cdot+\tau)=-y$, for some $\tau>0$,
With a class of non-monotone functions $f$, Dormayer \cite{Do89} proved that
  symmetric periodic solutions of $x'(t) = -\alpha f(x(t-1))$
   bifurcate from the primary branch at some critical parameter.

Clearly, the results above are completely
different from ours.

\begin{proof}[\bf  Proof of Theorem~\ref{th:bif-per3Delay}]

\noindent\textbf{Step 1}(\textsf{Reduction to bifurcations of a class of
Hamiltonian boundary value problems}).
Define $H:\Lambda\times \mathbb{R}\times({\mathbb{R}}^{n})^m\to\mathbb{R}$ by \begin{equation}\label{e:Delay4}
H\left(\lambda, x _ { 1 } ,  x _ { 2 } \right) = V \left(\lambda,  x _ { 1 } \right) + V \left(\lambda,  x _ {2} \right).
\end{equation}
Denote by $\nabla_2 {H}(\lambda,  v)$ and
$\nabla_2^2 {H}(\lambda,  v)$ the gradient and the Hessian of ${H}\left(\lambda, v \right)$ with respect to the second variable $v\in({\mathbb{R}}^{n})^2$, respectively.
Define the
$2n \times 2n$  matrices  $A _ { 2,n }$ and $T_{2,n}$  by
\begin{equation}\label{e:Delay7}
A _ { 2,n} = \left( \begin{array} { c c c  } 0 & I _ { n } \\
- I _ { n } & 0 \end{array} \right)=-J_n
\quad\text{and}\quad
T_{2,n}=\left( \begin{array} { c c c } 0 & I _ { n} \\ - I _ { n } & 0 \end{array} \right)=-J_n,
\end{equation}
 respectively. These matrices satisfy
\begin{align}\label{e:Delay8}
&T^2_{2,n}=-I_{2n},\quad T^{4}_{2,n}=I_{2n}, \quad T_{2,n}A_{2,n}(T_{2,n})^\top=A_{2,n},\nonumber\\
&H(\lambda, T_{2,n}z)=H(\lambda, z)\quad\forall (\lambda,t,z).
\end{align}
The following claim is easy to check
 (see \cite{KaYo74, Liu12}):

\begin{claim}\label{cl:delay1.1}
If the map $x:\mathbb{R}\to\mathbb{R}^n$ is a solution of (\ref{e:Delay6Auto13}),
  then  it is necessarily $4\tau$-periodic, and
\begin{equation}\label{e:Delay5}
\mathbb{R}\ni t\mapsto v(t):=(x_1(t)^\top, x_2(t)^\top)^\top\in ({\mathbb{R}}^{n})^2
\end{equation}
where $x _ { i } ( t ) = x ( t-(i-1)\tau)$ for $i=1,2$, satisfies
\begin{equation}\label{e:Delay6}
\dot{v}(t)=A_{2,n}\nabla_2 H(\lambda, v(t))\quad\text{and}\quad v(t+\tau)=T_{2,n}^{-1}v(t)\;\;\forall t\in \mathbb{R}.
\end{equation}
Conversely, if $v(t)=(x_1(t)^\top, x_2(t)^\top)^\top$, where $x _ { i } ( t )\in\mathbb{R}^n$ for $i=1,2$, satisfies (\ref{e:Delay6}),
then $x(t):=x_1(t)$ satisfies (\ref{e:Delay6Auto13}). 
\end{claim}

Since $A_{2,n}=T_{2,n}=-J_n$,
a direct computation yields the desired congruence relation
\[
\begin{pmatrix} 0 & I_n \\ I_n & 0 \end{pmatrix}^\top A_{2,n}^{-1} \begin{pmatrix} 0 & I_n \\ I_n & 0 \end{pmatrix}
= \begin{pmatrix} 0 & I_n \\ I_n & 0 \end{pmatrix} J_n \begin{pmatrix} 0 & I_n \\ I_n & 0 \end{pmatrix}
= -J_n = J_n^{-1}.
\]
Define
\begin{equation}\label{e:M-invariantDelay3-}
\Upsilon(2,n) = \begin{pmatrix} 0 & I_n \\ I_n & 0 \end{pmatrix}^{-1} = \begin{pmatrix} 0 & I_n \\ I_n & 0 \end{pmatrix},
\end{equation}
and
\begin{equation}\label{e:M-invariantDelay3}
\check{H}(\lambda, z)=H(\lambda, {\Upsilon}(2, n)^{-1}z)
\quad\text{for all
$(\lambda, z)\in \Lambda\times({\R}^{n})^2$.}
\end{equation}
Then
\begin{equation}\label{e:M-invariantDelay3+}
M_{2,n} = \Upsilon(2, n)T_{2,n}^{-1}\Upsilon(2, n)^{-1} = -\Upsilon(2, n)J_n^{-1}\Upsilon(2, n)^\top = -J_n
\end{equation}
is a symplectic orthogonal matrix with $(M_{2,n})^4=I_{2n}$,
and it is easily checked:
\begin{claim}\label{cl:delay1.2}
$v:\mathbb{R}\to\mathbb{R}^{2n}$ satisfies the problem (\ref{e:Delay6})
if and only if
$z(t):=\Upsilon(2,n)v(t)$ solves the following problem
\begin{equation}\label{e:Delay6Auto3}
\dot{z}(t)=J_n\nabla_2\check{H}\left(\lambda, z(t)\right)\quad\text{and}\quad
 z(t+\tau)=-J_nz(t)\;\;\forall t\in\mathbb{R}.
\end{equation}
and in this situation there holds
\begin{equation}\label{e:Delay6Auto3+}
\int^{\tau}_0\left[\frac{1}{2}(J_{n}\dot{z}(t), z(t))_{\mathbb{R}^{2n}}+ \check{H}(\lambda,  z(t))\right]dt
= \int^{\tau}_0\left[\frac{1}{2}(\dot{v}(t), J_nv(t))_{\mathbb{R}^{2n}}+ H(\lambda,
v(t))\right]dt.
\end{equation}
because $A_{2,n}^{-1}=J_n$.
Moreover,  the integral on the left-hand side of the preceding equality equals
$$
 ({x}(0), {x}(\tau))_{\mathbb{R}^n}-
 \int^{\tau}_0(\dot{{x}}(t), {x}(t-\tau))_{\mathbb{R}^n}dt+\int^{2\tau}_0V(\lambda,
 {x}(t))dt
$$
if $v$ is given by (\ref{e:Delay5}) and $x=x_1$
solves (\ref{e:Delay6Auto13}).
\end{claim}
For the proof of the final conclusion, see  \cite[\S3]{Lu14}.

Without confusion occurring,
we always use ${0}$ to denote not only the origin in
$\mathbb{R}^k$ (for different $k$) but also trivial solutions of
(\ref{e:Delay6Auto13}), (\ref{e:Delay6Auto3}) and (\ref{e:Delay6}).
Consider linearizations of the latter
three equations along the trivial solution ${0}$:
\begin{align}\label{e:LinearDelay1.1}
&\dot{x}( t ) = \nabla_2^2 V (\lambda, {0})x(t-\tau)\quad\text{and}\quad
 x( t + 2\tau ) = - x( t )\;\forall t,\\
& \dot{v}(t)=-J_n\nabla^2_2H\left(\lambda,{0}\right)v(t)\quad\text{and}\quad v(t+\tau)=J_nv(t)\;\;\forall t\in\mathbb{R},\label{e:Delay6Auto1}\\
&\dot{z}(t)=J_n\nabla^2_2\check{H}\left(\lambda,{0}\right)z(t)\quad\text{and}\quad
 z(t+\tau)=-J_nz(t)\;\;\forall t\in\mathbb{R}.\label{e:Delay6Auto4}
\end{align}
Here
$\nabla^2_2H\left(\lambda, {0}\right) =
\left( \begin{array} { c c c c }
\nabla^2_xV\left(\lambda, {0}\right) & 0 & \\
0 &\nabla^2_xV\left(\lambda, {0}\right) &
 \end{array} \right)$
 and
\begin{align*}
\nabla^2_2 \check{H}(\lambda,  {0})&=(\Upsilon(2,n)^{-1})^\top\nabla^2_2 {H}(\lambda, {0})\Upsilon(2,n)^{-1}\\
&=\left( \begin{array} {cc}
0 & I_n \\
I_n &0
 \end{array} \right)^T\left( \begin{array} { c c c c }
\nabla^2_xV\left(\lambda, {0}\right) & 0 & \\
0 &\nabla^2_xV\left(\lambda, {0}\right) &
 \end{array} \right)\left( \begin{array} {cc}
0 & I_n \\
I_n &0
 \end{array}\right)=
 \nabla^2_2{H}\left(\lambda,{0}\right).
\end{align*}
  Claims~\ref{cl:delay1.1} and \ref{cl:delay1.2} imply
\begin{align*}
 x:\mathbb{R}\to\mathbb{R}^{n}\;\text{satisfies (\ref{e:LinearDelay1.1})}\;
 &\Longleftrightarrow
  \mathbb{R}\ni t\mapsto v(t):=(x(t)^\top, x(t-\tau)^\top)^\top\in ({\mathbb{R}}^{n})^2
\;\text{satisfies (\ref{e:Delay6Auto1})}\\
&\Longleftrightarrow
  \mathbb{R}\ni t\mapsto z(t):=\Upsilon(2, 2n) v(t)\in ({\mathbb{R}}^{n})^2
\;\text{satisfies (\ref{e:Delay6Auto4})},
\end{align*}
and therefore
\begin{claim}\label{cl:delay1.3}
The problems (\ref{e:LinearDelay1.1}), (\ref{e:Delay6Auto1})
 and (\ref{e:Delay6Auto4}) have the same dimensions of solution spaces.
\end{claim}
 Let $\gamma_\lambda$ and $\Upsilon_\lambda$ be  the fundamental matrix solutions of (\ref{e:Delay6Auto1}) and (\ref{e:Delay6Auto4})
 with $\gamma_\lambda(0)=I_{2n}=\Upsilon_\lambda(0)$, respectively.

\noindent\textbf{Step 2}(\textsf{Proof of (I)}).
If $(\mu,0)$ with $\mu\in\Lambda$ is a bifurcation point of (\ref{e:Delay6Auto13}),
so is that of (\ref{e:Delay6Auto3})
by Claims~\ref{cl:delay1.1} and \ref{cl:delay1.2}.
It follows from \cite[Theorem~1.5(I)]{Lu11} that $\nu_{\tau,-J_n}(\gamma_\mu)\ne 0$.
Hence Theorem~\ref{th:PIndex1} (with $C=0$) implies
 $e_l(\mu)\in\frac{\pi(4\mathbb{Z}+3)}{2\tau}$ for some $l$.

\noindent\textbf{Step 3}(\textsf{Proof of (III)}).
For the present $V$ let us follow the notations above Theorem~\ref{th:bif-per3Delay}, and check that the conditions in \cite[Theorem~1.14]{Lu11} can be satisfied for $\check{H}$.
 \cite[Corollary~6.3.8]{HorJ} gives
 \begin{equation}\label{e:EigenContin}
\sum^n_{l=1}|e_l(\lambda)-e_l(\mu)|^2\le \|\nabla^2_xV\left(\lambda, {0}\right)-\nabla^2_xV\left(\mu, {0}\right)\|^2_2
\quad\forall\lambda,\mu\in\Lambda
\end{equation}
since real symmetric matrixes are not only normal but also Hermite.
By Theorem~\ref{th:PIndex1} (with $C=0$) we obtain
\begin{align}\label{e:DiagDelay10}
 \nu_{\tau,-J_n}(\Upsilon_\lambda)
    &=2\sharp\left\{l\,\Bigg|\, e_l(\lambda)=\frac{(4j+3)\pi}{2\tau}\;\hbox{with some $j\in\mathbb{Z}$}\right\},\\
  i_{\tau,-J_n}(\Upsilon_\lambda)&=2j_1(\lambda)+\cdots+ 2j_n(\lambda) +n+\left[-\frac{n}{2}\right]\label{e:DiagDelay10+}
  \end{align}
   where each $j_k(\lambda)\in\mathbb{Z}$ ($k=1,\dots,n$) is a unique integer from (\ref{e:uniqueInteger}).

Since $M=-J_n$ has only eigenvalues $\pm\sqrt{-1}$,
${\rm Ker}(M-I_{2n})=\{0\}$, that is, the condition (a) in \cite[Theorem~1.14]{Lu11} is satisfied.
Condition (A) and (\ref{e:DiagDelay10}) imply
\begin{equation}\label{e:EigenContinV}
 \nu_{\tau,-J_n}(\Upsilon_\mu)=2k.
\end{equation}

Suppose now that  condition (SB.1) holds.
Since  condition (A) implies
\begin{align*}
&2\tau e_{i_l}(\mu)\in \left((4j_{i_l}(\mu)-1)\pi, (4j_{i_l}(\mu)+3)\pi\right),\quad l=k+1,\dots,n,\\
&2\tau e_{i_l}(\mu)=(4j_{i_l}(\mu)+3)\pi,\quad l=1,\dots, k,
\end{align*}
using \eqref{e:EigenContin}, condition~(b) of Assumption~\ref{ass:BasiAss1Delay1}, and condition~(SB.1),  we can shrink $\epsilon>0$ so that
$$
e_{i_l}(\lambda)\in \left(\frac{(4j_{i_l}(\mu)-1)\pi}{2\tau}, \frac{(4j_{i_l}(\mu)+3)\pi}{2\tau}\right).
$$
for any $(\lambda, l)\in (\mu-\epsilon, \mu+\epsilon)\times\{k+1,\cdots,n\}$,
and that
$$
e_{i_l}(\lambda)\in
\begin{cases}
\left(\frac{(4j_{i_l}(\mu)-1)\pi}{2\tau}, \frac{(4j_{i_l}(\mu)+3)\pi}{2\tau}\right)\quad&\text{if $\lambda\in (\mu-\epsilon, \mu)$},\vspace{1mm}\\
\left(\frac{(4j_{i_l}(\mu)+3)\pi}{2\tau}, \frac{(4j_{i_l}(\mu)+7)\pi}{2\tau}\right)\quad&\text{if $\lambda\in (\mu, \mu+\epsilon)$}
\end{cases}
$$
for $l=1,\dots,k$. Then
\begin{equation*}
\nu_{\tau,-J_n}(\Upsilon_\lambda)=0\quad \forall\lambda\in (\mu-\epsilon, \mu+\epsilon)\setminus\{\mu\}
\end{equation*}
by (\ref{e:DiagDelay10}), and
$$
j_{i_l}(\lambda)=j_{i_l}(\mu)\quad\forall (\lambda, l)\in (\mu-\epsilon, \mu+\epsilon)\times\{k+1,\cdots,n\}
$$
and
\begin{equation*}
j_{i_l}(\lambda)=\begin{cases}
j_{i_l}(\mu)\quad&\text{if $(\lambda,l)\in (\mu-\epsilon, \mu)\times\{1,\cdots,k\}$},\vspace{1mm}\\
j_{i_l}(\mu)+1\quad&\text{if $(\lambda, l)\in (\mu, \mu+\epsilon)\times\{1,\cdots,k\}$}.
\end{cases}
\end{equation*}
By  (\ref{e:DiagDelay10+}), the last two lines imply
\begin{equation*}
    i_{\tau,-J_n}(\Upsilon_\lambda)=
    \begin{cases}
2j_1(\mu)+\cdots+ 2j_n(\mu) +n+[-n/2]\quad&\text{if $\lambda\in (\mu-\epsilon, \mu)$},\\
2j_1(\mu)+\cdots+ 2j_n(\mu)+ 2k +n+[-n/2]\quad&\text{if $\lambda\in (\mu, \mu+\epsilon)$}.
\end{cases}
\end{equation*}

Similarly, in the case of condition (SB.2), we have
\begin{equation*}
    i_{\tau,-J_n}(\Upsilon_\lambda)=
    \begin{cases}
2j_1(\mu)+\cdots+ 2j_n(\mu)+ 2k +n+[-n/2]\quad&\text{if $\lambda\in (\mu-\epsilon, \mu)$},\\
2j_1(\mu)+\cdots+ 2j_n(\mu) +n+[-n/2]\quad&\text{if $\lambda\in (\mu, \mu+\epsilon)$}
\end{cases}
\end{equation*}
and $\nu_{\tau,-J_n}(\Upsilon_\lambda)=0$ if $0<|\lambda-\mu|<\epsilon)$.

Applying the first part of Theorem~1.14 in  \cite{Lu11} to the problem (\ref{e:Delay6Auto3}), we get
the following alternatives:
 \begin{enumerate}
\item[\rm (i')] The problem (\ref{e:Delay6Auto3}) with $\lambda=\mu$ has a sequence of $\mathbb{R}$-distinct solutions,
$z^l\ne 0$ ($l=1,2,\dots$) such that $\|z^l|_I\|_{C^1}\to 0$ for any compact interval $I\subset\mathbb{R}$.
\item[\rm (ii')] There exist left and right  neighborhoods $\Lambda^-$ and $\Lambda^+$ of $\mu$ in $\Lambda$
and nonnegative integers $n^+$ and $n^-\ge 0$ satisfying $n^++n^-\ge \nu_{\tau,-J_n}(\Upsilon_\mu)/2=k$,
such that for $\lambda\in\Lambda^-\setminus\{\mu\}$ (resp. $\lambda\in\Lambda^+\setminus\{\mu\}$),
 problem (\ref{e:Delay6Auto3}) with parameter  $\lambda$  has at least $n^-$ (resp. $n^+$) $\R$-distinct
 solutions, $z_\lambda^i\ne 0$ ($i=1,\dots,n^-$ resp. $n^+$)
 whose restrictions to $[0,\tau]$  converge to zero in $C^1([0,\tau];(\mathbb{R}^{n})^m)$  as $\lambda\to\mu$.
\end{enumerate}

Define $u^l=\Upsilon(2,n)^{-1}z^l$ for $l=1,2,\dots$,
and $u^i_\lambda=\Upsilon(2,n)^{-1}z_\lambda^i\ne 0$ for $i=1,\dots,n^-$ (resp. $n^+$).
Then all $u^l\ne 0$ and satisfy the problem (\ref{e:Delay6}) with $\lambda=\mu$;
and for $\lambda\in\Lambda^-\setminus\{\mu\}$ (resp. $\lambda\in\Lambda^+\setminus\{\mu\}$),
$u_\lambda^i\ne 0$ ($i=1,\dots,n^-$ (resp. $n^+$)) and satisfy the problem (\ref{e:Delay6}) with parameter value $\lambda$.
Let us write
$$
u^l(t):=({x}^l_1(t)^\top, {x}^l_2(t)^\top)^\top\in ({\mathbb{R}}^{n})^2\quad\text{and}\quad
u^i_\lambda(t):=({x}^i_{\lambda,1}(t)^\top, {x}^i_{\lambda,2}(t)^\top)^\top\in ({\mathbb{R}}^{n})^2.
$$
By Claims~\ref{cl:delay1.1} and \ref{cl:delay1.2}, ${x}^l_2(t)={x}^l_1(t-\tau)$, ${x}^i_{\lambda,2}(t)={x}^i_{\lambda,1}(t-\tau)$,
and ${x}^l:={x}^l_1$ and ${x}^i_{\lambda}:= {x}^i_{\lambda,1}$
satisfy the expected alternatives (i) and (ii).

Next, if $k>1$, then $\nu_{\tau,-J_n}(\Upsilon_\mu)=2k>3$
by (\ref{e:EigenContinV}).
Applying the second part of Theorem~1.14 in  \cite{Lu11} to the problem (\ref{e:Delay6Auto3})
yields that at least one of (i')
or one of the following holds:
\begin{enumerate}
\item[\rm (iii')]  For every $\lambda\in\Lambda\setminus\{\mu\}$ sufficiently close to $\mu$, there is a
nonzero solution ${v}_\lambda$  of the problem (\ref{e:Delay6Auto3}) with parameter value $\lambda$, such that $v_\lambda$ converges to zero
in the $C^1$ norm on any compact interval $I\subset\R$  as $\lambda\to\mu$.

\item[\rm (iv')] For a given $\varepsilon>0$  there is a one-sided  neighborhood $\Lambda^0$ of $\mu$ within $\Lambda$ such that, for any $\lambda\in\Lambda^0\setminus\{\mu\}$,  problem (\ref{e:Delay6Auto3})
with parameter  $\lambda$ has either
\begin{itemize}
	\item[$\bullet$]
 infinitely many $\mathbb{R}$-distinct nonzero solutions
${v}_\lambda^l$ satisfying $\|{v}_\lambda^l|_{[0,\tau]}\|_{C^1}<
\varepsilon$ ($l=1,2,\dots$), or
\item[$\bullet$]  at least two $\mathbb{R}$-distinct nonzero solutions $\hat{v}_\lambda^1$ and $\hat{v}_\lambda^2$ satisfying inequalities $\|\hat{v}_\lambda^i|_{[0, \tau]}\|_{C^1}<\varepsilon$ ($i=1,2$) and
$$
\hspace{-5mm}\int^{\tau}_0\left[\frac{1}{2}(J_n\dot{\hat{v}}_\lambda^1(t),\hat{v}^1_\lambda(t))_{\mathbb{R}^{2n}}+ \check{H}(\lambda, \hat{v}_\lambda^1(t))\right]dt
\ne \int^{\tau}_0\left[\frac{1}{2}(J_n\dot{\hat{v}}^2_\lambda(t), \hat{v}_\lambda^2(t))_{\mathbb{R}^{2n}}+ \check{H}(\lambda, \hat{v}_\lambda^2(t))\right]dt.
$$
\end{itemize}
\end{enumerate}
Note that the second equality in the problem (\ref{e:Delay6Auto3}) implies
$\|{v}_\lambda^l|_{[0, 2\tau]}\|_{C^1}=\|{v}_\lambda^l|_{[0, \tau]}\|_{C^1}<\varepsilon$ ($l=1,2,\dots$)
and $\|\hat{v}_\lambda^i|_{[0, 2\tau]}\|_{C^1}=\|\hat{v}_\lambda^i|_{[0, \tau]}\|_{C^1}<\varepsilon$ ($i=1,2$).
As above let us define $u_\lambda=\Upsilon(2,n)^{-1}v_\lambda$,
 ${u}_\lambda^l=\Upsilon(2,n)^{-1}{v}_\lambda^l$ ($l=1,2,\dots$), and
$\hat{u}_\lambda^i=\Upsilon(2,n)^{-1}\hat{v}_\lambda^i$ ($i=1,2$).
They satisfy the problem (\ref{e:Delay6}) with parameter value $\lambda$, and by (\ref{e:Delay6Auto3+}) we have
$$
\int^{\tau}_0\left[\frac{1}{2}(J_n\dot{\hat{v}}_\lambda^i(t),\hat{v}^i_\lambda(t))_{\mathbb{R}^{2n}}+ \check{H}(\lambda, \hat{v}_\lambda^i(t))\right]dt
= \int^{\tau}_0\left[\frac{1}{2}(\dot{\hat{u}}^i_\lambda(t), J_n\hat{u}_\lambda^i(t))_{\mathbb{R}^{2n}}+ {H}(\lambda, \hat{u}_\lambda^i(t))\right]dt
$$
for $i=1,2$, and hence
$$
\int^{\tau}_0\left[\frac{1}{2}(\dot{\hat{u}}^1_\lambda(t), J_n\hat{u}_\lambda^1(t))_{\mathbb{R}^{2n}}+ {H}(\lambda, \hat{u}_\lambda^1(t))\right]dt
\ne \int^{\tau}_0\left[\frac{1}{2}(\dot{\hat{u}}^2_\lambda(t), J_n\hat{u}_\lambda^2(t))_{\mathbb{R}^{2n}}+ {H}(\lambda, \hat{u}_\lambda^2(t))\right]dt.
$$
As above we  write $u_\lambda(t):=({x}_{\lambda,1}(t)^\top, {x}_{\lambda,2}(t)^\top)^\top\in ({\mathbb{R}}^{n})^2$ and
$$
u^l_\lambda(t):=({x}^l_{\lambda,1}(t)^\top, {x}^l_{\lambda,2}(t)^\top)^\top\in ({\mathbb{R}}^{n})^2\quad\text{and}\quad
\hat{u}^i_\lambda(t):=(\hat{x}^i_{\lambda,1}(t)^\top, \hat{x}^i_{\lambda,2}(t)^\top)^\top\in ({\mathbb{R}}^{n})^2.
$$
Then ${x}_{\lambda,2}(t)={x}_{\lambda,1}(t-\tau)$,
${x}^l_{\lambda,2}(t)={x}^l_{\lambda,1}(t-\tau)$, $\hat{x}^i_{\lambda,2}(t)=\hat{x}^i_{\lambda,1}(t-\tau)$,
and
$$
{x}_{\lambda}(t):={x}_{\lambda,1}(t),\quad
{x}^l_\lambda:={x}^l_{\lambda,1}\quad\text{and}\quad \hat{x}^i_{\lambda}:= \hat{x}^i_{\lambda,1}
$$
satisfy the expected alternatives (iii) and (iv) by Claim~\ref{cl:delay1.2}.

\noindent\textbf{Step 4}(\textsf{Proof of (II)}).
Let $\epsilon$ be as in the paragraph below
(\ref{e:EigenContinV}). We may assume
$(\lambda_m^+)$ and $(\lambda_m^-)$ are contained in
 $(\mu-\epsilon, \mu+\epsilon)\setminus\{\mu\}$.
Since $e_{i_l}(\lambda_m^-)< e_{i_l}(\mu)<e_{i_l}(\lambda_m^+)$
for all $m\in\mathbb{N}$ and  $l=1,\dots,k$, it holds that
either $(\lambda_m^-)_m\subset (\mu-\epsilon, \mu)$ and
$(\lambda_m^+)_m\subset (\mu, \mu+\epsilon)$ or
$(\lambda_m^-)_m\subset (\mu, \mu+\epsilon)$ and
$(\lambda_m^+)_m\subset (\mu-\epsilon, \mu)$.
By the arguments in the paragraph below
(\ref{e:EigenContinV}), we have for any $m\in\mathbb{N}$,
\begin{equation*}
\nu_{\tau,-J_n}(\Upsilon_{\lambda_m^+})=\nu_{\tau,-J_n}(\Upsilon_{\lambda_m^-})=0,
\end{equation*}
and
\begin{equation*}
    i_{\tau,-J_n}(\Upsilon_\lambda)=
    \begin{cases}
2j_1(\mu)+\cdots+ 2j_n(\mu) +n+[-n/2]\quad&\text{if $\lambda=\lambda_m^-$},\\
2j_1(\mu)+\cdots+ 2j_n(\mu)+ 2k +n+[-n/2]\quad&\text{if $\lambda=\lambda_m^+$}
\end{cases}
\end{equation*}
in the first case, and
\begin{equation*}
    i_{\tau,-J_n}(\Upsilon_\lambda)=
    \begin{cases}
2j_1(\mu)+\cdots+ 2j_n(\mu)+ 2k +n+[-n/2]\quad&\text{if $\lambda=\lambda_m^+$},\\
2j_1(\mu)+\cdots+ 2j_n(\mu) +n+[-n/2]\quad&\text{if $\lambda=\lambda_m^-$}
\end{cases}
\end{equation*}
in the second case. Consequently, for all $m\in\mathbb{N}$,
$$
[i_{\tau,-J_n}(\Upsilon_{\lambda_m^-}), i_{\tau,-J_n}(\Upsilon_{\lambda_m^-})+
\nu_{\tau,-J_n}(\Upsilon_{\lambda_m^-})]\cap
[i_{\tau,-J_n}(\Upsilon_{\lambda_m^+}), i_{\tau,-J_n}(\Upsilon_{\lambda_m^+})+
\nu_{\tau,-J_n}(\Upsilon_{\lambda_m^+})]=\emptyset.
$$
It follows from \cite[Theorem~1.5(II)]{Lu11} that
$(\mu,0)$ is a bifurcation point of (\ref{e:Delay6Auto3}),
which implies that  $(\mu,0)$  is a bifurcation point of (\ref{e:Delay6Auto13}).
\end{proof}

\begin{proof}[\bf Proof of Corollary~\ref{cor:bif-per3DelayC}]
Define $\Lambda=\mathbb{R}$ and $V:\mathbb{R}\times\mathbb{R}^n\to\mathbb{R}$ by $V(\lambda,x)=\lambda W(x)$.
Clearly, $V$ satisfies Assumption~\ref{ass:BasiAss1Delay1}, and $e_i(\lambda)=\lambda e_i$ ($i=1,\dots,n$).

Let $(\mu,0)$ with $\mu\in\mathbb{R}\setminus\{0\}$ be a bifurcation point of (\ref{e:Delay6Auto13C-}). This very $\mu$ gives rise to a bifurcation point of the system
\begin{equation}\label{e:Delay6Auto13C}
  \dot{x}( t ) = \lambda\nabla W (x ( t - 1) )\quad\text{and}\quad
  x( t + 2) = - x( t )\;\forall t\in\mathbb{R}.
    \end{equation}
Applying  Theorem~\ref{th:bif-per3Delay} with $\tau = 1$
to (\ref{e:Delay6Auto13C}) yields
 $\mu e_l\in\frac{\pi(4\mathbb{Z}+3)}{2\tau}$ for some $l$,
which establishes the first conclusion.

It remains to prove the remaining conclusions.
For some $\mu\in\mathbb{R}$, assume that condition (G) holds.
It easily follows that $V$ satisfies conditions (A)
and either (SB.1) or (SB.2) in Theorem~\ref{th:bif-per3Delay} with $\tau = 1$.
 Then,  for the problem (\ref{e:Delay6Auto13C})
   there exist two possible alternatives:
 \begin{enumerate}
\item[\rm (i')] The problem (\ref{e:Delay6Auto13C})  with $\lambda=\mu$ admits a sequence of $\mathbb{R}$-distinct solutions $u^l\ne 0$ ($l=1,2,\dots$) such that, for any compact interval $I\subset\mathbb{R}$,
     $\|u^l|_I\|_{C^1}\to 0$ as $l\to\infty$.
\item[\rm (ii')] There exist left and right  neighborhoods $\Lambda^-$ and $\Lambda^+$ of $\mu$ in $\Lambda$ not containing the point
$0\in\mathbb{R}$ (since $\mu\ne 0$)
and nonnegative integers $n^+$ and $n^-\ge 0$ satisfying $n^++n^-\ge k$, such that
for $\lambda\in\Lambda^-\setminus\{\mu\}$ (resp. $\lambda\in\Lambda^+\setminus\{\mu\}$),
(\ref{e:Delay6Auto13C}) with parameter value $\lambda$  has at least $n^-$ (resp. $n^+$) $\R$-distinct
 solutions $u_\lambda^i\ne 0$ ($i=1,\dots,n^-$ (resp. $n^+$))
which converge to zero in $C^1_{\rm loc}(\mathbb{R}, \mathbb{R}^n)$  as $\lambda\to\mu$.
\end{enumerate}
Moreover, if $n>1$ and $k>1$, then at least one of (i'), (iii'), and (iv') holds,
where (i') is as stated previously, and
\begin{enumerate}
\item[\rm (iii')]  For every $\lambda\in\Lambda\setminus\{\mu\}$ near $\mu$, there is a
 nonzero solution ${u}_\lambda$  of (\ref{e:Delay6Auto13C})
 with parameter value $\lambda$, such that $u_\lambda$ converges to zero
 in $C^1_{\rm loc}(\mathbb{R}, \mathbb{R}^n)$  as $\lambda\to\mu$.

\item[\rm (iv')] For a given $\varepsilon>0$,  there is a one-sided  neighborhood $\Lambda^0$ of $\mu$ in $\Lambda$ not containing the point
$0\in\mathbb{R}$, such that
for any $\lambda\in\Lambda^0\setminus\{\mu\}$, problem (\ref{e:Delay6Auto13C}) with parameter  $\lambda$
has either
\begin{itemize}
 \item[$\bullet$]  infinitely many $\mathbb{R}$-distinct nonzero solutions ${u}_\lambda^l$
satisfying $\|{u}_\lambda^l|_{[0,2]}\|_{C^1}<\varepsilon$, $l=1,2,\dots$,  or
\item[$\bullet$]
 at least two $\mathbb{R}$-distinct nonzero solutions $\hat{u}_\lambda^1$ and $\hat{u}_\lambda^2$
satisfying $\|\hat{u}_\lambda^i|_{[0, 2]}\|_{C^1}<\varepsilon$, $i=1,2$,
and such that the following  inequality holds:
\begin{align*}
&(\hat{u}^1_\lambda(0), \hat{u}^1_\lambda(1))_{\mathbb{R}^n}-
 \int^{1}_0(\dot{\hat{u}}^1_\lambda(t), \hat{u}^1_\lambda(t-1))_{\mathbb{R}^n}dt+\lambda\int^{2}_0 W(\hat{u}^1_\lambda(t))dt\nonumber\\
&\ne (\hat{u}^2_\lambda(0), \hat{u}^2_\lambda(1))_{\mathbb{R}^n}-
\int^{1}_0(\dot{\hat{u}}^2_\lambda(t), \hat{u}^2_\lambda(t-1))_{\mathbb{R}^n}dt+ \lambda\int^{2}_0W(\hat{u}^2_\lambda(t))dt.
\end{align*}
\end{itemize}
\end{enumerate}
For each function $u^l$ in (i'), the function $x^l$ defined by
$x^l(t):=u^l(t/\mu)$ satisfies (\ref{e:Delay6Auto13C-})  with $\lambda=\mu$.
Moreover, $x^l\to 0$ in $C^1_{\rm loc}(\mathbb{R}, \mathbb{R}^n)$
as $l\to\infty$.

Similarly, the sequences $u_\lambda^i$ in (ii'), for $i=1,\dots,n^-$
and $i=1,\dots,n^+$ respectively, yield $x_\lambda^i(t):=u_\lambda^i(t/\lambda)$.
 These $x^i_\lambda$ satisfy (\ref{e:Delay6Auto13C-}) with parameter value $\lambda$,
  and $x^i_\lambda\to 0$ in $C^1_{\rm loc}(\mathbb{R}, \mathbb{R}^n)$
as $l\to\infty$.

Each $u_\lambda$ in (iii') gives rise to  $x_\lambda(t):=u_\lambda(t/\lambda)$
satisfying (\ref{e:Delay6Auto13C-}) with parameter value $\lambda$,
and $x_\lambda\to 0$  in $C^1_{\rm loc}(\mathbb{R}, \mathbb{R}^n)$  as $\lambda\to\mu$.

For ${u}_\lambda^l$, $\hat{u}_\lambda^1$ and $\hat{u}_\lambda^2$ in (iv'), we define
$$
{x}_\lambda^l(t):={u}_\lambda^l(t/\lambda),\quad
\hat{x}_\lambda^1(t):=\hat{u}_\lambda^1(t/\lambda),\quad
\hat{x}_\lambda^2(t):=\hat{u}_\lambda^2(t/\lambda).
$$
They satisfy (\ref{e:Delay6Auto13C-}) with parameter value $\lambda$.
A direct calculation leads to the following identity:
$$
\|{u}_\lambda^l\|_{C^1([0, 2])}=
\begin{cases}
 \|{x}_\lambda^l\|_{C^0([0, 2\lambda])}+\lambda
\|\dot{x}_\lambda^l\|_{C^0([0, 2\lambda])}\;&\text{if}\;\lambda>0,\vspace{1mm}\\
\|{x}_\lambda^l\|_{C^0([2\lambda,0])}+|\lambda|
\|\dot{x}_\lambda^l\|_{C^0([2\lambda,0])}\;&\text{if}\;\lambda<0.
   \end{cases}
$$
When $\lambda<0$, since ${u}_\lambda^l(t+2)=-{u}_\lambda^l(t)$
 or equivalently (by a shift of the time variable),  ${u}_\lambda^l(t-2)=-{u}_\lambda^l(t)$,
we have ${x}_\lambda^l(t-2\lambda)=-{x}_\lambda^l(t)$
(hence $\dot{x}_\lambda^l(t-2\lambda)=-\dot{x}_\lambda^l(t)$).
From these relations we deduce
$$
\|{x}_\lambda^l\|_{C^0([2\lambda,0])}+|\lambda|
\|\dot{x}_\lambda^l\|_{C^0([2\lambda,0])}=\|{x}_\lambda^l\|_{C^0([0, -2\lambda])}+|\lambda|
\|\dot{x}_\lambda^l\|_{C^0([0, -2\lambda])}.
$$
Hence there always holds
$$
\|{u}_\lambda^l\|_{C^1([0, 2])}=\|{x}_\lambda^l\|_{C^0([0, 2|\lambda|])}+|\lambda|
\|\dot{x}_\lambda^l\|_{C^0([0, 2|\lambda|])}.
$$
The same holds for $\hat{u}^i_\lambda$ and $\hat{x}^i_\lambda$ ($i=1,2$).
Moreover, it is easily computed that for $i=1,2$,
\begin{align*}
&(\hat{u}^i_\lambda(0), \hat{u}^i_\lambda(1))_{\mathbb{R}^n} - \int^{1}_0(\dot{\hat{u}}^i_\lambda(t), \hat{u}^i_\lambda(t-1))_{\mathbb{R}^n}dt + \lambda\int^{2}_0 W(\hat{u}^i_\lambda(t))dt \\
&= (\hat{x}^i_\lambda(0), \hat{x}^i_\lambda(\lambda))_{\mathbb{R}^n} - \int^{\lambda}_0(\dot{\hat{x}}^i_\lambda(t), \hat{x}^i_\lambda(t-\lambda))_{\mathbb{R}^n}dt + \int^{2\lambda}_0 W(\hat{x}^i_\lambda(t))dt.
\end{align*}
From these and (iv') we immediately arrive at (iv).
\end{proof}


 When the function $V$ in Assumption~\ref{ass:BasiAss1Delay1} is affine in $\lambda$, the conditions on the eigenvalues $e_k(\lambda)$ in Theorem~\ref{th:bif-per3Delay} are easily proved.  In this case, Corollary~1.15 of \cite{Lu11} yields the following result.

\begin{theorem}\label{th:bif-per2Delay++}
Let $V_0, V_1:  \mathbb{R}^n \to \mathbb{R}$ be even $C^2$ functions,
with the Hessian $\nabla^2{V}_1({0})$ being either positive or negative definite.
Denoted the zero solution of the problem
\begin{eqnarray}\label{e:LinearDelay2+}
 \left\{\begin{array}{ll}
 &\dot{x}( t ) = \nabla V_0 (x(t-\tau))+
 \lambda\nabla V_1 (x ( t - \tau )),\\
 & x( t + 2 \tau ) = - x( t )\;\forall t
   \end{array}\right.
  \end{eqnarray}
  by ${0}^\lambda$, and
 the dimension of the solution space of the linear delay problem
 \begin{eqnarray}\label{e:LinearDelay2}
 \left\{\begin{array}{ll}
 &\dot{x}( t ) = \nabla^2 V_0 ({0})x(t-\tau) +
 \lambda\nabla^2 V_1 ({0})x ( t - \tau ),\\
 & x( t + 2\tau ) = - x( t )\;\forall t
   \end{array}\right.
  \end{eqnarray}
by $\nu_{\tau}({0}^\lambda)$. 
Then
\begin{enumerate}
\item[\rm (I)] $\Sigma_\tau:=\{\lambda\in\mathbb{R}\,|\, \nu_{\tau}({0}^\lambda)>0\}$
is a discrete set in $\mathbb{R}$.

\item[\rm (II)]
 For each $\mu\in{\Sigma}_\tau$,
	 at least  one of the following alternatives holds true:
	\begin{enumerate}
		\item[\rm (i)] System (\ref{e:LinearDelay2+}) with $\lambda=\mu$ admits a sequence of $\R$-distinct nonzero solutions
		${x}^{k}$ for $k=1,2,\cdots$,
		which  converges to zero in $C^1_{\rm loc}(\mathbb{R},\R^{n})$ as $k\to\infty$.
		\item[\rm (ii)]
		There exist a left neighborhood $\Lambda^-$ and a right neighborhood $\Lambda^+$ of $\mu$ in $\mathbb{R}$,
		and nonnegative integers $n^+$ and $n^-$ with $n^++n^-\ge \nu_{\tau}(0^\mu)/2$,
		such that for each $\lambda\in\Lambda^-\setminus\{\mu\}$ (resp. $\lambda\in\Lambda^+\setminus\{\mu\}$),
		system (\ref{e:LinearDelay2+}) with parameter $\lambda$ possesses at least $n^-$ (resp. $n^+$) $\mathbb{R}$-distinct nonzero solutions $x^i_{\lambda}$, $i=1,\dots,n^-$ (resp. $i=1,\dots,n^+$),
		which satisfy $x^i_{\lambda}\to 0$ in $C^1_{\rm loc}(\mathbb{R},\mathbb{R}^{n})$ as $\lambda\to\mu$.
			\end{enumerate}
	Moreover, if $\nu_{\tau}(0^\mu)\ge 3$,
then at least one of  (i), (iii), and (iv) holds, with (i) as stated previously, and with (iii) and (iv) given by:
	\begin{enumerate}
		\item[\rm (iii)]
		For every $\lambda\in\Lambda\setminus\{\mu\}$ sufficiently close to $\mu$, there exists a nonzero solution $\bar{x}^\lambda$ of system (\ref{e:LinearDelay2+}) with parameter $\lambda$, such that $\bar{x}^\lambda\to 0$ in $C^1_{\rm loc}(\mathbb{R},\mathbb{R}^{n})$ as $\lambda\to\mu$.
		
		\item[\rm (iv)] For any given $\varepsilon>0$, there exists a one-sided neighborhood $\Lambda^0$ of $\mu$ in $\Lambda$ such that for each $\lambda\in\Lambda^0\setminus\{\mu\}$, system (\ref{e:LinearDelay2+}) with parameter $\lambda$ has either:
		\begin{itemize}
			\item infinitely many $\mathbb{R}$-distinct nonzero solutions $\bar{x}^k_{\lambda}$ ($k=1,2,\dots$) satisfying \\ $\|\bar{x}^k_{\lambda}|_{[0,\tau]}\|_{C^1}<\varepsilon$, or
			\item at least two $\mathbb{R}$-distinct nonzero solutions $\hat{x}^1_{\lambda},\hat{x}^2_{\lambda}$ with $\|\hat{x}^i_{\lambda}|_{[0,\tau]}\|_{C^1}<\varepsilon$ ($i=1,2$) and
\begin{align*}
        & \bigl(\hat{x}_\lambda^1(0),  \hat{x}_\lambda^1(\tau) \bigr)_{\mathbb{R}^{n}}
          - \int_0^{\tau} \bigl( \dot{\hat{x}}_\lambda^1(t),  \hat{x}_\lambda^1(t-\tau) \bigr)_{\mathbb{R}^{n}} \, dt
          + \int_{0}^{2\tau} V\bigl(\lambda, \hat{x}_\lambda^1(t)\bigr) \, dt \\
        \ne \; &
        \bigl(\hat{x}_\lambda^2(0),  \hat{x}_\lambda^2(\tau) \bigr)_{\mathbb{R}^{n}}
          - \int_0^{\tau} \bigl( \dot{\hat{x}}_\lambda^2(t),  \hat{x}_\lambda^2(t-\tau) \bigr)_{\mathbb{R}^{n}} \, dt
          + \int_{0}^{2\tau} V\bigl(\lambda, \hat{x}_\lambda^2(t)\bigr) \, dt,
    \end{align*}
where $V(\lambda,  x) = V_0(x) + \lambda V_1(x)$.
		\end{itemize}
		\end{enumerate}
\end{enumerate}
\end{theorem}

\begin{proof}[\bf Proof]
 Define  $H_0, {H}_1: ({\mathbb{R}}^{n})^2\to\mathbb{R}$  by
$H_0 \left(x _ { 1 } ,  x _ { 2 } \right) =
 \sum^2_{i=1}V_0 \left(x _ {i } \right)$
 and ${H}_1 \left(x _ { 1 } , x _ {2} \right) = \sum^2_{i=1}{V}_1\left(x _ {i } \right)$.
Let $\Upsilon(2,n)$ be given by (\ref{e:M-invariantDelay3-}). Define
\begin{eqnarray*}
	\check{H}_i(z)=H_i({\Upsilon}(2, n)^{-1}z)\quad\forall
	z\in ({\R}^{n})^2,\;i=0,1,
\end{eqnarray*}
and ${H}, \check{H}: \mathbb{R}\times({\R}^{n})^2\to\R$  by
${H}(\lambda, z)={H}_0(z)+\lambda {H}_1(z)$
and $\check{H}(\lambda, z)=\check{H}_0(z)+\lambda \check{H}_1(z)$.
By the arguments below (\ref{e:Delay6Auto4}),
$$\nabla^2H_i\left({0}\right) =
\left( \begin{array} { c c c c }
\nabla^2V_i\left({0}\right) & 0 & \\
0 &\nabla^2V_i\left({0}\right) &
 \end{array} \right)=\nabla^2\check{H}_i({0}),\quad i=0,1.
$$
Thus, $\nabla^2\check{H}_1(0)$ is (positive or negative) definite
if and only if
the Hessian $\nabla^2{V}_1({0})$ is (positive or negative) definite.
Since $V_i$ are even and $\check{H}_i(x_1,x_2)=V_i(x_1)+V_i(x_2)$
for $i=0,1$, it is easily checked
that $\check{H}(\lambda, -J_nz)=\check{H}(\lambda, z)$
for all $(\lambda, z)\in \Lambda\times({\mathbb{R}}^{n})^2$, i.e.,
each $\check{H}(\lambda, \cdot)$
is invariant for the action of $-J_n$
on $({\mathbb{R}}^{n})^2\equiv\mathbb{R}^{2n}$.
Noting that ${\rm Ker}(-J_n-I_{2n})=\{0\}$, problem
(\ref{e:Delay6Auto4})  has no nonzero constant solutions belonging
to ${\rm Ker}(-J_n-I_{2n})$.
Let $\Upsilon_\lambda$ is  the fundamental matrix solution (with $\Upsilon_\lambda(0)=I_{2n}$) to (\ref{e:Delay6Auto4})
for the specific Hamiltonian
$\check{H}(\lambda, z)=\check{H}_0(z)+\lambda \check{H}_1(z)$.
Applying Corollary~1.15 in \cite{Lu11} to (\ref{e:Delay6Auto4}),
where it is always understood that
$\check{H}(\lambda, z)=\check{H}_0(z)+\lambda \check{H}_1(z)$,
 we obtain the following result.
  	\begin{enumerate}
		\item[\rm (I*)]
	For any $\tau>0$, the set
	$\check{\Sigma}_\tau:=\{\lambda\in\mathbb{R} \mid \nu_{\tau,-J_n}(\Upsilon_\lambda)>0\}$
		is discrete in $\mathbb{R}$.

\item[\rm (II*)] For each $\mu\in\check{\Sigma}_\tau$,
	 at least  one of the following alternatives holds true:
	\begin{enumerate}
		\item[\rm (i*)] System (\ref{e:Delay6Auto4}) with $\lambda=\mu$ admits a sequence of $\R$-distinct nonzero solutions
		${z}^{\mu,k}$ for $k=1,2,\cdots$,
		which  converges to zero in $C^1_{\rm loc}(\mathbb{R},\R^{2n})$ as $k\to\infty$.
		\item[\rm (ii*)]
		There exist a left neighborhood $\Lambda^-$ and a right neighborhood $\Lambda^+$ of $\mu$ in $\mathbb{R}$,
		and nonnegative integers $n^+$ and $n^-$ with $n^++n^-\ge \nu_{\tau,-J_n}(\Upsilon_\mu)/2$,
		such that for each $\lambda\in\Lambda^-\setminus\{\mu\}$ (resp. $\lambda\in\Lambda^+\setminus\{\mu\}$),
		system (\ref{e:Delay6Auto4}) with parameter $\lambda$ possesses at least $n^-$ (resp. $n^+$) $\mathbb{R}$-distinct nonzero solutions $z^{\lambda,i}$, $i=1,\dots,n^-$ (resp. $i=1,\dots,n^+$),
		which satisfy $z^{\lambda,i}\to 0$ in $C^1_{\rm loc}(\mathbb{R},\mathbb{R}^{2n})$ as $\lambda\to\mu$.
			\end{enumerate}
	Moreover, if $\nu_{\tau,-J_n}(\Upsilon_\mu)\ge 3$, then at least one of (i*), (iii*)
	and (iv*) holds, where
	\begin{enumerate}
		\item[\rm (iii*)]
		For every $\lambda\in\Lambda\setminus\{\mu\}$ sufficiently close to $\mu$, there exists a nonzero solution $\bar{z}^\lambda$ of system (\ref{e:Delay6Auto4}) with parameter $\lambda$, such that $\bar{z}^\lambda\to 0$ in $C^1_{\rm loc}(\mathbb{R},\mathbb{R}^{2n})$ as $\lambda\to\mu$.
		
		\item[\rm (iv*)] For any given $\varepsilon>0$, there exists a one-sided neighborhood $\Lambda^0$ of $\mu$ in $\Lambda$ such that for each $\lambda\in\Lambda^0\setminus\{\mu\}$, system (\ref{e:Delay6Auto4}) with parameter $\lambda$ has either:
		\begin{itemize}
			\item infinitely many $\mathbb{R}$-distinct nonzero solutions $\bar{z}^{\lambda,k}$ ($k=1,2,\dots$) satisfying \\ $\|\bar{z}^{\lambda,k}|_{[0,\tau]}\|_{C^1}<\varepsilon$, or
			\item at least two $\mathbb{R}$-distinct nonzero solutions $\hat{z}^{\lambda,1},\hat{z}^{\lambda,2}$ with $\|\hat{z}^{\lambda,i}|_{[0,\tau]}\|_{C^1}<\varepsilon$ ($i=1,2$) and
$\mathcal{A}_\lambda({\hat{z}}_\lambda^1)\ne
\mathcal{A}_\lambda({\hat{z}}_\lambda^2)$, where
\begin{eqnarray*}
\mathcal{A}_\lambda({\hat{z}}_\lambda^i):=
\int^{\tau}_0\left[\frac{1}{2}(J_n\dot{\hat{z}}_\lambda^i(t),
\hat{z}^i_\lambda(t))_{\mathbb{R}^{2n}}+
 H(\lambda, \hat{z}_\lambda^i(t))\right]dt,\quad i=1,2.
\end{eqnarray*}
		\end{itemize}
		\end{enumerate}
	\end{enumerate}

By Claim~\ref{cl:delay1.3}, $\nu_{\tau,-J_n}(\Upsilon_\lambda)$ equals
the dimension $\nu_{\tau}({0}^\lambda)$
of the solution space for the linear delay problem
(\ref{e:LinearDelay2}), and hence
$\check{\Sigma}_\tau=\Sigma_\tau$.
Then the conclusion (I) follows from  (I*).

As in the proof of Theorem~\ref{th:bif-per3Delay},
for $z^{\mu,k}$ in (i*) and $z^{\lambda,i}$ in (ii*), we
define $v^{\mu,k}=\Upsilon(2,n)^{-1}z^{k,\mu}$ for $k=1,2,\dots$,
and $v^{\lambda,i}=\Upsilon(2,n)^{-1}z^{\lambda,i}\ne 0$ for $i=1,\dots,n^-$ (resp. $n^+$),
and write
$$
v^{\mu,k}(t):=({x}^{\mu,k}_1(t)^\top, {x}^{\mu,k}_2(t)^\top)^\top\in ({\mathbb{R}}^{n})^2\quad\text{and}\quad
v^{\lambda,i}(t):=({x}^{\lambda,i}_1(t)^\top, {x}^{\lambda,i}_2(t)^\top)^\top\in ({\mathbb{R}}^{n})^2.
$$
It follows from  Claims~\ref{cl:delay1.1} and \ref{cl:delay1.2}  that ${x}^{\mu,k}_2(t)={x}^{\mu,k}_1(t-\tau)$ and ${x}^{\lambda,i}_{2}(t)={x}^{\lambda,i}_{1}(t-\tau)$,
and that the functions ${x}^k:={x}^{\mu,k}_1$ and ${x}^i_{\lambda}:= {x}^{\lambda,i}_{1}$
fulfill the alternatives (i) and (ii).

Assume $\nu_{\tau,-J_n}(\Upsilon_\mu)=\nu_{\tau}({0}^\lambda)\ge 3$.
For $\bar{z}^{\lambda}$ in (iii*), and
$\bar{z}^{\lambda,k}$ and $\hat{z}_\lambda^i$ (with $i=1,2$) in (iv*),
 respectively, as above we  write
 $\Upsilon(2,n)^{-1}\bar{z}^\lambda(t):=(\bar{x}^{\lambda}_1(t)^\top, \bar{x}^{\lambda}_2(t)^\top)^\top\in ({\mathbb{R}}^{n})^2$ and
\begin{align*}
&\Upsilon(2,n)^{-1}\bar{z}^{\lambda,k}(t):=(\bar{x}^k_{\lambda,1}(t)^\top, \bar{x}^k_{\lambda,2}(t)^\top)^\top\in ({\mathbb{R}}^{n})^2,\\
&\Upsilon(2,n)^{-1}\hat{z}^i_\lambda(t):=(\hat{x}^i_{\lambda,1}(t)^\top, \hat{x}^i_{\lambda,2}(t)^\top)^\top\in ({\mathbb{R}}^{n})^2.
\end{align*}
Then $\bar{x}^{\lambda}_2(t)=\bar{x}^{\lambda}_1(t-\tau)$,
$\bar{x}^k_{\lambda,2}(t)=\bar{x}^k_{\lambda,1}(t-\tau)$, $\hat{x}^i_{\lambda,2}(t)=\hat{x}^i_{\lambda,1}(t-\tau)$,
and by Claim~\ref{cl:delay1.2} functions
$$
\bar{x}^{\lambda}(t):=\bar{x}^{\lambda}_1(t),\quad
\bar{x}^k_\lambda:=\bar{x}^k_{\lambda,1}\quad\text{and}\quad \hat{x}^i_{\lambda}:= \hat{x}^i_{\lambda,1}
$$
satisfy the required alternatives (iii) and (iv)
as in the proof of Theorem~\ref{th:bif-per3Delay}.
\end{proof}



\section{Bifurcation near the trivial equilibrium of system (\ref{e:2-Delay6Auto13})}\label{sec:delay2}

\begin{theorem}\label{th:bif-per3DelayII}
Under Assumption~\ref{ass:BasiAss1Delay1},
let  $e_1(\lambda)\le e_2(\lambda)\le\cdots\le e_n(\lambda)$ be all eigenvalues of  $\nabla^2_2V\left(\lambda, 0\right)$.
Fix a delay parameter $\tau>0$.
\begin{enumerate}
\item[\rm (I)]{\rm (\textsf{Necessary condition}):}
If $(\mu,0)$ with $\mu\in\Lambda$ is a bifurcation point of
(\ref{e:2-Delay6Auto13}), then
$$
N(\mu):=\left\{1\le j\le n\,\Big|\,  \sqrt{3}e_j(\mu)\tau\in \frac{2\pi}{3}+ \pi(2\mathbb{Z}+1)\right\}\ne\emptyset.
$$

\item[\rm (II)]{\rm (\textsf{Sufficient condition}):}
Let $\mu$ be an interior point of $\Lambda$. Then
$(\mu,0)$ is a bifurcation point of
(\ref{e:2-Delay6Auto13}), provided that the following conditions hold:
   \begin{enumerate}
  \item[\rm (A)]  If $e_j(\mu)>0$, then $\sqrt{3}e_j(\mu)\tau\notin \frac{2\pi}{3}+ \pi(2\mathbb{Z}+1)$.
 \item[\rm (B)] $N(\mu)\ne\emptyset$ and
  there exist two sequences $(\lambda_m^\pm)_m\subset\Lambda$  converging to $\mu$ such that
$e_{j}(\lambda_m^-)< e_{j}(\mu)<e_{j}(\lambda_m^+)$
for all $m\in\mathbb{N}$ and  $j\in N(\mu)$.
\end{enumerate}

\item[\rm (III)]{\rm (\textsf{Alternative bifurcations of Fadell-Rabinowitz type }):}
    Let $\mu$ be an interior point of $\Lambda$ satisfying condition (A) in (II) and the following condition:
   \begin{enumerate}
   \item[\rm (SB)] $N(\mu)\ne\emptyset$ and  one of the following two conditions holds for some small $\epsilon>0$:
\begin{itemize}
\item[\rm (SB.1)] $e_{j}(\lambda)< e_{}(\mu)<e_{j}(\lambda')$
for all $(\lambda, \lambda')\in (\mu-\epsilon, \mu)\times (\mu, \mu+\epsilon)$ and $j\in N(\mu)$.

\item[\rm (SB.2)]
$e_{j}(\lambda)> e_{j}(\mu)>e_{j}(\lambda')$
for all $(\lambda, \lambda')\in (\mu-\epsilon, \mu)\times (\mu, \mu+\epsilon)$ and $j\in N(\mu)$.
  \end{itemize}
\end{enumerate}
 Then for the problem (\ref{e:2-Delay6Auto13}),
     there exist two possible alternatives:
\begin{itemize}
\item[\rm (i)] The problem (\ref{e:2-Delay6Auto13}) with $\lambda=\mu$ admits a sequence of nonzero distinct solutions
$\{x^l\}^{\infty}_{l=1}$  such that $x^l\to 0$ in $C^1_{\rm loc}(\mathbb{R},\mathbb{R}^n)$
as $l\to\infty$.
\item[\rm (ii)] There exist left and right  neighborhoods $\Lambda^-$ and $\Lambda^+$ of $\mu$ in $\Lambda$
and nonnegative integers $n^+$ and $n^-\ge 0$ satisfying $n^++n^-\ge 2\,\sharp N(\mu)$, such that for $\lambda\in\Lambda^-\setminus\{\mu\}$ (resp. $\lambda\in\Lambda^+\setminus\{\mu\}$),
(\ref{e:2-Delay6Auto13}) with parameter  $\lambda$  has at least $n^-$ (resp. $n^+$) distinct pairs of nontrivial solutions
$\{x_\lambda^i, -x_\lambda^i\}$ ($i=1,\dots,n^-$ (resp. $n^+$))
which converge to zero in $C^1_{\rm loc}(\mathbb{R}, \mathbb{R}^n)$
as $\lambda\to\mu$.
\end{itemize}
\end{enumerate}
\end{theorem}

\begin{corollary}\label{cor:bif-per4DelayIIB}
Let $\Lambda \subset \mathbb{R}$ be an interval, and $V \in C(\Lambda \times \mathbb{R}, \mathbb{R})$ satisfy
the following conditions:
  \begin{enumerate}
\item[\rm (C)] For each $\lambda \in \Lambda$,  $V(\lambda, \cdot): \mathbb{R}\to \mathbb{R}$ is even and $C^2$,
which implies $\frac{\partial V}{\partial x}(\lambda, 0) = 0$.
\item[\rm (D)] $\frac{\partial V}{\partial x}(\lambda, x)$
and $\frac{\partial^2 V}{\partial x^2}(\lambda, x)$
are continuous in $(\lambda, x)$.
    \end{enumerate}
Fix a delay parameter $\tau>0$. Consider the problem
\begin{eqnarray}\label{e:2-Delay6Auto17*}
  \dot{x}( t ) = \frac{\partial V}{\partial x}(\lambda, x ( t - \tau) )+
  \frac{\partial V}{\partial x}(\lambda, x ( t - 2\tau) )\quad\hbox{and}\quad
  x( t +3\tau) = - x( t )\;\forall t\in\mathbb{R}.
   \end{eqnarray}
   If $(\mu,0)$ with $\mu\in\Lambda$ is a bifurcation point of
(\ref{e:2-Delay6Auto17*}), then
$$
\sqrt{3}\frac{\partial^2 V}{\partial x^2}(\mu, 0)\tau= \frac{2\pi}{3}+ (2j+1)\pi
\quad\text{for some $j\in\mathbb{Z}$}.
$$
Conversely, let $\mu$ be an interior point  of $\Lambda$. Then
$(\mu,0)$ is a bifurcation point of (\ref{e:2-Delay6Auto17*}) provided that the following conditions hold:
\begin{enumerate}
\item[\rm (E)] $\frac{\partial^2 V}{\partial x^2}(\mu, 0)<0$ and $\sqrt{3}\frac{\partial^2 V}{\partial x^2}(\mu, 0)\tau=
 \frac{2\pi}{3}+ (2j+1)\pi$  for some $j\in\mathbb{Z}$.
\item[\rm (F)] There exist two sequences $(\lambda_m^+)_m$ and $(\lambda_m^-)_m$ in
$\Lambda$  with $\lim_{m\to\infty}\lambda_m^\pm=\mu$ such that
$$
\frac{\partial^2 V}{\partial x^2}(\lambda_m^-, 0)<
 \frac{\partial^2 V}{\partial x^2}(\mu, 0)<\frac{\partial^2 V}{\partial x^2}(\lambda_m^+, 0)\quad\text{for all $m\in\mathbb{N}$}.
 $$
\end{enumerate}
Furthermore, suppose that
condition (F) is replaced by the stronger condition that
in a deleted neighborhood of $\mu$,
 $\frac{\partial^2 V}{\partial x^2}(\lambda, 0)-
 \frac{\partial^2 V}{\partial x^2}(\mu, 0)\ne 0$ and changes sign
 as $\lambda$ across $\mu$.
Then, for the problem (\ref{e:2-Delay6Auto17*}),
one of the following alternatives holds:
\begin{itemize}
	\item[\rm (i)] The problem (\ref{e:2-Delay6Auto17*}) with $\lambda=\mu$ admits a sequence of nonzero distinct solutions
	$\{x^l\}^{\infty}_{l=1}$  such that $x^l\to 0$ in $C^1_{\rm loc}(\mathbb{R},\mathbb{R})$
	as $l\to\infty$.
	\item[\rm (ii)] There exist left and right  neighborhoods $\Lambda^-$ and $\Lambda^+$ of $\mu$ in $\Lambda$
	and nonnegative integers $n^+$ and $n^-\ge 0$ satisfying $n^++n^-\ge 1$, such that for $\lambda\in\Lambda^-\setminus\{\mu\}$ (resp. $\lambda\in\Lambda^+\setminus\{\mu\}$),
	(\ref{e:2-Delay6Auto17*}) with parameter  $\lambda$  has at least $n^-$ (resp. $n^+$) distinct pairs of nontrivial solutions
	$\{x_\lambda^i, -x_\lambda^i\}$ ($i=1,\dots,n^-$ (resp. $n^+$))
	which converge to zero in $C^1_{\rm loc}(\mathbb{R}, \mathbb{R})$
	as $\lambda\to\mu$.
\end{itemize}
\end{corollary}

Similarly, Theorem~\ref{th:bif-per3DelayII}
implies:

\begin{corollary}\label{cor:bif-per3DelayIIC}
	Let $W:{\mathbb{R}}^n\to\mathbb{R}$ be even and $C^2$, and let
	$e_1\le e_2\le\cdots\le e_n$ be all eigenvalues of  $\nabla^2W\left({\bf 0}\right)$.
	Consider  the problem
	\begin{eqnarray}\label{e:2-Delay6Auto13C}
		\dot{x}( t ) = \nabla W ( x ( t - \lambda) )+\nabla W (x ( t - 2\lambda) )\quad\hbox{and}\quad
		x( t +3\lambda) = - x( t )\;\forall t\in\mathbb{R}.
	\end{eqnarray}
If $(\mu,0)$ with $\mu\in\mathbb{R}$ is a bifurcation point of
(\ref{e:2-Delay6Auto13C}),
then
$$
\hat{N}(\mu):=\left\{1\le j\le n\,\Big|\, \sqrt{3}\mu e_j\in \frac{2\pi}{3}+ \pi(2\mathbb{Z}+1)\right\}\ne\emptyset.
$$
Conversely, for some $\mu\in\mathbb{R}$,  suppose that  the following  conditions are satisfied:
	\begin{enumerate}
		\item[\rm (G)]  If $\mu e_j>0$, then $\sqrt{3}\mu e_j\notin \frac{2\pi}{3}+ \pi(2\mathbb{Z}+1)$.
		\item[\rm (H)] $\hat{N}(\mu)\ne\emptyset$,
 and all numbers in $\{e_j\,|\,j\in \hat{N}(\mu)\}$ have the same plus-minus sign.
			\end{enumerate}
		Then for the problem (\ref{e:2-Delay6Auto13C}),
	there exist two possible alternatives:
	\begin{itemize}
		\item[\rm (i)] The problem (\ref{e:2-Delay6Auto13C}) with $\lambda=\mu$ admits a sequence of nonzero distinct solutions
		$\{x^l\}^{\infty}_{l=1}$  such that $x^l\to 0$ in $C^1_{\rm loc}(\mathbb{R},\mathbb{R}^n)$
		as $l\to\infty$.
		\item[\rm (ii)] There exist left and right  neighborhoods $\Lambda^-$ and $\Lambda^+$ of $\mu$ in $\Lambda$
		and nonnegative integers $n^+$ and $n^-\ge 0$ satisfying $n^++n^-\ge
		2\,\sharp \hat{N}(\mu)$, such that for $\lambda\in\Lambda^-\setminus\{\mu\}$ (resp. $\lambda\in\Lambda^+\setminus\{\mu\}$),
		(\ref{e:2-Delay6Auto13C}) with parameter  $\lambda$  has at least $n^-$ (resp. $n^+$) distinct pairs of nontrivial solutions
		$\{x_\lambda^i, -x_\lambda^i\}$ ($i=1,\dots,n^-$ (resp. $n^+$))
		which converge to zero in $C^1_{\rm loc}(\mathbb{R}, \mathbb{R}^n)$
		as $\lambda\to\mu$.
	\end{itemize}
\end{corollary}
\begin{proof}[\bf Proof]
	Define $V:\mathbb{R}\times\mathbb{R}^n\to\mathbb{R}$ by $V(\lambda,x)=\lambda W(x)$. 	It satisfies Assumption~\ref{ass:BasiAss1Delay1}  with  $\Lambda=\mathbb{R}$, and $e_i(\lambda)=\lambda e_i$ 	($i=1,\dots,n$).
As in the proof of Corollary~\ref{cor:bif-per3DelayC}, the first assertion follows from the conclusion of Theorem~\ref{th:bif-per3DelayII}(I) with $\tau=1$.

From condition (G), it readily follows that $V$ satisfies condition (A) in Theorem~\ref{th:bif-per3DelayII} with $\tau=1$, and from condition (H), we similarly obtain condition (SB). Hence, Theorem~\ref{th:bif-per3DelayII}(III) yields the remaining assertions.
\end{proof}

\begin{proof}[\bf Proof of Theorem~\ref{th:bif-per3DelayII}]
\noindent\textbf{Step 1}(\textsf{Reduction to bifurcation of a class of
Hamiltonian boundary value problems}).
Let $A_{2,n}=-J_n$ be as in (\ref{e:Delay7}).
Define $\tilde{H}:\Lambda\times({\R}^{n})^{2}\to\R$  by
$$
\tilde{H}\left(\lambda,  x _ { 1 } ,  x _ {2} \right) =
V \left(\lambda,  x _ { 1 } \right)
+ V \left(\lambda,  x _ { 2} \right)+V \left(\lambda,  -x_1+x_2 \right).
$$
Clearly, each $\tilde{H}(\lambda,\cdot)$ is even.
Denote by $\nabla_2 H(\lambda, v)$  the gradient of $H\left(\lambda, v \right)$ with respect to $v\in({\R}^{n})^2$.
By  \cite[\S1.4]{Liu12},  $x:\mathbb{R}\to\mathbb{R}^n$ solves
 the problem (\ref{e:2-Delay6Auto13}) if and only if
$v:\mathbb{R}\to ({\R}^{n})^2$ defined by
$v(t):=(x_1(t)^\top, x_2(t)^\top)^\top\in ({\R}^{n})^2$,
where $x _ { 2 } ( t ) = x ( t-\tau)$, satisfies
 \begin{equation}\label{e:NewAdd2}
\dot{v}(t)=A_{2,n}\nabla_2\tilde{H}(\lambda, v(t))\quad\hbox{and}\quad v(t+\tau)=(\tilde { B } _ {2,n})^{-1}v(t)\;\;\forall t\in \mathbb{R}.
\end{equation}
where $\tilde { B } _ {2,n} = \left( \begin{array} { c c c c c } 0 & I _ { n }\\
- I _ { n } & I _ { n }\end{array} \right)$ and hence
\begin{equation}\label{e:NewAdd1}
(\tilde { B } _ {2,n})^{-1} = \left( \begin{array} { c c c c c }
I_n & -I _ { n }\\
 I _ { n } & 0\end{array} \right).
\end{equation}
In particular,   $v\equiv{0}\in\mathbb{R}^{2n}$ is a solution of (\ref{e:NewAdd2}) because  $x\equiv 0\in\mathbb{R}^n$ satisfies
(\ref{e:2-Delay6Auto13}).

Let ${\Upsilon}(2, n)=\left( \begin{array} { cc}
0 & I_n \\
I_n &0
 \end{array} \right)$ be as in (\ref{e:M-invariantDelay3-}). Define
  \begin{eqnarray}\label{e:NewAdd6}
&&M_{3,n}={\Upsilon}(2, n)\tilde{B}_{2,n}^{-1}{\Upsilon}(2, n)^{-1}=\left( \begin{array} { c c c c c }
0 & I _ { n }\\
 -I _ { n } & I_n\end{array} \right)\quad\hbox{and}\\
&&\check{H}(\lambda, z)=\tilde{H}(\lambda, {\Upsilon}(2, n)^{-1}z)\quad\forall
(\lambda, z)\in \Lambda\times\R\times({\R}^{n})^{2},\label{e:NewAdd7}
\end{eqnarray}
respectively. For $z=({x}^\top, {y}^\top)^\top\in ({\R}^{n})^2$ we have $M_{3,n}z=(y^\top, -x^\top+y^\top)^\top$ and
$\check{H}(\lambda, z)=V(\lambda,y)+V(\lambda,x)+V(\lambda, -y+x)$.
Since each $V(\lambda,\cdot)$ is even, $\check{H}(\lambda, \cdot)$
is also even, and
$\check{H}(\lambda,M_{3,n}z)=V(\lambda, y)+ V(\lambda, -x+y)+V(\lambda, -x)=
\check{H}(\lambda,z)$, i.e.,
 each $\check{H}(\lambda, \cdot)$ is $M_{3,n}$-invariant.

Recall that  $A_{2,n}^{-1}={\Upsilon}(2, n)^\top J_n^{-1}{\Upsilon}(2, n)$.
A straightforward calculation  shows that
$v:\mathbb{R}\to\mathbb{R}^{2n}$ is a solution of the problem (\ref{e:NewAdd2})
if and only if $z(t):=\Upsilon(2,n)v(t)$ satisfies the following problem
\begin{equation}\label{e:NewAdd8}
\dot{z}(t)=J_n\nabla_2\check{H}\left(\lambda, z(t)\right)\quad\hbox{and}\quad z(t+\tau)=M_{3,n}z(t)\;\;\forall t\in \R.
\end{equation}
and in this situation we have also
\begin{eqnarray}\label{e:NewAdd9}
\int^{\tau}_0\left[\frac{1}{2}(J_{n}\dot{z}(t), z(t))_{\mathbb{R}^{2n}}+ \check{H}(\lambda,  z(t))\right]dt
= \int^{\tau}_0\left[\frac{1}{2}(\dot{v}(t), J_nv(t))_{\mathbb{R}^{2n}}+ \tilde{H}(\lambda,
v(t))\right]dt.
\end{eqnarray}

Consider the linearizations  of the problems
(\ref{e:NewAdd2}) and (\ref{e:NewAdd8}) at trivial solutions $0$,
\begin{align}\label{e:NewAdd4}
\dot{u}(t)&=-J_n\nabla^2_2\tilde{H}(\lambda, {0})u(t)\quad\hbox{and}\quad u(t+\tau)=(\tilde { B } _ {2,n})^{-1}u(t)\;\;\forall t\in \mathbb{R},\\
\dot{z}(t)&=J_n\nabla^2_2\check{H}\left(\lambda,{0}\right)z(t)\quad\hbox{and}\quad z(t+\tau)=M_{3,n}z(t)\;\;\forall t\in \R.
\label{e:NewAdd10}
\end{align}
Here, by a straightforward computation, we have
\begin{eqnarray}\label{e:NewAdd5}
 \nabla^2_2\tilde{H}\left(\lambda, {0}\right) &=&
\left( \begin{array} { c c c c }
\nabla^2_2V \left(\lambda, {0} \right) & 0 \\
0 &\nabla^2_2V\left(\lambda, {0}\right) \end{array} \right)
+\left(-I_{n}\;I_{n}\right)^\top \cdot \nabla^2_2V\left(\lambda, {\bf 0}\right) \cdot \left(-I_{n}\;I_{n}\right)\nonumber\\
&=&\left( \begin{array} { c c c c }
2\nabla^2_2V \left(\lambda, {0} \right) & -\nabla^2_2V \left(\lambda, {0} \right) \\
-\nabla^2_2V \left(\lambda, {0} \right) &2\nabla^2_2V\left(\lambda, {0}\right) \end{array} \right),
\end{eqnarray}
and hence
\begin{eqnarray*}
\nabla^2_2 \check{H}(\lambda,  {0})&=&(\Upsilon(2,n)^{-1})^\top\nabla^2_2 \tilde{H}(\lambda, {0})\Upsilon(2,n)^{-1}\\
&=&\left( \begin{array} {cc}
0 & I_n \\
I_n &0
 \end{array} \right)^\top
 \left( \begin{array} { c c c c }
2\nabla^2_2V \left(\lambda, {0} \right) & -\nabla^2_2V \left(\lambda, {0} \right) \\
-\nabla^2_2V \left(\lambda, {0} \right) &2\nabla^2_2V\left(\lambda, {0}\right) \end{array} \right)
 \left( \begin{array} {cc}
0 & I_n \\
I_n &0
 \end{array} \right)\\
 &=&\left( \begin{array} { c c c c }
2\nabla^2_2V \left(\lambda, {0} \right) & -\nabla^2_2V \left(\lambda, {0} \right) \\
-\nabla^2_2V \left(\lambda, {0} \right) &2\nabla^2_2V\left(\lambda, {0}\right) \end{array} \right)
 =\nabla^2_2\tilde{H}\left(\lambda,{0}\right).
\end{eqnarray*}
  Let $\gamma_\lambda$ (resp. $\Upsilon_\lambda$) be  the fundamental matrix solution of
  $\dot{u}(t)=-J_n\nabla^2_2\tilde{H}\left(\lambda,{0}\right)u(t)$
[resp. $\dot{z}(t)=J_n\nabla^2_2\check{H}\left(\lambda,{0}\right)z(t)$]
 with $\gamma_\lambda(0)=I_{2n}$ (resp. $\Upsilon_\lambda(0)=I_{2n}$).
 Then $\Upsilon_{\lambda}=\Upsilon(2,n)\gamma_{\lambda}\Upsilon(2,n)^{-1}$
 and hence
  \begin{equation}\label{e:NewAdd11}
\left\{\begin{array}{ll}
&i_{\tau, M_{3,n}}(\Upsilon_\lambda)=i_{\tau, M_{3,n}}\big(\Upsilon(2,n)\gamma_{\lambda}\Upsilon(2,n)^{-1}\big)\\
&\nu_{\tau, M_{3,n}}(\Upsilon_{\lambda})=\nu_{\tau, M_{3,n}}\big(\Upsilon(2,n)\gamma_{\lambda}
\Upsilon(2,n)^{-1}\big).
\end{array}\right.
\end{equation}

By  (\ref{e:Add2}) in Theorem~\ref{th:PIndex3} we have
\begin{eqnarray}\label{e:NewAdd12-0}
\nu_{\tau, M_{3,n}}(\Upsilon_\mu)=2\,\sharp N(\mu).
\end{eqnarray}

\noindent\textbf{Step 2}(\textsf{Proof of (I)}).
Assume that $(\mu,0)$ with $\mu\in\Lambda$ is a bifurcation point of (\ref{e:2-Delay6Auto13}).
By the arguments preceding (\ref{e:NewAdd9}), this occurs
 if and only if this very $\mu$ gives rise to a bifurcation point $(\mu,0)$ of (\ref{e:NewAdd8}). Consequently, (\ref{e:NewAdd8}) has
  $(\mu,0)$ as a bifurcation point.
 An application of \cite[Theorem~1.5]{Lu11} now yields $\nu_{\tau, M_{3,n}}(\Upsilon_\mu)\ne 0$.  In view of (\ref{e:NewAdd12-0}), this implies
 $N(\mu)\ne\emptyset$, which establishes the conclusion in (I).

\noindent\textbf{Step 3}(\textsf{Proof of (III)}).
Note that (\ref{e:NewAdd12-0}) together with assumptions (A) and (SB) gives
\begin{eqnarray}\label{e:NewAdd12}
\nu_{\tau, M_{3,n}}(\Upsilon_\mu)=
2\sharp\left\{1\le j\le n\,\Big|\, e_j(\mu)<0 \text{ and } \sqrt{3}e_j(\mu)\tau\in \frac{2\pi}{3}+ \pi(2\mathbb{Z}+1)\right\}.
\end{eqnarray}
Moreover, (\ref{e:Add3}) in Theorem~\ref{th:PIndex3} gives
\begin{align}\label{e:NewAdd12+0}
 i_{\tau,M_{3,n}}(\Upsilon_\mu)&=\left[n-\frac{n\arctan 2}{\pi}\right]+2\sum_{e_j(\mu)>0}\sharp\left\{0<t<\tau\,\Big|\, \sqrt{3}e_j(\mu)t\in \frac{2\pi}{3}+ \pi(2\mathbb{Z}+1)\right\}\nonumber\\
&\quad-2\sum_{e_j(\mu)<0}\sharp\left\{0<t\le\tau\,\Big|\, \sqrt{3}e_j(\mu)t\in \frac{2\pi}{3}+ \pi(2\mathbb{Z}+1)\right\}\nonumber\\
&=\Xi-\nu_{\tau, M_{3,n}}(\Upsilon_\mu),
\end{align}
where
\begin{eqnarray*}
  \Xi&=&\left[n-\frac{n\arctan 2}{\pi}\right]+2\sum_{e_i(\mu)>0}\sharp\left\{0<t<\tau\,\Big|\, \sqrt{3}e_i(\mu)t\in \frac{2\pi}{3}+ \pi(2\mathbb{Z}+1)\right\}\nonumber\\
&&-2\sum_{e_i(\mu)<0}\sharp\left\{0<t<\tau\,\Big|\, \sqrt{3}e_i(\mu)t\in \frac{2\pi}{3}+ \pi(2\mathbb{Z}+1)\right\}.
\end{eqnarray*}

We aim to prove that, for some $\epsilon>0$, the identity for
$i_{\tau,M_{3,n}}(\Upsilon_\lambda)$
is given by
\begin{equation}\label{e:NewAdd12+}
   i_{\tau,M_{3,n}}(\Upsilon_\lambda)=
    \left\{\begin{array}{ll}
\Xi\quad&\hbox{if $\lambda\in (\mu-\epsilon, \mu)$},\\
\Xi-\nu_{\tau,M_{3,n}}(\Upsilon_\mu)\quad&\hbox{if $\lambda\in (\mu, \mu+\epsilon)$}
\end{array}\right.
\end{equation}
 or
 \begin{equation}\label{e:NewAdd12++}
   i_{\tau,M_{3,n}}(\Upsilon_\lambda)=
    \left\{\begin{array}{ll}
\Xi-\nu_{\tau,M_{3,n}}(\Upsilon_\mu)\quad&\hbox{if $\lambda\in (\mu-\epsilon, \mu)$},\\
\Xi\quad&\hbox{if $\lambda\in (\mu, \mu+\epsilon)$},
\end{array}\right.
\end{equation}
  according to whether (SB.1) or (SB.2)  of Assumption (SB) holds, respectively.

Since  $e_1(\mu)\le\cdots\le e_n(\mu)$, it follows from
(\ref{e:NewAdd12}) that $e_1(\mu)<0$. We may assume $e_1(\mu)\le\cdots\le e_k(\mu)<0$ and $e_n(\mu)\ge\cdots\ge e_{k+1}(\mu)\ge 0$.

Let us first examine the properties of $e_j(\lambda)$($1\le j\le n$)
in a neighborhood of $\lambda=\mu$, considering four distinct cases.

{\bf Case 1}($e_j(\mu)>0$ and so $\sqrt{3}e_j(\mu)\tau\notin \frac{2\pi}{3}+ \pi(2\mathbb{Z}+1)$ by condition (A)). By the continuity of $e_j(\lambda)$ at $\lambda=\mu$,
there exists $\epsilon>0$ such that for $\lambda\in (\mu-\epsilon,\mu+\epsilon)$,
$e_j(\lambda)>0$ and $\sqrt{3}e_j(\lambda)\tau\notin \frac{2\pi}{3}+ \pi(2\mathbb{Z}+1)$. If
$\sharp\left\{0<t<\tau\,\Big|\, \sqrt{3}e_j(\mu)t\in \frac{2\pi}{3}+ \pi(2\mathbb{Z}+1)\right\}>0$, let
$$
\left\{0<t<\tau\,\Big|\, \sqrt{3}e_j(\mu)t\in \frac{2\pi}{3}+ \pi(2\mathbb{Z}+1)\right\}
=\{t_1<\cdots<t_s\}.
$$
Then $\sqrt{3}e_j(\mu)t_i=\frac{2\pi}{3}+ (2m_{j,i}+1)\pi$
for integers $m_{j,i}$, $i=1,\cdots,s$.
 Since $e_j(\lambda)$ is continuous at $\mu$,
 by shrinking $\epsilon>0$ we can obtain, for each $\lambda\in (\mu-\epsilon,\mu+\epsilon)$, a unique $t_{i,\lambda}\in (0,\tau)$ near $t_i$
such that $\sqrt{3}e_j(\lambda)t_{i,\lambda}\in \frac{2\pi}{3}+ \pi(2\mathbb{Z}+1)$,
that is,
$$
t_{i,\lambda}=\frac{e_j(\mu)}{e_j(\lambda)}t_i,\quad i=1,\cdots,s.
$$

It follows that for any $\lambda\in (\mu-\epsilon,\mu+\epsilon)$,
$$
\sharp\left\{0<t<\tau\,\Big|\, \sqrt{3}e_j(\lambda)t\in \frac{2\pi}{3}+ \pi(2\mathbb{Z}+1)\right\}=\sharp\left\{0<t<\tau\,\Big|\, \sqrt{3}e_j(\mu)t\in \frac{2\pi}{3}+ \pi(2\mathbb{Z}+1)\right\}.
$$

{\bf Case 2}($e_j(\mu)<0$ and $\sqrt{3}e_j(\mu)\tau\notin \frac{2\pi}{3}+ \pi(2\mathbb{Z}+1)$). As above we have
 $\epsilon>0$ such that for $\lambda\in (\mu-\epsilon,\mu+\epsilon)$,
 \begin{eqnarray*}
 &&e_j(\lambda)<0,\quad \sqrt{3}e_j(\lambda)\tau\notin \frac{2\pi}{3}+ \pi(2\mathbb{Z}+1)\quad\hbox{and}\\
 &&\sharp\left\{0<t<\tau\,\Big|\, \sqrt{3}e_j(\lambda)t\in \frac{2\pi}{3}+ \pi(2\mathbb{Z}+1)\right\}=\sharp\left\{0<t<\tau\,\Big|\, \sqrt{3}e_j(\mu)t\in \frac{2\pi}{3}+ \pi(2\mathbb{Z}+1)\right\}.
 \end{eqnarray*}

{\bf Case 3}($e_j(\mu)<0$ and $\sqrt{3}e_j(\mu)\tau\in \frac{2\pi}{3}+ \pi(2\mathbb{Z}+1)$, i.e., $j\in N(\mu)$).
Since  $e_j(\lambda)$ is continuous at $\mu$, after shrinking
 $\epsilon>0$ we may assume that for $\lambda\in (\mu-\epsilon,\mu+\epsilon)\setminus\{\mu\}$,
  \begin{eqnarray*}
 e_j(\lambda)<0,\quad \sqrt{3}e_j(\lambda)\tau\notin \frac{2\pi}{3}+ \pi(2\mathbb{Z}+1).
 \end{eqnarray*}
 Let $t_\lambda=\frac{e_j(\mu)}{e_j(\lambda)}\tau$, which is near $\tau$.

 Suppose (SB.1) holds. Then $t_\lambda>\tau$
 for $\lambda\in (\mu-\epsilon, \mu)$, and $t_\lambda<\tau$
 for $\lambda\in (\mu, \mu+\epsilon)$. Consequently,
  \begin{eqnarray*}
 \sharp\left\{0<t\le\tau\,\Big|\, \sqrt{3}e_j(\lambda)t\in \frac{2\pi}{3}+ \pi(2\mathbb{Z}+1)\right\}=\sharp\left\{0<t\le\tau\,\Big|\, \sqrt{3}e_j(\mu)t\in \frac{2\pi}{3}+ \pi(2\mathbb{Z}+1)\right\}-1
 \end{eqnarray*}
for $\lambda\in (\mu-\epsilon, \mu)$, and
  \begin{eqnarray*}
 \sharp\left\{0<t\le\tau\,\Big|\, \sqrt{3}e_j(\lambda)t\in \frac{2\pi}{3}+ \pi(2\mathbb{Z}+1)\right\}=\sharp\left\{0<t\le\tau\,\Big|\, \sqrt{3}e_j(\mu)t\in \frac{2\pi}{3}+ \pi(2\mathbb{Z}+1)\right\}
 \end{eqnarray*}
for $\lambda\in (\mu, \mu+\epsilon)$.

Similarly, if (SB.2) holds,
 then $t_\lambda<\tau$
 for $\lambda\in (\mu-\epsilon, \mu)$, and $t_\lambda>\tau$
 for $\lambda\in (\mu, \mu+\epsilon)$, and therefore,
  \begin{eqnarray*}
 \sharp\left\{0<t\le\tau\,\Big|\, \sqrt{3}e_j(\lambda)t\in \frac{2\pi}{3}+ \pi(2\mathbb{Z}+1)\right\}=\sharp\left\{0<t\le\tau\,\Big|\, \sqrt{3}e_j(\mu)t\in \frac{2\pi}{3}+ \pi(2\mathbb{Z}+1)\right\}
 \end{eqnarray*}
for $\lambda\in (\mu-\epsilon, \mu)$, and
  \begin{eqnarray*}
 \sharp\left\{0<t\le\tau\,\Big|\, \sqrt{3}e_j(\lambda)t\in \frac{2\pi}{3}+ \pi(2\mathbb{Z}+1)\right\}=\sharp\left\{0<t\le\tau\,\Big|\, \sqrt{3}e_j(\mu)t\in \frac{2\pi}{3}+ \pi(2\mathbb{Z}+1)\right\}-1
 \end{eqnarray*}
for  $\lambda\in (\mu, \mu+\epsilon)$.

{\bf Case 4}($e_j(\mu)=0$). By the continuity of $e_j(\lambda)$ at $\lambda=\mu$,
after shrinking  $\epsilon>0$ we have
 $\sqrt{3}e_j(\lambda)t\notin \frac{2\pi}{3}+ \pi(2\mathbb{Z}+1)$
for $(\lambda, t)\in (\mu-\epsilon,\mu+\epsilon)\times [0, \tau]$.\\

The preceding four cases establish that, for each
 $j=1,\dots,n$,
$\sqrt{3}\,e_j(\lambda)\tau\notin \frac{2\pi}{3}+ \pi(2\mathbb{Z}+1)$
for all $\lambda\in (\mu-\epsilon,\mu+\epsilon)\setminus\{\mu\}$.
 Consequently,
\begin{equation}\label{e:NewAdd13-}
i_{\tau,M_{3,n}}(\Upsilon_\lambda)=0,\quad \text{for all}\;\lambda\in (\mu-\epsilon,\mu+\epsilon)\setminus\{\mu\}.
\end{equation}


\textbf{Proof of \eqref{e:NewAdd12+}}.
Put $N_2=N(\mu)$, $N_3=\{j\,|\, e_j(\mu)=0\}$, and
\begin{eqnarray*}
N_1=\{j\,|\, \hbox{$e_j(\mu)<0$ and $\sqrt{3}e_j(\mu)\tau\notin \frac{2\pi}{3}+ \pi(2\mathbb{Z}+1)$}\}.
\end{eqnarray*}
\underline{Firstly, we assume $\lambda\in (\mu-\epsilon, \epsilon)$}.
Let  $e_i(\lambda)<0$. If $\sqrt{3}e_i(\lambda)\tau\in \frac{2\pi}{3}+ \pi(2\mathbb{Z}+1)$,
then $e_i(\mu)$ cannot satisfy any one of the four cases above, and so a contradiction is obtained. 
It follows that $\sqrt{3}e_i(\lambda)\tau\notin \frac{2\pi}{3}+ \pi(2\mathbb{Z}+1)$,
and therefore $e_i(\mu)$  satisfies at least one of Case 2, Case 3 and Case 4
by the above analysis, i.e., $i\in N_1\cup N_2\cup N_3$.
Hence
\begin{eqnarray}\label{e:NewAdd13}
&&-2\sum_{e_i(\lambda)<0}\sharp\left\{0<t\le\tau\,\Big|\, \sqrt{3}e_i(\lambda)t\in \frac{2\pi}{3}+ \pi(2\mathbb{Z}+1)\right\}\nonumber\\
&=&-2\sum_{i\in N_1}\sharp\left\{0<t\le\tau\,\Big|\, \sqrt{3}e_i(\lambda)t\in \frac{2\pi}{3}+ \pi(2\mathbb{Z}+1)\right\}\nonumber\\
&&-2\sum_{i\in N_2}\sharp\left\{0<t\le\tau\,\Big|\, \sqrt{3}e_i(\lambda)t\in \frac{2\pi}{3}+ \pi(2\mathbb{Z}+1)\right\}\nonumber\\
&&-2\sum_{i\in N_3}\sharp\left\{0<t\le\tau\,\Big|\, \sqrt{3}e_i(\lambda)t\in \frac{2\pi}{3}+ \pi(2\mathbb{Z}+1)\right\}\nonumber\\
&=&-2\sum_{i\in N_1}\sharp\left\{0<t<\tau\,\Big|\, \sqrt{3}e_i(\mu)t\in \frac{2\pi}{3}+ \pi(2\mathbb{Z}+1)\right\}\nonumber\\
&&-2\sum_{i\in N_2}\left(\sharp\left\{0<t\le\tau\,\Big|\, \sqrt{3}e_i(\mu)t\in \frac{2\pi}{3}+ \pi(2\mathbb{Z}+1)\right\}-1\right)\nonumber\\
&=&-2\sum_{i\in N_1}\sharp\left\{0<t<\tau\,\Big|\, \sqrt{3}e_i(\mu)t\in \frac{2\pi}{3}+ \pi(2\mathbb{Z}+1)\right\}\nonumber\\
&&-2\sum_{i\in N_2}\sharp\left\{0<t<\tau\,\Big|\, \sqrt{3}e_i(\mu)t\in \frac{2\pi}{3}+ \pi(2\mathbb{Z}+1)\right\}\nonumber\\
&=&-2\sum_{e_i(\mu)<0}\sharp\left\{0<t<\tau\,\Big|\, \sqrt{3}e_i(\mu)t\in \frac{2\pi}{3}+ \pi(2\mathbb{Z}+1)\right\}-\nu_{\tau,M_{3,n}}(\Upsilon_\mu).
\end{eqnarray}

If $e_i(\lambda)>0$, the analysis in the four cases above shows that
$e_i(\mu)$ satisfies at least one of Case 1 and Case 4.
If Case 4  occurs for $e_i(\mu)$, then  $\sqrt{3}e_i(\lambda)t\notin \frac{2\pi}{3}+ \pi(2\mathbb{Z}+1)$
for $(\lambda, t)\in (\mu-\epsilon,\mu+\epsilon)\times [0, \tau]$
by the assumptions. Thus
\begin{eqnarray}\label{e:NewAdd14}
  &&2\sum_{e_i(\lambda)>0}\sharp\left\{0<t<\tau\,\Big|\, \sqrt{3}e_i(\lambda)t\in \frac{2\pi}{3}+ \pi(2\mathbb{Z}+1)\right\}\nonumber\\
  &=&2\sum_{e_i(\mu)>0}\sharp\left\{0<t<\tau\,\Big|\, \sqrt{3}e_i(\lambda)t\in \frac{2\pi}{3}+ \pi(2\mathbb{Z}+1)\right\}\nonumber\\
  &&+2\sum_{e_i(\mu)=0}\sharp\left\{0<t<\tau\,\Big|\, \sqrt{3}e_i(\lambda)t\in \frac{2\pi}{3}+ \pi(2\mathbb{Z}+1)\right\}\nonumber\\
 &=&2\sum_{e_i(\mu)>0}\sharp\left\{0<t<\tau\,\Big|\, \sqrt{3}e_j(\mu)t\in \frac{2\pi}{3}+ \pi(2\mathbb{Z}+1)\right\}.
\end{eqnarray}
Using this and (\ref{e:NewAdd13}), we derive from (\ref{e:Add3})
\begin{eqnarray}\label{e:NewAdd15}
  i_{\tau,M_{3,n}}(\Upsilon_\lambda)&=&\left[n-\frac{n\arctan 2}{\pi}\right]+2\sum_{e_i(\mu)>0}\sharp\left\{0<t<\tau\,\Big|\, \sqrt{3}e_i(\mu)t\in \frac{2\pi}{3}+ \pi(2\mathbb{Z}+1)\right\}\nonumber\\
&&-2\sum_{e_i(\mu)<0}\sharp\left\{0<t<\tau\,\Big|\, \sqrt{3}e_i(\mu)t\in \frac{2\pi}{3}+ \pi(2\mathbb{Z}+1)\right\}=\Xi.
\end{eqnarray}

\underline{Next, we consider the case where $\lambda\in (\mu, \epsilon+\epsilon)$}.
Let  $e_i(\lambda)<0$.  
As above we have $\sqrt{3}e_i(\lambda)\tau\notin \frac{2\pi}{3}+ \pi(2\mathbb{Z}+1)$,
and hence $i\in N_1\cup N_2\cup N_3$. Then
\begin{eqnarray}\label{e:NewAdd16}
&&-2\sum_{e_i(\lambda)<0}\sharp\left\{0<t\le\tau\,\Big|\, \sqrt{3}e_i(\lambda)t\in \frac{2\pi}{3}+ \pi(2\mathbb{Z}+1)\right\}\nonumber\\
&=&-2\sum_{i\in N_1}\sharp\left\{0<t\le\tau\,\Big|\, \sqrt{3}e_i(\lambda)t\in \frac{2\pi}{3}+ \pi(2\mathbb{Z}+1)\right\}\nonumber\\
&&-2\sum_{i\in N_2}\sharp\left\{0<t\le\tau\,\Big|\, \sqrt{3}e_i(\lambda)t\in \frac{2\pi}{3}+ \pi(2\mathbb{Z}+1)\right\}\nonumber\\
&&-2\sum_{i\in N_3}\sharp\left\{0<t\le\tau\,\Big|\, \sqrt{3}e_i(\lambda)t\in \frac{2\pi}{3}+ \pi(2\mathbb{Z}+1)\right\}\nonumber\\
&=&-2\sum_{i\in N_1}\sharp\left\{0<t<\tau\,\Big|\, \sqrt{3}e_i(\mu)t\in \frac{2\pi}{3}+ \pi(2\mathbb{Z}+1)\right\}\nonumber\\
&&-2\sum_{i\in N_2}\sharp\left\{0<t\le\tau\,\Big|\, \sqrt{3}e_i(\mu)t\in \frac{2\pi}{3}+ \pi(2\mathbb{Z}+1)\right\}\nonumber\\
&=&-2\sum_{e_i(\mu)<0}\sharp\left\{0<t<\tau\,\Big|\, \sqrt{3}e_i(\mu)t\in \frac{2\pi}{3}+ \pi(2\mathbb{Z}+1)\right\}-\nu_{\tau,M_{3,n}}(\Upsilon_\mu)
\end{eqnarray}
by (\ref{e:NewAdd12}).

If $e_i(\lambda)>0$, as above $e_i(\mu)$ satisfies at least one of Case 1 and Case 4, and we have
\begin{eqnarray}\label{e:NewAdd17}
  &&2\sum_{e_i(\lambda)>0}\sharp\left\{0<t<\tau\,\Big|\, \sqrt{3}e_i(\lambda)t\in \frac{2\pi}{3}+ \pi(2\mathbb{Z}+1)\right\}\nonumber\\
 &=&2\sum_{e_i(\mu)>0}\sharp\left\{0<t<\tau\,\Big|\, \sqrt{3}e_j(\mu)t\in \frac{2\pi}{3}+ \pi(2\mathbb{Z}+1)\right\}.
\end{eqnarray}
Using this and (\ref{e:NewAdd16}), we derive from (\ref{e:Add3})
\begin{eqnarray}\label{e:NewAdd18}
  i_{\tau,M_{3,n}}(\Upsilon_\lambda)&=&\left[n-\frac{n\arctan 2}{\pi}\right]+2\sum_{e_i(\mu)>0}\sharp\left\{0<t<\tau\,\Big|\, \sqrt{3}e_i(\mu)t\in \frac{2\pi}{3}+ \pi(2\mathbb{Z}+1)\right\}\nonumber\\
&&-2\sum_{e_i(\mu)<0}\sharp\left\{0<t<\tau\,\Big|\, \sqrt{3}e_i(\mu)t\in \frac{2\pi}{3}+ \pi(2\mathbb{Z}+1)\right\}-\nu_{\tau,M_{3,n}}(\Upsilon_\mu)\nonumber\\
&=&\Xi-\nu_{\tau,M_{3,n}}(\Upsilon_\mu).
\end{eqnarray}
Combining this with (\ref{e:NewAdd15}) gives (\ref{e:NewAdd12+}).

 The result (\ref{e:NewAdd12++}) follows by a similar argument.


%
%


Note that the matrix $\tilde { B } _ {2,n}$
in (\ref{e:NewAdd1}) is not orthogonal.
Theorem~1.14 from \cite{Lu11} cannot be
applied to problem (\ref{e:NewAdd8}).
However, since each  $\check{H}(\lambda, \cdot)$
is  even, applying the second part of Theorem~1.8 from \cite{Lu11} to problem (\ref{e:NewAdd8}), we obtain the following two possible alternatives:
\begin{itemize}
\item[\rm (i*)] The problem (\ref{e:NewAdd8})  with $\lambda=\mu$ admits a sequence of nonzero distinct solutions
$\{z^l\}^{\infty}_{l=1}$  such that $z^l\to 0$ in $C^1_{\rm loc}(\mathbb{R},\mathbb{R}^{2n})$
as $l\to\infty$.
\item[\rm (ii*)] There exist left and right  neighborhoods $\Lambda^-$ and $\Lambda^+$ of $\mu$ in $\Lambda$
and nonnegative integers $n^+$ and $n^-\ge 0$ satisfying $n^++n^-\ge \nu_{\tau, M_{3,n}}(\Upsilon_\mu)=2\,\sharp N(\mu)$, such that for $\lambda\in\Lambda^-\setminus\{\mu\}$ (resp. $\lambda\in\Lambda^+\setminus\{\mu\}$),
(\ref{e:NewAdd8}) with parameter  $\lambda$  has at least $n^-$ (resp. $n^+$) distinct pairs of nontrivial solutions
$\{z_\lambda^i, -z_\lambda^i\}$ ($i=1,\dots,n^-$ (resp. $n^+$))
which converge to zero in $C^1_{\rm loc}(\mathbb{R}, \mathbb{R}^{2n})$
as $\lambda\to\mu$.
\end{itemize}

Let
 $\Upsilon(2,n)^{-1}{z}^l(t):=({x}^{l}_1(t)^\top, {x}^{l}_2(t)^\top)^\top\in ({\mathbb{R}}^{n})^2$ and
$\Upsilon(2,n)^{-1}{z}^i_\lambda(t):=({x}^i_{\lambda,1}(t)^\top, {x}^i_{\lambda,2}(t)^\top)^\top\in ({\mathbb{R}}^{n})^2$.
These functions satisfy (\ref{e:NewAdd2}). It follows that
 $x^l_ { 2 } ( t ) = x^l_1( t-\tau)$ and
 ${x}^i_{\lambda,2}(t)={x}^i_{\lambda,1}(t-\tau)$
 and that the functions
$x^l_1$ and ${x}^i_{\lambda,1}$ satisfy (i) and (ii).

\noindent\textbf{Step 4}(\textsf{Proof of (II)}).
As in the proof of Theorem~\ref{th:bif-per3Delay}(II), we have
either $(\lambda_m^-)_m\subset (\mu-\epsilon, \mu)$ and
$(\lambda_m^+)_m\subset (\mu, \mu+\epsilon)$ or
$(\lambda_m^-)_m\subset (\mu, \mu+\epsilon)$ and
$(\lambda_m^+)_m\subset (\mu-\epsilon, \mu)$.
From the proofs of (\ref{e:NewAdd12+}) and (\ref{e:NewAdd12++})
one easily sees that for any $m\in\mathbb{N}$,
\begin{equation*}
   i_{\tau,M_{3,n}}(\Upsilon_\lambda)=
    \left\{\begin{array}{ll}
\Xi\quad&\hbox{if $\lambda=\lambda_m^-$},\\
\Xi-\nu_{\tau,M_{3,n}}(\Upsilon_\mu)\quad&\hbox{if $\lambda=\lambda_m^+$}
\end{array}\right.
\end{equation*}
 in the first case, and
  \begin{equation*}
   i_{\tau,M_{3,n}}(\Upsilon_\lambda)=
    \left\{\begin{array}{ll}
\Xi-\nu_{\tau,M_{3,n}}(\Upsilon_\mu)\quad&\hbox{if $\lambda=\lambda_m^+$},\\
\Xi\quad&\hbox{if $\lambda=\lambda_m^-$}
\end{array}\right.
\end{equation*}
in the second case.  Together with (\ref{e:NewAdd13-}),
these results show that
 for all $m\in\mathbb{N}$,
$$
[i_{\tau,M_{3,n}}(\Upsilon_{\lambda_m^-}), i_{\tau,M_{3,n}}(\Upsilon_{\lambda_m^-})+
\nu_{\tau,M_{3,n}}(\Upsilon_{\lambda_m^-})]\cap
[i_{\tau,M_{3,n}}(\Upsilon_{\lambda_m^+}), i_{\tau,M_{3,n}}(\Upsilon_{\lambda_m^+})+
\nu_{\tau,M_{3,n}}(\Upsilon_{\lambda_m^+})]=\emptyset.
$$
Therefore, $(\mu,0)$ is a bifurcation point of (\ref{e:NewAdd8}) by
 \cite[Theorem~1.5(II)]{Lu11}, and the desired assertion (II) follows.

\end{proof}


Corollary~1.12 of \cite{Lu11} can yield
the following  parallel result to
Theorem~\ref{th:bif-per2Delay++}.

\begin{theorem}\label{th:2-bif-per2Delay++}
	Let $V_0, V_1:  \mathbb{R}^n \to \mathbb{R}$ be even $C^2$ functions such
	that the Hessian $\nabla^2{V}_1({0})$ is  positive definite or negative definite.
		Denoted the zero solution of the problem
	\begin{eqnarray}\label{e:2-LinearDelay2+}
		\left\{\begin{array}{ll}
			&\dot{x}( t ) = \sum^2_{j=1}\nabla V_0 (x(t-j\tau))+
			\lambda\sum^2_{j=1}\nabla V_1 (x ( t - j\tau )),\\
			& x( t +3 \tau ) = - x( t )\;\forall t
		\end{array}\right.
	\end{eqnarray}
	by ${0}^\lambda$, and
	the dimension of the solution space of the linear delay problem
	\begin{eqnarray}\label{e:2-LinearDelay2}
		\left\{\begin{array}{ll}
			&\dot{x}( t ) = \nabla^2 V_0 ({0})\sum^2_{j=1}x(t-j\tau) +
			\lambda\nabla^2 V_1 ({0})
			\sum^2_{j=1}x ( t - j\tau ),\\
			& x( t +3\tau ) = - x( t )\;\forall t
		\end{array}\right.
	\end{eqnarray}
	by $\nu_{\tau}({0}^\lambda)$. 
	Then
	\begin{enumerate}
		\item[\rm (I)] $\Sigma_\tau:=\{\lambda\in\mathbb{R}\,|\, \nu_{\tau}({0}^\lambda)>0\}$
		is a discrete set in $\mathbb{R}$.
		
		\item[\rm (II)]
		For each $\mu\in{\Sigma}_\tau$,
		at least  one of the following alternatives holds true:
		\begin{enumerate}
			\item[\rm (i)] System (\ref{e:2-LinearDelay2+}) with $\lambda=\mu$ admits a sequence of distinct nonzero solutions
			${x}^{k}$ for $k=1,2,\cdots$,
			which  converges to zero in $C^1_{\rm loc}(\mathbb{R},\R^{n})$ as $k\to\infty$.
			\item[\rm (ii)]
			There exist left and right  neighborhoods $\Lambda^-$ and $\Lambda^+$ of $\mu$ in $\Lambda$
			and nonnegative integers $n^+$ and $n^-\ge 0$ satisfying $n^++n^-\ge \nu_\tau(0^\mu)$, such that for $\lambda\in\Lambda^-\setminus\{\mu\}$ (resp. $\lambda\in\Lambda^+\setminus\{\mu\}$),
		(\ref{e:2-LinearDelay2+})	 with parameter  $\lambda$  has at least $n^-$ (resp. $n^+$) distinct pairs of nontrivial solutions
			$\{x_\lambda^i, -x_\lambda^i\}$ ($i=1,\dots,n^-$ (resp. $n^+$))
			which converge to zero in $C^1_{\rm loc}(\mathbb{R}, \mathbb{R}^n)$
			as $\lambda\to\mu$.			
		\end{enumerate}
		\end{enumerate}
\end{theorem}
\begin{proof}[\bf Proof]
For $j=0,1$, let us define $\tilde{H}_j, \check{H}_j:\mathbb{R}\times({\R}^{n})^{2}\to\R$  by
	\begin{align*}
	\tilde{H}_j\left( x _ { 1 } ,  x _ {2} \right) &=
	V_j \left(x _ { 1 } \right)
	+ V_j \left( x _ { 2} \right)+V_j \left(  -x_1+x_2 \right),\\
	\check{H}_j(z)&=\tilde{H}_j({\Upsilon}(2, n)^{-1}z)\quad\forall
	 z\in ({\R}^{n})^{2}.
	\end{align*}
For $j=0,1$, the functions $\tilde{H}_j$ and
$\check{H}_j$ are even; furthermore,
$\check{H}_j$ is $M_{3,n}$-invariant
and the Hessian $\nabla^2\check{H}_1({0})$ is  positive definite or negative definite.
Let $\Upsilon_\lambda$ be  the fundamental matrix solution of
 $\dot{z}(t)=J_n(\nabla^2_2\check{H}_0(0)+
 \lambda\nabla^2_2\check{H}_1(0))z(t)$
with $\Upsilon_\lambda(0)=I_{2n}$.
As in the proof of Theorem~\ref{th:bif-per3Delay}, we have
$\nu_{\tau, M_{3,n}}(\Upsilon_\lambda)=\nu_\tau(0^\lambda)$.
It follows from Corollary~1.12 in \cite{Lu11}
that
 $\Sigma_\tau=\{\lambda\in\mathbb{R}\,|\, \nu_{\tau, M_{3,n}}(\Upsilon_\lambda)>0\}$
is a discrete set in $\mathbb{R}$ and that
for each $\mu\in{\Sigma}_\tau$,
at least  one of the following alternatives holds true if $\check{H}(\lambda,z)$
in  (\ref{e:NewAdd8})
is understood as $\check{H}_0(z)+
\lambda\check{H}_1(z)$:
\begin{itemize}
	\item[\rm (i')] The problem (\ref{e:NewAdd8})  with $\lambda=\mu$ admits a sequence of nonzero distinct solutions
	$\{z^l\}^{\infty}_{l=1}$  such that $z^l\to 0$ in $C^1_{\rm loc}(\mathbb{R},\mathbb{R}^{2n})$
	as $l\to\infty$.
	\item[\rm (ii')] There exist left and right  neighborhoods $\Lambda^-$ and $\Lambda^+$ of $\mu$ in $\Lambda$
	and nonnegative integers $n^+$ and $n^-\ge 0$ satisfying $n^++n^-\ge \nu_{\tau, M_{3,n}}(\Upsilon_\mu)=\nu_\tau(0^\mu)$, such that for $\lambda\in\Lambda^-\setminus\{\mu\}$ (resp. $\lambda\in\Lambda^+\setminus\{\mu\}$),
	(\ref{e:NewAdd8}) with parameter  $\lambda$  has at least $n^-$ (resp. $n^+$) distinct pairs of nontrivial solutions
	$\{z_\lambda^i, -z_\lambda^i\}$ ($i=1,\dots,n^-$ (resp. $n^+$))
	which converge to zero in $C^1_{\rm loc}(\mathbb{R}, \mathbb{R}^{2n})$
	as $\lambda\to\mu$.
\end{itemize}
The desired statements follow in the same way as at the end of the proof of Theorem~\ref{th:bif-per3DelayII}.
\end{proof}

\section{Bifurcation near the trivial equilibrium of  system (\ref{e:Delay2Auto})}\label{sec:delay3}

Under Assumption~\ref{ass:bif-per2Delay2Auto}, a result parallel to Theorem~\ref{th:bif-per3Delay} holds for the  system
(\ref{e:Delay2Auto}) at the origin (the trivial equilibrium).

\begin{theorem}\label{th:bif-per3+3}
Under Assumption~\ref{ass:bif-per2Delay2Auto}, let $e_1(\lambda)\le \cdots\le e_n(\lambda)$ be all eigenvalues of the Hermitian matrix
$D(\lambda)+\sqrt{-1}E(\lambda)$.
Fix a delay parameter $\tau>0$.
\begin{enumerate}
\item[\rm (I)]{\rm (\textsf{Necessary condition}):}
If $(\mu,0)$ with $\mu\in\Lambda$ is a bifurcation point of
 (\ref{e:Delay2Auto}), then $e_l(\mu)\in\frac{\pi\mathbb{Z}}{\tau}$ for some $l$.

\item[\rm (II)]{\rm (\textsf{Sufficient condition}):}
Let $\mu$ be an interior point of $\Lambda$. Then
$(\mu,0)$ is a bifurcation point of (\ref{e:Delay6Auto13}),
 provided that  there exists a partition
  $\{i_1,\dots,i_q\}\cup\{i_{q+1},\dots,i_n\}$  of $\{1,\dots,n\}$
  with $1 \le q \le n$, such that
 \begin{itemize}
\item[\rm (A)] $e_{i_l}(\mu)=\frac{p_l\pi}{\tau}\in\frac{\pi}{\tau}\mathbb{Z}$ for $l=1,\dots, q$,
and $e_{i_l}(\mu)\notin\frac{\pi}{\tau}\mathbb{Z}$ for  $l=q+1,\dots,n$;

\item[\rm (B)] there exist two sequences $(\lambda_m^+)_m$ and $(\lambda_m^-)_m$ in $\Lambda$  converging to $\mu$ such that
$e_{i_l}(\lambda_m^-)< e_{i_l}(\mu)<e_{i_l}(\lambda_m^+)$
for all $m\in\mathbb{N}$ and all $l=1,\dots,q$.
\end{itemize}
%

 \item[\rm (III)]{\rm (\textsf{Alternative bifurcations of Fadell-Rabinowitz type and of Rabinowitz type}):}
Let the assumptions of (II) hold,
 replacing condition (B) by the following stronger requirement:
 \begin{itemize}
 \item[\rm (SB)] For some $\epsilon>0$, one of the following two conditions holds:
 \begin{itemize}
\item[\rm (SB.1)] $e_{i_l}(\lambda)< e_{i_l}(\mu)<e_{i_l}(\lambda')$
for all $(\lambda, \lambda')\in (\mu-\epsilon, \mu)\times (\mu, \mu+\epsilon)$ and
 $l=1,\dots,q$.

\item[\rm (SB.2)]
$e_{i_l}(\lambda)> e_{i_l}(\mu)>e_{i_l}(\lambda')$
for all $(\lambda, \lambda')\in (\mu-\epsilon, \mu)\times (\mu, \mu+\epsilon)$ and $l=1,\dots,q$.
 \end{itemize}
\end{itemize}

   Then  one of the following alternatives occurs:
   \begin{itemize}
\item[\rm (i)] The problem (\ref{e:Delay2Auto})
   with $\lambda=\mu$  admits a sequence of nonzero, $\R$-distinct solutions
$\{x^l\}_{l=1}^{\infty}$  such that $x^l\to 0$ in
$C^1_{\rm loc}(\mathbb{R},\mathbb{R}^{2n})$
as $l\to\infty$.
\item[\rm (ii)] There exist left and right neighborhoods $\Lambda^-$ and $\Lambda^+$ of $\mu$
in $\Lambda$, and nonnegative integers $n^-$ and $n^+$ satisfying $n^- + n^+ \geq q$,
such that for $\lambda \in \Lambda^- \setminus \{\mu\}$ (resp. $\lambda \in \Lambda^+ \setminus \{\mu\}$),
the problem (\ref{e:Delay2Auto}) with parameter $\lambda$ has at least $n^-$ (resp. $n^+$) $\mathbb{R}$-distinct nonzero solutions
$x_\lambda^i \neq 0$ ($i = 1, \dots, n^-$ (resp. $n^+$))
which converge to zero in $C^1_{\rm loc}(\mathbb{R}, \mathbb{R}^{2n})$
as $\lambda\to\mu$.
  \end{itemize}
Moreover, if $n>1$ and $q>1$, then at least one of (i), (iii), and (iv)  holds,
where (i) is as stated previously, and
\begin{itemize}
\item[\rm (iii)]  For every $\lambda\in\Lambda\setminus\{\mu\}$ near $\mu$, there is a
 nonzero solution ${x}_\lambda$  of problem (\ref{e:Delay2Auto})
 with parameter  $\lambda$, such that
 $x_\lambda\to 0$ in $C^1_{\rm loc}(\mathbb{R},\mathbb{R}^{2n})$ as $\lambda\to\mu$.

\item[\rm (iv)] For a given $\varepsilon > 0$, there exists a one-sided neighbourhood $\Lambda^0$ of $\mu$ in $\Lambda$ such that, for any $\lambda \in \Lambda^0 \setminus \{\mu\}$,
    problem (\ref{e:Delay2Auto}) with parameter $\lambda$  has either
    \begin{itemize}
    \item[$\bullet$]
     infinitely many $\mathbb{R}$-distinct nonzero solutions $\{x_\lambda^l\}_{l=1}^\infty$ such that $\|x_\lambda^l|_{[0,2\tau]}\|_{C^1}<\varepsilon$
        for all $l \in \mathbb{N}$, or \item[$\bullet$]
         at least two $\mathbb{R}$-distinct nonzero solutions $\hat{x}_\lambda^1$ and $\hat{x}_\lambda^2$ satisfying the following inequalities:
        $\|\hat{x}_\lambda^i|_{[0,2\tau]}\|_{C^1}<\varepsilon$ for $i=1,2$, and
    \begin{align*}
        & \bigl(\hat{x}_\lambda^1(0), J_n \hat{x}_\lambda^1(\tau) \bigr)_{\mathbb{R}^{2n}}
          + \int_0^{\tau} \bigl( \dot{\hat{x}}_\lambda^1(t), J_n \hat{x}_\lambda^1(t-\tau) \bigr)_{\mathbb{R}^{2n}} \, dt
          + \int_{-\tau}^{\tau} G\bigl(\lambda, \hat{x}_\lambda^1(t)\bigr) \, dt \\
        \ne \; &
        \bigl(\hat{x}_\lambda^2(0), J_n \hat{x}_\lambda^2(\tau) \bigr)_{\mathbb{R}^{2n}}
          + \int_0^{\tau} \bigl( \dot{\hat{x}}_\lambda^2(t), J_n \hat{x}_\lambda^2(t-\tau) \bigr)_{\mathbb{R}^{2n}} \, dt
          + \int_{-\tau}^{\tau} G\bigl(\lambda, \hat{x}_\lambda^2(t)\bigr) \, dt.
    \end{align*}
\end{itemize}
 \end{itemize}
 \end{enumerate}
\end{theorem}

When $n=1$, Theorem~\ref{th:bif-per3+3} has the following simple version.

\begin{corollary}\label{cor:bif-per3+3}
Let $\Lambda \subset \mathbb{R}$ be an interval, and let $G \in C(\Lambda \times \mathbb{R}^{2}, \mathbb{R})$
 be such that:
  \begin{itemize}
\item[\rm (C)]
For every $\lambda \in \Lambda$, the function $G(\lambda, \cdot): \mathbb{R}^{2} \to \mathbb{R}$ is of class $C^2$, and its first and second partial derivatives with respect to the second variable,
 $\nabla_2G(\lambda, x)$ and $\nabla^2_2G\left(\lambda, x\right)$,   are continuous in $(\lambda, x)$.

 \item[\rm (D)] $\nabla_2G(\lambda, {0})=0$ and $\nabla^2_2G\left(\lambda, {0}\right)=\operatorname{diag}(e(\lambda), e(\lambda))$ for a function $e:\Lambda\to\mathbb{R}$ and all $\lambda\in\Lambda$.
\end{itemize}
Fix a delay parameter $\tau>0$.
 If $(\mu,0)$ with $\mu\in\Lambda$ is a bifurcation point of
 (\ref{e:Delay2Auto}) with $n=1$,  then $e(\mu)\in\frac{\pi\mathbb{Z}}{\tau}$.
Conversely, let $\mu$ be an interior point of $\Lambda$ satisfying
$e(\mu)\in\frac{\pi\mathbb{Z}}{\tau}$.
Then $(\mu,0)$ is a bifurcation point of (\ref{e:Delay2Auto}) with $n=1$,
provided that there exist sequences $(\lambda_m^\pm)_m\subset\Lambda$  converging to
$\mu$ such that
$e(\lambda_m^-)<e(\mu)<e(\lambda_m^+)$ for all $m\in\mathbb{N}$.
Furthermore, suppose that for some interior point
 $\mu$ of $\Lambda$, we have that $e(\mu)\in\frac{\pi}{\tau}\mathbb{Z}$
and that $e(\lambda)-e(\mu)\ne 0$ and changes sign as $\lambda$
across $\mu$
in a deleted neighborhood of $\mu$.
Then for  problem (\ref{e:Delay2Auto}) with $n=1$,
one of the following alternatives holds: 
 \begin{itemize}
\item[\rm (i)] The problem (\ref{e:Delay2Auto}) with $n=1$  and $\lambda=\mu$ possesses a sequence of $\mathbb{R}$-distinct solutions, $x^k\ne 0$ ($k=1,2,\dots$) such that
     $x^k\to 0$ in $C^1_{\rm loc}(\mathbb{R}, \mathbb{R}^2)$ as $k\to\infty$.

\item[\rm (ii)] There exist left and right  neighborhoods $\Lambda^-$ and $\Lambda^+$ of $\mu$ in $\mathbb{R}$
and nonnegative integers $n^+$ and $n^-$ satisfying $n^++n^-\ge 1$, such that
 for $\lambda\in\Lambda^-\setminus\{\mu\}$ (resp. $\lambda\in\Lambda^+\setminus\{\mu\}$), problem
(\ref{e:Delay2Auto}) with $n=1$ and parameter  $\lambda$  has at least $n^-$ (resp. $n^+$) $\R$-distinct
 solutions, $x_\lambda^i \neq 0$ ($i = 1, \dots, n^-$ resp. $n^+$)
  which, as $\lambda \to \mu$, converge to  zero in
   $C^1_{\rm loc}(\mathbb{R}, \mathbb{R}^2)$.
  \end{itemize}
\end{corollary}

Parallel to  Corollary~\ref{cor:bif-per3DelayC} we have:

\begin{corollary}\label{cor:bif-per3+4}
Let  $K: \mathbb{R}^{2n} \to \mathbb{R}$ be a $C^2$ function satisfying
$$
\nabla K({0})=0\quad\hbox{and}\quad
\nabla^2 K\left({0}\right)=
\left( \begin{array} { c c }
 	D & -E \\
 	E &D
 \end{array} \right),
 $$
where $D$ is an $n\times n$ symmetric matrix and  $E$ is an $n\times n$ skew-symmetric matrix.
Let $e_1\le e_2\le\cdots\le e_n$ be all eigenvalues of the Hermitian matrix
 $D+\sqrt{-1}E$.
Consider the problem
  \begin{eqnarray}\label{e:Delay2AutoC}
   \dot{x}( t ) =  J_n\nabla K( x ( t - \lambda) )\quad\text{and}\quad
   x( t + 2\lambda) =  x( t )\quad\forall t.
   \end{eqnarray}
  If $(\mu,0)$ with $\mu\in\mathbb{R}\setminus\{0\}$ is a bifurcation point of
(\ref{e:Delay2AutoC}), then $\mu e_{l}\in\pi\mathbb{Z}$ for some $l$.
Conversely, for some  $\mu\in\mathbb{R}$, suppose  that there exists a partition
  $\{i_1,\dots,i_q\}\cup\{i_{q+1},\dots,i_n\}$
  of $\{1,\dots,n\}$ with $1\le q\le n$ such that
    \begin{itemize}
\item[\rm (E)] $\mu e_{i_l}\in\pi\mathbb{Z}\setminus\{0\}$ for  $l=1,\dots,q$ (which implies that
$\mu\ne 0$ and $e_{i_l}\ne 0$ for $l=1,\dots,q$),
$e_{i_l}$($l=1,\dots,q$) have the same plus-minus sign, and
 $\mu e_{i_l}\notin\pi\mathbb{Z}$ for any $l=q+1,\dots,n$.
 \end{itemize}
 (In particular, if $n=1$, then $e_1=D$, and Condition (E) becomes $\mu D\in\pi\mathbb{Z}\setminus\{0\}$.)
  Then  one of the following alternatives occurs:
   \begin{itemize}
\item[\rm (i)] The problem (\ref{e:Delay2AutoC})   with $\lambda=\mu$
admits a sequence of nonzero,  $\R$-distinct solutions $\{x^k\}^{\infty}_{k=1}$ such that
$x^k\to 0$ in $C^1_{\rm loc}(\mathbb{R}, \mathbb{R}^{2n})$
as $k\to\infty$.
\item[\rm (ii)]
    There exist left and right neighborhoods $\Lambda^-$ and $\Lambda^+$ of $\mu$ in $\Lambda$,
and nonnegative integers $n^-$ and $n^+$ with $n^- + n^+ \geq q$,
such that for $\lambda \in \Lambda^- \setminus \{\mu\}$ (resp. $\lambda \in \Lambda^+ \setminus \{\mu\}$),
problem (\ref{e:Delay2AutoC}) with parameter $\lambda$ has at least $n^-$ (resp. $n^+$)
$\mathbb{R}$-distinct nonzero solutions
$x_\lambda^i \neq 0$ ($i = 1, \dots, n^-$ (resp. $n^+$)),
which converge to the zero function in $C^1_{\rm loc}(\mathbb{R}, \mathbb{R}^{2n})$
 as $\lambda \to \mu$.
 \end{itemize}
Moreover, if $n>1$ and $q>1$, then at least one of  (i), (iii), and (iv) holds, with (i) as stated previously, and with (iii) and (iv) given by:
\begin{itemize}
\item[\rm (iii)]  For every $\lambda\in\Lambda\setminus\{\mu\}$ near $\mu$, there is a
 nonzero solution ${x}_\lambda$  of problem (\ref{e:Delay2AutoC})
 with parameter  $\lambda$, such that $x_\lambda\to 0$
 in $C^1_{\rm loc}(\mathbb{R}, \mathbb{R}^{2n})$  as $\lambda\to\mu$.

\item[\rm (iv)] For a given $\varepsilon>0$,  there is a one-sided  neighborhood $\Lambda^0$ of $\mu$ in $\Lambda$ such that, for any $\lambda\in\Lambda^0\setminus\{\mu\}$,
    problem (\ref{e:Delay2AutoC}) with parameter  $\lambda$
    has either
    \begin{itemize}
    	\item [$\bullet$]
     infinitely many $\mathbb{R}$-distinct nonzero solutions $\{x_\lambda^l\}_{l=1}^\infty$ such that
$$
\|{x}_\lambda^l\|_{C^0([0, 2|\lambda|])}+|\lambda|
\|\dot{x}_\lambda^l\|_{C^0([0, 2|\lambda|])}<\varepsilon
$$
for $l=1,2,\dots$, or
\item[$\bullet$]  at least two $\mathbb{R}$-distinct nonzero solutions $\hat{x}_\lambda^1$ and $\hat{x}_\lambda^2$ satisfying  the following inequalities
   $$
\|\hat{x}_\lambda^i\|_{C^0([0, 2|\lambda|])}+|\lambda|
\|\dot{\hat{x}}_\lambda^i\|_{C^0([0, 2|\lambda|])}<\varepsilon, \;i=1,2,
$$
and
\begin{align*}
&&(\hat{x}^1_\lambda(0), J_n\hat{x}^1_\lambda(\lambda))_{\mathbb{R}^{2n}}+
 \int^{\lambda}_0(\dot{\hat{x}}^1_\lambda(t), J_n\hat{x}^1_\lambda(t-1))_{\mathbb{R}^{2n}}dt+\int^{\lambda}_{-\lambda}K(\hat{x}^1_\lambda(t))dt
\nonumber\\
&&\ne (\hat{x}^2_\lambda(0), J_n\hat{x}^2_\lambda(\lambda))_{\mathbb{R}^{2n}}+
 \int^{\lambda}_0(\dot{\hat{x}}^2_\lambda(t), J_n\hat{x}^2_\lambda(t-1))_{\mathbb{R}^{2n}}dt+\int^{\lambda}_{-\lambda}K(\hat{x}^2_\lambda(t))dt.
\end{align*}
\end{itemize}
\end{itemize}
\end{corollary}

\begin{proof}
	[\bf Proof of Theorem~\ref{th:bif-per3+3}]
\noindent{\bf Step~1}.
Under Assumption~\ref{ass:bif-per2Delay2Auto}, define $\hat{H}:\Lambda\times ({\mathbb{R}}^{2n})^2\to\mathbb{R}$  by
\begin{equation}\label{e:Delay14}
\hat{H} \left(\lambda, (x_{ 1 }^\top ,  x_{ 2 }^\top)^\top \right) =
 G\left(\lambda,  x_{ 1 } \right) +  G\left(\lambda,  x_{ 2 } \right).
\end{equation}
Denote by $\nabla_2 \hat{H}(\lambda,  v)$ and
$\nabla_2^2 \hat{H}(\lambda,  v)$ the gradient and the Hessian of $\hat{H}\left(\lambda, v \right)$ with respect to the second variable $v\in({\mathbb{R}}^{2n})^2$, respectively.
Then
\begin{equation}\label{e:Delay29}
\nabla^2_2\hat{H}\left(\lambda, (x_{ 1 }^\top ,  x_{ 2 }^\top)^\top \right) =
\left( \begin{array} { c c }
\nabla^2_2G\left(\lambda, x_1 \right) & 0   \\
 0 &
\nabla^2_2G\left(\lambda, x_2 \right)\end{array}
\right).
\end{equation}
Define $4n \times 4n$ matrices  $J_ {n, 2 }$ and $P_{n,2}$  by
\begin{equation}\label{e:Delay17}
J _ { n , 2 } = \left( \begin{array} { c c c c c }
0 & J _ { n } &  \\
J _ { n } & 0 &  \end{array} \right)
\quad\text{and}\quad
P_{n, 2 } =  \left( \begin{array} { c c c }
0 & I _ {2n } \\ I _ { 2n } & 0 \end{array} \right),
\end{equation}
 respectively. These matrices satisfy the following properties:
\begin{align}\label{e:Delay18}
&P_{n , 2 }^{ - 1 } = P_{n,2 }^{ \top }, \quad P_{n,2 }^{2 } = I_{2n},\quad
P_{n, 2 } J_{ n,2 } = J_{n,2}P_{n, 2 }, \nonumber\\
&\hat{H}(\lambda, z)=\hat{H}(\lambda, P_{n,2}z),\quad\forall (\lambda,t,z).
\end{align}
As in Claim~\ref{cl:delay1.1} we have (see \cite[\S1.4]{Liu12}):

\begin{claim}\label{cl:delay2.1}
If  $x$ is a  solution of (\ref{e:Delay2Auto}),
then the function
\begin{equation}\label{e:Delay15}
\mathbb{R}\ni t\mapsto v(t):=(x_1(t)^\top, x_2(t)^\top)^\top\in ({\mathbb{R}}^{2n})^2
\end{equation}
where  $x_{ i } ( t ) = x ( t-(i-1)\tau)$ for $i=1,2$,
satisfies
\begin{equation}\label{e:Delay16}
\dot{v}(t)=J_{n,2}\nabla_2 \hat{H}(\lambda, v(t))\quad\text{and}\quad v(t+\tau)=P_{n,2}^{-1}v(t)\;\;\forall t\in \mathbb{R}.
\end{equation}
Conversely, if $v(t)=(x_1(t)^\top, x_2(t)^\top)^\top$,
where $x_{ i } ( t )\in\mathbb{R}^{2n}$ for $i=1,2$,
satisfies (\ref{e:Delay16}),
then $x(t):=x_1(t)$ is a  solution of (\ref{e:Delay2Auto}).
\end{claim}

\noindent{\bf Step~2}.
Since both $J_{n,2}$ and $J_{2n}$  are non-degenerate skew-symmetric bilinear forms on $({\mathbb{R}}^{2n})^2$,
there exists a congruence transformation between $J_{n,2}$ and $J_{2n}$,
given by an invertible real
matrix $\Gamma(2, 2n)$ of order $4n$ satisfying
$\Gamma(2, 2n)^\top J_{n,2} \Gamma(2, 2n) = J_{2n}$.
For example, one explicit example of such matrices is
 \begin{equation}\label{e:LinearDelay2G+}
\Gamma(2,2n)=\left( \begin{array} { c c c } -J_n & 0 \\ 0 &I _ { 2n }\end{array} \right).
\end{equation}
In this case, consider the matrix defined by
\begin{align}\label{e:LinearDelay2G++}
M_{2,n}:=\Gamma(2, 2n)^\top P_{n,2}^{-1}(\Gamma(2, 2n)^\top)^{-1}&=\left( \begin{array} { c c c } J_n & 0 \\ 0 &I _ { 2n }\end{array} \right)
\left( \begin{array} { c c c } 0 & I _ { 2n} \\ I _ { 2n } & 0 \end{array} \right)
\left( \begin{array} { c c c } -J_n & 0 \\ 0 &I _ { 2n }\end{array} \right)
\nonumber\\
&=\left( \begin{array} { c c c } 0 & J_n \\ -J_n & 0 \end{array} \right)
\in {\rm Sp}(4n)^0.
\end{align}
(The final inclusion relation follows from the equality
 $D(M_{2,n}):=(-1)^{2n-1}\det(M_{2,n}-I_{4n})=0$, which is given by Lemma~3 in \cite[page~38]{Long02}.)
 $M_{2,n}$  is   orthogonal and satisfies $(M_{2,n})^2=I_{4n}$.



Define $\check{H}:\Lambda\times({\mathbb{R}}^{2n})^2\to\mathbb{R}$ by
\begin{equation}\label{e:Delay30.2}
\check{H}(\lambda,z)=\hat{H}(\lambda, (\Gamma(2, 2n)^\top)^{-1}z)=
\hat{H}(\lambda, \Gamma(2, 2n)z).
\end{equation}
Then, with the matrix $M_{2,2n}$ given by (\ref{e:LinearDelay2G++}),
for all $z=(x_1^\top, x_2^\top)^\top\in(\mathbb{R}^{2n})^2$,
we have
\begin{equation}\label{e:Delay14check}
\check{H}(\lambda, M_{2,n}z)=G(\lambda,x_2)+G(\lambda,-J_nx_1)=\check{H}(\lambda, z),\quad\forall\lambda\in\Lambda.
\end{equation}
By a straightforward computation, we have

\begin{claim}\label{cl:delay2.2}
$v:\mathbb{R}\to ({\mathbb{R}}^{2n})^2$ solves  the problem (\ref{e:Delay16})
 if and only if
$u(t):=\Gamma(2, 2n)^\top v(t)$ satisfies 
\begin{equation}\label{e:Delay6Auto3G}
\dot{u}(t)={J}_{2n} \nabla_2 \check{H}(\lambda,  u(t))\quad\text{and}\quad
u(t+\tau)={M}_{2,n}u(t)\;\;\forall t\in\mathbb{R},
\end{equation}
and in this situation there holds
$$
\int^{\tau}_0\left[\frac{1}{2}(J_{2n}\dot{u}(t), u(t))_{\mathbb{R}^{4n}}+ \check{H}(\lambda, u(t))\right]dt
= \int^{\tau}_0\left[\frac{1}{2}(\dot{v}(t), J_{n,2}^{-1}v(t))_{\mathbb{R}^{4n}}+
\hat{H}(\lambda,  v(t))\right]dt.
$$
Moreover,  the integral on the left-hand side of the preceding equality equals
$$
({x}^\lambda(0), J_n{x}^\lambda(\tau))_{\mathbb{R}^{2n}}+
 \int^{\tau}_0(\dot{{x}}^\lambda(t), J_n{x}^\lambda(t-\tau))_{\mathbb{R}^{2n}}dt+\int^{\tau}_{-\tau}G(\lambda,t,
 {x}^\lambda(t))dt
$$
provide that  $x$ is a  solution of (\ref{e:Delay2Auto}),
and $v$ is given by (\ref{e:Delay15}).
\end{claim}
In order to see the final claim, note that $(J_{n,2})^{-1}=-J_{n,2}$ implies
$$
(\dot{{v}}(t), J_{n,2}^{-1}{v}(t))_{\mathbb{R}^{4n}}=
(\dot{{x}}_1(t), J_{n}{x}_2(t))_{\mathbb{R}^{2n}}+
(\dot{{x}}_2(t), J_{n}{x}_1(t))_{\mathbb{R}^{2n}}.
$$
Applying integration by parts yields
$$
\int^{\tau}_0(\dot{{v}}(t), J_{n,2}^{-1}{v}(t))_{\mathbb{R}^{4n}}=
2({{x}}(0), J_n{x}(\tau))_{\mathbb{R}^{2n}}+
2\int^{\tau}_0({{x}}(t), J_n{x}(t-\tau))_{\mathbb{R}^{2n}}dt.
$$
since ${x}$ is $2\tau$-periodic. The latter and
the identity $\hat{H}(\lambda, {v}(t))=G(\lambda,
{x}(t))+ G(\lambda, {x}(t-\tau))$ yield
$$
\int^\tau_0\hat{H}(\lambda, {v}(t))dt=\int^\tau_{-\tau}G(\lambda,
{x}(t))dt.
$$

\noindent{\bf Step~3}.
By Claims~\ref{cl:delay2.1} and \ref{cl:delay2.2}, the trivial solution of (\ref{e:Delay2Auto}) corresponds to the trivial solution of (\ref{e:Delay16}) and to the trivial solution of (\ref{e:Delay6Auto3G}). All of them are denoted by ${0}$ without confusion.
Consider the linearizations of (\ref{e:Delay2Auto}),
(\ref{e:Delay16}) and (\ref{e:Delay6Auto3G}) at trivial solutions,
\begin{align}
&\dot{x}( t ) = J_n\nabla_2^2 G (\lambda, {0})x(t)\quad\text{and}\quad
    x( t + 2\tau ) =  x( t ), \quad \forall t,\label{e:LineDelay2Auto}\\
&\dot{v}(t)=J_{n,2}\nabla_2^2 \hat{H}(\lambda, {0})v(t)\quad\text{and}\quad v(t+\tau)=P_{n,2}^{-1}v(t)\;\;\forall t\in \mathbb{R},\label{e:LineDelay16}\\
&\dot{u}(t)={J}_{2n} \nabla_2^2 \check{H}(\lambda,  {0})u(t)\quad\text{and}\quad
u(t+\tau)={M}_{2,n}u(t)\;\;\forall t\in\mathbb{R}.\label{e:LineDelay6Auto3G}
\end{align}
It follows from Claims~\ref{cl:delay2.1} and \ref{cl:delay2.2} that
\begin{align*}
 x:\mathbb{R}\to\mathbb{R}^{2n}\;\text{satisfies (\ref{e:LineDelay2Auto})}\;
 &\Longleftrightarrow
  \mathbb{R}\ni t\mapsto v(t):=(x(t)^\top, x(t-\tau)^\top)^\top\in ({\mathbb{R}}^{2n})^2
\;\text{satisfies (\ref{e:LineDelay16})}\\
&\Longleftrightarrow
  \mathbb{R}\ni t\mapsto u(t):=\Gamma(2, 2n)^\top v(t)\in ({\mathbb{R}}^{2n})^2
\;\text{satisfies (\ref{e:LineDelay6Auto3G})}.
\end{align*}
Therefore, we obtain
\begin{claim}\label{cl:delay2.3}
The solution spaces of problems (\ref{e:LineDelay2Auto}), (\ref{e:LineDelay16}) and (\ref{e:LineDelay6Auto3G}) have the same dimension.
\end{claim}


As in the proof of Theorem~\ref{th:bif-per3Delay}, all conclusions may follow from \cite[Theorem~1.14]{Lu11}.
 Firstly, as in (\ref{e:EigenContin}), by \cite[Corollary~6.3.8]{HorJ} and  (\ref{e:V-function1G}) we obtain
  \begin{align}\label{e:EigenContinG}
\sum^n_{l=1}|e_l(\lambda)-e_l(\mu)|^2&\le
2\|D(\lambda)+\sqrt{-1}E(\lambda)
-D(\mu)-\sqrt{-1}E(\mu)\|_2^2\nonumber\\
&=\|\nabla^2_xG\left(\lambda, {{0}}\right)-\nabla^2_xG\left(\mu, {{0}}\right)\|^2_2,
\quad\forall\lambda,\mu\in\Lambda.
\end{align}

Note that $\Gamma(2,2n)$ is an orthogonal matrix.
By the second equality in (\ref{e:M-invariantDelay3}) we obtain
\begin{align*}
\nabla^2_2 \check{H}(\lambda, {{0}})&=\Gamma(2, 2n)^{-1}\nabla^2_2 \hat{H}(\lambda, {{0}})\Gamma(2, 2n)\\
&=\left( \begin{array} {cc}
J_{n} & 0 \\
0 &I_{2n}
 \end{array} \right)^{-1}\left( \begin{array} { c c c c }
\nabla^2_xG\left(\lambda, {{0}}\right) & 0  \\
0 &\nabla^2_xG\left(\lambda, {{0}}\right)
 \end{array} \right)\left( \begin{array} {cc}
J_{n} & 0 \\
0 &I_{2n}
 \end{array} \right)\\
 &=\left( \begin{array} {cc}
-J_n\nabla^2_xG\left(\lambda, {{0}}\right)J_n & 0 \\
0 &\nabla^2_xG\left(\lambda, {{0}}\right)
 \end{array} \right)\\
 &=\left( \begin{array} {cc}
\nabla^2_xG\left(\lambda, {{0}}\right) & 0 \\
0 &\nabla^2_xG\left(\lambda, {{0}}\right)
 \end{array} \right),
 \end{align*}
where the final equality comes from $\nabla^2_xG\left(\lambda, {{0}}\right)=\left( \begin{array} {cc}
D(\lambda) & -E(\lambda) \\
E(\lambda) &D(\lambda)
 \end{array} \right)$ by Assumption~\ref{ass:bif-per2Delay2Auto}.
Denote by  $\Upsilon_\lambda$  the fundamental matrix solution of
$$
\dot{z}(t)=J_{2n}\nabla^2_2 \check{H}(\lambda, {{0}})z(t)=
J_{2n}{\rm Diag}\left(\left( \begin{array} {cc}
D(\lambda) & -E(\lambda) \\
E(\lambda) &D(\lambda)
 \end{array} \right), \left( \begin{array} {cc}
D(\lambda) & -E(\lambda) \\
E(\lambda) &D(\lambda)
 \end{array} \right)\right)z(t)
$$
 with $\Upsilon_\lambda(0)=I_{4n}$.

\noindent\textbf{Step 4}(\textsf{Proof of (I)}).
Assume that $(\mu,0)$ with $\mu\in\Lambda$ is a bifurcation point of (\ref{e:Delay2Auto}). By Claims~\ref{cl:delay2.1} and \ref{cl:delay2.2}, this holds if and only if (\ref{e:Delay6Auto3G})
admits a bifurcation point at $(\mu,0)$ for this very $\mu$. Consequently, $(\mu,0)$ is a bifurcation point of (\ref{e:Delay6Auto3G}). Then \cite[Theorem~1.5]{Lu11} yields $\nu_{\tau, M_{3,n}}(\Upsilon_\mu)\ne 0$. Note that
$$
\nu_{\tau, M_{2,n}}(\Upsilon_\mu) = 2\sharp\{\, 1\le j\le n \mid \tau e_j(\mu) \in \pi\mathbb{Z} \,\}
$$
by (\ref{e:DiagPindex2*}) in Theorem~\ref{th:PIndex4}. This leads to the
 conclusion in (I).

\noindent\textbf{Step 5}(\textsf{Proof of (III)}).
 For each $j=1,\dots,n$, let $m_j(\lambda)$ be the unique integer such that
$m_j(\lambda)< e_j(\lambda)\tau\le (m_j(\lambda)+1)\pi$.
By the condition (A) we have
\begin{align}\label{e:EigenContinG1A}
 &\{k\,|\,\tau e_k(\mu)\in\pi\mathbb{Z}\}=
\{j\,|\, e_j(\mu)\tau=(m_j(\lambda)+1)\pi\}=\{i_1,\cdots,i_q\},\\
 &e_{i_l}(\mu)\tau\in\bigl(m_{i_l}(\mu)\pi,   (m_{i_l}(\mu)+1)\pi\bigr), \quad l=q+1,\dots, n.\label{e:EigenContinG1B}
 \end{align}
From (\ref{e:EigenContinG1A}) and  (\ref{e:DiagPindex1*})
 in Theorem~\ref{th:PIndex4}, we derive from
\begin{equation}\label{e:EigenContinG1C}
\nu_{\tau, M_{2,n}}(\Upsilon_{\mu})=2q.
\end{equation}
By (\ref{e:DiagPindex2*}) in Theorem~\ref{th:PIndex4}
 we obtain
  \begin{align}\label{e:EigenContinG8A}
 i_{\tau, M_{2,n}}(\Upsilon_\mu)&=2\sum^n_{l=1}m_{i_l}(\mu)+n.
\end{align}

By (\ref{e:EigenContinG}) and (\ref{e:EigenContinG1B})
 we can shrink $\epsilon>0$ so that
\begin{equation}\label{e:EigenContinG2}
e_{i_l}(\lambda)\tau\in\bigl(m_{i_l}(\mu)\pi,   (m_{i_l}(\mu)+1)\pi\bigr)
\end{equation}
for any $(\lambda, l)\in (\mu-\epsilon, \mu+\epsilon)\times\{q+1,\cdots,n\}$.
 We can further shrink $\epsilon>0$ so that
 \begin{equation}\label{e:EigenContinG3A}
e_{i_l}(\lambda)\tau\in\begin{cases}
\left(m_{i_l}(\mu)\pi, (m_{i_l}(\mu)+1)\pi\right)\quad&\hbox{if $\lambda\in (\mu-\epsilon, \mu)$},\vspace{1mm}\\
\left((m_{i_l}(\mu)+1)\pi, (m_{i_l}(\mu)+2)\pi\right)\quad&\hbox{if $\lambda\in (\mu, \mu+\epsilon)$}
\end{cases}
\end{equation}
for $l=1,\dots,q$ if the condition (SB.1) holds, and
 \begin{equation}\label{e:EigenContinG3B}
e_{i_l}(\lambda)\tau\in\begin{cases}
\left((m_{i_l}(\mu)+1)\pi, (m_{i_l}(\mu)+2)\pi\right)\quad&\hbox{if $\lambda\in (\mu-\epsilon, \mu)$},\vspace{1mm}\\
\left(m_{i_l}(\mu)\pi, (m_{i_l}(\mu)+1)\pi\right)\quad&\hbox{if $\lambda\in (\mu, \mu+\epsilon)$}
\end{cases}
\end{equation}
for $l=1,\dots,q$ if the condition (SB.2) is satisfied.

By (\ref{e:DiagPindex1*}) in Theorem~\ref{th:PIndex4},
 under two cases we have always
\begin{equation}\label{e:EigenContinG4}
\nu_{\tau, M_{2,n}}(\Upsilon_\lambda)=0,\quad \forall\lambda\in (\mu-\epsilon, \mu+\epsilon)\setminus\{\mu\}.
\end{equation}

Suppose now the condition (SB.2) holds.
We compute $i_{\tau, M_{2,n}}(\Upsilon_\lambda)$
for $\lambda\in (\mu-\epsilon, \mu+\epsilon)\setminus\{\mu\}$.

 Observe that (\ref{e:EigenContinG2}) and (\ref{e:EigenContinG3B}) imply, respectively,
 \begin{equation}\label{e:EigenContinG8B}
m_{i_l}(\lambda)=m_{i_l}(\mu)\quad\forall
(\lambda, l)\in (\mu-\epsilon, \mu+\epsilon)\times\{q+1,\cdots,n\},
\end{equation}
and
\begin{equation}\label{e:EigenContinG8C}
m_{i_l}(\lambda)=\begin{cases}
(m_{i_l}(\mu)+1)\pi\quad&\hbox{if $\lambda\in (\mu-\epsilon, \mu)$},\vspace{1mm}\\
m_{i_l}(\mu)\pi\quad&\hbox{if $\lambda\in (\mu, \mu+\epsilon)$}
\end{cases}
\end{equation}
for $l=1,\dots,q$.
Using (\ref{e:EigenContinG8B}) and (\ref{e:EigenContinG8C}),
   we derive from (\ref{e:DiagPindex2*}) in Theorem~\ref{th:PIndex4}
  \begin{align*}\label{e:EigenContinG8D}
 i_{\tau, M_{2,n}}(\Upsilon_\lambda)&=2\sum^n_{l=1}m_{i_l}(\mu)+
  n=i_{\tau, M_{2,n}}(\Upsilon_\mu)\quad\forall\lambda\in (\mu, \mu+\epsilon),\\
  i_{\tau, M_{2,n}}(\Upsilon_\lambda)&=2\sum^n_{l=q+1}m_{i_l}(\mu)+
  2\sum^q_{l=1}(m_{i_l}(\mu)+1)+
  n\nonumber\\
  &=i_{\tau, M_{2,n}}(\Upsilon_\mu)+2q
  \quad\forall\lambda\in (\mu-\epsilon, \mu).
  \end{align*}
 That is,  the conditions (A) and (SB.2), together with (\ref{e:EigenContinG1C}), imply
\begin{equation}\label{e:EigenContinG8D}
i_{\tau, M_{2,n}}(\Upsilon_\lambda)=\begin{cases}
i_{\tau, M_{2,n}}(\Upsilon_\mu)+\nu_{\tau, M_{2,n}}(\Upsilon_{\mu})\quad&\hbox{if $\lambda\in (\mu-\epsilon, \mu)$},\vspace{1mm}\\
i_{\tau, M_{2,n}}(\Upsilon_\mu)\quad&\hbox{if $\lambda\in [\mu, \mu+\epsilon)$}.
\end{cases}
\end{equation}

Similarly, under the conditions (A) and (SB.1),
from (\ref{e:EigenContinG3A}) we can derive
 \begin{equation}\label{e:EigenContinG8E}
i_{\tau, M_{2,n}}(\Upsilon_\lambda)=\begin{cases}
i_{\tau, M_{2,n}}(\Upsilon_\mu)\quad&\hbox{if $\lambda\in (\mu-\epsilon, \mu]$},\vspace{1mm}\\
i_{\tau, M_{2,n}}(\Upsilon_\mu)+\nu_{\tau, M_{2,n}}(\Upsilon_{\mu})\quad&\hbox{if $\lambda\in (\mu, \mu+\epsilon)$}.
\end{cases}
\end{equation}

The above equations (\ref{e:EigenContinG4}),
(\ref{e:EigenContinG8D}) and (\ref{e:EigenContinG8E})  show that the condition (b) in \cite[Theorem~1.14]{Lu11} is satisfied
for $H=\check{H}$.

  Although  ${\rm Ker}(M_{2,n}-I_{4n})=\{(u_1^\top, u_2^\top, u_2^\top, -u_1^\top)^\top\,|\, u_1, u_2\in\mathbb{R}^n\}\ne \{0\}$,
  we nevertheless have
    $$
  {\rm Ker}(M_{2,n}-I_{4n})\cap{\rm Ker}(\nabla_z^2\check{H}({\mu}, {{0}}))=\{0\}.
  $$
  This follows because condition (A) implies that  $D(\mu)+\sqrt{-1}E(\mu)$ is non-degenerate, which in turn forces
  ${\rm Ker}(\nabla_x^2{G}({\mu}, {{0}}))=\{0\}$ and hence
    ${\rm Ker}(\nabla_z^2\check{H}({\mu}, {{0}}))=\{0\}$.
    Consequently, the above equality shows that condition (a) of Theorem~1.14 in \cite{Lu11} is satisfied for $H=\check{H}$.

Applying the first part of \cite[Theorem~1.14]{Lu11} to the problem (\ref{e:Delay6Auto3G})  we get
the following alternatives:
 \begin{enumerate}
\item[\rm (i')] The problem (\ref{e:Delay6Auto3G}) with $\lambda=\mu$ has a sequence of $\mathbb{R}$-distinct solutions,
$z^l\ne 0$ ($l=1,2,\dots$) such that $\|z^l|_I\|_{C^1}\to 0$ for any compact interval $I\subset\R$;
\item[\rm (ii')] There exist left and right  neighborhoods $\Lambda^-$ and $\Lambda^+$ of $\mu$ in $\Lambda$
and nonnegative integers $n^+$ and $n^-$ satisfying $n^++n^-\ge \nu_{\tau, M_{2,n}}(\Upsilon_\mu)/2=q$,
such that for $\lambda\in\Lambda^-\setminus\{\mu\}$ (resp. $\lambda\in\Lambda^+\setminus\{\mu\}$),
the problem (\ref{e:Delay6Auto3G}) with parameter value $\lambda$  has at least $n^-$ (resp. $n^+$) $\mathbb{R}$-distinct
 solutions, $z_\lambda^i\ne 0$ ($i=1,\dots,n^-$ resp. $n^+$)
such that all $\|z_\lambda^i|_I\|_{C^1}\to 0$  for any compact interval $I\subset\R$  as $\lambda\to\mu$.
\end{enumerate}
As in the proof of Theorem~\ref{th:bif-per3Delay}
let us define $v^l=\Gamma(2,2n)z^l$ ($l=1,2,\dots$), and $v^i_\lambda=\Gamma(2,2n)z_\lambda^i\ne 0$ ($i=1,\dots,n^-$ resp. $n^+$).
Then all $v^l\ne 0$ and satisfy the problem (\ref{e:Delay16}) with $\lambda=\mu$;
and for $\lambda\in\Lambda^-\setminus\{\mu\}$ (resp. $\lambda\in\Lambda^+\setminus\{\mu\}$),
all $v_\lambda^i\ne 0$ ($i=1,\dots,n^-$ (resp. $n^+$)) and satisfy the problem (\ref{e:Delay16}) with parameter value $\lambda$.
Let us write
$$
v^l(t):=({x}^l_1(t)^\top, {x}^l_2(t)^\top)^\top\in ({\mathbb{R}}^{2n})^2\quad\text{and}\quad
v^i_\lambda(t):=({x}^i_{\lambda,1}(t)^\top, {x}^i_{\lambda,2}(t)^\top)^\top\in ({\mathbb{R}}^{2n})^2.
$$
By Claim~\ref{cl:delay2.1}, ${x}^l:={x}^l_1$ and ${x}^i_{\lambda}:= {x}^i_{\lambda,1}$
satisfy the expected alternatives (i) and (ii).

Next, if $n>1$ and $q>1$, then
$\nu_{\tau, M_{2,n}}(\Upsilon_{\mu})=2q>3$ by (\ref{e:EigenContinG1C}).
Applying the second part of Theorem~1.14 in  \cite{Lu11} to the problem (\ref{e:Delay6Auto3G}),
other conclusions can be obtained as in the proof of Theorem~\ref{th:bif-per3Delay} via Claim~\ref{cl:delay2.2}.

\noindent\textbf{Step 6}(\textsf{Proof of (II)}).
A slight modification of the proof of (II) in Theorem~\ref{th:bif-per3Delay} or \ref{th:bif-per3DelayII} yields the desired conclusion.
\end{proof}

\begin{proof}[\bf Proof of Corollary~\ref{cor:bif-per3+4}]
Define $\Lambda=\mathbb{R}$ and $G:\mathbb{R}\times\mathbb{R}^{2n}
\to\mathbb{R}$ by $G(\lambda,x)=\lambda K(x)$.
Then $G$ satisfies Assumption~\ref{ass:bif-per2Delay2Auto},
 and $e_i(\lambda)=\lambda e_i$ ($i=1,\dots,n$).

Suppose $(\mu,0)$ with $\mu\in\mathbb{R}\setminus\{0\}$ is a bifurcation point of
(\ref{e:Delay2AutoC}). One readily verifies that
$(\mu,0)$  is a bifurcation point of the system
  \begin{equation}\label{e:Delay2AutoC**}
   \dot{x}( t ) = \lambda J_n\nabla K( x ( t - 1) )\quad\text{and}\quad
   x( t + 2) =  x( t )\;\forall t.
    \end{equation}
Hence, by Theorem~\ref{th:bif-per3+3} with $\tau = 1$,
we obtain $\mu e_l\in\pi\mathbb{Z}$ for some $l$,
which establishes the first conclusion.

We now prove the remaining conclusions.
  Let $\mu\in\mathbb{R}$ satisfy condition (E).
  It is easy to see that $G$ satisfies
condition (A) and one of conditions (SB.1) or (SB.2) in  Theorem~\ref{th:bif-per3+3}
with $\tau = 1$. Therefore, for the problem (\ref{e:Delay2AutoC**}),
  one of the following alternatives occurs:
   \begin{enumerate}
\item[\rm (i')] The problem (\ref{e:Delay2AutoC**})  with $\lambda=\mu$ has a sequence of
nonzero, $\R$-distinct solutions $\{u^l\}^\infty_{l=1}$ such that
$u^l\to 0$ in $C^1_{\rm loc}(\mathbb{R}, \mathbb{R}^{2n})$
as $l\to\infty$. for any compact interval $I\subset\R$.
\item[\rm (ii')]
    There exist left and right neighborhoods $\Lambda^-$ and $\Lambda^+$ of $\mu$ in $\Lambda$ not containing the point
    $0\in\mathbb{R}$,
and nonnegative integers $n^-$ and $n^+$ with $n^- + n^+ \geq q$,
such that for $\lambda \in \Lambda^- \setminus \{\mu\}$ (resp. $\lambda \in \Lambda^+ \setminus \{\mu\}$),
the equation (\ref{e:Delay2AutoC**}) with parameter $\lambda$ has at least $n^-$ (resp. $n^+$) $\mathbb{R}$-distinct solutions
$u_\lambda^i \neq 0$ ($i = 1, \dots, n^-$ resp. $n^+$),
which converge to zero in $C^1_{\rm loc}(\mathbb{R}, \mathbb{R}^{2n})$ as $\lambda \to \mu$.
the zero function in the $C^1$ norm on every compact interval $I \subset \mathbb{R}$ as $\lambda \to \mu$.
 \end{enumerate}
Moreover, if $n>1$ and $q>1$, then at least one of (i'), (iii'), and (iv')  holds,
where (i) is as stated previously, and
\begin{enumerate}
\item[\rm (iii')]  For every $\lambda\in\Lambda\setminus\{\mu\}$ near $\mu$, there is a
 nonzero solution ${u}_\lambda$  of (\ref{e:Delay2AutoC**})
 with parameter value $\lambda$, such that $u_\lambda$ converges to zero
 $C^1_{\rm loc}(\mathbb{R}, \mathbb{R}^{2n})$  as $\lambda \to \mu$.
 in $C^1$ on any compact interval $I\subset\R$ as $\lambda\to\mu$.

\item[\rm (iv')] For a given $\varepsilon>0$,  there is a one-sided  neighborhood $\Lambda^0$ of $\mu$ in $\Lambda$ not containing the point
$0\in\mathbb{R}$, such that
for any $\lambda\in\Lambda^0\setminus\{\mu\}$, (\ref{e:Delay2AutoC**}) with parameter value $\lambda$
has either:
\begin{enumerate}
\item[$\bullet$]  infinitely many $\mathbb{R}$-distinct nonzero solutions ${x}_\lambda^l$
satisfying $\|{u}_\lambda^l|_{[0,2]}\|_{C^1}<\varepsilon$, $l=1,2,\dots$, or
\item[$\bullet$] at least two $\mathbb{R}$-distinct nonzero solutions $\hat{u}_\lambda^1$ and $\hat{u}_\lambda^2$ satisfying $\|\hat{u}_\lambda^i|_{[0, 2]}\|_{C^1}<\varepsilon$, $i=1,2$, and such that the following integral inequality holds:
 \begin{eqnarray*}
	&&(\hat{u}^1_\lambda(0), J_n\hat{u}^1_\lambda(1))_{\mathbb{R}^{2n}}+
	\int^{1}_0(\dot{\hat{u}}^1_\lambda(t), J_n\hat{u}^1_\lambda(t-1))_{\mathbb{R}^{2n}}dt+\lambda\int^{1}_{-1}K(\hat{u}^1_\lambda(t))dt
	\nonumber\\
	&&\ne (\hat{u}^2_\lambda(0), J_n\hat{u}^2_\lambda(1))_{\mathbb{R}^{2n}}+
	\int^{1}_0(\dot{\hat{u}}^2_\lambda(t), J_n\hat{u}^2_\lambda(t-1))_{\mathbb{R}^{2n}}dt+\lambda\int^{1}_{-1}K(\hat{u}^2_\lambda(t))dt.
\end{eqnarray*}
\end{enumerate}
\end{enumerate}
Since $\mu\ne 0$ and all $\lambda$
appearing in $u^i_\lambda$, $u_\lambda$
and $\hat{u}^s_\lambda$($s=1,2$)
are nonvanishing,
as in the proof of Corollary~\ref{cor:bif-per3DelayC},
we may define
$$
x_\lambda^i(t):=u_\lambda^i(t/\lambda),\quad
{x}_\lambda^l(t):={u}_\lambda^l(t/\lambda),\quad
\hat{x}_\lambda^1(t):=\hat{u}_\lambda^1(t/\lambda),\quad
\hat{x}_\lambda^2(t):=\hat{u}_\lambda^2(t/\lambda).
$$
These satisfy Corollary~\ref{cor:bif-per3+4}.
\end{proof}

%
%
%
%

Following the approach in the proof of Theorem~\ref{th:bif-per2Delay++}, we obtain from Corollary~1.15 of \cite{Lu11} a parallel result, which serves as a complement to Theorem~\ref{th:bif-per3+3}.

\begin{theorem}\label{th:bif-per2Delay++G}
Let $G_0, G_1:  \mathbb{R}^{2n} \to \mathbb{R}$ be $C^2$ functions such that
for $j=0,1$,
$$
\nabla G_j({0})=0\quad\hbox{and}\quad
\nabla^2 G_j\left({0}\right)=
\left( \begin{array} { c c }
	D_j & -E_j \\
	E_j &D_j
\end{array} \right),
$$
where $D_j$ are $n\times n$ symmetric matrices and  $E_j$ are $n\times n$ skew-symmetric matrices.
 Suppose in addition that
  $D_1+\sqrt{-1}E_1$ is  definite (either positive or negative).
Denoted the zero solution of the problem
\begin{eqnarray}\label{e:LinearDelay2+G}
 \left\{\begin{array}{ll}
 &\dot{x}( t ) = J_n\nabla G_0 (x(t-\tau))+
 \lambda J_n\nabla G_1 (x ( t - \tau )),\\
 & x( t + 2 \tau ) = - x( t )\;\forall t
   \end{array}\right.
  \end{eqnarray}
  by ${0}^\lambda$, and
 the dimension of the solution space of the linear delay problem
 \begin{eqnarray}\label{e:LinearDelay2G}
 \left\{\begin{array}{ll}
 &\dot{x}( t ) = J_n\nabla^2 G_0 ({0})x(t-\tau) +
 \lambda J_n\nabla^2 G_1 ({0})x ( t - \tau ),\\
 & x( t + 2\tau ) = - x( t )\;\forall t
   \end{array}\right.
  \end{eqnarray}
by $\nu_{\tau}({0}^\lambda)$. 
Then
\begin{enumerate}
\item[\rm (I)] $\Sigma_\tau:=\{\lambda\in\mathbb{R}\,|\, \nu_{\tau}({0}^\lambda)>0\}$
is a discrete set in $\mathbb{R}$.

\item[\rm (II)]
 If $\mu\in{\Sigma}_\tau$ is not an eigenvalue of $(D_1+\sqrt{-1}E_1)^{-1}
 (D_0+\sqrt{-1}E_0)$,
	 at least  one of the following alternatives holds true:
	\begin{enumerate}
		\item[\rm (i)] System (\ref{e:LinearDelay2+G}) with $\lambda=\mu$ admits a sequence of $\R$-distinct nonzero solutions
		${x}^{k}$ for $k=1,2,\cdots$,
		which  converges to zero in $C^1_{\rm loc}(\mathbb{R},\R^{2n})$ as $k\to\infty$.
		\item[\rm (ii)]
		There exist a left neighborhood $\Lambda^-$ and a right neighborhood $\Lambda^+$ of $\mu$ in $\mathbb{R}$,
		and nonnegative integers $n^+$ and $n^-$ with $n^++n^-\ge \nu_{\tau}(0^\mu)/2$,
		such that for each $\lambda\in\Lambda^-\setminus\{\mu\}$ (resp. $\lambda\in\Lambda^+\setminus\{\mu\}$),
		system (\ref{e:LinearDelay2+G}) with parameter $\lambda$ possesses at least $n^-$ (resp. $n^+$) $\mathbb{R}$-distinct nonzero solutions $x^i_{\lambda}$, $i=1,\dots,n^-$ (resp. $i=1,\dots,n^+$),
		which satisfy $x^i_{\lambda}\to 0$ in $C^1_{\rm loc}(\mathbb{R},\mathbb{R}^{2n})$ as $\lambda\to\mu$.
			\end{enumerate}
	Moreover, if $\nu_{\tau}(0^\mu)\ge 3$,
then at least one of (i), (iii) and (iv) holds, where (i) is as stated above, and the assertions of (iii) and (iv) are as follows:
	\begin{enumerate}
		\item[\rm (iii)]
		For every $\lambda\in\Lambda\setminus\{\mu\}$ sufficiently close to $\mu$, there exists a nonzero solution $\bar{x}^\lambda$ of system (\ref{e:LinearDelay2+G}) with parameter $\lambda$, such that $\bar{x}^\lambda\to 0$ in $C^1_{\rm loc}(\mathbb{R},\mathbb{R}^{n})$ as $\lambda\to\mu$.
		
		\item[\rm (iv)] For any given $\varepsilon>0$, there exists a one-sided neighborhood $\Lambda^0$ of $\mu$ in $\Lambda$ such that for each $\lambda\in\Lambda^0\setminus\{\mu\}$, system (\ref{e:LinearDelay2+G}) with parameter $\lambda$ has either:
		\begin{itemize}
			\item infinitely many $\mathbb{R}$-distinct nonzero solutions $\bar{x}^k_{\lambda}$ ($k=1,2,\dots$) satisfying \\ $\|\bar{x}^k_{\lambda}|_{[0,\tau]}\|_{C^1}<\varepsilon$, or
			\item at least two $\mathbb{R}$-distinct nonzero solutions $\hat{x}^1_{\lambda},\hat{x}^2_{\lambda}$ with $\|\hat{x}^i_{\lambda}|_{[0,\tau]}\|_{C^1}<\varepsilon$ ($i=1,2$) and
    \begin{align*}
    	& \bigl(\hat{x}_\lambda^1(0), J_n \hat{x}_\lambda^1(\tau) \bigr)_{\mathbb{R}^{2n}}
    	+ \int_0^{\tau} \bigl( \dot{\hat{x}}_\lambda^1(t), J_n \hat{x}_\lambda^1(t-\tau) \bigr)_{\mathbb{R}^{2n}} \, dt
    	+ \int_{-\tau}^{\tau} G\bigl(\lambda, \hat{x}_\lambda^1(t)\bigr) \, dt \\
    	\ne \; &
    	\bigl(\hat{x}_\lambda^2(0), J_n \hat{x}_\lambda^2(\tau) \bigr)_{\mathbb{R}^{2n}}
    	+ \int_0^{\tau} \bigl( \dot{\hat{x}}_\lambda^2(t), J_n \hat{x}_\lambda^2(t-\tau) \bigr)_{\mathbb{R}^{2n}} \, dt
    	+ \int_{-\tau}^{\tau} G\bigl(\lambda, \hat{x}_\lambda^2(t)\bigr) \, dt.
    \end{align*}
where $G(\lambda, x) = G_0(x)+
\lambda G_1(x)$.
		\end{itemize}
		\end{enumerate}
\end{enumerate}
\end{theorem}

\begin{proof}[\bf Proof]
 Define  $\hat{H}_j, \check{H}_j: ({\mathbb{R}}^{2n})^2\to\mathbb{R}$ and
 $\hat{H}, \check{H}: \mathbb{R}\times({\R}^{2n})^2\to\R$  by
 \begin{align*}
 \hat{H}_j \left(x _ { 1 } ,  x _ { 2 } \right) &= G_j \left(x _ {1 } \right)+
  G_j \left(x _ {2 } \right),\quad j=0,1,\\
 	\check{H}_j(z)&=\hat{H}_j({\Gamma}(2, 2n)z)\quad\forall
 z\in ({\R}^{2n})^2,\;j=0,1,\\
 \hat{H}(\lambda, z)&=\hat{H}_0(z)+\lambda \hat{H}_1(z),\qquad
 \check{H}(\lambda, z)=\check{H}_0(z)+\lambda \check{H}_1(z),
 \end{align*}
where $\Gamma(2,2n)$ is given by (\ref{e:LinearDelay2G+}).
Let $M_{2,n}$ be as in (\ref{e:LinearDelay2G++}).
(\ref{e:Delay14check}) shows that
the above each $\check{H}_j(\lambda, \cdot)$
is invariant for the action of $M_{2,n}$ for $j=0,1$.
According to the arguments below (\ref{e:EigenContinG}), we have
$$
\nabla^2 \check{H}_j({{0}})=
{\rm Diag}\left(\left( \begin{array} {cc}
	D_j & -E_j \\
	E_j &D_j
\end{array} \right), \left( \begin{array} {cc}
	D_j & -E_j \\
	E_j &D_j
\end{array} \right)\right),\quad j=0,1,
$$
and hence $\nabla^2 \check{H}_1({{0}})$ is definite (positive or negative).
Let $D(\lambda)=D_0+\lambda D_1$
and $E(\lambda)=E_0+\lambda E_1$, and let
  $\Upsilon_\lambda$  denote the fundamental matrix solution of
$$
\dot{z}(t)=J_{2n}\nabla^2_2 \check{H}(\lambda, {{0}})z(t)=J_{2n}
{\rm Diag}\left(\left( \begin{array} {cc}
	D(\lambda) & -E(\lambda) \\
	E(\lambda) &D(\lambda)
\end{array} \right), \left( \begin{array} {cc}
	D(\lambda) & -E(\lambda) \\
	E(\lambda) &D(\lambda)
\end{array} \right)\right)z(t)
$$
 with $\Upsilon_\lambda(0)=I_{4n}$.
 Claim~\ref{cl:delay2.3} shows  $\nu_\tau(0^\lambda)=\nu_{\tau,M_{2,n}}(\Upsilon_\lambda)$.
Then, by Corollary~1.15 in \cite{Lu11},
	for any $\tau>0$, the following set is discrete in $\mathbb{R}$:
$$
\Sigma_\tau:=\{\lambda\in\mathbb{R}\,|\, \nu_{\tau}({0}^\lambda)>0\}
=\{\lambda\in\mathbb{R} \mid \nu_{\tau, M_{2,n}}(\Upsilon_\lambda)>0\}.
$$
This proves (I).

Suppose that
$\mu\in{\Sigma}_\tau$ is not an eigenvalue of $(D_1+\sqrt{-1}E_1)^{-1}
(D_0+\sqrt{-1}E_0)$.
 Then it is also not an eigenvalue of $(\nabla^2{G}_1({0}))^{-1}\nabla^2{G}_0({0})$, and hence not an eigenvalue of $(\nabla^2{H}_1({0}))^{-1}\nabla^2{H}_0({0})$.
It follows that ${\rm Ker}(\nabla_z^2\check{H}({\mu}, {{0}}))=\{0\}$
and hence ${\rm Ker}(M_{2,n}-I_{4n})\cap{\rm Ker}(\nabla_z^2\check{H}({\mu}, {{0}}))=\{0\}$.
 In the following, we always understand
 $\check{H}(\lambda, z)=\check{H}_0(z)+\lambda \check{H}_1(z)$
 in (\ref{e:Delay6Auto3G}). Applying Corollary~1.15 in \cite{Lu11} to (\ref{e:Delay6Auto3G}) then gives that at least one of the following holds:
 	\begin{enumerate}
		\item[\rm (i*)] System (\ref{e:Delay6Auto3G}) with $\lambda=\mu$ admits a sequence of $\R$-distinct nonzero solutions
		${u}^{\mu,k}$ for $k=1,2,\cdots$,
		which  converges to zero in $C^1_{\rm loc}(\mathbb{R},\R^{4n})$ as $k\to\infty$.
		\item[\rm (ii*)]
		There exist a left neighborhood $\Lambda^-$ and a right neighborhood $\Lambda^+$ of $\mu$ in $\mathbb{R}$,
		and nonnegative integers $n^+$ and $n^-$ with $n^++n^-\ge \nu_{\tau, M_{2,n}}(\Upsilon_\mu)/2=\nu_\tau(0^\mu)/2$,
		such that for each $\lambda\in\Lambda^-\setminus\{\mu\}$ (resp. $\lambda\in\Lambda^+\setminus\{\mu\}$),
		system (\ref{e:Delay6Auto3G}) with parameter $\lambda$ possesses at least $n^-$ (resp. $n^+$) $\mathbb{R}$-distinct nonzero solutions $u^{\lambda,i}$, $i=1,\dots,n^-$ (resp. $i=1,\dots,n^+$),
		which satisfy $u^{\lambda,i}\to 0$ in $C^1_{\rm loc}(\mathbb{R},\mathbb{R}^{4n})$ as $\lambda\to\mu$.
			\end{enumerate}
	Moreover, if $\nu_{\tau}(0^\mu)\ge 3$,
then at least one of  (i*), (iii*), and (iv*) holds, with (i*) as stated previously, and with (iii*) and (iv*) given by:
	\begin{enumerate}
		\item[\rm (iii*)]
		For every $\lambda\in\Lambda\setminus\{\mu\}$ sufficiently close to $\mu$, there exists a nonzero solution $\bar{u}^\lambda$ of system (\ref{e:Delay6Auto3G}) with parameter $\lambda$, such that $\bar{u}^\lambda\to 0$ in $C^1_{\rm loc}(\mathbb{R},\mathbb{R}^{4n})$ as $\lambda\to\mu$.
		
		\item[\rm (iv*)] For any given $\varepsilon>0$, there exists a one-sided neighborhood $\Lambda^0$ of $\mu$ in $\Lambda$ such that for each $\lambda\in\Lambda^0\setminus\{\mu\}$, system (\ref{e:Delay6Auto3G}) with parameter $\lambda$ has either:
		\begin{itemize}
			\item infinitely many $\mathbb{R}$-distinct nonzero solutions $\bar{u}^{\lambda,k}$ ($k=1,2,\dots$) satisfying \\ $\|\bar{u}^{\lambda,k}|_{[0,\tau]}\|_{C^1}<\varepsilon$, or
			\item at least two $\mathbb{R}$-distinct nonzero solutions $\hat{u}^{\lambda,1},\hat{z}^{\lambda,2}$ with $\|\hat{u}^{\lambda,i}|_{[0,\tau]}\|_{C^1}<\varepsilon$ ($i=1,2$) and
$\mathcal{A}_\lambda({\hat{u}}_\lambda^1)\ne
\mathcal{A}_\lambda({\hat{u}}_\lambda^2)$, where
\begin{eqnarray*}
\mathcal{A}_\lambda({\hat{u}}_\lambda^i):=
\int^{\tau}_0\left[\frac{1}{2}(J_{2n}\dot{\hat{u}}_\lambda^i(t),
\hat{u}^i_\lambda(t))_{\mathbb{R}^{4n}}+
 \check{H}(\lambda, \hat{u}_\lambda^i(t))\right]dt,\quad i=1,2.
\end{eqnarray*}
		\end{itemize}
		\end{enumerate}

As in the proof of Theorem~\ref{th:bif-per3+3},
for $u^{\mu,k}$ in (i*) and $u^{\lambda,i}$ in (ii*), we
can write
\begin{align*}
&\Gamma(2,2n)u^{\mu,k}(t):=({x}^{\mu,k}_1(t)^\top, {x}^{\mu,k}_2(t)^\top)^\top\in ({\mathbb{R}}^{2n})^2, \quad k=1,2,\dots,\\
&\Gamma(2,2n)u^{\lambda,i}(t):=({x}^{\lambda,i}_1(t)^\top, {x}^{\lambda,i}_2(t)^\top)^\top\in ({\mathbb{R}}^{2n})^2,
\quad\hbox{$i=1,\dots,n^-$ (resp. $n^+$)}.
\end{align*}
Then  Claims~\ref{cl:delay2.1} and \ref{cl:delay2.2} imply that ${x}^{\mu,k}_2(t)={x}^{\mu,k}_1(t-\tau)$ and ${x}^{\lambda,i}_{2}(t)={x}^{\lambda,i}_{1}(t-\tau)$,
and that the functions ${x}^k:={x}^{\mu,k}_1$ and ${x}^i_{\lambda}:= {x}^{\lambda,i}_{1}$
fulfill the alternatives (i) and (ii).

Assume $\nu_{\tau}({0}^\lambda)\ge 3$.
For $\bar{u}^{\lambda}$ in (iii*), and
$\bar{u}^{\lambda,k}$ and $\hat{u}_\lambda^i$ (with $i=1,2$) in (iv*),
 respectively, put
 $\Gamma(2,2n)\bar{u}^\lambda(t):=(\bar{x}^{\lambda}_1(t)^\top, \bar{x}^{\lambda}_2(t)^\top)^\top\in ({\mathbb{R}}^{2n})^2$ and
\begin{align*}
&\Gamma(2,2n)\bar{u}^{\lambda,k}(t):=(\bar{x}^k_{\lambda,1}(t)^\top, \bar{x}^k_{\lambda,2}(t)^\top)^\top\in ({\mathbb{R}}^{2n})^2,\\
&\Gamma(2,2n)\hat{u}^i_\lambda(t):=(\hat{x}^i_{\lambda,1}(t)^\top, \hat{x}^i_{\lambda,2}(t)^\top)^\top\in ({\mathbb{R}}^{2n})^2.
\end{align*}
By Claim~\ref{cl:delay2.2} functions
$\bar{x}^{\lambda}(t):=\bar{x}^{\lambda}_1(t)$,
$\bar{x}^k_\lambda:=\bar{x}^k_{\lambda,1}$ and
$\hat{x}^i_{\lambda}:= \hat{x}^i_{\lambda,1}$
satisfy the required alternatives (iii) and (iv)
as in the proof of Theorem~\ref{th:bif-per3+3}.
\end{proof}

 \section{Bifurcation near the trivial equilibrium of system (\ref{e:crm1})}\label{sec:delay4}



By Assumption~\ref{ass:Crm1}(ii),
 for each fixed $\lambda$, the function $H(\lambda,  \cdot)$ is invariant under the transformation
$(x, y) \mapsto (y, -x)$. The latter transformation
corresponds to the action of $-J_n$ on $\mathbb{R}^{2n} \equiv \mathbb{R}^n \times \mathbb{R}^n$ defined by
\[
(x_1, \cdots, x_n, y_1, \cdots, y_n)^\top \mapsto -J_n(x_1, \cdots, x_n, y_1, \cdots, y_n)^\top = (y_1, \cdots, y_n, -x_1, \cdots, -x_n)^\top.
\]
(Since $-J_n=J_n^{-1}$, $H(\lambda, J_n^{-1}\cdot z)=H(\lambda, (-J_n)\cdot z)=H(\lambda,t,  z)$
if and only if $H(\lambda, J_n\cdot z)=H(\lambda,  z)$, that is,
each $H(\lambda, \cdot)$ is $J_n$-invariant.

Denote by $\nabla_zH(\lambda,\cdot)$  the euclidian gradient of
$H(\lambda,\cdot)$ with respect to the $\mathbb{R}^n\times\mathbb{R}^n\equiv\mathbb{R}^{2n}$-variable, and by $\nabla^2_zH(\lambda,\cdot)=D_z(\nabla_zH(\lambda,\cdot))\in\mathcal{L}_s(\mathbb{R}^{2n})$.
Let $\nabla_2H(\lambda, x, y)$ and $\nabla_3H(\lambda, x, y)$ denote
the euclidian gradient of  $H(\lambda, x, y)$ with respect to the $x$-variable and $y$-variable, respectively.
The second equality in Assumption~\ref{ass:Crm1}(ii) yields
\begin{equation}\label{e:crm0}
	\nabla_2H(\lambda, x, y)=-\nabla_3H(\lambda, y, -x),\quad\forall (\lambda,  x,y).
\end{equation}

 Let $x\equiv 0$ satisfy (\ref{e:crm1})
 for each $\lambda\in\Lambda$. Then
 $\nabla_2H(\lambda, 0, 0)=0$ and therefore
 \begin{equation}\label{e:crm0A}
 \nabla_zH(\lambda, 0, 0)=0
 \end{equation}
 since  $\nabla_3H(\lambda, 0, 0)=0$ by (\ref{e:crm0}). Clearly, (\ref{e:crm0})
 also implies that for $(\lambda,  x,y)$,
\begin{align*}\label{e:crm0}
	H_{xx}(\lambda, x, y)=
	H_{yy}(\lambda, y, -x)\quad\text{and}\quad
	H_{xy}(\lambda, x, y)=
	-H_{yx}(\lambda, y, -x).	
\end{align*}
It follows that for $(\lambda,  x,y)$,
\begin{equation}\label{e:crm0C}
	H_{xx}(\lambda, 0, 0)=
H_{yy}(\lambda, 0, 0)\quad\text{and}\quad
H_{xy}(\lambda, 0, 0)
=-H_{yx}(\lambda, 0, 0)=-(H_{xy}(\lambda, 0, 0))^\top.	
\end{equation}

The following relation
stated in \cite[Section~4]{WaLiu14}.
(See \cite{Lu14} for a detailed proof.)
\begin{claim}\label{cl:crm1}
	Under Assumption~\ref{ass:Crm1}(i), if  $z(t)=(x(t)^\top, y(t)^\top)^{\top}$ satisfies
	\begin{equation}\label{e:crm2}
		\dot{z}(t)=J_n\nabla_z H(\lambda, z(t))\quad\text{and}\quad z(t+\tau)=J_nz(t),\;\forall t\in \mathbb{R},
	\end{equation}
	then  $x(t)$  satisfies (\ref{e:crm1}).
		Conversely, under Assumption~\ref{ass:Crm1}, if $x(t)$ satisfies (\ref{e:crm1}),
		then $z(t)=(x(t)^\top, y(t)^\top)^{\top}$, where $y(t):=x(t-\tau)$,
		satisfies  (\ref{e:crm2}).
	\end{claim}

Applying \cite[Theorem~1.5(I),(II)]{Lu11} and
the first part of Theorem~1.14 in  \cite{Lu11} to problem (\ref{e:crm2}),
we obtain an analogue of Theorems~\ref{th:bif-per3Delay} and~\ref{th:bif-per3+3}.

\begin{theorem}\label{th:bif-per3DelayCrm}
	Under Assumption~\ref{ass:Crm1},
	let $x\equiv 0$ satisfy (\ref{e:crm1})
	for each $\lambda\in\Lambda$, and
	let  $e_1(\lambda)\le e_2(\lambda)\le\cdots\le e_n(\lambda)$ be all eigenvalues of the Hermitian matrix
	$H_{xx}(\lambda, 0, 0)-\sqrt{-1}
	H_{xy}(\lambda, 0, 0)$.
Given a delay parameter $\tau>0$,  for each $e_k(\lambda)$, there exists a unique integer $j_k(\lambda) \in \mathbb{Z}$ such that
	\begin{equation}\label{e:uniqueIntegerCrm}
		e_k(\lambda) \in  \left( \frac{(4j_k(\lambda) + 1)\pi}{2\tau}, \frac{(4j_k(\lambda) + 5)\pi}{2\tau} \right].
	\end{equation}
\begin{enumerate}
\item[\rm (I)]{\rm (\textsf{Necessary condition}):}
If $(\mu,0)$ with $\mu\in\Lambda$ is a bifurcation point of
 (\ref{e:Delay2Auto}), then $e_l(\mu)\in
 \frac{\pi(4\mathbb{Z}+5)}{2\tau}$ for some $l$.

\item[\rm (II)]{\rm (\textsf{Sufficient condition}):}
Let $\mu$ be an interior point of $\Lambda$. Then
$(\mu,0)$ is a bifurcation point of (\ref{e:Delay2Auto}),
 provided that  there exists a partition
 $\{i_1,\dots,i_k\}\cup\{i_{k+1},\dots,i_n\}$  of $\{1,\dots,n\}$
 with $1 \le k \le n$, such that
	\begin{itemize}
		\item[\rm (A)] $e_{i_l}(\mu)=\frac{(4j_{i_l}(\mu)+5)\pi}{2\tau}$ for  $l=1,\dots,k$, and $e_{i_l}(\mu)\notin\frac{\pi(4\mathbb{Z}+5)}{2\tau}$ for $l=k+1,\dots,n$;

\item[\rm (B)] there exist two sequences $(\lambda_m^+)_m$
and $(\lambda_m^-)_m$ in
$\Lambda$  converging to $\mu$ such that
$e_{i_l}(\lambda_m^-)< e_{i_l}(\mu)<e_{i_l}(\lambda_m^+)$
for all $m\in\mathbb{N}$ and all $l=1,\dots,k$.
	\end{itemize}
 \item[\rm (III)]{\rm (\textsf{Alternative bifurcations of Fadell-Rabinowitz type and of Rabinowitz type}):}
 Assume the hypotheses of (II), subject to the following strengthening of condition (B):
 \begin{itemize}
\item[\rm (SB)]
 For some $\epsilon>0$, one of the following two conditions holds:
\begin{itemize}
\item[\rm (SB.1)] $e_{i_l}(\lambda)< e_{i_l}(\mu)<e_{i_l}(\lambda')$
for all $(\lambda, \lambda')\in (\mu-\epsilon, \mu)\times (\mu, \mu+\epsilon)$ and
 $l=1,\dots,k$.
\item[\rm (SB.2)]
$e_{i_l}(\lambda)> e_{i_l}(\mu)>e_{i_l}(\lambda')$
for all $(\lambda, \lambda')\in (\mu-\epsilon, \mu)\times (\mu, \mu+\epsilon)$ and $l=1,\dots,k$.
 \end{itemize}
 \end{itemize}
 Then for the problem (\ref{e:crm1}),
	there exist two possible alternatives:
	\begin{itemize}
		\item[\rm (i)] The problem (\ref{e:crm1})  with $\lambda=\mu$ admits a sequence of
		nonzero, $\R$-distinct solutions
		$\{x^l\}^{\infty}_{l=1}$  such that $x^l\to 0$ in $C^1_{\rm loc}(\mathbb{R},\mathbb{R}^n)$
		as $l\to\infty$.
		\item[\rm (ii)] There exist left and right  neighborhoods $\Lambda^-$ and $\Lambda^+$ of $\mu$ in $\Lambda$
		and nonnegative integers $n^+$ and $n^-\ge 0$ satisfying $n^++n^-\ge k$, such that
		for $\lambda\in\Lambda^-\setminus\{\mu\}$ (resp. $\lambda\in\Lambda^+\setminus\{\mu\}$),
		(\ref{e:crm1}) with parameter  $\lambda$  has at least $n^-$ (resp. $n^+$) $\R$-distinct
		solutions $x_\lambda^i\ne 0$ ($i=1,\dots,n^-$ (resp. $n^+$))
		which converge to zero in $C^1_{\rm loc}(\mathbb{R}, \mathbb{R}^n)$
		as $\lambda\to\mu$.
	\end{itemize}
	Moreover, if $n>1$ and $k>1$, then at least one of (i), (iii), and (iv)  holds,
	where (i) is as stated previously, and
	\begin{itemize}
		\item[\rm (iii)]  For every $\lambda\in\Lambda\setminus\{\mu\}$ near $\mu$, there is a
		nonzero solution ${x}_\lambda$  of (\ref{e:crm1})
		with parameter $\lambda$, such that  $x_\lambda\to 0$ in $C^1_{\rm loc}(\mathbb{R},\mathbb{R}^n)$
		as $\lambda\to\mu$.
				
		\item[\rm (iv)] For a given $\varepsilon > 0$, there exists a one-sided neighbourhood $\Lambda^0$ of $\mu$ in $\Lambda$ such that, for any $\lambda \in \Lambda^0 \setminus \{\mu\}$, problem
	(\ref{e:crm1}) with parameter $\lambda$ has either
	\begin{itemize}
		\item[$\bullet$]
	 infinitely many $\mathbb{R}$-distinct nonzero solutions $\{x_\lambda^l\}_{l=1}^\infty$ such that $\|x_\lambda^l|_{[0,2\tau]}\|_{C^1} < \varepsilon$ for all $l \in \mathbb{N}$,
	 or
	\item[$\bullet$]  at least two $\mathbb{R}$-distinct nonzero solutions $\hat{x}_\lambda^1$ and $\hat{x}_\lambda^2$ satisfying the following inequalities: $\|\hat{x}_\lambda^i|_{[0,2\tau]}\|_{C^1} < \varepsilon$ for $i=1,2$, and
		\begin{align*}
			& -\bigl(\hat{x}_\lambda^1(0),  \hat{x}_\lambda^1(\tau) \bigr)_{\mathbb{R}^{n}}
			+ \int_0^{\tau} \bigl( \dot{\hat{x}}_\lambda^1(t),  \hat{x}_\lambda^1(t-\tau) \bigr)_{\mathbb{R}^{n}} \, dt
			+ \int_{0}^{\tau} H\bigl(\lambda,
			\hat{x}_\lambda^1(t), \hat{x}_\lambda^1(t-\tau)\bigr) \, dt \\
			\ne \; &
			-\bigl(\hat{x}_\lambda^2(0),  \hat{x}_\lambda^2(\tau) \bigr)_{\mathbb{R}^{n}}
			+ \int_0^{\tau} \bigl( \dot{\hat{x}}_\lambda^2(t),  \hat{x}_\lambda^2(t-\tau) \bigr)_{\mathbb{R}^{n}} \, dt
			+ \int_{0}^{\tau} H\bigl(\lambda,
			\hat{x}_\lambda^2(t),
			\hat{x}_\lambda^2(t-\tau)
			\bigr) \, dt.
		\end{align*}
	\end{itemize}
\end{itemize}
\end{enumerate}
\end{theorem}

For $n=1$, Theorem~\ref{th:bif-per3DelayCrm} simplifies to the following.

\begin{corollary}\label{cor:bif-per3+3Crm}
	Let $\Lambda \subset \mathbb{R}$ be an interval, and let $H \in C(\Lambda \times \mathbb{R}^{2}, \mathbb{R})$
	be such that:
	\begin{itemize}
		\item[\rm (C)]
		For every $\lambda \in \Lambda$, the function $H(\lambda, \cdot): \mathbb{R}^{2} \to \mathbb{R}$ is of class $C^2$, and its first and second partial derivatives with respect to the second variable,
		$\nabla_2H(\lambda, z)$ and $\nabla^2_2H\left(\lambda, z\right)$,   are continuous in $(\lambda, z)$.
		
		\item[\rm (D)] $\nabla_y H(\lambda, 0, 0)=0\;\forall\lambda\in\Lambda$,
		$H(\lambda, x, y) = H(\lambda, y, -x)$ for all $(\lambda, x, y)$,
		and so 		
		$\nabla_zH(\lambda, 0,0)=0$ and $\nabla^2_zH\left(\lambda, 0, 0\right)=
		\operatorname{diag}(e(\lambda), e(\lambda))$ for a continuous function $e:\Lambda\to\mathbb{R}$ and all $\lambda\in\Lambda$.
	\end{itemize}
Fix a delay parameter $\tau>0$.
 If $(\mu,0)$ with $\mu\in\Lambda$ is a bifurcation point of
(\ref{e:crm1}) with $n=1$,  then $e(\mu)\in\frac{\pi(4\mathbb{Z}+5)}{2\tau}$.
Conversely, let $\mu$ be an interior point of $\Lambda$ satisfying
$e(\mu)\in\frac{\pi(4\mathbb{Z}+5)}{2\tau}$.
Then $(\mu,0)$ is a bifurcation point of (\ref{e:crm1}) with $n=1$,
provided that there exist sequences $(\lambda_m^\pm)_m\subset\Lambda$  converging to
$\mu$ such that
$e(\lambda_m^-)<e(\mu)<e(\lambda_m^+)$ for all $m\in\mathbb{N}$.
Furthermore, suppose that for some interior point
 $\mu$ of $\Lambda$, we have that
 $e(\mu)\in\frac{\pi(4\mathbb{Z}+5)}{2\tau}$
and that $e(\lambda)-e(\mu)\ne 0$ and changes sign as $\lambda$ across $\mu$
in a deleted neighborhood of $\mu$.
Then for  problem (\ref{e:crm1}) with $n=1$,
	there are  the following alternatives:
	\begin{itemize}
		\item[\rm (i)] The problem (\ref{e:crm1}) with $n=1$  and $\lambda=\mu$ possesses a sequence of $\mathbb{R}$-distinct solutions, $x^k\ne 0$ ($k=1,2,\dots$) such that
		$x^k\to 0$ in $C^1_{\rm loc}(\mathbb{R}, \mathbb{R})$ as $k\to\infty$.
		
		\item[\rm (ii)] There exist left and right  neighborhoods $\Lambda^-$ and $\Lambda^+$ of $\mu$ in $\mathbb{R}$
		and nonnegative integers $n^+$ and $n^-$ satisfying $n^++n^-\ge 1$, such that for $\lambda\in\Lambda^-\setminus\{\mu\}$ (resp. $\lambda\in\Lambda^+\setminus\{\mu\}$), problem
		(\ref{e:crm1}) with $n=1$ and parameter  $\lambda$  has at least $n^-$ (resp. $n^+$) $\R$-distinct
		solutions, $x_\lambda^i \neq 0$ ($i = 1, \dots, n^-$ resp. $n^+$)
		which, as $\lambda \to \mu$, converge to  zero in
		$C^1_{\rm loc}(\mathbb{R}, \mathbb{R})$.
	\end{itemize}
\end{corollary}

Similar to  Corollaries~\ref{cor:bif-per3DelayC}
and \ref{cor:bif-per3+4}
we have:

\begin{corollary}\label{cor:bif-per3+4Cri}
	Let  $K: \mathbb{R}^{2n} \to \mathbb{R}$ be a $C^2$ function satisfying
	$\nabla K(0, 0)=0$ and
	$K(x,y)=K(y,-x)$ for all $(x,y)\in\mathbb{R}^n\times\mathbb{R}^n
	\equiv\mathbb{R}^{2n}$. This implies
	that the matrix $K_{xx}(0, 0)- \sqrt{-1}
	K_{xy}(0, 0)$ is Hermitian.
	Let  $e_1\le e_2\le\cdots\le e_n$ be all eigenvalues of $K_{xx}(0, 0)-\sqrt{-1}
	K_{xy}(0, 0)$ (arranged in nondecreasing order).
		Consider the problem
	\begin{eqnarray}\label{e:Delay2AutoCCrm}
		dot{x}( t ) =  -\nabla_y K(x(t), x ( t - \lambda) )\quad\text{and}\quad
			x( t + 2\lambda) =  -x( t )\;\forall t.
		\end{eqnarray}
If $(\mu,0)$ with $\mu\in\mathbb{R}$ is a bifurcation point of
the problem (\ref{e:Delay2AutoCCrm}), then $\mu e_l\in
 \frac{\pi(4\mathbb{Z}+5)}{2}$ for some $l$.
 Conversely, for some $\mu\in\mathbb{R}$,
 suppose there exists a partition of $\{1,\cdots,n\}$,
	$\{i_1,\cdots,i_q\}$ and $\{i_{q+1},\cdots,i_n\}$,
where $1 \le q \le n$, such that the following  condition is satisfied:
	\begin{enumerate}
		\item[\rm (E)] $\mu e_{i_l}=\frac{(4j_{i_l}(\mu)+5)\pi}{2}$ for  $l=1,\dots,q$ (which implies that
		$\mu\ne 0$ and $e_{i_l}\ne 0$ for $l=1,\dots,q$),
		$e_{i_l}$($l=1,\dots,q$) have the same plus-minus sign, and
		$\mu e_{i_l}\notin\frac{\pi(4\mathbb{Z}+5)}{2}$ for any $l=q+1,\dots,n$.
	\end{enumerate}
	(In particular, if $n=1$, then
	$K_{xy}(0,0)=0$, $e_1=K_{xx}(0,0)=K_{yy}(0,0)$, and Condition (E) becomes $2\mu K_{xx}(0,0)\in \pi(4\mathbb{Z}+5)$.)
		Then  one of the following alternatives occurs:
	\begin{itemize}
		\item[\rm (i)] The problem (\ref{e:Delay2AutoCCrm})   with $\lambda=\mu$
		admits a sequence of nonzero,  $\R$-distinct solutions $\{x^k\}^{\infty}_{k=1}$ such that
		$x^k\to 0$ in $C^1_{\rm loc}(\mathbb{R}, \mathbb{R}^{n})$
		as $k\to\infty$.
		\item[\rm (ii)]
		There exist left and right neighborhoods $\Lambda^-$ and $\Lambda^+$ of $\mu$ in $\Lambda$,
		and nonnegative integers $n^-$ and $n^+$ with $n^- + n^+ \geq q$,
		such that for $\lambda \in \Lambda^- \setminus \{\mu\}$ (resp. $\lambda \in \Lambda^+ \setminus \{\mu\}$),
		problem (\ref{e:Delay2AutoCCrm}) with parameter $\lambda$ has at least $n^-$ (resp. $n^+$)
		$\mathbb{R}$-distinct nonzero solutions
		$x_\lambda^i \neq 0$ ($i = 1, \dots, n^-$ (resp. $n^+$)),
		which converge to the zero function in $C^1_{\rm loc}(\mathbb{R}, \mathbb{R}^{n})$
		as $\lambda \to \mu$.
	\end{itemize}
	Moreover, if $n>1$ and $q>1$,
then at least one of (i), (iii) and (iv) holds, where (i) is as stated above, and the assertions of (iii) and (iv) are as follows:
	\begin{itemize}
		\item[\rm (iii)]  For every $\lambda\in\Lambda\setminus\{\mu\}$ near $\mu$, there is a
		nonzero solution ${x}_\lambda$  of problem (\ref{e:Delay2AutoCCrm})
		with parameter  $\lambda$, such that $x_\lambda\to 0$
		in $C^1_{\rm loc}(\mathbb{R}, \mathbb{R}^{2n})$  as $\lambda\to\mu$.
		
		\item[\rm (iv)] For a given $\varepsilon>0$,  there is a one-sided  neighborhood $\Lambda^0$ of $\mu$ in $\Lambda$ not containing
		the point $0\in\mathbb{R}$, such that for any $\lambda\in\Lambda^0\setminus\{\mu\}$,
		problem (\ref{e:Delay2AutoCCrm}) with parameter  $\lambda$
		has either
		\begin{itemize}
			\item [$\bullet$]
		 infinitely many $\mathbb{R}$-distinct nonzero solutions $\{x_\lambda^l\}_{l=1}^\infty$ such that
		$$
		\|{x}_\lambda^l\|_{C^0([0, 2|\lambda|])}+|\lambda|
		\|\dot{x}_\lambda^l\|_{C^0([0, 2|\lambda|])}<\varepsilon \;
		$$
	for $l=1,2,\dots$,	or
	\item[$\bullet$]  at least two $\mathbb{R}$-distinct nonzero solutions $\hat{x}_\lambda^1$ and $\hat{x}_\lambda^2$ satisfying  the following inequalities
		$$
		\|\hat{x}_\lambda^i\|_{C^0([0, 2|\lambda|])}+|\lambda|
		\|\dot{\hat{x}}_\lambda^i\|_{C^0([0, 2|\lambda|])}<\varepsilon, \;i=1,2,
		$$
		and
				\begin{align*}
			& -\bigl(\hat{x}_\lambda^1(0),  \hat{x}_\lambda^1(\lambda) \bigr)_{\mathbb{R}^{n}}
			+ \int_0^{\lambda} \bigl( \dot{\hat{x}}_\lambda^1(t),  \hat{x}_\lambda^1(t-\lambda) \bigr)_{\mathbb{R}^{n}} \, dt
			+ \int_{0}^{\lambda} K\bigl(
			\hat{x}_\lambda^1(t), \hat{x}_\lambda^1(t-\lambda)\bigr) \, dt \\
			\ne \; &
			-\bigl(\hat{x}_\lambda^2(0),  \hat{x}_\lambda^2(\lambda) \bigr)_{\mathbb{R}^{n}}
			+ \int_0^{\lambda} \bigl( \dot{\hat{x}}_\lambda^2(t),  \hat{x}_\lambda^2(t-\lambda) \bigr)_{\mathbb{R}^{n}} \, dt
			+ \int_{0}^{\lambda} K\bigl(
			\hat{x}_\lambda^2(t),
			\hat{x}_\lambda^2(t-\lambda)
			\bigr) \, dt.
		\end{align*}
	\end{itemize}
\end{itemize}
\end{corollary}

This can be directly derived from Theorem~\ref{th:bif-per3DelayCrm}
as in the proofs of Corollaries~\ref{cor:bif-per3DelayC}
and~\ref{cor:bif-per3+4}.

The conditions in the theorem and corollaries above ensure a quantitative characterization of the bifurcation parameter values.
 For a subclass of functions $H$ that satisfy Assumption~\ref{ass:Crm1}, a qualitative bifurcation result can be established (Theorem~\ref{cor:crm5}).

For the advanced system (\ref{e:crm1Ad}), the following claim is the direct analogue of Claim~\ref{cl:crm1}.

\begin{claim}\label{cl:crm2}
	Under Assumption~\ref{ass:Crm1}(i), if $ z(t) = (x(t)^\top, y(t)^\top)^\top $ satisfies
	\begin{equation}\label{e:crm2Ad}
		\dot{z}(t) = J_n \nabla_z H(\lambda,  z(t))\quad \text{and} \quad z(t + \tau) = -J_n z(t),\;
		\quad \forall t \in \mathbb{R},
	\end{equation}
	then $ x(t) $ satisfies
	(\ref{e:crm1Ad}).	
	Conversely, under Assumption~\ref{ass:Crm1}, if $ x(t) $ satisfies \eqref{e:crm1Ad},
	then $ z(t) = (x(t)^\top, y(t)^\top)^\top $ with $ y(t) := x(t+ \tau) $, satisfies \eqref{e:crm2Ad}.
\end{claim}


This allows us to derive  results analogous
to Theorem~\ref{th:bif-per3DelayCrm} and Corollaries~\ref{cor:bif-per3+3Crm} and \ref{cor:bif-per3+4Cri}
for the advanced system (\ref{e:crm1Ad}). For example, the conclusions of Theorem~\ref{th:bif-per3DelayCrm} remain valid, subject to the following modifications:
\begin{itemize}
	\item Replace all (\ref{e:crm1}) with (\ref{e:crm1Ad}).
	\item Replace (\ref{e:uniqueIntegerCrm})
	with  \begin{equation*}
		e_k(\lambda) \in  \left( \frac{(4j_k(\lambda) - 1)\pi}{2\tau}, \frac{(4j_k(\lambda) + 3)\pi}{2\tau} \right].
	\end{equation*}
	\item Replace $(4j_{i_l}(\mu)+5)$ in (A)
	with $(4j_{i_l}(\mu)+3)$.
	\item Replace the final inequality
	in (iv) with the following	
		\begin{align*}
		& \bigl(\hat{x}_\lambda^1(0),  \hat{x}_\lambda^1(\tau) \bigr)_{\mathbb{R}^{n}}
		+ \int_0^{\tau} \bigl( \dot{\hat{x}}_\lambda^1(t),  \hat{x}_\lambda^1(t+\tau) \bigr)_{\mathbb{R}^{n}} \, dt
		+ \int_{0}^{\tau} H\bigl(\lambda,
		\hat{x}_\lambda^1(t), \hat{x}_\lambda^1(t+\tau)\bigr) \, dt \\
		\ne \; &
		\bigl(\hat{x}_\lambda^2(0),  \hat{x}_\lambda^2(\tau) \bigr)_{\mathbb{R}^{n}}
		+ \int_0^{\tau} \bigl( \dot{\hat{x}}_\lambda^2(t),  \hat{x}_\lambda^2(t+\tau) \bigr)_{\mathbb{R}^{n}} \, dt
		+ \int_{0}^{\tau} H\bigl(\lambda,
		\hat{x}_\lambda^2(t),
		\hat{x}_\lambda^2(t+\tau)
		\bigr) \, dt.
	\end{align*}
\end{itemize}

Actually, the systems (\ref{e:crm1}) and (\ref{e:crm1Ad}) are equivalent in the following sense: a function $x:\mathbb{R}\to\mathbb{R}^n$ satisfies (\ref{e:crm1}) if and only if it satisfies (\ref{e:crm1Ad}) with $H$ replaced by $G$, where $G:\Lambda\times {\mathbb{R}}^{n}\times{\mathbb{R}}^{n}
\to\mathbb{R}$ is defined by
$G(\lambda, u, v)=- H(\lambda, v, u)$.
For details, see \cite[Remark~6.1]{Lu14}.

\begin{proof}[\bf Proof of Theorem~\ref{th:bif-per3DelayCrm}]
For each $\lambda\in\Lambda$,	
let $B(\lambda)=H_{xx}(\lambda, 0, 0)$
and  $C(\lambda)=-H_{xy}(\lambda, 0, 0)$.
By (\ref{e:crm0A}) and (\ref{e:crm0C}),
$\nabla_zH(\lambda, 0, 0)=0$,
$B(\lambda)$ and $C(\lambda)$
are symmetric and skew-symmetric,
respectively,  and
	$$
\nabla^2_zH(\lambda, 0, 0)
=\left( \begin{array} {cc}
	B(\lambda) & -C(\lambda) \\
	C(\lambda) &	B(\lambda)
\end{array} \right).
$$	
	Denote by  $\Upsilon_\lambda$  the fundamental matrix solution of
	$\dot{z}(t)=J_{n}\nabla^2_z {H}(\lambda, 0,0)z(t)$
	with $\Upsilon_\lambda(0)=I_{2n}$.
	By Theorem~\ref{th:PIndex2}, we have
		\begin{align}\label{e:crm2Ad1}
	\nu_{\tau,J_n}(\Upsilon_\lambda)&=
	2\sharp\left\{k\,|\, 2\tau e_k(\lambda)=(4j+1)\pi\;\hbox{with some $j\in\mathbb{Z}$}\right\},\\
	i_{\tau,J_n}(\Upsilon_\lambda)&=
	2j_1(\lambda)+\cdots+ 2j_n(\lambda) +[3n/2],\label{e:crm2Ad2}
	\end{align}
	where for $l=1,\cdots,n$, $j_l(\lambda)\in\mathbb{Z}$ is the unique integer such that $\frac{(4j_l+1)\pi}{2\tau}<e_l(\lambda)\le\frac{(4j_l+5)\pi}{2\tau}$.

\noindent\textbf{Step 1}(\textsf{Proof of (I)}).
 If $(\mu,0)$ with $\mu\in\Lambda$ is a bifurcation point of
 (\ref{e:crm1}), by  Claim~\ref{cl:crm1}
 this $\mu$ corresponds to a bifurcation point $(\mu, 0)$ of (\ref{e:crm2}). Consequently, \cite[Theorem~1.5(I)]{Lu11} yields
$\nu_{\tau,J_n}(\Upsilon_\mu)>0$,
which, in conjunction with (\ref{e:crm2Ad1}), establishes the conclusion
that $e_l(\mu)\in\frac{\pi(4\mathbb{Z}+5)}{2\tau}$ for some $l$.

\noindent\textbf{Step 2}(\textsf{Proof of (III)}).
Note that condition (a) of Theorem~1.14 in \cite{Lu11} holds because  ${\rm Ker}(J_n-I_{2n})=\{0\}$.
By  assumption (A) and (\ref{e:crm2Ad1}), we have
		$\nu_{\tau,J_n}(\Upsilon_\mu)=2k$.
		As in the proof of Theorem~\ref{th:bif-per3Delay},	
	for sufficiently small $\epsilon>0$,
	we derive from (\ref{e:crm2Ad1}) and (\ref{e:crm2Ad2})
	 $$
	 \nu_{\tau,J_n}(\Upsilon_\lambda)=0\quad
	 \forall 0<|\lambda-\mu|<\epsilon,
	 $$
	and
	\begin{equation*}
		i_{\tau,-J_n}(\Upsilon_\lambda)=
		\begin{cases}
			2j_1(\mu)+\cdots+ 2j_n(\mu) +[3n/2]\quad&\text{if $\lambda\in (\mu-\epsilon, \mu)$},\\
			2j_1(\mu)+\cdots+ 2j_n(\mu)+ 2k +[3n/2]\quad&\text{if $\lambda\in (\mu, \mu+\epsilon)$}
		\end{cases}
	\end{equation*}
(that is, $i_{\tau,-J_n}(\Upsilon_\lambda)=i_{\tau,-J_n}(\Upsilon_\mu)$
for any $\lambda\in (\mu-\epsilon, \mu)$,
and $i_{\tau,-J_n}(\Upsilon_\lambda)=
i_{\tau,-J_n}(\Upsilon_\mu)+ \nu_{\tau,-J_n}(\Upsilon_\mu)$
for any $\lambda\in (\mu, \mu+\epsilon)$)
	in the case (SB.1), and
		\begin{equation*}
		i_{\tau,-J_1}(\Upsilon_\lambda)=
		\begin{cases}
			2j_1(\mu)+\cdots+ 2j_n(\mu)+ 2k +[3n/2]\quad&\text{if $\lambda\in (\mu-\epsilon, \mu)$},\\
			2j_1(\mu)+\cdots+ 2j_n(\mu) +[3n/2]\quad&\text{if $\lambda\in (\mu, \mu+\epsilon)$}
		\end{cases}
	\end{equation*}
(namely, $i_{\tau,-J_n}(\Upsilon_\lambda)=i_{\tau,-J_n}(\Upsilon_\mu)
+ \nu_{\tau,-J_n}(\Upsilon_\mu)$
for any $\lambda\in (\mu-\epsilon, \mu)$,
and $i_{\tau,-J_n}(\Upsilon_\lambda)=
i_{\tau,-J_n}(\Upsilon_\mu)$
for any $\lambda\in (\mu, \mu+\epsilon)$)
in the case (SB.2).

Therefore, applying the first part of Theorem~1.14 in  \cite{Lu11} to the problem (\ref{e:crm2}) yields	the following alternatives:
	\begin{enumerate}
		\item[\rm (i')] The problem (\ref{e:crm2}) with $\lambda=\mu$ has a sequence of $\mathbb{R}$-distinct solutions,
		$z^l\ne 0$ ($l=1,2,\dots$) such that $\|z^l|_I\|_{C^1}\to 0$ for any compact interval $I\subset\mathbb{R}$;
		\item[\rm (ii')] There exist left and right  neighborhoods $\Lambda^-$ and $\Lambda^+$ of $\mu$ in $\Lambda$
		and nonnegative integers $n^+$ and $n^-\ge 0$ satisfying $n^++n^-\ge \nu_{\tau,J_n}(\Upsilon_\mu)/2=k$,
		such that for $\lambda\in\Lambda^-\setminus\{\mu\}$ (resp. $\lambda\in\Lambda^+\setminus\{\mu\}$),
		problem (\ref{e:crm2}) with parameter  $\lambda$  has at least $n^-$ (resp. $n^+$) $\R$-distinct
		solutions, $z_\lambda^i\ne 0$ ($i=1,\dots,n^-$ resp. $n^+$)
		whose restrictions to $[0,\tau]$  converge to zero in $C^1([0,\tau];\mathbb{R}^{2n})$  as $\lambda\to\mu$.
	\end{enumerate}
Write $z^l(t)=(x^l(t)^\top, y^l(t)^\top)^{\top}$
and $z^i_\lambda(t)=(x^i_\lambda(t)^\top, y^i_\lambda(t)^\top)^{\top}$.
By  Claim~\ref{cl:crm1},
  $y^l(t)=x^l(t-\tau)$
and  $y^i_\lambda(t)=x^i_\lambda(t-\tau)$,
and these $x^l$ and $x^i_\lambda$ satisfy
	the conclusions (i) and (ii)
	in Theorem~\ref{th:bif-per3DelayCrm}.

From now on we assume $n>1$ and $k>1$.
Then $\nu_{\tau,-J_n}(\Upsilon_\mu)=2k>3$.
Applying the second part of Theorem~1.14 in  \cite{Lu11} to the problem (\ref{e:crm2}),
we conclude that  at least one of (i')
or one of the following holds:
\begin{enumerate}
	\item[\rm (iii')]  For every $\lambda\in\Lambda\setminus\{\mu\}$ sufficiently close to $\mu$, there is a
	nonzero solution ${z}_\lambda$  of the problem (\ref{e:crm2}) with parameter value $\lambda$, such that $z_\lambda$ converges to zero
	in the $C^1$ norm on any compact interval $I\subset\R$  as $\lambda\to\mu$.
	
	\item[\rm (iv')] For a given $\varepsilon>0$  there exists a one-sided  neighborhood $\Lambda^0$ of $\mu$ within $\Lambda$ such that,
	for any $\lambda\in\Lambda^0\setminus\{\mu\}$,  problem (\ref{e:crm2})
	with parameter  $\lambda$ has either
	\begin{itemize}
		\item[$\bullet$]
	 infinitely many $\mathbb{R}$-distinct nonzero solutions
	${z}_\lambda^l$ satisfying $\|{z}_\lambda^l|_{[0, \tau]}\|_{C^1}<\varepsilon$ ($l=1,2,\dots$), or
\item[$\bullet$] at least two $\mathbb{R}$-distinct nonzero solutions $\hat{z}_\lambda^1$ and $\hat{z}_\lambda^2$ satisfying inequalities $\|\hat{z}_\lambda^i|_{[0, \tau]}\|_{C^1}<\varepsilon$ ($i=1,2$) and
	$$
	\hspace{-5mm}\int^{\tau}_0\left[\frac{1}{2}(J_n\dot{\hat{z}}_\lambda^1(t),\hat{z}^1_\lambda(t))_{\mathbb{R}^{2n}}+ {H}(\lambda, \hat{z}_\lambda^1(t))\right]dt
	\ne \int^{\tau}_0\left[\frac{1}{2}(J_n\dot{\hat{z}}^2_\lambda(t), \hat{z}_\lambda^2(t))_{\mathbb{R}^{2n}}+ {H}(\lambda, \hat{z}_\lambda^2(t))\right]dt.
	$$
\end{itemize}
\end{enumerate}
As above let $z_\lambda(t)=(x_\lambda(t)^\top, y_\lambda(t)^\top)^{\top}$,
 $z^l_\lambda(t)=(x^l_\lambda(t)^\top, y^l_\lambda(t)^\top)^{\top}$
and $\hat{z}^i_\lambda(t)=
(\hat{x}^i_\lambda(t)^\top, \hat{y}^i_\lambda(t)^\top)^{\top}$.
By Claim~\ref{cl:crm1},
 $y_\lambda(t)=x_\lambda(t-\tau)$,
 $y^l_\lambda(t)=x^l_\lambda(t-\tau)$
and  $\hat{y}^i_\lambda(t)=\hat{x}^i_\lambda(t-\tau)$, and
 these $x_\lambda$,
$x^l_\lambda$ and $\hat{x}^i_\lambda$ satisfy
 (\ref{e:crm1}). Clearly, the inequality
$\|\hat{z}_\lambda^i|_{[0, \tau]}\|_{C^1}<\varepsilon$ implies that
$$
\|\hat{x}_\lambda^i|_{[0, \tau]}\|_{C^1}<\varepsilon\quad\hbox{and}\quad
\|\hat{x}_\lambda^i|_{[\tau, 2\tau]}\|_{C^1}=
\|\hat{y}_\lambda^i|_{[0, \tau]}\|_{C^1}<\varepsilon
$$
because  $y^l_\lambda(t)=x^l_\lambda(t-\tau)=
-x^l_\lambda(t+\tau)$. Moreover,
using integration by parts yields
$$
\frac{1}{2}\int^{\tau}_0(J_n\dot{\hat{z}}_\lambda^i(t),\hat{z}^i_\lambda(t))_{\mathbb{R}^{2n}}dt=
({\hat{x}}_\lambda^i(0),\hat{x}^i_\lambda(\tau))_{\mathbb{R}^{n}}dt+
\int^{\tau}_0(\dot{\hat{x}}_\lambda^i(t),\hat{y}^i_\lambda(t))_{\mathbb{R}^{n}}dt
$$
for $i=1,2$. We thus obtain conclusions (iii) and (iv) of Theorem~\ref{th:bif-per3DelayCrm}.

\noindent\textbf{Step 3}(\textsf{Proof of (II)}).
The desired conclusion follows by a routine modification of the proof of (II) in Theorem~\ref{th:bif-per3Delay} or \ref{th:bif-per3DelayII}.
\end{proof}



Analogously to the proof of Theorem~\ref{th:bif-per2Delay++}, Corollary~1.15 of \cite{Lu11} yields a parallel result that complements Theorem~\ref{th:bif-per3DelayCrm}.

\begin{theorem}\label{cor:crm5}
	Let $H_0, \hat{H}: \mathbb{R}^{2n}\to\mathbb{R}$ be $C^2$-functions that are $J_n$-invariant and satisfy $\nabla H_0(0)=0$ and $\nabla\hat{H}(0)=0$.
	Assume that $\nabla^2 H_0(0)$ is either positive definite or negative definite.
	Define $H(\lambda,x,y)=H_0(x,y)+\lambda\hat{H}(x,y)$ and $\Upsilon_\lambda(t)=\exp\left(tJ_n(\nabla^2 H_0(0) + \lambda\nabla^2\hat{H}(0))\right)$.
	Then:
	\begin{description}
		\item[(I)]
	For any $\tau>0$, the set
	$\Sigma_\tau:=\{\lambda\in\mathbb{R} \mid \nu_{\tau,-J_n}(\Upsilon_\lambda)>0\}$
		is discrete in $\mathbb{R}$.
	\item[(II)] For each $\mu\in\Sigma_\tau$,
	 at least  one of the following alternatives holds true:
	\begin{enumerate}
		\item[\rm (i)] System (\ref{e:crm1}) with $\lambda=\mu$ admits a sequence of $\R$-distinct nonzero solutions
		${x}^{\mu,k}$ for $k=1,2,\cdots$,
		which  converges to zero in $C^1_{\rm loc}(\mathbb{R},\R^{n})$ as $k\to\infty$.
		\item[\rm (ii)]
		There exist a left neighborhood $\Lambda^-$ and a right neighborhood $\Lambda^+$ of $\mu$ in $\mathbb{R}$,
		and nonnegative integers $n^+$ and $n^-$ with $n^++n^-\ge \nu_{\tau,-J_n}(\Upsilon_\mu)/2$,
		such that for each $\lambda\in\Lambda^-\setminus\{\mu\}$ (resp. $\lambda\in\Lambda^+\setminus\{\mu\}$),
		system (\ref{e:crm1}) with parameter $\lambda$ possesses at least $n^-$ (resp. $n^+$) $\mathbb{R}$-distinct nonzero solutions $x^{\lambda,i}$, $i=1,\dots,n^-$ (resp. $i=1,\dots,n^+$),
		which satisfy $x^{\lambda,i}\to 0$ in $C^1_{\rm loc}(\mathbb{R},\mathbb{R}^{n})$ as $\lambda\to\mu$.
			\end{enumerate}
	Moreover, if $\nu_{\tau,-J_n}(\Upsilon_\mu)\ge 3$,
then at least one of  (i), (iii), and (iv) holds, with (i) as stated previously, and with (iii) and (iv) given by:
	\begin{enumerate}
		\item[\rm (iii)]
		For every $\lambda\in\mathbb{R}\setminus\{\mu\}$ sufficiently close to $\mu$, there exists a nonzero solution $\bar{x}^\lambda$ of system (\ref{e:crm1}) with parameter $\lambda$, such that $\bar{x}^\lambda\to 0$ in $C^1_{\rm loc}(\mathbb{R},\mathbb{R}^{n})$ as $\lambda\to\mu$.
		
		\item[\rm (iv)] For any given $\varepsilon>0$, there exists a one-sided neighborhood $\Lambda^0$ of $\mu$ in $\mathbb{R}$ such that for each $\lambda\in\Lambda^0\setminus\{\mu\}$, system (\ref{e:crm1}) with parameter $\lambda$ has either:
		\begin{itemize}
			\item infinitely many $\mathbb{R}$-distinct nonzero solutions $\bar{x}^{\lambda,k}$ ($k=1,2,\dots$) satisfying \\ $\|\bar{x}^{\lambda,k}|_{[0,\tau]}\|_{C^1}<\varepsilon$, or
			\item at least two $\mathbb{R}$-distinct nonzero solutions $\hat{x}^{\lambda,1},\hat{x}^{\lambda,2}$ with $\|\hat{x}^{\lambda,i}|_{[0,\tau]}\|_{C^1}<\varepsilon$ ($i=1,2$) and
			\begin{align*}
				& -\bigl(\hat{x}_\lambda^1(0),  \hat{x}_\lambda^1(\tau) \bigr)_{\mathbb{R}^{n}}
				+ \int_0^{\tau} \bigl( \dot{\hat{x}}_\lambda^1(t),  \hat{x}_\lambda^1(t-\tau) \bigr)_{\mathbb{R}^{n}} \, dt
				+ \int_{0}^{\tau} H\bigl(\lambda,
				\hat{x}_\lambda^1(t), \hat{x}_\lambda^1(t-\tau)\bigr) \, dt \\
				\ne \; &
				-\bigl(\hat{x}_\lambda^2(0),  \hat{x}_\lambda^2(\tau) \bigr)_{\mathbb{R}^{n}}
				+ \int_0^{\tau} \bigl( \dot{\hat{x}}_\lambda^2(t),  \hat{x}_\lambda^2(t-\tau) \bigr)_{\mathbb{R}^{n}} \, dt
				+ \int_{0}^{\tau} H\bigl(\lambda,
				\hat{x}_\lambda^2(t),
				\hat{x}_\lambda^2(t-\tau)
				\bigr) \, dt.
			\end{align*}			
		\end{itemize}
		\end{enumerate}
	\end{description}
	\end{theorem}


\section{Bifurcation near the trivial equilibrium of system (\ref{e:Bifu-distributedDelay1})}\label{sec:delay5}

Under Assumption~\ref{ass:Distri1Delay1}, let $H(\lambda,z) = 2G(\lambda,x) + F(\lambda, y)$
for $z=(x^\top,y^\top)^\top\in\mathbb{R}^{2n}$.
By the proof of Lemma~2.1 in \cite{LiuWZ25} we have:
\begin{itemize}
\item[$\bullet$]   If a differentiable curve
$x:\mathbb{R}\to \mathbb{R}^n$ satisfies (\ref{e:Bifu-distributedDelay1}),
setting $y(t):=\int_0^\tau
\nabla_2 G(\lambda, x(t-s)) \mathrm{d}s$, then  $z(t)=(x(t)^\top, y(t)^\top)^\top$
 satisfies
\begin{equation}\label{e:distributedDelay2}
\dot{z}(t) = J_n\nabla_2 H(\lambda,  z(t))\quad\hbox{and}\quad
z(t+\tau) = - z(t)\;\forall t.
\end{equation}
\item[$\bullet$] Conversely, if $z(t)=(x(t)^\top, y(t)^\top)^\top$ satisfies (\ref{e:distributedDelay2}), then
 $y(t)=\int_0^\tau\nabla_2 G(\lambda, x(t-s)) \mathrm{d}s$ and $x(t)$ satisfies
(\ref{e:Bifu-distributedDelay1}).
\end{itemize}
In particular, by Assumption~\ref{ass:Distri1Delay1}(iii), the zero map $\mathbb{R}\to\mathbb{R}^{2n}$
(denoted by $0$ by a slight abuse of notation) satisfies (\ref{e:distributedDelay2}) for any $\lambda\in\Lambda$.


\begin{theorem}\label{th:DistribifDelay1}
Under Assumption~\ref{ass:Distri1Delay1}, with $n=1$, the distributed-delay differential system (\ref{e:Bifu-distributedDelay1})
becomes
\begin{equation}\label{e:Bifu-distributedDelay1=n}
\left\{\begin{array}{ll}
\dot{x}(t) = -\frac{\partial F}{\partial x}\left(\lambda,  \displaystyle\int_0^\tau
\frac{\partial G}{\partial x}(\lambda, x(t-s)) \,\mathrm{d}s\right), & x(t) \in \mathbb{R},
\\
x(t+\tau) = -x(t) \; \forall t.
\end{array}\right.
\end{equation}
Let $f(\lambda)=\frac{\partial^2 F}{\partial x^2}(\lambda, 0)$
and $g(\lambda)=\frac{\partial^2 G}{\partial x^2}(\lambda, 0)$.
\begin{enumerate}
\item[\rm (I)]{\rm (\textsf{Necessary condition}):}
If $(\mu,0)$ with $\mu\in\Lambda$ is a bifurcation point of
the problem (\ref{e:Bifu-distributedDelay1=n}), then
$\tau\sqrt{2g(\mu)f(\mu)}\in \pi(2\mathbb{N}-1)$.

\item[\rm (II)]{\rm (\textsf{Sufficient condition}):}
Let $\mu$ be an interior point of $\Lambda$  satisfying
the following conditions:
\begin{enumerate}
\item[\rm (a)] $f(\mu)g(\mu)>0$ and
$\tau\sqrt{2g(\mu)f(\mu)}\in\pi(2\mathbb{N}-1)$.
\item[\rm (b)] There exist sequences $(\lambda_m^\pm)_m\subset\Lambda$  converging to
$\mu$ such that
$f(\lambda_m^-)g(\lambda_m^-)<f(\mu)g(\mu)<f(\lambda_m^+)g(\lambda_m^+)$ for all $m\in\mathbb{N}$.
 \end{enumerate}
Then $(\mu,0)$ is a bifurcation point of (\ref{e:Bifu-distributedDelay1=n}).

\item[\rm (III)]{\rm (\textsf{Alternative bifurcations of
Fadell-Rabinowitz type }):}
    Let the assumptions of (II) hold, with condition (b)  replaced by the stronger condition:
    \begin{enumerate}
\item[\rm (sb)] In a deleted neighborhood of $\mu$,
 $f(\lambda)g(\lambda)-f(\mu)g(\mu)\ne 0$ and changes sign
 as $\lambda$ across $\mu$.
 \end{enumerate}
 (This can be satisfied if $g(\lambda)f(\lambda)$ has a nonzero derivative
 at $\lambda=\mu$.)
Then, for the problem (\ref{e:Bifu-distributedDelay1=n}),
there exist the following  alternatives:
\begin{itemize}
	\item[\rm (i)] The problem (\ref{e:Bifu-distributedDelay1=n}) with $\lambda=\mu$ admits a sequence of nonzero distinct solutions
	$\{x^l\}^{\infty}_{l=1}$  such that $x^l\to 0$ in $C^1_{\rm loc}(\mathbb{R},\mathbb{R})$
	as $l\to\infty$.
	\item[\rm (ii)] There exist left and right  neighborhoods $\Lambda^-$ and $\Lambda^+$ of $\mu$ in $\Lambda$
	and nonnegative integers $n^+$ and $n^-$ satisfying $n^++n^-\ge 1$, such that for $\lambda\in\Lambda^-\setminus\{\mu\}$ (resp. $\lambda\in\Lambda^+\setminus\{\mu\}$),
	(\ref{e:Bifu-distributedDelay1=n}) with parameter  $\lambda$  has at least $n^-$ (resp. $n^+$) $\mathbb{R}$-distinct  nontrivial solutions
	$x_\lambda^i$ ($i=1,\dots,n^-$ (resp. $n^+$))
	which converge to zero in $C^1_{\rm loc}(\mathbb{R}, \mathbb{R})$
	as $\lambda\to\mu$.
\end{itemize}
\end{enumerate}
\end{theorem}
\begin{proof}[\bf Proof]
\noindent\textbf{Step 1}(\textsf{Proof of (I)}).
 Let $(\mu,0)$ with $\mu\in\Lambda$ be a bifurcation point of
the problem (\ref{e:Bifu-distributedDelay1=n}).
As argued in the paragraph preceding Theorem~\ref{th:DistribifDelay1},
this very $\mu$ yields a bifurcation point $(\mu,0)$ of
(\ref{e:distributedDelay2}).
Hence, \cite[Theorem~1.5(I)]{Lu11} yields
$\nu_{\tau,-I_{2}}(\gamma_\mu)>0$, where
${\gamma}_\lambda$  is the fundamental matrix solution  of
$\dot{z}(t) = J_1\nabla^2_{2}H(\lambda, 0)\, z(t)$ with $\gamma_\lambda(0)=I_{2}$.
 In view of the block-diagonal structure
$\nabla^2_{2}H(\lambda, 0)={\rm diag}(2g(\lambda), f(\lambda))$,
 $\gamma_\lambda$ agrees with $\gamma_{2g(\lambda), f(\lambda)}$ from Lemma~\ref{lem:distri2}.
 The conclusion in (I) now follows since Lemma~\ref{lem:distri2} forces
  $\sqrt{2g(\mu)f(\mu)}\tau$ to lie in $\pi(2\mathbb{N}-1)$.

\noindent\textbf{Step 2}(\textsf{Proof of (III)}).
It suffices to prove that \cite[Theorem~1.14]{Lu11} can be applied to (\ref{e:distributedDelay2}) with $n=1$.
Since $M=-I_2$ implies $\operatorname{Ker}(M-I_{2})=\{0\}$,
 condition (a) in \cite[Theorem~1.14]{Lu11} is satisfied.
By condition (a),
 Lemma~\ref{lem:distri2} gives
\begin{equation}\label{e:distributedDelay3}
 \nu_{\tau,-I_{2}}(\gamma_\mu)=2\quad\text{and}\quad   i_{\tau,-I_{2}}(\gamma_\mu)=
 \begin{cases}
2(j-1),\quad&\hbox{if $g(\mu)>0$},\\
-2j,\quad&\hbox{if $g(\mu)<0$}.
\end{cases}
 \end{equation}
In order to compute $i_{\tau,-I_{2}}(\gamma_\lambda)$ and
$\nu_{\tau,-I_{2}}(\gamma_\lambda)$ for $\lambda$ near $\mu$,
note that condition (sb) may be precisely formulated as follows:
  There exists a $\delta>0$ such that
  either of the following holds:
  \begin{itemize}
\item[\rm (sb.1)] $g(\lambda)f(\lambda)< g(\mu)f(\mu)<g(\lambda')f(\lambda')$
for all $(\lambda, \lambda')\in (\mu-\delta,\mu)\times(\mu,\mu+\delta)$.

  \item[\rm (sb.2)] $g(\lambda)f(\lambda)> g(\mu)f(\mu)>g(\lambda')f(\lambda')$
for all $(\lambda, \lambda')\in (\mu-\delta,\mu)\times(\mu,\mu+\delta)$.
  \end{itemize}

\textbf{Case $g(\mu)>0$}. Then $f(\mu)>0$ as well. Since both $g(\lambda)$ and
$f(\lambda)$ are continuous, by further shrinking $\delta>0$ if necessary
 we may assume that $g(\lambda)>0$ and $f(\lambda)>0$ for all $\lambda\in (\mu-\delta,\mu+\delta)$ and that 
 (sb.1) and (sb.2) imply the following two claims, respectively:
 \begin{itemize}
\item[$\bullet$] $(2j-2)\pi<\tau\sqrt{2g(\lambda)f(\lambda)}<(2j-1)\pi
<\tau\sqrt{2g(\lambda')f(\lambda')}<2j\pi$
for all $(\lambda, \lambda')\in (\mu-\delta,\mu)\times(\mu,\mu+\delta)$.

\item[$\bullet$] $(2j-2)\pi<\tau\sqrt{2g(\lambda')f(\lambda')}<(2j-1)\pi<\tau\sqrt{2g(\lambda)f(\lambda)}<2j\pi$
for all $(\lambda, \lambda')\in (\mu-\delta,\mu)\times(\mu,\mu+\delta)$.
\end{itemize}

 Clearly, in both cases we have $\nu_{\tau,-I_2}(\gamma_\lambda)=0$
 for all $0<|\lambda-\mu|<\delta$.
In the first case it holds
\begin{align*}
&j-\frac{3}{2}<\frac{\tau\sqrt{2g(\lambda)f(\lambda)}}{2\pi}-\frac{1}{2}<j-1,\quad
\forall \lambda\in (\mu-\delta,\mu),\\
&j-1<\frac{\tau\sqrt{2g(\lambda)f(\lambda)}}{2\pi}-\frac{1}{2}<j-\frac{1}{2},\quad
\forall \lambda\in (\mu,\mu+\delta).
\end{align*}
It follows from these and Lemma~\ref{lem:distri2} that
\begin{equation}\label{e:distributedDelay4}
    i_{\tau,-I_{2}}(\gamma_\lambda)=
 \begin{cases}
 2j-2=i_{\tau,-I_{2}}(\gamma_\mu),&\forall \,\lambda\in (\mu-\delta,\mu),\\
2j=i_{\tau,-I_{2}}(\gamma_\mu)+\nu_{\tau,-I_{2}}(\gamma_\mu),&\forall \,\lambda\in (\mu,\mu+\delta).
\end{cases}
 \end{equation}
Similarly, in the second case we obtain
\begin{equation}\label{e:distributedDelay5}
    i_{\tau,-I_{2}}(\gamma_\lambda)=
 \begin{cases}
2j=i_{\tau,-I_{2}}(\gamma_\mu)+\nu_{\tau,-I_{2}}(\gamma_\mu),&\forall\,\lambda\in (\mu-\delta,\mu),\\
2j-2=i_{\tau,-I_{2}}(\gamma_\mu),&\forall \,\lambda\in (\mu,\mu+\delta).
\end{cases}
 \end{equation}

\textbf{Case $g(\mu)<0$}. Thus $f(\mu)<0$ as well, and
 by further shrinking $\delta>0$ if necessary
 we may assume that $g(\lambda)<0$ and $f(\lambda)<0$ for all $\lambda\in (\mu-\delta,\mu+\delta)$ and that (sb.1) and (sb.2) imply
 the following two claims, respectively:
 \begin{itemize}
\item[$\bullet$] $(-2j)\pi<-\tau\sqrt{2g(\lambda')f(\lambda')}<(-2j+1)\pi<-\tau\sqrt{2g(\lambda)f(\lambda)}<(-2j+2)\pi$
for all $(\lambda, \lambda')\in (\mu-\delta,\mu)\times(\mu,\mu+\delta)$.
\item[$\bullet$] $(-2j)\pi<-\tau\sqrt{2g(\lambda)f(\lambda)}<(-2j+1)\pi
<-\tau\sqrt{2g(\lambda')f(\lambda')}<(-2j+2)\pi$
for all $(\lambda, \lambda')\in (\mu-\delta,\mu)\times(\mu,\mu+\delta)$.
\end{itemize}
 In each case, $\nu_{\tau,-I_2}(\gamma_\lambda)=0$
 whenever  $0<|\lambda-\mu|<\delta$.
As above, we also obtain
\begin{equation}\label{e:distributedDelay6}
     i_{\tau,-I_{2}}(\gamma_\lambda)=
 \begin{cases}
 -2j+2=i_{\tau,-I_{2}}(\gamma_\mu)+\nu_{\tau,-I_{2}}(\gamma_\mu),&\forall\,\lambda\in (\mu-\delta,\mu),\\
-2j=i_{\tau,-I_{2}}(\gamma_\mu),&\forall \,\lambda\in (\mu,\mu+\delta)
\end{cases}
 \end{equation}
 in the first case, and
\begin{equation}\label{e:distributedDelay7}
   i_{\tau,-I_{2}}(\gamma_\lambda)=
 \begin{cases}
-2j=i_{\tau,-I_{2}}(\gamma_\mu),&\forall \,\lambda\in (\mu-\delta,\mu),\\
-2j+2=i_{\tau,-I_{2}}(\gamma_\mu)+\nu_{\tau,-I_{2}}(\gamma_\mu),&\forall \,\lambda\in (\mu,\mu+\delta)
\end{cases}
 \end{equation}
in the second case.

By (\ref{e:distributedDelay4}), (\ref{e:distributedDelay5}),
(\ref{e:distributedDelay6}) and (\ref{e:distributedDelay7}),
we apply the first part of Theorem~1.14 of  \cite{Lu11} to the problem (\ref{e:distributedDelay2}) with $n=1$ and obtain
the following alternatives:
 \begin{enumerate}
\item[\rm (i')] The problem (\ref{e:distributedDelay2}) with $n=1$
 and $\lambda=\mu$ has a sequence of $\mathbb{R}$-distinct solutions,
$z^l\ne 0$ ($l=1,2,\dots$) such that $\|z^l|_I\|_{C^1}\to 0$ for any compact interval $I\subset\mathbb{R}$.
\item[\rm (ii')] There exist left and right  neighborhoods $\Lambda^-$ and $\Lambda^+$ of $\mu$ in $\Lambda$
and nonnegative integers $n^+$ and $n^-$ satisfying $n^++n^-\ge \nu_{\tau,-I_2}(\gamma_\mu)/2=1$,
such that for $\lambda\in\Lambda^-\setminus\{\mu\}$ (resp. $\lambda\in\Lambda^+\setminus\{\mu\}$),
 problem (\ref{e:distributedDelay2}) with $n=1$ and parameter  $\lambda$  has at least $n^-$ (resp. $n^+$) $\R$-distinct
 solutions, $z_\lambda^i\ne 0$ ($i=1,\dots,n^-$ resp. $n^+$)
 whose restrictions to $[0,\tau]$  converge to zero in $C^1([0,\tau];\mathbb{R}^{2})$  as $\lambda\to\mu$.
\end{enumerate}
By the arguments  preceding Theorem~\ref{th:DistribifDelay1},
they are equivalent  to (i) and (ii) in Theorem~\ref{th:DistribifDelay1},
respectively.

\noindent\textbf{Step 3}(\textsf{Proof of (II)}).
By slightly modifying the proof of (II) in Theorem~\ref{th:bif-per3Delay} or \ref{th:bif-per3DelayII}, we obtain the desired conclusion.
\end{proof}

We now turn to the case $n>1$, for which we need:

\begin{lemma}\label{lem:Kato}
Let  $A(\lambda)$ and $B(\lambda)$ be continuous, real symmetric $n\times n$ matrix-valued functions on a topological space $\Lambda$
that commute for all $\lambda\in\Lambda$. For some $\mu\in\Lambda$,
suppose that $A(\mu)$ has $n$ distinct eigenvalues $a_j(\mu)$, $j=1,\cdots,n$, and
therefore there exists a $\varepsilon>0$ such that intervals
$$
[a_j(\mu)-2\varepsilon,  a_j(\mu)+2\varepsilon],\;   j=1,\cdots,n
$$
are pairwise disjoint. For each $j=1,\cdots,n$, fix a unit eigenvector $v_j$ of $A(\mu)$ corresponding to the eigenvalue $a_j(\mu)$.
Then there exists a neighborhood $\Lambda_0$ of $\mu$ in $\Lambda$ such that
for each $\lambda\in\Lambda_0$, $A(\lambda)$ has a unique eigenvalue in
$[a_j(\mu)-\varepsilon,  a_j(\mu)+\varepsilon]$, denoted by $a_j(\lambda)$, $j=1,\cdots,n$.
Denote by $P_j(\lambda)$ the orthogonal projection onto the eigenspace of $A(\lambda)$
corresponding to $a_j(\lambda)$, $j=1,\cdots,n$.
By the argument in the proof of Proposition~B.1 in \cite{Lu12}, these projections depend continuously on $\lambda$.
As shown in \cite[II. \S3.3]{Kato76}, the vectors defined by
$$
v_j(\lambda):=((P(\lambda)v_j, v_j)_{\mathbb{R}^n})^{-1}P(\lambda)v_j
$$
are normalized eigenvectors of  $A(\lambda)$  corresponding to the eigenvalue $a_j(\lambda)$, for $j=1,\cdots,n$.
Clearly, these eigenvectors are continuous in $\lambda$, and hence the functions
$$
a_j(\lambda)=(A(\lambda)v_j(\lambda), v_j(\lambda))_{\mathbb{R}^n},\quad
j=1,\cdots,n,
$$
 are continuous real functions on $\Lambda_0$.
Consequently, since $A(\lambda)$ and $B(\lambda)$ commute,
each $v_j(\lambda)$ is a normalized eigenvector of $B(\lambda)$  with eigenvalue
$$
b_j(\lambda):=(B(\lambda)v_j(\lambda), v_j(\lambda))_{\mathbb{R}^n},\quad
j=1,\cdots,n,
$$
which imply that $b_1(\lambda),\cdots, b_n(\lambda)$ are continuous real functions on $\Lambda_0$.
\end{lemma}

\begin{theorem}\label{th:DistribifDelay2}
	Let Assumption~\ref{ass:Distri1Delay1} hold with $n>1$.
\begin{enumerate}
\item[\rm (I)]{\rm (\textsf{Necessary condition}):}
If $(\mu,0)$ with $\mu\in\Lambda$ is a bifurcation point of
the problem (\ref{e:Bifu-distributedDelay1}),
and if $\nabla^2_2F(\mu, 0)$ and $\nabla^2_2G(\mu, 0)$ commute,
so that
\begin{align*}
\nabla^2_2G(\mu, 0)=E^\top\operatorname{diag}(g_1(\mu),\dots, g_n(\mu))
\quad\hbox{and}\quad
 \nabla^2_2F(\mu, 0)=E^\top\operatorname{diag}(f_1(\mu),\dots, f_n(\mu))E
\end{align*}
for some real orthogonal  $n\times n$ matrix $E$,
then $\tau\sqrt{2g_l(\mu)f_l(\mu)}\in \pi(2\mathbb{N}-1)$ for some $l$.

\item[\rm (II)]{\rm (\textsf{Sufficient condition}):}
For each $\lambda\in\Lambda$, suppose  $\nabla^2_2F(\lambda, 0)$ and $\nabla^2_2G(\lambda, 0)$ commute. Assume further that for some
 $\mu\in{\rm Int}(\Lambda)$, the matrix $\nabla^2_2G(\mu, 0)$ has $n$ distinct eigenvalues $g_j(\mu)$, $j=1,\dots,n$.
Thus Lemma~\ref{lem:Kato} assures that in a neighborhood $\Lambda_0$ of $\mu$ in $\Lambda$
there exist continuous real functions $g_j, f_j$ for $j=1,\cdots,n$, and
a continuous real orthogonal $n\times n$ matrix-valued function $\Xi$ such that for each $\lambda\in\Lambda_0$,
\begin{align*}
\nabla^2_2G(\lambda, 0)&=\Xi(\lambda)^\top{\rm diag}(g_1(\lambda),\cdots, g_n(\lambda))\Xi(\lambda)
\quad\hbox{and}\\
 \nabla^2_2F(\lambda, 0)&=\Xi(\lambda)^\top{\rm diag}(f_1(\lambda),\cdots, f_n(\lambda))\Xi(\lambda).
\end{align*}
Suppose there exists a partition
	$\{i_1,\dots,i_k\}\cup\{i_{k+1},\dots,i_n\}$  of $\{1,\dots,n\}$
with $1 \le k \le n$, such that the following two conditions hold:
	\begin{itemize}
\item[\rm (A)] For each $l=k+1,\dots,n$, either $
2g_{i_l}(\mu)f_{i_l}(\mu)<(\pi/\tau)^2$ or $$
    (2j_l(\mu)-1)\pi<\tau\sqrt{2g_{i_l}(\mu)f_{i_l}(\mu)}<
    (2j_l(\mu)+1)\pi
    $$
    for some  $j_l(\mu)\in\mathbb{N}$. (When $k=n$, there is no such $l$, so the condition is vacuous.)
		\item[\rm (B)]  For each $l=1,\dots,k$, $g_{i_l}(\mu)f_{i_l}(\mu)>0$ and
$\sqrt{2g_{i_l}(\mu)f_{i_l}(\mu)}\tau=(2j_l(\mu)-1)\pi$ for some
 $j_l(\mu)\in\mathbb{N}$.

 \item[\rm (C)]There exist sequences $(\lambda_m^\pm)_m\subset\Lambda$
  converging to $\mu$ such that
$f_{i_l}(\lambda_m^-)g_{i_l}(\lambda_m^-)<f_{i_l}(\mu)g_{i_l}(\mu)
<f_{i_l}(\lambda_m^+)g_{i_l}(\lambda_m^+)$ for all $m\in\mathbb{N}$
 and for all $l=1,\cdots,k$.
 \end{itemize}
 Then $(\mu,0)$  is a bifurcation point of
 (\ref{e:Bifu-distributedDelay1}).

\item[\rm (III)]{\rm (\textsf{Alternative bifurcations of Fadell-Rabinowitz type and of Rabinowitz type}):}
  Let the assumptions of (II) hold, with condition (C) strengthened as follows:
  \begin{itemize}
 \item[\rm (SC)] For some $\delta>0$ and for each $l=1,\cdots,k$, one of the following two conditions holds:
   \begin{itemize}
		\item[\rm (SC.1)]  $g_{i_l}(\lambda)f_{i_l}(\lambda)<g_{i_l}(\mu)f_{i_l}(\mu)<
g_{i_l}(\lambda')f_{i_l}(\lambda')$ for $\mu-\delta<\lambda<\mu<\lambda'<\mu+\delta$
if $g_{i_l}(\mu)>0$; and
$g_{i_l}(\lambda)f_{i_l}(\lambda)>g_{i_l}(\mu)f_{i_l}(\mu)>
g_{i_l}(\lambda')f_{i_l}(\lambda')$ for $\mu-\delta<\lambda<\mu<\lambda'<\mu+\delta$
if $g_{i_l}(\mu)<0$.

		\item[\rm (SC.2)] $g_{i_l}(\lambda)f_{i_l}(\lambda)>g_{i_l}(\mu)f_{i_l}(\mu)>
g_{i_l}(\lambda')f_{i_l}(\lambda')$ for $\mu-\delta<\lambda<\mu<\lambda'<\mu+\delta$
if $g_{i_l}(\mu)>0$; and
$g_{i_l}(\lambda)f_{i_l}(\lambda)<g_{i_l}(\mu)f_{i_l}(\mu)<
g_{i_l}(\lambda')f_{i_l}(\lambda')$ for $\mu-\delta<\lambda<\mu<\lambda'<\mu+\delta$
if $g_{i_l}(\mu)<0$.
\end{itemize}
\end{itemize}
    Then for the problem (\ref{e:Bifu-distributedDelay1}),
	there exist two possible alternatives:
	\begin{itemize}
		\item[\rm (i)] The problem (\ref{e:Bifu-distributedDelay1})  with $\lambda=\mu$ admits a sequence of
		nonzero, $\R$-distinct solutions
		$\{x^l\}^{\infty}_{l=1}$  such that $x^l\to 0$ in $C^1_{\rm loc}(\mathbb{R},\mathbb{R}^n)$
		as $l\to\infty$.
		\item[\rm (ii)] There exist left and right  neighborhoods $\Lambda^-$ and $\Lambda^+$ of $\mu$ in $\Lambda$
		and nonnegative integers $n^+$ and $n^-\ge 0$ satisfying $n^++n^-\ge k$, such that
		for $\lambda\in\Lambda^-\setminus\{\mu\}$ (resp. $\lambda\in\Lambda^+\setminus\{\mu\}$),
		(\ref{e:Bifu-distributedDelay1}) with parameter  $\lambda$  has at least $n^-$ (resp. $n^+$) $\R$-distinct
		solutions $x_\lambda^i\ne 0$ ($i=1,\dots,n^-$ (resp. $n^+$))
		which converge to zero in $C^1_{\rm loc}(\mathbb{R}, \mathbb{R}^n)$
		as $\lambda\to\mu$.
	\end{itemize}
	Moreover, if  $k>1$, then at least one of (i), (iii), and (iv)  holds,
	where (i) is as stated previously, and
	\begin{itemize}
		\item[\rm (iii)]  For every $\lambda\in\Lambda\setminus\{\mu\}$ near $\mu$, there is a
		nonzero solution ${x}_\lambda$  of (\ref{e:Bifu-distributedDelay1})
		with parameter $\lambda$, such that  $x_\lambda\to 0$ in $C^1_{\rm loc}(\mathbb{R},\mathbb{R}^n)$
		as $\lambda\to\mu$.
				
		\item[\rm (iv)] For a given $\varepsilon > 0$, there exists a one-sided neighbourhood $\Lambda^0$ of $\mu$ in $\Lambda$ such that, for any $\lambda \in \Lambda^0 \setminus \{\mu\}$, problem
	(\ref{e:Bifu-distributedDelay1}) with parameter $\lambda$ has either
	\begin{itemize}
		\item[$\bullet$]
	 infinitely many $\mathbb{R}$-distinct nonzero solutions $\{x_\lambda^l\}_{l=1}^\infty$ such that $\|x_\lambda^l|_{[0,2\tau]}\|_{C^1} < \varepsilon$ for all $l \in \mathbb{N}$,
	 or
	\item[$\bullet$]  at least two $\mathbb{R}$-distinct nonzero solutions $\hat{x}_\lambda^1$ and $\hat{x}_\lambda^2$ satisfying the following inequalities: $\|\hat{x}_\lambda^i|_{[0,2\tau]}\|_{C^1} < \varepsilon$ for $i=1,2$, and
		\begin{align*}
			& \int_0^{\tau} \bigl(\dot{\hat{x}}_\lambda^1(t),  \hat{y}_\lambda^1(t) \bigr)_{\mathbb{R}^{n}} \, dt
			+ 2\int_0^{\tau} G(\lambda, {\hat{x}}_\lambda^1(t))\, dt
			+ \int_{0}^{\tau} F\bigl(\lambda,
			\hat{y}_\lambda^1(t)\bigr) \, dt \\
			\ne \; &
		\int_0^{\tau} \bigl(\dot{\hat{x}}_\lambda^2(t),  \hat{y}_\lambda^2(t) \bigr)_{\mathbb{R}^{n}} \, dt
			+ 2\int_0^{\tau} G(\lambda, {\hat{x}}_\lambda^2(t))\, dt
			+ \int_{0}^{\tau} F\bigl(\lambda,
			\hat{y}_\lambda^2(t)\bigr) \, dt,	
		\end{align*}
where $\hat{y}_\lambda^j(t)=\int^\tau_0\nabla_2G(\lambda, \hat{x}_\lambda^j(t-s))ds$, $j=1,2$.
	\end{itemize}
\end{itemize}
\end{enumerate}
\end{theorem}
\begin{proof}[\bf Proof]
Let ${\gamma}_\lambda$ be the fundamental matrix solution of
$\dot{z}(t) = J_n\nabla^2_{2}H(\lambda, 0)\, z(t)$ with $\gamma_\lambda(0)=I_{2n}$.
Note that $\nabla^2_{2}H(\lambda, 0)=\begin{pmatrix}
2\nabla^2_{2}G(\lambda, 0) & 0 \\
0 & \nabla^2_{2}F(\lambda, 0)
\end{pmatrix}$.
Suppose $(\mu,0)$ with $\mu\in\Lambda$ is a bifurcation point of
the problem (\ref{e:Bifu-distributedDelay1}).
Then, by the same argument as in the proof of Theorem~\ref{th:DistribifDelay1}(I), we obtain  $\nu_{\tau,-I_{2n}}(\gamma_\mu)>0$, which establishes the necessity
in (I) since
$$
\nu_{\tau,-I_{2n}}(\gamma_\lambda)=2\sharp\left\{j\,\Big|\, \tau\sqrt{2g_j(\mu) f_j(\mu)}\in \pi(2\mathbb{N}-1)\right\}
$$
by Theorem~\ref{th:PIndex5}.

The proof of (II) follows by adapting the argument used in the proof of Theorem~\ref{th:bif-per3Delay} or \ref{th:bif-per3DelayII},
with only minor changes.

\underline{It remains to establish (III)}.
Observe that $M=-I_{2n}$ implies $\operatorname{Ker}(M-I_{2n})=\{0\}$; thus,
condition (a) in \cite[Theorem~1.14]{Lu11} holds.
Toward applying this theorem to (\ref{e:distributedDelay2}), we analyze the behavior of $i_{\tau,-I_{2n}}(\gamma_\lambda)$ and $\nu_{\tau,-I_{2n}}(\gamma_\lambda)$
for $\lambda$ in a neighborhood of $\mu$.
For this purpose,  let $\gamma_{2g_j(\lambda), f_j(\lambda)}$ be as in Lemma~\ref{lem:distri2}.
Then  Theorem~\ref{th:PIndex5} gives
\begin{equation}\label{e:distributedDelay11}
\nu_{\tau,-I_{2n}}(\gamma_\lambda)=2\sharp\left\{j\,\Big|\, \tau\sqrt{2g_j(\lambda) f_j(\lambda)}\in \pi(2\mathbb{N}-1)\right\}.
\end{equation}
The conditions (A) and (B) show
\begin{equation}\label{e:distributedDelay12}
\nu_{\tau,-I_{2n}}(\gamma_\mu)=2k.
\end{equation}

Let $\Gamma\subset\{k+1,\cdots,n\}$ be the set of indices $l$ for which the second inequality in (A) holds.
 Using Lemma~\ref{lem:distri2},  condition (A) yields that for each $l\in\{k+1,\dots,n\}$,
\begin{equation}\label{e:distributedDelay12+}
i_{\tau,-I_{2}}(\gamma_{2g_{i_l}(\mu), f_{i_l}(\mu)})=
\begin{cases}
0 &\quad\text{if $l\notin\Gamma$},\\
2j_{l}(\mu)&\quad\text{if $l\in\Gamma$},
\end{cases}
\end{equation}
 whereas  condition (B) implies that for each $l\in\{1,\dots,k\}$,
\begin{equation}\label{e:distributedDelay12++}
i_{\tau,-I_{2}}(\gamma_{2g_{i_l}(\mu), f_{i_l}(\mu)})=
\begin{cases}
2(j_{l}(\mu)-1) &\quad\text{if $g_{i_l}(\mu)>0$},\\
-2j_{l}(\mu)&\quad\text{if $g_{i_l}(\mu)<0$}.
\end{cases}
\end{equation}

Combining the proof of Theorem~\ref{th:PIndex5} with the above relations, we obtain
\begin{align}\label{e:distributedDelay13}
i_{\tau,-I_{2n}}(\gamma_\mu)&=\sum^n_{j=1}i_{\tau,-I_{2}}(\gamma_{2g_j(\mu), f_j(\mu)})\nonumber\\
&=\sum_{k\ge l:g_{i_l}(\mu)>0}2(j_{l}(\mu)-1)+
\sum_{k\ge l:g_{i_l}(\mu)<0}(-2j_{l}(\mu))
+\sum_{l\in\Gamma}2j_l(\mu).
\end{align}

Now let us compute $i_{\tau,-I_{2n}}(\gamma_\lambda)$ and
$i_{\tau,-I_{2n}}(\gamma_{\lambda'})$
for $\mu-\delta<\lambda<\mu<\lambda'<\mu+\delta$.

Since $g_i(\lambda)f_i(\lambda)$ ($i=1,\cdots,n$) are continuous,
after further shrinking $\delta>0$ if necessary,
conditions (A) and (SC) can be formulated as follows:
\begin{itemize}
\item[\rm (A')] For each $l=k+1,\dots,n$, the following hold:
\begin{align*}
&2g_{i_l}(\lambda)f_{i_l}(\lambda)<(\pi/\tau)^2\quad
\text{for all $\lambda\in (\mu-\delta,\mu+\delta)$,\;if
$l\notin\Gamma$},\\
 &   (2j_l(\mu)-1)\pi<\tau\sqrt{2g_{i_l}(\lambda)f_{i_l}(\lambda)}<
    (2j_l(\mu)+1)\pi\quad
\text{for all $\lambda\in (\mu-\delta,\mu+\delta)$,\; if $l\in\Gamma$}.
  \end{align*}

   	\item[\rm (SC')]  For each $l=1,\dots,k$, one of the following two conditions holds:
   	\begin{itemize}
		\item[\rm (SC'.1)] If $g_{i_l}(\mu)>0$, then
for any $(\lambda,\lambda')\in (\mu-\delta,\mu)\times (\mu,\mu+\delta)$, we have $g_{i_l}(\lambda)>0$,
$g_{i_l}(\lambda')>0$ and
$$
(2j_l(\mu)-2)\pi<\tau\sqrt{2g_{i_l}(\lambda)f_{i_l}(\lambda)}<
\tau\sqrt{2g_{i_l}(\mu)f_{i_l}(\mu)}<
\tau\sqrt{2g_{i_l}(\lambda')f_{i_l}(\lambda')}<2j_l(\mu)\pi;
$$
and
if $g_{i_l}(\mu)<0$, then
for any $(\lambda,\lambda')\in (\mu-\delta,\mu)\times (\mu,\mu+\delta)$, we have
$g_{i_l}(\lambda)<0$, $g_{i_l}(\lambda')<0$ and
$$
2j_l(\mu)\pi>\tau\sqrt{2g_{i_l}(\lambda)f_{i_l}(\lambda)}>\tau
\sqrt{2g_{i_l}(\mu)f_{i_l}(\mu)}>
\tau\sqrt{2g_{i_l}(\lambda')f_{i_l}(\lambda')}>(2j_l(\mu)-2)\pi.
$$
\item[\rm (SC'.2)] If $g_{i_l}(\mu)>0$, then
for any $(\lambda,\lambda')\in (\mu-\delta,\mu)\times (\mu,\mu+\delta)$, we have
 $g_{i_l}(\lambda)>0$,
$g_{i_l}(\lambda')>0$ and
$$
2j_l(\mu)\pi>\tau\sqrt{2g_{i_l}(\lambda)f_{i_l}(\lambda)}>\tau\sqrt{2g_{i_l}(\mu)f_{i_l}(\mu)}>
\tau\sqrt{2g_{i_l}(\lambda')f_{i_l}(\lambda')}>(2j_l(\mu)-2)\pi;
$$
 and if $g_{i_l}(\mu)<0$, then
 for any $(\lambda,\lambda')\in (\mu-\delta,\mu)\times (\mu,\mu+\delta)$, we have $g_{i_l}(\lambda)<0$, $g_{i_l}(\lambda')<0$ and
$$
(2j_l(\mu)-2)\pi<\tau\sqrt{2g_{i_l}(\lambda)f_{i_l}(\lambda)}<
\tau\sqrt{2g_{i_l}(\mu)f_{i_l}(\mu)}<
\tau\sqrt{2g_{i_l}(\lambda')f_{i_l}(\lambda')}<2j_l(\mu)\pi.
$$
\end{itemize}
\end{itemize}

As in the proof of (\ref{e:distributedDelay13}),
 conditions (A') implies  that for any $\lambda\in (\mu-\delta,\mu+\delta)$,
  \begin{align}\label{e:distributedDelay14}
i_{\tau,-I_{2n}}(\gamma_\lambda)&=\sum^n_{j=1}i_{\tau,-I_{2}}(\gamma_{2g_j(\lambda), f_j(\mu)})\nonumber\\
&=\sum^k_{l=1}i_{\tau,-I_{2}}(\gamma_{2g_{i_l}(\lambda), f_{i_l}(\lambda)})+
\sum_{l\in\Gamma}2j_l(\mu).
\end{align}

{Firstly, suppose (SC'.1) holds}.
For any $\lambda\in(\mu-\delta,\mu)$, since $g_{i_l}(\lambda)g_{i_l}(\mu)>0$,
 \begin{align}\label{e:distributedDelay15}
\sum^k_{l=1}i_{\tau,-I_{2}}(\gamma_{2g_{i_l}(\lambda), f_{i_l}(\lambda)})&=
\sum_{k\ge l:g_{i_l}(\mu)>0}i_{\tau,-I_{2}}(\gamma_{2g_{i_l}(\lambda), f_{i_l}(\lambda)})
+\sum_{k\ge l:g_{i_l}(\mu)<0}i_{\tau,-I_{2}}(\gamma_{2g_{i_l}(\lambda), f_{i_l}(\lambda)})\nonumber\\
&=\sum_{k\ge l:g_{i_l}(\mu)>0}2(j_{i_l}(\mu)-1)
+\sum_{k\ge l:g_{i_l}(\mu)<0}(-2j_{i_l}(\mu))
\end{align}
by Lemma~\ref{lem:distri2}. Combining this with (\ref{e:distributedDelay13})
and (\ref{e:distributedDelay14}) we obtain
\begin{align}\label{e:distributedDelay16}
i_{\tau,-I_{2n}}(\gamma_\lambda)=i_{\tau,-I_{2n}}(\gamma_\mu),
\quad\forall \lambda\in(\mu-\delta,\mu).
\end{align}

For any $\lambda'\in(\mu, \mu+\delta)$, since $g_{i_l}(\lambda')g_{i_l}(\mu)>0$,
as above we derive from Lemma~\ref{lem:distri2},
 \begin{align*}
\sum^k_{l=1}i_{\tau,-I_{2}}(\gamma_{2g_{i_l}(\lambda'), f_{i_l}(\lambda')})&=
\sum_{k\ge l:g_{i_l}(\mu)>0}i_{\tau,-I_{2}}(\gamma_{2g_{i_l}(\lambda'), f_{i_l}(\lambda')})
+\sum_{k\ge l:g_{i_l}(\mu)<0}i_{\tau,-I_{2}}(\gamma_{2g_{i_l}(\lambda'), f_{i_l}(\lambda')})\nonumber\\
&=\sum_{k\ge l:g_{i_l}(\mu)>0}2j_{i_l}(\mu)
+\sum_{k\ge l:g_{i_l}(\mu)<0}(-2j_{i_l}(\mu)+2)\nonumber\\
&=\sum_{k\ge l:g_{i_l}(\mu)>0}2(j_{i_l}(\mu)-1)+
\sum_{k\ge l:g_{i_l}(\mu)<0}(-2j_{i_l}(\mu))+2k.
\end{align*}
Combining this with (\ref{e:distributedDelay12}), (\ref{e:distributedDelay13})
and (\ref{e:distributedDelay14}) yields
\begin{align}\label{e:distributedDelay17}
i_{\tau,-I_{2n}}(\gamma_{\lambda'})=i_{\tau,-I_{2n}}(\gamma_\mu)+ \nu_{\tau,-I_{2n}}(\gamma_\mu),
\quad\forall \lambda'\in(\mu, \mu+\delta).
\end{align}

Similarly, when (SC'.2) holds, we may prove
\begin{equation}\label{e:distributedDelay18}
   i_{\tau,-I_{2n}}(\gamma_\lambda)=
 \begin{cases}
i_{\tau,-I_{2n}}(\gamma_\mu)+\nu_{\tau,-I_{2}}(\gamma_\mu),&\forall \,\lambda\in (\mu-\delta,\mu),\\
i_{\tau,-I_{2n}}(\gamma_\mu),&\forall \,\lambda\in (\mu,\mu+\delta)
\end{cases}
 \end{equation}

(\ref{e:distributedDelay16}), (\ref{e:distributedDelay17}) and (\ref{e:distributedDelay18}) show that condition (b) of Theorem~1.14 of \cite{Lu11} is satisfied. Thus, applying the first part of that theorem to (\ref{e:distributedDelay2}), we obtain the following alternatives:
 \begin{enumerate}
\item[\rm (i')] The problem (\ref{e:distributedDelay2}) with
 and $\lambda=\mu$ has a sequence of $\mathbb{R}$-distinct solutions,
$z^l\ne 0$ ($l=1,2,\dots$) such that $\|z^l|_I\|_{C^1}\to 0$ for any compact interval $I\subset\mathbb{R}$.
\item[\rm (ii')] There exist left and right  neighborhoods $\Lambda^-$ and $\Lambda^+$ of $\mu$ in $\Lambda$
and nonnegative integers $n^+$ and $n^-$ satisfying $n^++n^-\ge \nu_{\tau,-I_2}(\gamma_\mu)/2=k$,
such that for $\lambda\in\Lambda^-\setminus\{\mu\}$ (resp. $\lambda\in\Lambda^+\setminus\{\mu\}$),
 problem (\ref{e:distributedDelay2}) with  parameter  $\lambda$  has at least $n^-$ (resp. $n^+$) $\R$-distinct
 solutions, $z_\lambda^i\ne 0$ ($i=1,\dots,n^-$ resp. $n^+$)
 whose restrictions to $[0,\tau]$  converge to zero in $C^1([0,\tau];\mathbb{R}^{2n})$  as $\lambda\to\mu$.
\end{enumerate}
Moreover, if  $k>1$, then $\nu_{\tau,-I_{2n}}(\gamma_\mu)=2k>3$,
and hence we may apply the second part of Theorem~1.14 of \cite{Lu11} to (\ref{e:distributedDelay2}) to conclude
that at least one of (i'), (iii'), and (iv')  holds,
	where (i') is as stated previously, and
	\begin{itemize}
		\item[\rm (iii')]  For every $\lambda\in\Lambda\setminus\{\mu\}$ near $\mu$, there is a nonzero solution ${z}_\lambda$  of (\ref{e:distributedDelay2})
		with parameter $\lambda$, such that  $z_\lambda\to 0$ in $C^1_{\rm loc}(\mathbb{R},\mathbb{R}^{2n})$
		as $\lambda\to\mu$.
				
		\item[\rm (iv')] For a given $\varepsilon > 0$, there exists a one-sided neighbourhood $\Lambda^0$ of $\mu$ in $\Lambda$ such that, for any $\lambda \in \Lambda^0 \setminus \{\mu\}$, problem
	(\ref{e:distributedDelay2}) with parameter $\lambda$ has either
	\begin{itemize}
		\item[$\bullet$]
	 infinitely many $\mathbb{R}$-distinct nonzero solutions $\{z_\lambda^l\}_{l=1}^\infty$ such that $\|z_\lambda^l|_{[0,2\tau]}\|_{C^1} < \varepsilon$ for all $l \in \mathbb{N}$,
	 or
	\item[$\bullet$]  at least two $\mathbb{R}$-distinct nonzero solutions $\hat{z}_\lambda^1$ and $\hat{z}_\lambda^2$ satisfying the following inequalities: $\|\hat{z}_\lambda^i|_{[0,2\tau]}\|_{C^1} < \varepsilon$ for $i=1,2$, and
		\begin{align*}
			&
			+ \frac{1}{2}\int_0^{\tau} \bigl( J_n\dot{\hat{z}}_\lambda^1(t),  \hat{z}_\lambda^1(t) \bigr)_{\mathbb{R}^{2n}} \, dt
			+ \int_{0}^{\tau} H\bigl(\lambda,
			\hat{z}_\lambda^1(t)\bigr) \, dt \\
			\ne \; &
			\frac{1}{2}\int_0^{\tau} \bigl( J_n\dot{\hat{z}}_\lambda^2(t),  \hat{z}_\lambda^2(t) \bigr)_{\mathbb{R}^{2n}} \, dt
			+ \int_{0}^{\tau} H\bigl(\lambda,
			\hat{z}_\lambda^2(t)\bigr) \, dt.
		\end{align*}
	\end{itemize}
\end{itemize}
For $j=1,2$, writing $\hat{z}_\lambda^j(t)=(\hat{x}_\lambda^j(t)^\top, \hat{y}_\lambda^j(t)^\top)^\top$,
a direct computation shows that
$$
\int_0^{\tau} \bigl( J_n\dot{\hat{z}}_\lambda^j(t),  \hat{z}_\lambda^j(t) \bigr)_{\mathbb{R}^{2n}} \, dt=2\int_0^{\tau} \bigl(\dot{\hat{x}}_\lambda^j(t),  \hat{y}_\lambda^j(t) \bigr)_{\mathbb{R}^{n}} \, dt.
$$
As before, these results and  the arguments  preceding
Theorem~\ref{th:DistribifDelay1} give the conclusions in (III).
\end{proof}

The following corollary is a consequence of
Theorems~\ref{th:DistribifDelay1} and \ref{th:DistribifDelay2}.

\begin{corollary}\label{cor:DistribifDelay3}
Let $\mathsf{F}:\mathbb{R}^n\to \mathbb{R}$ be a $C^2$ function
satisfying $\nabla\mathsf{F}(0)=0$, and let
$\mathsf{G}: \mathbb{R}^n\to \mathbb{R}$ be an even $C^2$ function (hence $\nabla\mathsf{G}(0)=0$). For a nonzero real $\lambda$,
consider the distributed-delay differential equation
\begin{equation}\label{e:distributedDelay19}
\left\{\begin{array}{ll}
\dot{x}(t) = -\nabla\mathsf{F}\left(\displaystyle\int_0^\lambda
\nabla\mathsf{G}(x(t-s)) \,\mathrm{d}s\right), & x(t) \in \mathbb{R},
\\
x(t+\lambda) = -x(t) \; \forall t.
\end{array}\right.
\end{equation}
Assume that matrices $\nabla^2\mathsf{F}(0)$ and $\nabla^2\mathsf{G}(0)$
 commute, and hence by \cite[Theorem~4.5.15]{HorJ}
 there exists  an orthogonal matrix $\Xi$ of order $n$ such that
$$
\nabla^2\mathsf{F}(0)=\Xi^\top{\rm diag}(f_1,\cdots,f_n)\Xi\quad\hbox{and}\quad \nabla^2\mathsf{G}(0)=\Xi^\top{\rm diag}(g_1,\cdots, g_n)\Xi.
$$
If $(\mu,0)$ with $\mu\in\mathbb{R}\setminus\{0\}$ is a bifurcation point of
 (\ref{e:distributedDelay19}),
then $\tau|\mu|\sqrt{2g_lf_l}\in \pi(2\mathbb{N}-1)$ for some $l$.
Conversely, suppose $\mu\ne 0$ and there exists a bipartition
	$\{i_1,\dots,i_k\}\cup\{i_{k+1},\dots,i_n\}$  of $\{1,\dots,n\}$
with $1 \le k \le n$ such that the following two conditions hold:
	\begin{itemize}
\item[\rm (D)] For each $l=k+1,\dots,n$, either $
2g_{i_l}f_{i_l}|\mu|^2<(\pi/\tau)^2$ or $$
    (2j_l(\mu)-1)\pi<|\mu|\sqrt{2g_{i_l}f_{i_l}}<
    (2j_l(\mu)+1)\pi
    $$
    for some  $j_l(\mu)\in\mathbb{N}$. (When $n=1$, there is no such $l$, so the condition is vacuous.)
		\item[\rm (E)]  For each $l=1,\dots,k$, $g_{i_l}f_{i_l}>0$ and
$|\mu|\sqrt{2g_{i_l}f_{i_l}}=(2j_l(\mu)-1)\pi$ for some
 $j_l(\mu)\in\mathbb{N}$.
 \end{itemize}
Then, for the problem (\ref{e:distributedDelay19}),
	there exist two possible alternatives:
	\begin{itemize}
		\item[\rm (i)] The problem (\ref{e:distributedDelay19})
with $\lambda=\mu$ admits a sequence of
		nonzero, $\R$-distinct solutions
		$\{x^l\}^{\infty}_{l=1}$  such that $x^l\to 0$ in $C^1_{\rm loc}(\mathbb{R},\mathbb{R}^n)$
		as $l\to\infty$.
	\item[\rm (ii)] There exist left and right  neighborhoods $\Lambda^-$ and $\Lambda^+$ of $\mu$ in $\mathbb{R}$, respectively, containing no $0$,
		and nonnegative integers $n^+$ and $n^-\ge 0$ satisfying $n^++n^-\ge k$, such that for $\lambda\in\Lambda^-\setminus\{\mu\}$ (resp. $\lambda\in\Lambda^+\setminus\{\mu\}$),
		(\ref{e:distributedDelay19}) with parameter  $\lambda$  has at least $n^-$ (resp. $n^+$) $\R$-distinct
		solutions $x_\lambda^i\ne 0$ ($i=1,\dots,n^-$ (resp. $n^+$))
		which converge to zero in $C^1_{\rm loc}(\mathbb{R}, \mathbb{R}^n)$
		as $\lambda\to\mu$.
			\end{itemize}
	Moreover, if  $k>1$, (which implies $n>1$),
then at least one of (i), (iii), and (iv)  holds,
	where (i) is as stated previously, and
	\begin{itemize}
		\item[\rm (iii)]  For every nonzero $\lambda\in\Lambda\setminus\{\mu\}$ near $\mu$, there is a nonzero solution ${x}_\lambda$  of (\ref{e:distributedDelay19})
		with parameter $\lambda$, such that  $x_\lambda\to 0$ in $C^1_{\rm loc}(\mathbb{R},\mathbb{R}^n)$
		as $\lambda\to\mu$.
				
		\item[\rm (iv)] For a given $\varepsilon > 0$, there exists a one-sided neighbourhood $\Lambda^0$ of $\mu$ in $\mathbb{R}$ such that, for any
nonzero $\lambda \in \Lambda^0 \setminus \{\mu\}$, problem
	(\ref{e:distributedDelay19}) with parameter $\lambda$ has either
	\begin{itemize}
		\item[$\bullet$]
	 infinitely many $\mathbb{R}$-distinct nonzero solutions $\{x_\lambda^l\}_{l=1}^\infty$ such that $\|x_\lambda^l|_{[0,2|\lambda|]}\|_{C^1} < \varepsilon$ for all $l \in \mathbb{N}$,
	 or
	\item[$\bullet$]  at least two $\mathbb{R}$-distinct nonzero solutions $\hat{x}_\lambda^1$ and $\hat{x}_\lambda^2$ satisfying the following inequalities: $\|\hat{x}_\lambda^i|_{[0,2|\lambda|]}\|_{C^1} < \varepsilon$ for $i=1,2$, and
		\begin{align*}
			& \int_0^{\lambda} \bigl(\dot{\hat{x}}_\lambda^1(t),  \hat{y}_\lambda^1(t) \bigr)_{\mathbb{R}^{n}} \, dt
			+ 2\int_0^{\lambda} \mathsf{G}({\hat{x}}_\lambda^1(t))\, dt
			+ \int_{0}^{\lambda} \mathsf{F}\bigl(
			\hat{y}_\lambda^1(t)\bigr) \, dt \\
			\ne \; &
		\int_0^{\lambda} \bigl(\dot{\hat{x}}_\lambda^2(t),  \hat{y}_\lambda^2(t) \bigr)_{\mathbb{R}^{n}} \, dt
			+ 2\int_0^{\lambda} \mathsf{G}({\hat{x}}_\lambda^2(t))\, dt
			+ \int_{0}^{\lambda} \mathsf{F}\bigl(
			\hat{y}_\lambda^2(t)\bigr) \, dt,	
		\end{align*}
where $\hat{y}_\lambda^j(t)=\int^\lambda_0\nabla\mathsf{G}(\hat{x}_\lambda^j(t-s))ds$, $j=1,2$.
	\end{itemize}
\end{itemize}
 \end{corollary}

When $n=1$, the assumptions listed after (\ref{e:distributedDelay19}) in Corollary~\ref{cor:DistribifDelay3} may be reformulated as follows:
 Suppose  $\mathsf{F}''(0)\mathsf{G}''(0)>0$ and
$$
\mu\in\left\{\pm \frac{(2k-1)\pi}{\sqrt{2\mathsf{F}''(0)\mathsf{G}''(0)}}\,\bigm|\,
k\in\mathbb{N}\right\}.
$$

\begin{proof}[\bf Proof of Corollary~\ref{cor:DistribifDelay3}]
As in the proofs of Corollary~\ref{cor:bif-per3DelayC} and Corollary~\ref{cor:bif-per3+4}, the first assertion follows from
conclusions (I) of Theorems~\ref{th:DistribifDelay1} and \ref{th:DistribifDelay2}.

We now establish the remaining conclusions.
The case $n=1$ easily follows from (III) of
Theorem~\ref{th:DistribifDelay1}. Henceforth,
we assume $n>1$.
For each fixed $\lambda \in \mathbb{R}$, define functions
$F(\lambda,\cdot), G(\lambda,\cdot): \mathbb{R}^n \to \mathbb{R}$ by
$F(\lambda, x)=\lambda\mathsf{F}(x)$ and $G(\lambda, x)=
\lambda\mathsf{G}(x)$.
Note that the assumption in Theorem~\ref{th:DistribifDelay2} that
 $\nabla^2_2G(\mu, 0)$ has $n$ distinct eigenvalues
 guarantees the applicability of Lemma~\ref{lem:Kato}.
In the present case, this assumption is not needed, since the
 orthogonal matrix $\Xi$ ensures that for each $\lambda\in\Lambda_0$,
\begin{align*}
\nabla^2_2G(\lambda, 0)=\Xi^\top{\rm diag}(\lambda g_1,\cdots, \lambda g_n)\Xi
\quad\text{and}\quad
 \nabla^2_2F(\lambda, 0)=\Xi^\top{\rm diag}(\lambda f_1,\cdots, \lambda f_n)\Xi.
\end{align*}
Clearly, condition (D) here corresponds to condition (A) with $\tau=1$ in Theorem~\ref{th:DistribifDelay2}, while condition (E) corresponds to
 conditions (B) and (SC.1) with $\tau=1$ in that theorem.
(In fact, for each $l=1,\dots,k$, since $g_{i_l}f_{i_l}>0$ and
$(\lambda g_{i_l})(\lambda f_{i_l})-(\mu g_{i_l})(\mu f_{i_l})=
(\lambda^2-\mu^2) g_{i_l}f_{i_l}$, the condition corresponding to
(SC.1) in Theorem~\ref{th:DistribifDelay2} holds.)
Hence, for the problem
\begin{equation}\label{e:distributedDelay20}
\left\{\begin{array}{ll}
\dot{u}(t) = -\lambda\nabla \mathsf{F}\left(  \displaystyle\int_0^1
\lambda\nabla \mathsf{G}( u(t-s)) \,\mathrm{d}s\right), & u(t) \in \mathbb{R}^n,
\\
u(t+1 = -u(t) \; \forall t,
\end{array}\right.
\end{equation}
	there exist two possible alternatives:
	\begin{itemize}
		\item[\rm (i')] The problem (\ref{e:distributedDelay20})  with $\lambda=\mu$ admits a sequence of
		nonzero, $\R$-distinct solutions
		$\{u^l\}^{\infty}_{l=1}$  such that $u^l\to 0$ in $C^1_{\rm loc}(\mathbb{R},\mathbb{R}^n)$
		as $l\to\infty$.
	\item[\rm (ii')] There exist left and right  neighborhoods $\Lambda^-$ and $\Lambda^+$ of $\mu$ in $\mathbb{R}$
		and nonnegative integers $n^+$ and $n^-\ge 0$ satisfying $n^++n^-\ge k$, such that for $\lambda\in\Lambda^-\setminus\{\mu\}$ (resp. $\lambda\in\Lambda^+\setminus\{\mu\}$),
		(\ref{e:distributedDelay20}) with parameter  $\lambda$  has at least $n^-$ (resp. $n^+$) $\R$-distinct
		solutions $u_\lambda^i\ne 0$ ($i=1,\dots,n^-$ (resp. $n^+$))
		which converge to zero in $C^1_{\rm loc}(\mathbb{R}, \mathbb{R}^n)$
		as $\lambda\to\mu$.
	\end{itemize}
	Moreover, if  $k>1$, then at least one of (i'), (iii'), and (iv')  holds,
	where (i') is as stated previously, and
	\begin{itemize}
		\item[\rm (iii')]  For every $\lambda\in\Lambda\setminus\{\mu\}$ near $\mu$, there is a nonzero solution ${u}_\lambda$  of (\ref{e:distributedDelay20})
		with parameter $\lambda$, such that  $u_\lambda\to 0$ in $C^1_{\rm loc}(\mathbb{R},\mathbb{R}^n)$
		as $\lambda\to\mu$.
				
		\item[\rm (iv')] For a given $\varepsilon > 0$, there exists a one-sided neighbourhood $\Lambda^0$ of $\mu$ in $\mathbb{R}$ such that, for any $\lambda \in \Lambda^0 \setminus \{\mu\}$, problem
	(\ref{e:distributedDelay20}) with parameter $\lambda$ has either
	\begin{itemize}
		\item[$\bullet$]
	 infinitely many $\mathbb{R}$-distinct nonzero solutions $\{u_\lambda^l\}_{l=1}^\infty$ such that $\|u_\lambda^l|_{[0,2]}\|_{C^1} < \varepsilon$ for all $l \in \mathbb{N}$,
	 or
	\item[$\bullet$]  at least two $\mathbb{R}$-distinct nonzero solutions $\hat{u}_\lambda^1$ and $\hat{u}_\lambda^2$ satisfying the following inequalities: $\|\hat{u}_\lambda^i|_{[0,2]}\|_{C^1} < \varepsilon$ for $i=1,2$, and
		\begin{align*}
			& \int_0^{1} \bigl(\dot{\hat{u}}_\lambda^1(t),  \hat{v}_\lambda^1(t) \bigr)_{\mathbb{R}^{n}} \, dt
			+ 2\lambda\int_0^{1} \mathsf{G}({\hat{u}}_\lambda^1(t))\, dt
			+ \lambda\int_{0}^{1} \mathsf{F}\bigl(
			\hat{v}_\lambda^1(t)\bigr) \, dt \\
			\ne \; &
		\int_0^{1} \bigl(\dot{\hat{u}}_\lambda^2(t),  \hat{v}_\lambda^2(t) \bigr)_{\mathbb{R}^{n}} \, dt
			+ 2\lambda\int_0^{1} \mathsf{G}({\hat{u}}_\lambda^2(t))\, dt
			+ \lambda\int_{0}^{1} \mathsf{F}\bigl(
			\hat{v}_\lambda^2(t)\bigr) \, dt,	
		\end{align*}
where $\hat{v}_\lambda^j(t)=\lambda\int^1_0\nabla\mathsf{G}(\hat{u}_\lambda^j(t-s))ds$, $j=1,2$.
	\end{itemize}
\end{itemize}
Since $\mu\ne 0$, all $\lambda$ above can be taken to be nonzero.
For $u^l$ and $u_\lambda^i$ in (i') and (ii'),
 one readily checks that $x^l(t)=u^l(t/\lambda)$
and $x_\lambda^i(t)=u_\lambda^i(t/\lambda)$
satisfy (i) and (ii). Similarly, for $u_\lambda$, $u^l_\lambda$ and $\hat{u}^i_\lambda$ in (iii') and (iv'),
the curves $x_\lambda(t)=u_\lambda(t/\lambda)$, $x^l_\lambda(t)=u^l_\lambda(t/\lambda)$
and $\hat{x}_\lambda^i(t)=\hat{u}_\lambda^i(t/\lambda)$
satisfy (iii) and (iv).
\end{proof}

\appendix

\section{Some computation results about Maslov-type indexes }\label{app:Maslov}

\noindent{\textbf{Notation and conventions}}.
Follow notation and conventions in Section~\ref{sec:Intro0}.
For any two $m\times m$ symmetric matrices $A_1$ and $A_2$, we write $A_1\le A_2$
(resp. $A_1<A_2$)  if $A_2-A_1$ is positive semi-definite (resp. positive definite). For any
 $B_1,B_2\in L^\infty([0,\tau];\mathcal{L}_s(\mathbb{R}^{2n}))$
 we write $B_1\le B_2$ if $B_1(t)\le B_2(t)$ for almost every $t\in [0,\tau]$,
 and   $B_1<B_2$ if $B_1\le B_2$ and $\{t\in [0,\tau]\,|\,
 B_1(t)<B_2(t)\}$ has positive (Lebesgue) measure.
The $\diamond$-product of  two square block matrices
  $$
  M _ { 1 } = \left( \begin{array} { c c } A _ { 1 } & B _ { 1 } \\ C _ { 1 } & D _ { 1 } \end{array} \right)\in
  \mathbb{K}^{i\times i}  \quad\text{and}\quad
   M _ { 2 } = \left( \begin{array} { c c } A _ { 2 } & B _ { 2 } \\ C _ { 2 } & D _ { 2 } \end{array} \right)\in
   \mathbb{K}^{j\times j}
  $$
   defined by
  \begin{equation}\label{e:diamond-product}
   M _ { 1 } \diamond M _ { 2 } = \left( \begin{array} { c c c c } A _ { 1 } & 0 & B _ { 1 } & 0 \\
    0 & A _ { 2 } & 0 & B _ { 2 } \\ C _ { 1 } & 0 & D _ { 1 } & 0 \\ 0 & C _ { 2 } & 0 & D _ { 2 } \end{array} \right)
   \end{equation}
   is  associative. Denote by $M ^ {\diamond k}$  the $k$-fold $\diamond$ product of $M$.
Let ${\rm Sp}(2n,\mathbb{R})$ (or simply ${\rm Sp}(2n)$ if there is no confusion)
be the symplectic group of real symplectic matrices  of order $2n$, i.e., ${\rm Sp}(2n,\mathbb{R})=\{M\in
\mathbb{R}^{2n\times 2n}\,|\, M^\top J_nM=J_n\}$,
where
\begin{equation}\label{e:standcompl}
J_n=\left(\begin{array}{cc}
             0 & -I_n \\
             I_n & 0 \\
           \end{array}
         \right)
\end{equation}
with the identity matrix $I_n$ of order $n$. $J_n$ gives  the standard complex structure  on $\mathbb{R}^{2n}$,
$$
 (q_1,\cdots,q_n, p_1,\cdots, p_n)^\top\mapsto J(q_1,\cdots,q_n, p_1,\cdots, p_n)^\top=(-p_1,\cdots,-p_n, q_1,\cdots, q_n)^\top.
$$
Each $S\in{\rm Sp}(2n)$ has a unique decomposition $PU$, where $P=\sqrt{SS^\top}$ and
$U=\left(\begin{array}{cc}
             U_1 & -U_2 \\
             U_2 & U_1 \\
           \end{array}
         \right)$
such that $\mathfrak{u}(S):=U_1+\sqrt{-1}U_2$ is a unitary matrix, i.e.,
$\mathfrak{u}(S)\in U(n,\mathbb{C})$.
Note that the map $\mathfrak{u}: {\rm Sp}(2n)\to U(n,\mathbb{C})$ restricts to a Lie group isomorphism
\begin{equation}\label{e:LieIsom}
\mathfrak{u}:{\rm Sp}(2n)\cap{\rm O}(2n)\to U(n,\mathbb{C}).
\end{equation}
For $M\in{\rm Sp}(2n)$, let $D(M)=( - 1 ) ^ { n - 1 } \operatorname { d e t } \left( M - I _ { 2 n }\right)$,
and let
$$
{\rm Sp}( 2 n ) ^ { 0 } = \{ M \in {\rm Sp}( 2 n )\,|\, D(M)= 0 \}.
$$
The complementary set ${\rm Sp}( 2 n ) ^ { * } := {\rm Sp}( 2 n ) \backslash{\rm Sp}( 2 n ) ^ { 0 }$
 has exactly two path-connected components
$$
{\rm Sp}( 2 n ) ^ { +} = \{ M \in {\rm Sp}( 2 n )\,|\, D(M) < 0 \}\quad\text{and}\quad
{\rm Sp}( 2 n ) ^ { -} = \{ M \in {\rm Sp}( 2 n )\,|\, -D(M) < 0 \},
$$
which contain, respectively,
$M _ { n } ^ { + } = D ( 2 ) ^ { \diamond n }$ and $M _ { n } ^ { - } = D ( - 2 ) \diamond D ( 2 ) ^ { \diamond ( n - 1 ) }$, where
$D (a) = \operatorname { d i a g } \left\{a , \frac { 1 } {a} \right\}$ for $a\in\mathbb{R}\setminus\{0\}$.
Every loop in these two components is contractible in ${\rm Sp}(2n)$
(cf. \cite[Lemma~1.7]{CoZe} and \cite[Lemma~3.2]{SaZe2}).
Therefore, each path $\gamma\in C([a, b], {\rm Sp}(2n))$ corresponds to a unique continuous path
$\mathfrak{u}_\gamma:[a,b]\to U(n,\C)$ given by $\mathfrak{u}_\gamma(t):=\mathfrak{u}(\gamma(t))$,
where $\mathfrak{u}$ is as in (\ref{e:LieIsom}).
By the lifting criterion for covering spaces
there exists a continuous real-valued function $\Delta_\gamma$ on $[a, b]$
such that $\det \mathfrak{u}_\gamma(t)=\exp(\sqrt{-1}\Delta_\gamma(t))$ for $t\in[a,b]$, and
$\Delta({\gamma}):=\Delta_\gamma(b)-\Delta_\gamma(a)$ is uniquely determined by $\gamma$.
For $\tau>0$, the concatenated path (or joint path) of $\gamma_1:[0,\tau]\to{\rm Sp}(2n)$ and $\gamma_2:[0,\tau]\to{\rm Sp}(2n)$,
 with $\gamma_1(\tau)=\gamma_2(0)$, is defined as the path
 $\gamma_2\ast\gamma_1:[0,\tau]\to{\rm Sp}(2n)$ given by
\begin{equation}\label{e:pathComp}
    \gamma_2\ast\gamma_1(t)=
 \begin{cases}
\gamma_1(2t),\quad&\hbox{if $t\in [0,\tau/2]$},\\
\gamma_2(2t-\tau),\quad&\hbox{if $t\in[\tau/2,\tau]$}.
\end{cases}
 \end{equation}
Let $\mathcal{P}_\tau(2n) =\{\gamma\in C([0,\tau],{\rm Sp}(2n))\,|\, \gamma(0)=I_{2n}\}$
and ${\cal P}^\ast_\tau(2n)=\{\gamma\in\mathcal{P}_\tau(2n)\,|\, \gamma(\tau)\in{\rm Sp}(2n)^\ast\}$.\\


\noindent{\textbf{Main results}}.
For $M\in{\rm Sp}(2n,\mathbb{R})$ and $\gamma\in\mathcal{P}_\tau(2n)$,
let $(i_{\tau,M}(\gamma), \nu_{\tau,M}(\gamma))$
be the Maslov-index of $\gamma$ relative to $M$ defined by
 (\ref{e:dongIndex}).
 The goal of this appendix is to give proofs of the following three theorems.

\begin{theorem}\label{th:PIndex1}
  	Let $B$ and $C$ be real matrices of order $n$, where $B$ is symmetric and $C$ is  skew-symmetric.
 Let $e_1\le e_2\le\cdots\le e_n$ be all eigenvalues of the Hermitian matrix
 $B+\sqrt{-1}C$; these eigenvalues are all real.
 Denote by  $\Upsilon$   the fundamental matrix solution of
 $$
 \dot{z}(t)=J_n\left( \begin{array} { c c c c }
 	B & -C \\
 	C &B
 \end{array} \right)z(t)
 $$
 satisfying $\Upsilon(0)=I_{2n}$, i.e., $\Upsilon(t)=\exp\left(tJ_n\left( \begin{array} { c c c c }
 	B & -C  \\
 	C &B
 \end{array} \right)\right)$.
 Then
\begin{align}\label{e:DiagPindex1}
 \nu_{\tau,-J_n}(\Upsilon)&=
    2\sharp\left\{k\,|\, 2\tau e_k=(4j+3)\pi\;\hbox{with some $j\in\mathbb{Z}$}\right\},\\
  i_{\tau,-J_n}(\Upsilon)&=i_{\tau}^{-J_n}(\Upsilon)+ [n/2]=2j_1+\cdots+ 2j_n +[n/2],
  \label{e:DiagPindex2}
  \end{align}
  where for $k=1,\cdots,n$, $j_k\in\mathbb{Z}$ is the unique integer such that $\frac{(4j_k-1)\pi}{2\tau}<e_k\le\frac{(4j_k+3)\pi}{2\tau}$, i.e.,
  $j_k$ is equal to $\frac{\tau e_k}{2\pi}-
  \frac{3}{4}$ if $\frac{\tau e_k}{2\pi}-
  \frac{3}{4}$ is an integer, and
  $\left[\frac{\tau e_k}{2\pi}-\frac{3}{4}\right]+1$
  otherwise.
  Furthermore, if $B$ is positive definite, then
\begin{eqnarray}\label{e:DiagPindex3}
i_{\tau,-J_n}(\Upsilon)
=[n/2]+ 2\sum_{0<t<\tau}\sharp\left\{k\,|\, 2te_k=(4j+3)\pi\;\hbox{with some $j\in\mathbb{Z}$}\right\},
\end{eqnarray}
 and if $B$ is negative definite, then
\begin{eqnarray}\label{e:DiagPindex4}
i_{\tau,-J_n}(\Upsilon)
=[n/2]- 2\sum_{0<t\le\tau}\sharp\left\{k\,|\, 2te_k=(4j+3)\pi\;\hbox{with some $j\in\mathbb{Z}$}\right\}.
\end{eqnarray}
 \end{theorem}


\begin{theorem}\label{th:PIndex2}
For $\Upsilon$ in Theorem~\ref{th:PIndex1},
the following holds:
	\begin{align}\label{e:DiagPindex1ADD}
		\nu_{\tau,J_n}(\Upsilon)&=
		2\sharp\left\{k\,|\, 2\tau e_k=(4j+1)\pi\;\text{for some $j\in\mathbb{Z}$}\right\},\\
		i_{\tau,J_n}(\Upsilon)&=i_{\tau}^{J_n}(\Upsilon)+ [-n/2]=2j_1+\cdots+ 2j_n +[3n/2],
		\label{e:DiagPindex2ADD}
	\end{align}
	where for $k=1,\cdots,n$, $j_k\in\mathbb{Z}$ is the unique integer such that $\frac{(4j_k+1)\pi}{2\tau}<e_k\le\frac{(4j_k+5)\pi}{2\tau}$, i.e.,
	$j_k$ is equal to $\frac{\tau e_k}{2\pi}-
	\frac{5}{4}$ if $\frac{\tau e_k}{2\pi}-
	\frac{5}{4}$ is an integer, and
	$\left[\frac{\tau e_k}{2\pi}-\frac{5}{4}\right]+1$
	otherwise.
	Moreover, if $B$ is positive definite, then
	\begin{eqnarray}\label{e:DiagPindex3ADD}
		i_{\tau,J_n}(\Upsilon)
		=[-n/2]+ 2\sum_{0<t<\tau}\sharp\left\{k\,|\, 2te_k=(4j+1)\pi\;\hbox{with some $j\in\mathbb{Z}$}\right\},
	\end{eqnarray}
	and if $B$ is negative definite, then
	\begin{eqnarray}\label{e:DiagPindex4ADD}
		i_{\tau,J_n}(\Upsilon)
		=[-n/2]- 2\sum_{0<t\le\tau}\sharp\left\{k\,|\, 2te_k=(4j+1)\pi\;\hbox{with some $j\in\mathbb{Z}$}\right\}.
	\end{eqnarray}
\end{theorem}

%
%

\begin{theorem}\label{th:PIndex3}
Let $\xi_{2n,\tau}(t)=I_{2n}$ for all $t\in [0, \tau]$,
let $e_1\le e_2\le\cdots\le e_n$ be all eigenvalues of a real symmetric matrix $B$ of order $n$, and
let $\Upsilon$ be  the fundamental matrix solution of
 \begin{eqnarray}\label{e:Add1}
 \dot{z}(t)=J_n\left( \begin{array} { c c c c }
2B & -B \\
-B &2B
 \end{array} \right)z(t)
 \end{eqnarray}
 with $\Upsilon(0)=I_{2n}$, i.e., $\Upsilon(t)=\exp\left(tJ_n\left( \begin{array} { c c c c }
2B & -B  \\
-B &2B
 \end{array} \right)\right)$.
Then for $M_{3,n}=\left( \begin{array} { c c c c c }
0 & I _ { n }\\
 -I _ { n } & I_n\end{array} \right)$,  it holds that
\begin{align}\label{e:Add2}
 \nu_{\tau, M_{3,n}}(\Upsilon)&=2\sharp\left\{j\,\Big|\, \sqrt{3}e_j\tau\in \frac{2\pi}{3}+ (2\mathbb{Z}+1)\pi\right\},\\
 i_{\tau,M_{3,n}}(\Upsilon)&=\left[n-\frac{n\arctan 2}{\pi}\right]+2\sum_{e_j>0}\sharp\left\{0<t<\tau\,\Big|\, \sqrt{3}e_jt\in \frac{2\pi}{3}+ (2\mathbb{Z}+1)\pi\right\}\nonumber\\
&\quad-2\sum_{e_j<0}\sharp\left\{0<t\le\tau\,\Big|\, \sqrt{3}e_jt\in \frac{2\pi}{3}+ (2\mathbb{Z}+1)\pi\right\}.\label{e:Add3}
\end{align}
 In particular, if $B$ is positive definite,
\begin{eqnarray}\label{e:Add4}
i_{\tau, M_{3,n}}(\Upsilon)
=\left[n-\frac{n\arctan 2}{\pi}\right]+2\sum_{0< t<\tau}\sharp\left\{j\,\Big|\, \sqrt{3}e_jt\in \frac{2\pi}{3}+ (2\mathbb{Z}+1)\pi\right\}
\end{eqnarray}
and if $B$ is negative definite,
\begin{eqnarray}\label{e:Add5}
i_{\tau, M_{3,n}}(\Upsilon)
=\left[n-\frac{n\arctan 2}{\pi}\right]-2\sum_{0<t\le\tau}\sharp\left\{j\,\Big|\, \sqrt{3}e_jt\in \frac{2\pi}{3}+ (2\mathbb{Z}+1)\pi\right\}.
\end{eqnarray}
 \end{theorem}


\begin{theorem}\label{th:PIndex4}
Let $B$, $C$, and $e_1\le e_2\le\cdots\le e_n$ be as in Theorem~\ref{th:PIndex1}.
Let $\Upsilon$ be  the fundamental matrix solution of
 $$
 \dot{Z}(t)=J_{2n}
\left( \begin{array} {cccc}
B& -C &0&0\\
C&B&0&0\\
0&0&B&-C\\
0&0&C&B
 \end{array} \right)Z(t)
 $$
 with $\Upsilon(0)=I_{4n}$, and let
 $M_n:=\left( \begin{array} { c c  } 0 & J_n \\ -J_n & 0 \end{array} \right)\in{\rm Sp}(4n)^0$.
 Then it holds that
\begin{align}\label{e:DiagPindex1*}
    \nu_{\tau, M_n}(\Upsilon) &= 2\sharp\{\, m \mid \tau e_m \in \pi\mathbb{Z} \,\},\\
    i_{\tau, M_n}(\Upsilon) &= 2\sum_{j=1}^{n}m_j+ n, \label{e:DiagPindex2*}
\end{align}
where for each $j\in\{1,\cdots,n\}$, $m_j \in \mathbb{Z}$ is the unique integer such that $m_j\pi < e_j\tau \le (m_j + 1)\pi$.
  Moreover, if $B$ is positive definite, then
\begin{equation}\label{e:positive-negativeAG}
i_{\tau,M_n}(\Upsilon)=2n+ i_{\tau, M_n}(\xi_{4n,\tau})
+2\sum_{0<t<\tau}\sharp\left\{k\mid te_k\in\pi\mathbb{Z}\right\},
\end{equation}
and if $B$ is negative definite, then
\begin{equation}\label{e:positive-negativeBG}
i_{\tau,M_n}(\Upsilon)=i_{\tau, M_n}(\xi_{4n,\tau})
-2\sum_{0<t\le\tau}\sharp\left\{k\mid te_k\in\pi\mathbb{Z}\right\}.
\end{equation}
Hereafter, $\xi_{m,\tau}(t)=I_{m}$ for all $t\in [0, \tau]$.
\end{theorem}


\begin{theorem}\label{th:PIndex5}
Let $A$ and $B$ be two commuting real symmetric matrices of order $n$.
Then there exists (see \cite[Theorem~4.5.15]{HorJ}) an orthogonal matrix $\Xi$ of order $n$ such that
$$
A=\Xi^\top{\rm diag}(a_1,\cdots,a_n)\Xi\quad\hbox{and}\quad B=\Xi^\top{\rm diag}(b_1,\cdots, b_n)\Xi.
$$
(Note that ordering of the eigenvalues $a_1,\cdots,a_n$ and $b_1,\cdots,b_n$ is not unique and depends on $\Xi$.
But $\Xi$ can be chosen such that $a_1\le a_2\le\cdots\le a_n$.)
  Denote by  $\Upsilon$   the fundamental matrix solution of
 $$
 \dot{z}(t)=J_n\left( \begin{array} { c c c c }
 	A & 0 \\
 	0 &B
 \end{array} \right)z(t)
 $$
 satisfying $\Upsilon(0)=I_{2n}$.  Then
\begin{align}\label{e:DistriDiagPindex1}
 \nu_{\tau,-I_{2n}}(\Upsilon)&=
    2\sharp\left\{k\,|\, \tau\sqrt{a_kb_k}\in (2\mathbb{N}-1)\pi\right\},\\
  i_{\tau,-I_{2n}}(\Upsilon)&=\sum_{a_kb_k>0\& a_k>0}2\lceil \frac{\sqrt{a_kb_k}\tau}{2\pi}-\frac{1}{2} \rceil+
\sum_{a_kb_k>0\& a_k<0}
2\lceil \frac{-\sqrt{a_kb_k}\tau}{2\pi}-\frac{1}{2} \rceil,
  \label{e:DistriDiagPindex2}
  \end{align}
   where $\lceil r \rceil=\min\{k\in\mathbb{Z}\,|\,k\ge r\}$ for a real $r$.
       Furthermore, if $A$ and $B$ are positive definite, then
\begin{eqnarray}\label{e:DistriDiagPindex3}
i_{\tau,-I_{2n}}(\Upsilon)
=\sum_{0< t < \tau}2\sharp\left\{k\,|\, \sqrt{a_kb_k}t\in (2\mathbb{N}-1)\pi\right\},
\end{eqnarray}
 and if $A$ and $B$ are negative definite, then
\begin{eqnarray}\label{e:DistriDiagPindex4}
i_{\tau,-I_{2n}}(\Upsilon)
=-\sum_{0< t \le\tau}2\sharp\left\{k\,|\, \sqrt{a_kb_k}t\in (2\mathbb{N}-1)\pi\right\}.
\end{eqnarray}
 \end{theorem}

The proofs of Theorems~\ref{th:PIndex1},~\ref{th:PIndex2} are given in Section~\ref{sec:PIndex2},
while those of Theorems~\ref{th:PIndex3}, \ref{th:PIndex4} and
\ref{th:PIndex5}
 are proved in Sections~\ref{sec:PIndex3},  \ref{sec:PIndex4}
 and \ref{sec:PIndex5}, respectively.

To ensure the proofs are self-contained, we first outline related Maslov-type indices and present some lemmas. Some of these findings are not available in existing references.

\subsection{Preliminary}\label{sec:Pre}

\subsubsection{A brief overview of related Maslov-type indexes and results}\label{app:Pre1}

 Conley and Zehnder \cite{CoZe} (for $n\ge 2$)
 and Long and  Zehnder \cite{LongZeh}  (for $n=1$) assigned an integer
 \begin{equation}\label{e:Conley-Zehnder}
 i_{\rm CZ}(\gamma)=\Delta({\beta\ast\gamma})/\pi
 \end{equation}
  to each $\gamma\in\mathcal{P}^\ast_\tau(2n)$,  the so-called \textsf{Conley-Zehnder index}  of $\gamma$, where
 $\beta$ is a path in ${\rm Sp}(2n)^\ast$ connecting $\gamma(\tau)$ to $M^+_n$ or $M^-_n$.
   An alternative exposition was presented in  \cite{SaZe2}.

For each $B\in C([0,\tau];\mathcal{L}_s(\mathbb{R}^{2n}))$
that is positive definite everywhere and satisfies $B(0) = B(\tau)$,
Ekeland \cite{Ek84} defined,  via dual variational methods, a pair of numbers
  $$
  (i^E(B), \nu^E(B))\in (\N\cup\{0\})\times\{0,1,\cdots,2n\},
  $$
    called the \textsf{Ekeland index of $B$}.
  He showed  that
\begin{equation}\label{e:Ekeland-index}
i^E(B)=\sum_{0<s<\tau}\nu^E(B|_{[0,s]})=\sum_{0<s<\tau}\dim{\rm Ker}(\Upsilon_B(s)-I_{2n})
\end{equation}
as stated in Theorem~6 of \cite[Chapter I, \S4]{Ek90},
 where $\Upsilon_B:[0,\tau]\to {\rm Sp}(2n,\mathbb{R})$ is the fundamental
 matrix solution of the linear system
$\dot{Z}(t)=JB(t)Z(t)$ with $\Upsilon_B(0)=I_{2n}$.
     Brousseau \cite{Bro86}  proved
\begin{equation}\label{e:long-Ekeland0}
i_{\rm CZ}(\Upsilon_B)=i^E(B)+ n\quad\hbox{if}\quad \nu^E(B)=0.
\end{equation}

As an extension of the Conley-Zehnder index $i_{\rm CZ}$ to paths in $\mathcal{P}_\tau(2n)\setminus\mathcal{P}^\ast_\tau(2n)$,
  Long \cite{Long97, Long02}  defined the \textsf{Maslov-type index} of $\gamma\in\mathcal{P}_\tau(2n)$ to be a pair of integers $(i_\tau(\gamma),\nu_\tau(\gamma))$, where $\nu_\tau(\gamma):=\dim{\rm Ker}(\gamma(\tau)-I_{2n})$,
  ${i}_\tau(\gamma)=i_{\rm CZ}(\gamma)$ if $\nu_\tau(\gamma)=0$, and
\begin{eqnarray}\label{e:long}
 {i}_\tau(\gamma)=\inf\{i_{\rm CZ}(\beta)\,|\, \beta\in{\cal
P}^\ast_\tau(2n)\;\hbox{is sufficiently $C^0$ close to $\gamma$
in}\;{\cal P}_\tau(2n)\}
\end{eqnarray}
if $\nu_\tau(\gamma)>0$. The Maslov-type index $i_\tau$ satisfies
 the following \textsf{homotopy invariance property}
\begin{equation}\label{e:homotopyInvari}
\left.\begin{array}{ll}
\hbox{Two paths $\gamma_0$ and $\gamma_1$ in $\mathcal{P}_\tau(2n)$
have the same index $i_\tau(\gamma_0)=i_\tau(\gamma_1)$ }\\
\hbox{
if there is a map $\delta\in C([0,1]\times[0,\tau], {\rm Sp}(2n))$ such that $\delta(0,\cdot)=\gamma_0$,}\\
\hbox{$\delta(1,\cdot)=\gamma_1$,
$\delta(s,0)=I_{2n}$ and $\nu_\tau(\delta(s,\cdot))$ is constant for $0\le s\le 1$}
\end{array}\right\}
\end{equation}
(\cite{Long97} and Theorem~3 in \cite[page~145]{Long02}), and the \textsf{symplectic additivity property}
\begin{equation}\label{e:sympAdd}
i_\tau(\gamma_0\diamond\gamma_1)=i_\tau(\gamma_0)+ i_\tau(\gamma_1)\quad\hbox{for $\gamma_j\in\mathcal{P}_\tau(2n_j)$, $j=0,1$}
\end{equation}
(\cite[Theorem~6, page 146]{Long02}), where $\gamma_0\diamond\gamma_1$ denotes the path
$t\mapsto\gamma_0(t)\diamond\gamma_1(t)$ .
For $a<b$ and any path $\gamma\in C([a,b],{\rm Sp}(2n))$, choose a path
$\beta\in{\cal P}_1(2n)$ such that  $\beta(1)=\gamma(a)$, and define a path
$\phi\in{\cal P}_1(2n)$ by setting
$$
\phi(t) =
 \begin{cases}
\beta(2t),\quad& 0\le t\le 1/2,\\
\gamma(a+(2t-1)(b-a)),\quad& 1/2\le t\le 1.
\end{cases}
 $$
 Long \cite{Long97} showed that the difference $i_1(\phi) - i_1(\beta)$ depends only on $\gamma$. Accordingly, he called
\begin{equation}\label{e:Long.1}
i(\gamma, [a,b]) := i_1(\phi) - i_1(\beta)
\end{equation}
 the \textsf{Maslov-type index} of the path $\gamma$.
  Clearly, $i(\gamma, [0,1])=i_1(\gamma)$ for any $\gamma\in{\cal P}_1(2n)$.
If $B$ and $\Upsilon_B$ are as in (\ref{e:Ekeland-index}),
      Long, respectively, generalized (\ref{e:Ekeland-index}) and (\ref{e:long-Ekeland0}) as
   \begin{equation}\label{e:Long-Ekeland-index}
i_\tau(\Upsilon_B)=\sum_{0<s<1}\nu_\tau(\Upsilon_B(s\cdot))+n=\sum_{0<s<\tau}\dim{\rm Ker}(\Upsilon_B(s)-I_{2n})+n
\end{equation}
(\cite[Proposition~15.3]{Long02}) and
 \begin{equation}\label{e:long-Ekeland}
i_\tau(\Upsilon_B)=i^E(B)+ n\quad\text{and}\quad \nu_\tau(\Upsilon_B)=\nu^E(B)
\end{equation}
(\cite{Long98}).
For a systematic study and applications of the index pair $(i_\tau, \nu_\tau)$, we refer to \cite{Long02}.

Robbin and Salamon \cite{RoSa93} gave another extension of the Conley-Zehnder index
$i_{\rm CZ}$ to paths in $\mathcal{P}_\tau(2n)\setminus\mathcal{P}^\ast_\tau(2n)$.
 Define the symplectic space $(F,\Omega)=(\mathbb{R}^{2n}\oplus\mathbb{R}^{2n}, (-\omega_0)\oplus\omega_0)$.
Let $\mathscr{L}(F,\Omega)$ denote the set (or manifold) of Lagrangian subspaces of $(F,\Omega)$.
Both
\begin{center}
$W:=\{(x^\top, x^\top)^\top\in\R^{4n}\,|\,x\in\R^{2n}\}$\quad  and\quad
${\rm Gr}(M):=\{(x^\top, (Mx)^\top)^\top\in\R^{4n}\,|\,x\in\R^{2n}\}$,
\end{center}
 where $M\in{\rm Sp}(2n)$,  are Lagrangian subspaces of $(F, \Omega)$,
 and hence belong to $\mathscr{L}(F, \Omega)$.
 The Robbin-Salamon index $\mu^{\rm RS}$, as defined in \cite{RoSa93}, assigns a half-integer
$\mu^{\rm RS}(\Lambda,\Lambda')$ to each pair of Lagrangian paths $\Lambda,\Lambda':[a,b]\to\mathscr{L}(F,\Omega)$.
 The \textsf{Conley-Zehnder index}  of $\gamma\in C([a,b],{\rm Sp}(2n))$ is defined  by
\begin{eqnarray}\label{e:RobbinSaCZ}
\mu_{\rm CZ}(\gamma)=\mu^{\rm RS}({\rm Gr}\left(\gamma), W\right)
\end{eqnarray}
(cf. \cite[Remark~5.35]{RoSa95}).
It follows from \cite[Remark~5.4]{RoSa93}  that
\begin{eqnarray}\label{e:RobbinSaCZ+}
\mu_{\rm CZ}(\gamma)=i_{\rm CZ}(\gamma),\quad\forall \gamma\in\mathcal{P}^\ast_\tau(2n).
\end{eqnarray}

There is a precise relation between $i(\gamma, [a,b])$ and $\mu_{\rm CZ}(\gamma)$
for each $\gamma\in C([a,b],{\rm Sp}(2n))$.
 Recall that the Cappell-Lee-Miller index $\mu^{\rm CLM}_F$, which is characterized by properties
I-VI in \cite[pp. 127-128]{CLM},  assigns an integer $\mu^{\rm CLM}_F(\Lambda,\Lambda')$ to each
pair of Lagrangian paths $\Lambda,\Lambda':[a,b]\to\mathscr{L}(F,\Omega)$.
Corollary 2.1 of \cite{LongZhu00} states that
\begin{equation}\label{e:Long.2}
i_\tau(\gamma)=\mu^{\rm CLM}_F(W, {\rm Gr}(\gamma), [0,\tau])-n,\quad\forall \gamma\in\mathcal{P}_\tau(2n).
\end{equation}
Combining (\ref{e:Long.1}), (\ref{e:Long.2}), and the path additivity of $\mu^{\rm CLM}_F$, we obtain
\begin{equation}\label{e:Long.3}
i(\gamma, [a,b])=\mu^{\rm CLM}_F(W, {\rm Gr}\left(\gamma),
[a,b]\right)
\end{equation}
for any $\gamma\in C([a,b],{\rm Sp}(2n))$.
Thus,
\begin{equation}\label{e:RobbinSa2}
i(\gamma, [a,b])
=\mu_{\rm CZ}(\gamma)-\frac{1}{2}(\dim
{\rm Ker}(I_{2n}-\gamma(b))-\dim
{\rm Ker}(I_{2n}-\gamma(a))).
\end{equation}
by combining (\ref{e:RobbinSaCZ}) with Theorem 3.1 in \cite{LongZhu00}.
 In particular, for every $\gamma\in\mathcal{P}_\tau(2n)$,
\begin{equation}\label{e:RobbinSa3}
i_\tau(\gamma)=\mu_{\mathrm{CZ}}(\gamma)-\frac{1}{2}\dim
{\mathrm{Ker}}(I_{2n}-\gamma(\tau)).
\end{equation}
For the paths $\phi$ and $\beta$ in (\ref{e:Long.1}), the Conley-Zehnder index $\mu_{\mathrm{CZ}}$ is additive under path concatenation.
Therefore,
\begin{align}\label{e:RobbinSa1}
\mu_{\mathrm{CZ}}(\gamma) = \mu_{\mathrm{CZ}}(\phi) - \mu_{\mathrm{CZ}}(\beta).
\end{align}
In addition,  both indices $i$ and $\mu_{\mathrm{CZ}}$
satisfy the following four properties: the vanishing property, the product property, and the following two properties:
\begin{description}
\item[(Homotopy invariant)] If for a continuous path $\ell:[0,1]\to C([a,b],{\rm Sp}(2n))$, the maps
\begin{equation}\label{e:homotopy1}
s\mapsto\dim{\rm Ker}(I_{2n}-\ell(s)(a))\quad\text{and}\quad
s\mapsto\dim{\rm Ker}(I_{2n}-\ell(s)(b))
\end{equation}
are constant  on $[0,1]$, then
\begin{align}\label{e:RobbinSa4}
i(\ell(s), [a,b])&=i(\ell(0), [a,b]),\quad\forall s\in [0,1],\\
\mu_{\mathrm{CZ}}(\ell(s))&=\mu_{\mathrm{CZ}}(\ell(0)),\quad\forall s\in [0,1].\label{e:RobbinSa5}
\end{align}
\item[(Naturality)] For any $\phi, \gamma\in C([a,b],{\rm Sp}(2n))$, the following holds:
\begin{equation}\label{e:RobbinSa6}
i(\phi\gamma\phi^{-1}, [a,b])=i(\gamma, [a,b])\quad\text{and}\quad
\mu_{\mathrm{CZ}}(\phi\gamma\phi^{-1})=\mu_{\mathrm{CZ}}(\gamma).
\end{equation}
From the second equality in (\ref{e:RobbinSa6}), together with (\ref{e:long}) and (\ref{e:RobbinSaCZ+}),
 we obtain the following generalization of \cite[Corollary~6.5]{Long97}:
\begin{equation}\label{e:RobbinSaCZ+RobbinSa6}
i_{\tau}(\phi\gamma\phi^{-1})=i_{\tau}(\gamma),\quad\forall \gamma\in\mathcal{P}_\tau(2n),\;\forall
\phi\in C([0, \tau],{\rm Sp}(2n,\R)).
\end{equation}
\end{description}

In fact, (\ref{e:RobbinSa4}) is derived as follows; this equality (or its derivation) implies (\ref{e:RobbinSa5}). Let $\alpha\in \mathcal{P}_1(2n)$. For each fixed $s\in [0, 1]$, we define  paths $\beta_s, \gamma_s:[0,1]\to {\rm Sp}(2n,\R)$ by
$$
\beta_s(t)=\ell(st)(a),\;t\in [0,1],\qquad \gamma_s(t)=\ell(s)((1-t)a+tb),\;t\in [0,1].
$$
Consider the concatenated paths $\beta_s\ast\alpha$ and  $\gamma_s\ast(\beta_s\ast\alpha)$.
Since $\beta_0$ is the constant identity path at $\ell(0)(a)$, the families
 $\{\beta_s\ast\alpha\,|\, 0\le s\le 1\}$ and $\{\gamma_s\ast(\beta_s\ast\alpha)\,|\, 0\le s\le 1\}$
  define, respectively,  homotopies (with fixed initial endpoints)
  from $\alpha$ to  $\beta_1\ast\alpha$ and from $\gamma_0\ast\alpha$ to  $\gamma_1\ast(\beta_1\ast\alpha)$
  in ${\rm Sp}(2n,\R)$. Note that
 $\nu_1(\beta_s\ast\alpha)=\dim{\rm Ker}(I_{2n}-\ell(s)(a))$ and
$\nu_1(\gamma_s\ast(\beta_s\ast\alpha))=\dim{\rm Ker}(I_{2n}-\ell(s)(b))$ are constant functions on $[0,1]$.
 The  homotopy invariance of ${i}_1$ as in (\ref{e:homotopyInvari})
leads to
\begin{align}\label{e:RobbinSaCZ+RobbinSa8}
&i_1(\beta_s\ast\alpha)\equiv i_1(\alpha),\quad\forall s\in [0,1],\\
&i_1(\gamma_s\ast(\beta_s\ast\alpha))\equiv i_1(\gamma_0\ast\alpha),\quad\forall s\in [0,1].
\label{e:RobbinSaCZ+RobbinSa9}
\end{align}
Moreover, by the definition in (\ref{e:Long.1}), for all $s\in [0,1]$ we have
$$
i(\ell(s), [a,b])=i(\gamma_s, [0,1])=i_1(\gamma_s\ast(\beta_s\ast\alpha))-i_1(\beta_s\ast\alpha).
$$
Therefore, combining the above identity with (\ref{e:RobbinSaCZ+RobbinSa8}) and (\ref{e:RobbinSaCZ+RobbinSa9}), we obtain for all $s\in [0, 1]$ that
$$
i(\ell(s), [a,b])\equiv i_1(\gamma_0\ast\alpha)-i_1(\alpha)=i(\ell(0), [a,b]),\quad
\forall s\in [0,1].
$$
This proves (\ref{e:RobbinSa4}), and hence we arrive at
  (\ref{e:RobbinSa5}) by (\ref{e:RobbinSa2}) and (\ref{e:homotopy1}).


In order to prove (\ref{e:RobbinSa6}), we consider a continuous path $\ell:[0,1]\to C([a,b],{\rm Sp}(2n))$ given by $\ell(s)(t)=\phi(st)\gamma(t)(\phi(st))^{-1}$.
Then (\ref{e:RobbinSa5}) implies
\begin{equation}\label{e:RobbinSaCZ+RobbinSa10}
\mu_{\mathrm{CZ}}(\phi\gamma\phi^{-1})=\mu_{\mathrm{CZ}}(\ell(1))=\mu_{\mathrm{CZ}}(\ell(0))=
\mu_{\mathrm{CZ}}(\phi(0)\gamma(\phi(0))^{-1}).
\end{equation}
Note that
$$
{\rm Gr}(\phi(0)\gamma(\phi(0))^{-1})(t)=\{(x^\top,(\phi(0)\gamma(t)(\phi(0))^{-1}x)^\top)^\top\in\R^{4n}\,|\,x\in\R^{2n}\}
=\Xi{\rm Gr}(\gamma)(t),
$$
where $\Xi:(F,\Omega)\to (F,\Omega)$ is the symplectic isomorphism given by
$\Xi(x\oplus y)=(\phi(0)x)\oplus(\phi(0)y)$.
Since $\Xi W=W$, the naturality property of $\mu^{\rm RS}$ (\cite[\S5]{RoSa95})
 together with (\ref{e:RobbinSaCZ}) yields
 $$
\mu_{\mathrm{CZ}}(\phi(0)\gamma(\phi(0))^{-1})=\mu^{\rm RS}(\Xi{\rm Gr}(\gamma), W)
=\mu^{\rm RS}({\rm Gr}(\gamma), W)
=\mu_{\mathrm{CZ}}(\gamma),
$$
Thus, the second equality in (\ref{e:RobbinSa6}) follows from (\ref{e:RobbinSaCZ+RobbinSa10}), and the first equality in (\ref{e:RobbinSa6}) is then a consequence of (\ref{e:RobbinSa2}).

Motivated by the study of different problems in Hamiltonian dynamics,
the various generalizations of the above Maslov-type index $(i_\tau,\nu_\tau)$ were
 developed by Long's school, see \cite{Liu17} for an overview.
This paper only concerns those in
\cite{Do06, Liu06}.

Dong \cite{Do06} and Liu \cite{Liu06} each extended
the Maslov-type index $(i_\tau(\gamma),\nu_\tau(\gamma))$ of $\gamma\in\mathcal{P}_\tau(2n)$
 to the case relative to a given symplectic matrix $M\in{\rm Sp}(2n,\mathbb{R})$,
but via  different methods. We denoted their respective extensions by
\begin{equation}\label{e:DongLiuIndex}
(i_{\tau, M}(\gamma),\nu_{\tau,M}(\gamma))\quad\text{and}\quad (i_{\tau}^M(\gamma),\nu_{\tau}^M(\gamma)).
\end{equation}
(The former was written as  $(i_M(\gamma),\nu_M(\gamma))$ in \cite{Do06}).
 They are more suitable and convenient for dual variational methods and saddle point reduction ones, respectively.
If $M=I_{2n}$, by \cite{Do06} and \cite[Definition~2.5 and Definition~2.6]{Liu06} there holds
\begin{equation}\label{e:dongLiuIndex+}
(i_{\tau,M}(\gamma),\nu_{\tau,M}(\gamma))=(i_{\tau}^M(\gamma),\nu_{\tau}^M(\gamma))=(i_\tau(\gamma),\nu_\tau(\gamma)).
\end{equation}
Let $\xi$ be any element in $\mathcal{P}_\tau(2n)$ satisfying $\xi(\tau)=M^{-1}$.
 Dong  defined
 \begin{equation}\label{e:dongIndex}
  \begin{aligned}
  &\nu_{\tau,M}(\gamma)=\dim{\rm Ker}(\gamma(\tau)-M)\quad\hbox{and}\\
&i_{\tau,M}(\gamma)=[i_\tau((\gamma M^{-1})\ast\xi)-\Delta({\xi})/\pi]
   \end{aligned}
 \end{equation}
(\cite[Definitions~2.1, 2.2]{Do06}), and  Liu (\cite[Definition~2.7 and Remark~2.8]{LiuTang15}) defined
\begin{equation}\label{e:LiuTang-index}
\begin{aligned}
&\nu_{\tau}^M(\gamma)=\dim{\rm Ker}(\gamma(\tau)-M)\quad\hbox{and}\\
&i^{M}_\tau(\gamma)=i_\tau((M^{-1}\gamma)\ast\xi)-i_\tau({\xi})
\end{aligned}
 \end{equation}
 for $M\ne I_{2n}$,  and $(i_{\tau}^M(\gamma),\nu_{\tau}^M(\gamma))=(i_\tau(\gamma),\nu_\tau(\gamma))$
  for $M= I_{2n}$.
 (Note that $i_\tau((M^{-1}\gamma)\ast\xi)$ can be replaced by $i_\tau((\gamma M^{-1})\ast\xi))$,
 see the proof of \cite[Proposition~A.1]{Lu11}.
It was shown in \cite[Remark~2.8]{LiuTang15} that when $M=I_{2n}$
the right side of the second equality in (\ref{e:LiuTang-index}) is $i_\tau(\gamma)+n$.)
We derived the following exact relation between
$i_{\tau, M}$ and $i^{M}_\tau$ in \cite{Lu11}.

\begin{proposition}[\hbox{\cite[Proposition~A.1]{Lu11}}]\label{prop:twoDef}
 For any $(M,\gamma)\in ({\rm Sp}(2n,\R)\setminus\{I_{2n}\})\times\mathcal{P}_\tau(2n)$
 and  any $\xi\in\mathcal{P}_\tau(2n)$ satisfying $\xi(\tau)=M^{-1}$
 it holds that
  \begin{equation}\label{e:indexrelation}
 i_{\tau,M}(\gamma)=[i^{M}_\tau(\gamma)+i_\tau(\xi)-\Delta({\xi})/\pi]=i^{M}_\tau(\gamma)+ [i_\tau(\xi)-\Delta({\xi})/\pi].
\end{equation}
 {\rm (}The term $[i_\tau(\xi)-\Delta({\xi})/\pi]$
only depends on $M$.{\rm )}
   \end{proposition}

  Lemma~\ref{lem:Dong06}(iii) shows that $i_{\tau,M}$ and $i^M_\tau$ may be different.
  Therefore, we cannot expect $i_{\tau,M}$ to have the same good properties as $i^M_\tau$ listed in
\cite{Liu17}, see, for example, Lemma~\ref{lem:Dong06+}.
Moreover, as a generalization of the symplectic additivity (\ref{e:sympAdd}) of $i_\tau$,
Liu and Tang \cite[Proposition~2.13]{LiuTang15} proved that
for any $M_j\in{\rm Sp}(2n_j)$, $j=0,1$,
\begin{equation}\label{e:sympAddLiuT}
i^{M_0\diamond M_1}_\tau(\gamma_0\diamond\gamma_1)=i^{M_0}_\tau(\gamma_0)+ i^{M_1}_\tau(\gamma_1).
\end{equation}
However,  by (\ref{e:dongIndex}), (\ref{e:sympAdd}), and Lemma~\ref{lem:diamond-product},  we see that  there holds only
\begin{equation}\label{e:sympAddDong}
i_{\tau, M_0\diamond M_1}(\gamma_0\diamond\gamma_1)=[i_\tau((\gamma_0 M_0^{-1})\ast\xi_0)+i_\tau((\gamma_1 M_1^{-1})\ast\xi_1)-\Delta({\xi}_0\diamond\xi_1)/\pi]
\end{equation}
for $M_j\in{\rm Sp}(2n_j)$ and $\xi_j\in\mathcal{P}_\tau(2n_j)$ satisfying $\xi_j(\tau)=M_j^{-1}$, $j=0,1$.
If $M_0$ and $M_1$ are orthogonal symplectic matrices,  $\xi_j\in\mathcal{P}_\tau(2n_j)$ with
$\xi_j(\tau)=M_j^{-1}$ such that all $\xi_j(t)$ are also orthogonal symplectic matrices, $j=0,1$, then
Lemma~\ref{lem:diamond-product}(iii) leads to $\Delta({\xi}_0\diamond\xi_1)=\Delta({\xi}_0)+\Delta(\xi_1)$ and
therefore (\ref{e:sympAddDong}) becomes
\begin{equation}\label{e:sympAddDong1}
i_{\tau, M_0\diamond M_1}(\gamma_0\diamond\gamma_1)=[i_\tau((\gamma_0 M_0^{-1})\ast\xi_0)+i_\tau((\gamma_1 M_1^{-1})\ast\xi_1)-(\Delta({\xi}_0)+\Delta(\xi_1))/\pi].
\end{equation}

\subsubsection{Some lemmas}\label{sec:Pre2}

The following lemma establishes some properties of the $\diamond$-product of square block matrices defined in (\ref{e:diamond-product}).

\begin{lemma}\label{lem:diamond-product}
 \begin{enumerate}
\item[\rm (i)] For matrices $M_1, N_1\in\mathbb{R}^{2i\times 2i}$ and
$M_2, N_2\in\mathbb{R}^{2j\times 2j}$, there holds
\begin{equation}\label{e:diamond-pro}
(M_1\diamond M_2)(N_1\diamond N_2)=(M_1N_1)\diamond(M_2N_2).
\end{equation}

\item[\rm (ii)] $I_{2i}\diamond I_{2j}=I_{2(i+j)}$;  the $\diamond$-product $M_1\diamond M_2$
of invertible  matrices $M_1\in\mathbb{R}^{2i\times 2i}$ and
$M_2\in\mathbb{R}^{2j\times 2j}$ is invertible, and $(M_1\diamond M_2)^{-1}=M_1^{-1}\diamond M_2^{-1}$.

\item[\rm (iii)]
The $\diamond$-product satisfies $(M_1\diamond M_2)^\top = M_1^\top \diamond M_2^\top$.
Consequently, the set of orthogonal symplectic matrices is closed under the
$\diamond$-product. Moreover,  for any two symplectic matrices  $M_1$ and $M_2$, we have
\begin{equation}\label{e:diamond-pro1}
\mathfrak{u}(M_1\diamond M_2) =
\left( \begin{array}{cc}
\mathfrak{u}(M_1) & 0 \\
0 & \mathfrak{u}(M_2)
\end{array} \right),
\end{equation}
where $\mathfrak{u}$ is the map defined above (\ref{e:LieIsom}).

\item[\rm (iv)] $M_1\diamond\cdots \diamond M_k+ N_1\diamond\cdots\diamond N_k=(M_1+N_1)\diamond\cdots\diamond(M_k+N_k)$.

\item[\rm (v)] $(M_1+\cdots+M_k)\diamond(N_1+\cdots+N_k)=M_1\diamond N_1+\cdots+M_k\diamond N_k$ and
$$\exp(M_1\diamond\cdots\diamond M_k)=\exp(M_1)\diamond\cdots\diamond\exp(M_k).
$$

\item[\rm (vi)] For the matrices $M_1$ and $M_2$ in (\ref{e:diamond-product}) and
vectors $X_1, X_2\in\mathbb{K}^i$
and $Y_1, Y_2\in\mathbb{K}^j$, it holds that
$$
M_1\diamond M_2\left( \begin{array} { c}
X _ { 1 } \\
    Y_1 \\ X_2 \\ Y_2 \end{array} \right)=
\left( \begin{array} { c }
A _ { 1 }X_1+ B _ { 1 }X_2 \\
     A _ { 2 }Y_1+ B _ { 2 }Y_2 \\
     C _ { 1 }X_1+ D _ { 1 }X_2 \\
      C _ { 2 }Y_1+D _ { 2 }Y_2 \end{array} \right).
      $$
\end{enumerate}
\end{lemma}
\begin{proof}[\bf Proof]
We prove only (iii). The first statement is immediate. This, combined with (i) and (ii), leads to the second claim. Consequently, (\ref{e:diamond-pro1}) holds for orthogonal symplectic matrices $M_1$ and $M_2$.
The general case follows from a polar decomposition argument.
Define  $R_j=\sqrt{M_jM_j^\top}$ for $j=1,2$,  which are positive definite and satisfy
 $R_j=\sqrt{M_jM_j^\top}$.
Note that $R_1\diamond R_2$ is also positive definite.
Applying  (i) and the first claim in (iii), we have:
$$
(R_1\diamond R_2)^2=R_1^2\diamond R_2^2=(M_1M_1^\top)\diamond(M_2M_2^\top)=(M_1\diamond M_2)(M_1\diamond M_2)^\top.
$$
Hence, we may identify $\sqrt{(M_1\diamond M_2)(M_1\diamond M_2)^\top}$ with $R_1\diamond R_2$.
Now, by (i) and (ii), the  matrix
$$
(R_1\diamond R_2)^{-1}(M_1\diamond M_2)=(R_1^{-1}\diamond R_2^{-1})(M_1\diamond M_2)=
(R_1^{-1}M_1)\diamond (R_2^{-1}M_2),
$$
is orthogonal symplectic.  Consequently,
$$
\mathfrak{u}((R_1\diamond R_2)^{-1}(M_1\diamond M_2)) =\mathfrak{u}((R_1^{-1}M_1)\diamond (R_2^{-1}M_2))
=\left( \begin{array}{cc}
\mathfrak{u}(R_1^{-1}M_1) & 0 \\
0 & \mathfrak{u}(R_2^{-1}M_2)
\end{array} \right).
$$
By the definition of $\mathfrak{u}$, we have $\mathfrak{u}(M_1)=\mathfrak{u}((R_1^{-1}M_1)$, $\mathfrak{u}(M_2)=\mathfrak{u}((R_2^{-1}M_2)$, and
$$
\mathfrak{u}(M_1\diamond M_2)=\mathfrak{u}((R_1\diamond R_2)^{-1}(M_1\diamond M_2)).
$$
Combining these identities completes the proof of (\ref{e:diamond-pro1}).
\end{proof}

\begin{lemma}\label{lem:Dong06}
  \begin{enumerate}
\item[\rm (i)] Let $\gamma\in\mathcal{P}_\tau(2n)$.
If $\gamma(\tau)M^{-1} \in \mathrm{Sp}(2n)^*$, then
 \begin{equation}\label{e:dongLiuIndex*}
i_{\tau,M}(\gamma)= \left[ \frac{\Delta(\beta)}{\pi} + \frac{\Delta(\gamma M^{-1})}{\pi} \right]
\end{equation}
  for any path  $\beta \in C([0, \tau], \mathrm{Sp}(2n)^*)$
 connecting $\beta(0) = \gamma(\tau)M^{-1}$ to $\beta(\tau) = M_n^+$ (or $M_n^-$). (Such a path always exists by \cite[Theorem~2.4.1]{Long02}.)
(If $\gamma(\tau)M^{-1} \in \mathrm{Sp}(2n)^0$,
it was stated in \cite[Remark~2.3]{Do06} that
formula \eqref{e:dongLiuIndex*} holds
 for some path $\beta \in C([0, \tau], \mathrm{Sp}(2n))$, depending solely on
 $\gamma(\tau)M^{-1}$,  such that $\beta(t) \in \mathrm{Sp}(2n)^*$ for all $t \in (0, \tau]$, with endpoints
  $\beta(0) = \gamma(\tau)M^{-1}$ and $\beta(\tau) = M_n^+$ (or $M_n^-$).)

\item[\rm (ii)] Let $\xi_{2n,\tau}(t)=I_{2n}$ for all $t\in [0, \tau]$. Define
$$
R(\theta)=\left(
 \begin{array} {cc}
\cos (\theta) & -\sin (\theta) \\
\sin (\theta) &\cos (\theta)\end{array} \right),\quad \theta\in\mathbb{R},
$$
and for $c\in\mathbb{R}$ set $\gamma_c(t)=R(ct)=e^{tcJ_1}$. Then
$$
i_{\tau, R ( \theta ) } \left(\xi_{2,\tau}\right) =
\begin{cases}
-1\quad&\hbox{if $\theta\in [0, \pi)$},\\
0\quad&\hbox{if $\theta\in [\pi, 2\pi)$};
\end{cases}
$$
and
$$
i_{\tau, R ( \theta ) } \left( \gamma_c \right) =
\begin{cases}
2k-1\quad&\hbox{if $\theta\in [0, \pi)$ and $\theta + 2 ( k - 1 ) \pi < c \tau \leqslant \theta + 2 k \pi$},\\
2k\quad&\hbox{if $\theta\in [\pi, 2\pi)$ and $\theta + 2 ( k - 1 ) \pi < c \tau \leqslant \theta + 2 k \pi$}.
\end{cases}
$$
{\rm (}See the paragraph below \cite[Remark~2.4]{Do06}{\rm )}.

\item[\rm (iii)] $i_{\tau}^{-J_n}(\xi_{2n,\tau})=0$ and $i_{\tau, -J_n}(\xi_{2n,\tau})=[n/2]$.

\item[\rm (iv)] $i_{\tau}^{J_n}(\xi_{2n,\tau})=0$ and $i_{\tau, J_n}(\xi_{2n,\tau})=[-n/2]$.
  \end{enumerate}
  \end{lemma}
\begin{proof}[\bf Proof]
{\bf (I)}[\textsf{Proof of the first part in} (i)].\quad
Suppose that $\gamma(\tau )M^{-1}\in \mathrm{Sp}(2n)^{+}$ (resp. $\mathrm{Sp}(2n)^{-}$) and
let $\beta\in C ( [ 0 , \tau ] , \mathrm{Sp}( 2 n )^\ast)$ be a path
 satisfying $\beta( 0 ) =\gamma(\tau )M^ { - 1 }$ and
$\beta( \tau ) =M_n^+$ (resp. $M_n^-$).
By the definition of $i_\tau$ in \cite{Long02}, we have
$$i_\tau((\gamma M^{-1})\ast\xi)= \frac { \Delta\left( \beta * \left( \gamma M ^ { - 1 } \right) * \xi \right) } { \pi } =
 \frac { \Delta( \beta ) } { \pi } + \frac { \Delta\left( \gamma M^ { - 1 } \right) } { \pi } +
 \frac { \Delta( \xi ) } { \pi }.$$
The desired equality then follows directly from the definition in (\ref{e:dongIndex}).

A theorem due to Conley and Zehnder \cite{CoZe}, Salamon and Zehnder \cite{SaZe2}, and Long \cite{Long99} states that both $\mathrm{Sp}(2n)^+$ and $\mathrm{Sp}^-(2n)$ are simply connected subsets in $\mathrm{Sp}(2n)$.
 Consequently,  for any path
 $\beta^\star\in C ( [ 0 , \tau ] , \mathrm{Sp}( 2 n )^\ast)$
 satisfying $\beta^\star( 0 ) =\gamma(\tau )M^ { - 1 }$ and $\beta^\star( \tau ) = \beta(\tau)$,
 we have $\Delta( \beta )=\Delta(\beta^\star)$.
  (See the proof of Lemma~6 in \cite[page~120]{Long02}.)

{\bf (II)}[\textsf{Proof of} (iii)].\quad
Since $-J_n$ and $\xi_{2n,\tau}$ are,  respectively,
 the $n$-fold $\diamond$-product of $-J_1$ and $\xi_{2,\tau}$,
 it follows from (\ref{e:sympAddLiuT}) that
 $$
 i_{\tau}^{-J_n}(\xi_{2n,\tau})= ni_{\tau}^{-J_1}(\xi_{2,\tau}).
 $$
Thus, it suffices to prove $i_{\tau}^{-J_1}(\xi_{2,\tau})=0$.
By (\ref{e:LiuTang-index}),
$$
i^{-J_1}_\tau(\xi_{2,\tau})=i_\tau(((-J_1)^{-1}\xi_{2,\tau})\ast\xi)-i_\tau({\xi}),
$$
where $\xi\in\mathcal{P}_\tau(2)$ is any path with endpoint $\xi(\tau)=(-J_1)^{-1}=J_1$.
Using (\ref{e:pathComp}), one easily verifies that
\begin{equation*}
  ((-J_1)^{-1}\xi_{2,\tau})\ast\xi(t) =
   \begin{cases}
\xi(2t)\quad&\hbox{if $t\in[0, \tau/2]$},\\
J_1\quad&\hbox{if $t\in [\tau/2, \tau]$}.
\end{cases}
\end{equation*}
Clearly,   the paths $((-J_1)^{-1}\xi_{2,\tau})\ast\xi$ and $\xi$ are homotopic
within $\mathcal{P}_\tau(2)$ while keeping their endpoints fixed.
Therefore,
 $$
 i_\tau(((-J_1)^{-1}\xi_{2,\tau})\ast\xi)=i_\tau({\xi}),
 $$
 which implies $i^{-J_1}_\tau(\xi_{2,\tau})=0$ as desired.

Next, we compute  $i_{\tau,-J_n}(\xi_{2n,\tau})$. By (\ref{e:dongIndex}),
$$
i_{\tau,-J_n}(\xi_{2n,\tau})=[i_\tau((\gamma (-J_n)^{-1})\ast\xi)-\Delta({\xi})/\pi],
$$
 where $\xi\in\mathcal{P}_\tau(2n)$ is any path with endpoint $\xi(\tau)=(-J_n)^{-1}=J_n$.
A natural choice for $\xi$ is
$\xi^{(n)}(t)=\exp(J_n\frac{\pi t}{2\tau})$.
Define $\xi^{(1)}=\gamma_c(t)=R(ct)$ with $c=\frac{\pi}{2\tau}$.
Observe  that
$\xi^{(n)}$ is the $n$-fold $\diamond$-product of  $\xi^{(1)}$.
Applying the additivity formula (\ref{e:sympAddDong}) yields
$$
    i_{\tau,-J_n}(\xi_{2n,\tau})
   =\left[ni_{\tau}((\xi_{2,\tau} J_1)\ast\xi^{(1)})-n\Delta({\xi^{(1)}})/\pi\right].
$$
As noted earlier, the concatenated path satisfies
$$
(\xi_{2,\tau} J_1)\ast\xi^{(1)}(t)=
\begin{cases}
\xi^{(1)}(2t)\quad&\text{for $t\in[0, \tau/2]$},\\
J_1\quad &\text{for $t\in[\tau/2,\tau]$. }
\end{cases}
$$
Therefore,  $(\xi_{2,\tau} J_1)\ast\xi^{(1)}$
is homotopic to  $\xi$ within $\mathcal{P}_\tau(2)$ with fixed endpoints. Consequently,
$$
i_{\tau}((\xi_{2,\tau} J_1)\ast\xi^{(1)})=i_\tau(\xi^{(1)}),
$$
and hence
$$
i_{\tau,-J_n}(\xi_{2n,\tau})=\left[ni_{\tau}(\xi^{(1)})-n\Delta({\xi^{(1)}})/\pi\right].
$$
Setting $\theta=0$ in the second statement of Lemma~\ref{lem:Dong06}(ii) gives
$i_\tau(\xi^{(1)})=i_{\tau, R (0) } \left( \xi^{(1)} \right) =
i_{\tau, R (0) } \left( \gamma_c \right)=
1$.
Clearly, $\Delta({\xi^{(1)}})=\pi/2$.
 Substituting these values, we finally obtain
$$
i_{\tau,-J_n}(\xi_{2n,\tau})=\left[ni_{\tau}(\xi^{(1)})-n\Delta({\xi^{(1)}})/\pi\right]
=\left[\frac{n}{2}\right].
$$

{\bf (III)}[\textsf{Proof of} (iv)].\quad
The proof is almost identical to that of (iii). Nevertheless, we will outline it here, as this will be convenient for the proof of Theorem~\ref{th:PIndex2}.

Since $i_{\tau}^{J_n}(\xi_{2n,\tau})= ni_{\tau}^{J_1}(\xi_{2,\tau})$
as in (II), we only need to prove $i_{\tau}^{J_1}(\xi_{2,\tau})=0$.
For any path $\xi\in\mathcal{P}_\tau(2)$
with endpoint $\xi(\tau)=(J_1)^{-1}=-J_1$,
 (\ref{e:LiuTang-index}) yields
$$
i^{J_1}_\tau(\xi_{2,\tau})=i_\tau(((J_1)^{-1}\xi_{2,\tau})\ast\xi)-i_\tau({\xi}).
$$
As in (II) we may prove $i^{J_1}_\tau(\xi_{2,\tau})=0$ because
\begin{equation*}
	((J_1)^{-1}\xi_{2,\tau})\ast\xi(t) =
	\begin{cases}
		\xi(2t)\quad&\hbox{if $t\in[0, \tau/2]$},\\
		-J_1\quad&\hbox{if $t\in [\tau/2, \tau]$},
	\end{cases}
\end{equation*}
shows that the paths $((J_1)^{-1}\xi_{2,\tau})\ast\xi$ and $\xi$ are homotopic
within $\mathcal{P}_\tau(2)$ while keeping their endpoints fixed.

Next, we compute  $i_{\tau,J_n}(\xi_{2n,\tau})$. By (\ref{e:dongIndex}),
$$
i_{\tau,J_n}(\xi_{2n,\tau})=[i_\tau((\gamma (J_n)^{-1})\ast\xi)-\Delta({\xi})/\pi],
$$
where $\xi\in\mathcal{P}_\tau(2n)$ is any path with endpoint $\xi(\tau)=(J_n)^{-1}=-J_n$.
A natural choice for $\xi$ is
$\zeta^{(n)}(t)=\exp(
-J_n\frac{\pi t}{2\tau})$.
Define $\zeta^{(1)}=\gamma_c(t)=R(ct)$ with $c=-\frac{\pi}{2\tau}$.
Clearly, $\zeta^{(n)}$ is the $n$-fold $\diamond$-product of  $\zeta^{(1)}$.
Applying the additivity formula (\ref{e:sympAddDong}) yields
$$
i_{\tau,J_n}(\xi_{2n,\tau})
=\left[ni_{\tau}((\xi_{2,\tau} (-J_1))\ast\zeta^{(1)})-n\Delta({\zeta^{(1)}})/\pi\right].
$$
Since the concatenated path satisfies
$$
(\xi_{2,\tau} (-J_1))\ast\zeta^{(1)}(t)=
\begin{cases}
	\zeta^{(1)}(2t)\quad&\text{for $t\in[0, \tau/2]$},\\
	-J_1\quad &\text{for $t\in[\tau/2,\tau]$, }
\end{cases}
$$
the path  $(\xi_{2,\tau} (-J_1))\ast\zeta^{(1)}$
is homotopic to  $\xi$ within $\mathcal{P}_\tau(2)$ with fixed endpoints. Then
$i_{\tau}((\xi_{2,\tau} (-J_1))\ast\zeta^{(1)})=i_\tau(\zeta^{(1)})$,
and hence
$$
i_{\tau,J_n}(\xi_{2n,\tau})=
\left[ni_{\tau}(\zeta^{(1)})-n\Delta({\zeta^{(1)}})/\pi\right].
$$
Taking $\theta=0$ in the second statement of Lemma~\ref{lem:Dong06}(ii) we get
$i_\tau(\zeta^{(1)})=i_{\tau, R (0) } \left( \gamma_c \right) =-1$.
Clearly, $\Delta({\zeta^{(1)}})=-\pi/2$.
Substituting these values, we finally obtain
$$
i_{\tau, J_n}(\xi_{2n,\tau})=\left[ni_{\tau}(\zeta^{(1)})-n\Delta({\zeta^{(1)}})/\pi\right]
=\left[-\frac{n}{2}\right].
$$
\end{proof}

%

For $\tau > 0$ and $\omega \in \mathbf{U}$,
let $\operatorname{Sp}(2n)_\omega^*=\{M\in \operatorname{Sp}(2n)\,|\,\det(M-\omega I_{2n})\ne 0\}$ and
\[
\mathcal{P}_{\tau,\omega}^*(2n) = \bigl\{ \gamma \in \mathcal{P}_\tau(2n) \,\bigm|\, \gamma(\tau) \in \operatorname{Sp}(2n)^* \bigr\}, \quad
\mathcal{P}_{\tau,\omega}^0(2n) = \mathcal{P}_\tau(2n) \setminus \mathcal{P}_{\tau,\omega}^*(2n).
\]
 For $\gamma \in \mathcal{P}_\tau(2n)$,  define
$\nu_\omega(\gamma) = \nu_\omega\bigl(\gamma(\tau)\bigr) = \dim_{\mathbb{C}} \ker_{\mathbb{C}}\bigl(\gamma(\tau) - \omega I\bigr)$.
If $\gamma \in \mathcal{P}_{\tau,\omega}^0(2n)$, define
\begin{equation}\label{e:Long142A-}
i_\omega(\gamma) = i_\omega(\gamma_{-s}) \quad \forall s \in (0,1]
\end{equation}
(\cite[page~129, (15)]{Long02}),
where $\gamma_{-s}$ is the rotational perturbation path of $\gamma$ defined by (\cite[page~129, (6)]{Long02}), and $i_\omega(\gamma_{-s})$
was defined by (\cite[page~120, (11)]{Long02}).
In applications, it is not easy  to compute
$i_\omega(\gamma)$ for $\gamma \in \mathcal{P}_{\tau,\omega}^0(2n)$.
As a corollary of the following result we
can obtain a useful method.

\begin{theorem}[\hbox{\cite[page~142, Th.8]{Long02}}]\label{th:Long142}
 For $\omega \in \mathbf{U}$ and $\tau > 0$, let $\gamma \in \mathcal{P}_{\tau,\omega}^0(2n)$. Then for any two paths $\alpha$ and $\beta$ in $\mathcal{P}_{\tau,\omega}^\ast(2n)$ which are sufficiently $C^0$-close to $\gamma$, and any $s \in (0,1]$,
\[
i_\omega(\gamma_{-s}) \le i_\omega(\alpha) \le i_\omega(\gamma_s) = i_\omega(\gamma_{-s}) + \nu_\omega(\gamma),
\]
\[
|i_\omega(\beta) - i_\omega(\alpha)| \le \nu_\omega(\gamma).
\]
\end{theorem}

\begin{corollary}\label{cor:Long142A}
Let $\omega \in \mathbf{U}$ and $\tau > 0$.
\begin{itemize}
\item[\rm (A)] For $\gamma \in \mathcal{P}_{\tau,\omega}^0(2n)$, suppose that there exist two sequences $(\alpha_k)$ and $(\beta_k)$ in
 $\mathcal{P}_{\tau,\omega}^\ast(2n)$,
 both converging  to $\gamma$ in $C^0$-topology,  such that
 the following conditions are satisfied:
 \begin{itemize}
 \item[\rm (i)] $i_\omega(\alpha_k)\equiv m_1$ for all $k=1,2,\cdots$,
 \item[\rm (ii)] $i_\omega(\beta_k)\equiv m_2$ for all $k=1,2,\cdots$,
 \item[\rm (iii)] $|m_2-m_1|=\nu_\omega(\gamma)$.
  \end{itemize}
 Then $i_\omega(\gamma)=\min\{m_1,m_2\}$.
 \item[\rm (B)]
 For  $0<\epsilon\ll\tau$,
let $\Upsilon:[0, \tau+\epsilon]\to \operatorname{Sp}(2n)$ be
a continuous path such that
$\Upsilon(\tau)\in\operatorname{Sp}(2n)_\omega^0=\{M\in \operatorname{Sp}(2n)\,|\,\det(M-\omega I_{2n})= 0\}$ and
$\Upsilon(t)\in\operatorname{Sp}(2n)^*$ for all
$t\in [\tau-\epsilon, \tau+\epsilon]\setminus\{\tau\}$.
For $\rho\in (0, \tau+\epsilon]$, define $\Upsilon_\rho(t)=\Upsilon(\frac{\rho}{\tau}t)$ for $t\in [0, \tau]$.
Then $\Upsilon_\tau=\Upsilon|_{[0,\tau]}\in \mathcal{P}_{\tau,\omega}^0(2n)$
and $\Upsilon_\rho\in \mathcal{P}_{\tau,\omega}^\ast(2n)$
for $\rho\in [\tau-\epsilon, \tau+\epsilon]\setminus\{\tau\}$, and
(by the homotopy invariance of $i_\omega$
(\cite[page~145, Th.3]{Long02})) therefore
\begin{align*}
&i_\omega(\Upsilon|_{[0,\rho]})=i_\omega(\Upsilon_\rho)\equiv
i_\omega(\Upsilon|_{[0,\tau-\epsilon]})\quad\forall\rho\in [\tau-\epsilon, \tau),\\
&i_\omega(\Upsilon|_{[0,\rho]})=i_\omega(\Upsilon_\rho)\equiv
i_\omega(\Upsilon|_{[0,\tau+\epsilon]})\quad\forall\rho\in (\tau, \tau+\epsilon]. \end{align*}
Moreover, $i_\omega(\Upsilon_\tau)=
\min\{i_\omega(\Upsilon|_{[0,\tau-\epsilon]}),
i_\omega(\Upsilon|_{[0,\tau+\epsilon]})\}$ provided that
$$
\max\{i_\omega(\Upsilon|_{[0,\tau-\epsilon]}),
i_\omega(\Upsilon|_{[0,\tau+\epsilon]})\}-
\min\{i_\omega(\Upsilon|_{[0,\tau-\epsilon]}),
i_\omega(\Upsilon|_{[0,\tau+\epsilon]})\}=\nu_\omega(\Upsilon_\tau).
$$
 \end{itemize}
\end{corollary}
\begin{proof}[\bf Proof]
Clearly, $\Upsilon_\rho$ converges  to $\Upsilon_\tau$ in $C^0$-topology
as $\rho\to\tau$. Let $\epsilon<\epsilon_k\downarrow 0$.
Then (B) can be derived from (A) by applying
to $\gamma=\Upsilon_\tau$, $\alpha_k=\Upsilon_{\tau-\epsilon_k}$ and
$\beta_k=\Upsilon_{\tau+\epsilon_k}$ for $k=1,2,\cdots$.

In order to prove (A), without loss of generality, we may assume $m_1<m_2$. Then $m_2-m_1=\nu_\omega(\gamma)$.
Let $\gamma_s$ and $\gamma_{-s}$ be as in Theorem~\ref{th:Long142}.
Then by Theorem~\ref{th:Long142}, for $k$ large enough it holds that
$$
i_\omega(\gamma_{-s})=i_\omega(\gamma)\le i_\omega(\alpha_k)
\le i_\omega(\gamma_s) = i_\omega(\gamma_{-s}) + \nu_\omega(\gamma)
$$
and therefore $m_1\ge i_\omega(\gamma)$.
On the other hand, for lagre $k$, Theorem~\ref{th:Long142} also leads to
$$
m_1+ \nu_\omega(\gamma)= m_2=i_\omega(\beta_k)
\le i_\omega(\gamma_s) = i_\omega(\gamma_{-s}) + \nu_\omega(\gamma)=
 i_\omega(\gamma)+\nu_\omega(\gamma)
$$
and hence $ i_\omega(\gamma)\ge m_1$. The desired claim is proved.
\end{proof}

\begin{corollary}\label{cor:Long142B}
Let $M\in{\rm Sp}(2n,\mathbb{R})$ and $\tau > 0$.
For $\gamma\in\mathcal{P}_\tau(2n)$,
let $i_{\tau, M}(\gamma)$, $i_{\tau}^M(\gamma)$ and $\nu_{\tau,M}(\gamma)=\nu_{\tau}^M(\gamma)$ be as in
(\ref{e:DongLiuIndex}). Fix $i\in \{i_{\tau, M}, i_{\tau}^M\}$ below.
\begin{itemize}
\item[\rm (A)]
For $\gamma \in \mathcal{P}_{\tau}(2n)$ satisfying $\det(\gamma(\tau)-M)=0$, suppose that there exist two sequences $(\alpha_k)$ and $(\beta_k)$ in
 $\mathcal{P}_{\tau,M}^\ast(2n):=\{\gamma \in \mathcal{P}_{\tau}(2n)\,|\, \det(\gamma(\tau)-M)\ne 0\}$,
 both converging  to $\gamma$ in $C^0$-topology,  such that
 the following conditions are satisfied:
 \begin{itemize}
 \item[\rm (i)] $i(\alpha_k)\equiv m_1$ for all $k=1,2,\cdots$,
 \item[\rm (ii)] $i(\beta_k)\equiv m_2$ for all $k=1,2,\cdots$,
 \item[\rm (iii)] $|m_2-m_1|=\nu_{\tau,M}(\gamma)=\nu_{\tau}^M(\gamma)$.
  \end{itemize}
 Then $i(\gamma)=\min\{m_1,m_2\}$.

\item[\rm (B)]
For $0<\epsilon\ll\tau$,
let $\Upsilon:[0, \tau+\epsilon]\to \operatorname{Sp}(2n)$ be
a continuous path such that
$\det(\Upsilon(\tau)-M)=0$ and
$\det(\Upsilon(t)-M)\ne 0$ for all
$t\in [\tau-\epsilon, \tau+\epsilon]\setminus\{\tau\}$.
For $\rho\in (0, \tau+\epsilon]$, define $\Upsilon_\rho(t)=\Upsilon(\frac{\rho}{\tau}t)$ for $t\in [0, \tau]$.
Then there hold:
\begin{align}\label{e:Long142C}
&i(\Upsilon|_{[0,\rho]})=i(\Upsilon_\rho)\equiv
i(\Upsilon|_{[0,\tau-\epsilon]})\quad\forall\rho\in [\tau-\epsilon, \tau),\\
&i(\Upsilon|_{[0,\rho]})=i(\Upsilon_\rho)\equiv
i(\Upsilon|_{[0,\tau+\epsilon]})\quad\forall\rho\in (\tau, \tau+\epsilon].
\label{e:Long142D}
\end{align}
Moreover, $i(\Upsilon_\tau)=
\min\{i(\Upsilon|_{[0,\tau-\epsilon]}),
i(\Upsilon|_{[0,\tau+\epsilon]})\}$ provided that
$$
\max\{i(\Upsilon|_{[0,\tau-\epsilon]}),
i(\Upsilon|_{[0,\tau+\epsilon]})\}-
\min\{i(\Upsilon|_{[0,\tau-\epsilon]}),
i(\Upsilon|_{[0,\tau+\epsilon]})\}=
\nu_{\tau,M}(\Upsilon_\tau)=\nu_{\tau}^M(\Upsilon_\tau).
$$
\end{itemize}
\end{corollary}
\begin{proof}[\bf Proof]
Corollary~\ref{cor:Long142A} contains the case $M=I_{2n}$.
From now on, we assume $M\ne I_{2n}$.
As in the proof of Corollary~\ref{cor:Long142A}, (B) may follows from (A).

By Proposition~\ref{prop:twoDef} we only need
to prove Corollary~\ref{cor:Long142B}(A) for $i=i_{\tau}^M$.
Let $\xi$ be any element in $\mathcal{P}_\tau(2n)$ satisfying $\xi(\tau)=M^{-1}$.
By (\ref{e:LiuTang-index}) and the assumptions (i)-(iii), we have
  \begin{align*}
    i_\tau((M^{-1}\gamma)\ast\xi)-i_\tau({\xi})&=i_\tau^M(\gamma),\\
    i_\tau((M^{-1}\alpha_k)\ast\xi)-i_\tau({\xi})&=i_\tau^M(\alpha_k)=m_1,
    \;\forall k=1,2,\cdots,\\
    i_\tau((M^{-1}\beta_k)\ast\xi)-i_\tau({\xi})&=
    i_\tau^M(\beta_k)=m_2,\;\forall k=1,2,\cdots,\\
  |i_\tau((M^{-1}\alpha_k)\ast\xi)-(M^{-1}\beta_k)\ast\xi|&=
|i_\tau^M(\alpha_k)-i_\tau^M(\beta_k)|\\
&=\dim{\rm Ker}(\gamma(\tau)-M)=
\nu_\tau((M^{-1}\gamma)\ast\xi),
       \end{align*}
Note that $(M^{-1}\gamma)\ast\xi\in\mathcal{P}_{\tau}^0(2n)$,
 and $(M^{-1}\gamma)\ast\xi,(M^{-1}\beta_k)\ast\xi\in
\mathcal{P}_{\tau}^\ast(2n)$ for all $k=1,2,\cdots$.
By Corollary~\ref{cor:Long142A}, we obtain
$$
i_\tau((M^{-1}\gamma)\ast\xi)=\min\{m_1-i_\tau({\xi}), m_2-i_\tau({\xi})\}
$$
and hence the desired claim.
\end{proof}

As a generalization of the index $i_\omega$ in
(\ref{e:Long142A-}), Liu and Tang \cite{LiuTang15}  introduced the Maslov $(M;\omega)$-index $i_\omega^M$ and gave a generalization of
Theorem~\ref{th:Long142} in \cite[Theorem~2.16]{LiuTang15}.
Therefore,  the corresponding results
of Corollary~\ref{cor:Long142B} can be derived with the same methods.

\begin{lemma}\label{lem:Dong06+}
For any $\gamma\in \mathcal { P } _ { \tau } ( 2 n )$ and $M, Q\in  \operatorname{Sp}(2n)$,
 the following identities hold:
 \begin{align}\label{e:RobbinSaCZ+RobbinSa7-}
 \nu_{\tau, Q^{-1}MQ}(Q^{-1}\gamma Q)&=\nu_{\tau, M}(\gamma)=\nu_{\tau}^M(\gamma)=\nu_{\tau}^{Q^{-1}MQ}(Q^{-1}\gamma Q),\nonumber\\
 i_{\tau}^{Q^{-1}MQ}(Q^{-1}\gamma Q)&=i_{\tau}^M(\gamma).
\end{align}
Moreover, if $M$ and $Q$ are also orthogonal symplectic matrices, then
\begin{equation}\label{e:RobbinSaCZ+RobbinSa7}
i_{\tau, Q^{-1}MQ}(Q^{-1}\gamma Q)=i_{\tau, M}(\gamma).
\end{equation}
 \end{lemma}
\begin{proof}[\bf Proof]
The equalities in the second line of Lemma~\ref{lem:Dong06+} are clear.

\textbf{(I) [Proof of (\ref{e:RobbinSaCZ+RobbinSa7-})].} \quad
If $M = I$, then \eqref{e:RobbinSaCZ+RobbinSa7} is a special case of \eqref{e:RobbinSaCZ+RobbinSa6}
by \cite[Definition~2.5 and Definition~2.6]{Liu06}.

Suppose now that $M \ne I_{2n}$. Let $\xi\in \mathcal{P}_\tau(2n)$ be any path
 satisfying $\xi(\tau) = M^{-1}$.
Since $Q^{-1}\xi Q \in \mathcal{P}_\tau(2n)$ takes values
$Q^{-1}M^{-1}Q = (Q^{-1}MQ)^{-1}$ at $t = \tau$ and
\[
\bigl(Q^{-1}(M^{-1}\gamma)Q\bigr) \ast \bigl(Q^{-1}\xi Q\bigr) = Q^{-1}\bigl((M^{-1}\gamma) \ast \xi\bigr)Q,
\]
by \eqref{e:LiuTang-index} and \eqref{e:RobbinSaCZ+RobbinSa6} we deduce
\[
\begin{aligned}
i_{\tau}^{Q^{-1}MQ}(Q^{-1}\gamma Q) &= i_\tau\Bigl(Q^{-1}\bigl((M^{-1}\gamma) \ast \xi\bigr)Q\Bigr) - i_\tau\bigl(Q^{-1}\xi Q\bigr) \\
&= i_\tau\bigl((M^{-1}\gamma) \ast \xi\bigr) - i_\tau(\xi) \\
&= i_\tau^M(\gamma).
\end{aligned}
\]

\textbf{ (II) [{Proof of} (\ref{e:RobbinSaCZ+RobbinSa7})].}\quad
Since $M$ is orthogonal symplectic, so is $M^{-1}$. Hence, we can choose a path $\xi \in \mathcal{P}_\tau(2n)$ with $\xi(\tau) = M^{-1}$ such that $\xi(t)$ is orthogonal symplectic for all $t$.

As $Q^{-1}\xi Q \in \mathcal{P}_\tau(2n)$ and
\[
\bigl(Q^{-1}(\gamma M^{-1})Q\bigr) \ast \bigl(Q^{-1}\xi Q\bigr) = Q^{-1}\bigl((\gamma M^{-1}) \ast \xi\bigr)Q,
\]
we may apply \eqref{e:dongIndex} and \eqref{e:RobbinSaCZ+RobbinSa6} as before to obtain
\begin{align}\label{e:dongIndex1}
i_{\tau, Q^{-1}MQ}(Q^{-1}\gamma Q)&=\left[i_\tau((Q^{-1}\gamma Q Q^{-1}M^{-1}Q)\ast(Q^{-1}\xi Q))-\frac{\Delta({Q^{-1}\xi Q})}{\pi}\right]\nonumber\\
&=\left[i_\tau((Q^{-1}(\gamma M^{-1})Q)\ast(Q^{-1}\xi Q))-\frac{\Delta({Q^{-1}\xi Q})}{\pi}\right]\nonumber\\
&=\left[i_\tau((\gamma M^{-1})\ast\xi )-\frac{\Delta({Q^{-1}\xi Q})}{\pi}\right].
\end{align}

Now, because $Q$, $Q^{-1}$, and all $\xi(t)$ are orthogonal symplectic matrices, and because \eqref{e:LieIsom} is a Lie group isomorphism, we have
\[
\mathfrak{u}\bigl(Q^{-1}\xi(t)Q\bigr) = \mathfrak{u}(Q^{-1})\, \mathfrak{u}(\xi(t))\, \mathfrak{u}(Q).
\]
Consequently,
\[
\det\!\mathfrak{u}\bigl(Q^{-1}\xi(t)Q\bigr) = \det\!\mathfrak{u}(Q^{-1}) \det\!\mathfrak{u}(\xi(t)) \det\!\mathfrak{u}(Q) = \det\!\mathfrak{u}(\xi(t)),
\]
which yields $\Delta(Q^{-1}\xi Q) = \Delta(\xi)$. Substituting this into \eqref{e:dongIndex1}
 and comparing with \eqref{e:dongIndex} (or the definition of $i_{\tau,M}$) gives
 \eqref{e:RobbinSaCZ+RobbinSa7} immediately.
\end{proof}

For $B \in L^\infty([0,\tau]; \mathcal{L}_s(\mathbb{R}^{2n}))$, let $\Upsilon_B$ be as in \eqref{e:Ekeland-index}. For each $0 \le t \le \tau$, define
\[
\nu_{t,M}\bigl( \Upsilon_B|_{[0,t]} \bigr) = \dim \operatorname{Ker}\bigl( \Upsilon_B(t) - M \bigr),
\]
so that in particular
\[
\nu_{0,M}\bigl( \Upsilon_B|_{[0,0]} \bigr) = \dim \operatorname{Ker}(I_{2n} - M).
\]
The following lemma generalizes \eqref{e:Ekeland-index}.

\begin{lemma}[\hbox{\cite[Lemma~4.2]{Do06}}]\label{lem:Dong4.2}
Under the above assumptions it holds that
\begin{equation}\label{e:positive-negativeA}
i_{\tau,M}(\Upsilon_B)=i_{\tau,M}(\xi_{2n,\tau})
+\sum_{0\le t<\tau}\nu_{t,M}(\Upsilon_B|_{[0,t]})
\end{equation}
if $B$ is positive definite, and
\begin{equation}\label{e:positive-negativeB}
i_{\tau,M}(\Upsilon_B)=i_{\tau,M}(\xi_{2n,\tau})
-\sum_{0<t\le\tau}\nu_{t,M}(\Upsilon_B|_{[0,t]})
\end{equation}
if $B$ is negative definite. Here $\xi_{2n,\tau}(t)=I_{2n}$ for all $t\in [0, \tau]$.
\end{lemma}

A remark concerning (\ref{e:positive-negativeB}) can be found in \cite[Remark~A.4]{Lu11}.


\subsection{Proofs of Theorems~\ref{th:PIndex1} and \ref{th:PIndex2}}\label{sec:PIndex2}

\begin{proof}[\bf Proof of Theorem~\ref{th:PIndex1}]
	
	As agreed at the end of Section 1, all vectors in $\mathbb{R}^m$ and $\mathbb{C}^m$ are understood as column vectors.
From linear algebra, there exists an orthonormal set of complex eigenvectors
\( z_k = x_k + iy_k \) (with column
vectors \( x_k, y_k \in \mathbb{R}^n \))
of the Hermitian matrix $B+\sqrt{-1}C$
belonging to $e_k$, such that
$(B+\sqrt{-1}C)z_k = e_k z_k$
and $z_k^* z_l = \delta_{kl}$
for $k,l=1,\cdots,n$.
It can be easily verified that
real vectors
\[
u_k = \begin{pmatrix} x_k \\ y_k \end{pmatrix}, \quad v_k = \begin{pmatrix} -y_k \\ x_k \end{pmatrix}.
\]
are orthogonal unit eigenvectors of
$\left( \begin{array} { c c c c }
	B & -C \\
	C &B
\end{array} \right)$
corresponding to the eigenvalue \( e_k \).
A straightforward computation shows that
\[
Q =
\begin{pmatrix}
	x_1 & x_2 & \cdots & x_n & -y_1 & -y_2 & \cdots & -y_n \\
	y_1 & y_2 & \cdots & y_n & x_1 & x_2 & \cdots & x_n
\end{pmatrix}
\]
is a  real orthogonal and symplectic matrix satisfying
\begin{equation}\label{e:DiagDelayA.1-}
Q^\top\left( \begin{array} { c c c c }
	B & -C \\
	C &B
\end{array} \right)Q=
\operatorname{diag}(e_1, e_2, \dots, e_n, e_1, e_2, \dots, e_n).
\end{equation}
We therefore obtain  identity
\begin{equation}\label{e:DiagDelayA.1}
	J_n{\rm diag}(e_1,\cdots,
	e_n, e_1,\cdots, e_n)=J_nQ^{-1}\left( \begin{array} { c c c c }
		B & -C  \\
		C &B
	\end{array} \right)Q=Q^{-1}J_n\left( \begin{array} { c c c c }
		B & -C  \\
		C &B
	\end{array} \right)Q.
\end{equation}
Denote by $\hat\Upsilon$  the fundamental matrix solution of
$\dot{z}(t)=J_n{\rm diag}(e_1,\cdots,  e_n, e_1,\cdots,  e_n)z(t)$
with $\hat\Upsilon(0)=I_{2n}$; that is,  $\hat\Upsilon(t)=\exp(J_n{\rm diag}(e_1,\cdots,  e_n, e_1,\cdots,  e_n)t)$.
Then (\ref{e:DiagDelayA.1}) implies
\begin{equation}\label{e:DiagDelay1}
	Q^{-1}\Upsilon(t) Q=\hat\Upsilon(t).
\end{equation}
Since $Q^{-1}(-J_n)Q=-J_n$, from (\ref{e:DiagDelay1}) and Lemma~\ref{lem:Dong06+} we derive
\begin{equation}\label{e:DiagDelay2}
 i_{\tau,-J_n}(\Upsilon)=i_{\tau,-J_n}(\hat\Upsilon)
\quad\text{and}\quad
\nu_{\tau,-J_n}(\Upsilon)
=\nu_{\tau,-J_n}(\hat\Upsilon).
\end{equation}

For each $j=1,\cdots,n$, let $\hat\Upsilon^j$ be  the fundamental matrix solution of
$\dot{z}=J_1\operatorname{diag}(e_j, e_j)z$ with $\hat\Upsilon^j(0)=I_2$; that is,
\begin{equation}\label{e:fundSoluI}
\hat\Upsilon^j(t)=\exp\left(J_1\left( \begin{array} { c c c c }
e_j & 0  \\
0 &e_j
 \end{array} \right)t\right)=R(e_jt),
 \end{equation}
where  $R(t)=\left(
 \begin{array} {cc}
\cos t & -\sin t \\
\sin t &\cos t\end{array} \right)=\cos t I_2+ \sin t J_1=\exp({tJ_1})$.

Observe that $J_{n}=(J_1)^{\diamond n}$ (the $n$-fold $\diamond$ product of $J_1$, defined by (\ref{e:diamond-product})), and that
\begin{eqnarray*}
{\rm diag}(e_1,\cdots,
 e_n, e_1,\cdots,
 e_n)
=\left( \begin{array} { c c c c }
e_1 & 0  \\
0 &e_1
 \end{array} \right)\diamond\cdots\diamond
 \left( \begin{array} { c c c c }
e_n & 0  \\
0 &e_n
 \end{array} \right).
 \end{eqnarray*}
Combined with Lemma~\ref{lem:diamond-product}(i), this yields
\begin{eqnarray*}
J_n{\rm diag}(e_1,\cdots,
 e_n, e_1,\cdots,
 e_n)
=J_1\left( \begin{array} { c c c c }
e_1 & 0  \\
0 &e_1
 \end{array} \right)\diamond\cdots\diamond
 J_1\left( \begin{array} { c c c c }
e_n & 0  \\
0 &e_n
 \end{array} \right).
 \end{eqnarray*}
 Consequently, by Lemma~\ref{lem:diamond-product}(v) we obtain
 \begin{equation}\label{e:UpsilonProduct}
\hat\Upsilon(t)=\hat\Upsilon^1(t)\diamond\cdots\diamond\hat\Upsilon^n(t).
\end{equation}
\eqref{e:UpsilonProduct} together with Lemma~\ref{lem:diamond-product}(iv) yields
 $$
 \hat\Upsilon(\tau)-(-J_n)=(\hat\Upsilon^1(\tau)+J_1)\diamond\cdots\diamond(\hat\Upsilon^n(\tau)+J_1).
 $$
 It is easily verified that 
 $$
 {\rm Ker}(\hat\Upsilon(\tau)+J_n)=\{(x_1,\cdots,x_n,y_1,\cdots,y_n)^\top\in\mathbb{R}^{2n}\,|\,
 (x_i, y_i)^\top\in {\rm Ker}(\hat\Upsilon^i(\tau)+J_1),\;i=1,\cdots,n\}.
 $$
 Therefore, we obtain
 \[
\dim \operatorname{Ker}\!\bigl(\hat{\Upsilon}(\tau) + J_n\bigr)
= \sum_{i=1}^{n} \dim \operatorname{Ker}\!\bigl(\hat{\Upsilon}^i(\tau) + J_1\bigr),
\]
and consequently
 \begin{equation}\label{e:DiagDelay4}
 \nu_{\tau, -J_n}(\hat{\Upsilon}) = \sum_{i=1}^{n} \nu_{\tau, -J_1}(\hat{\Upsilon}^i).
 %
 \end{equation}
Moreover, using \eqref{e:fundSoluI} it is straightforward to verify that
\[
\det\!\bigl(\hat{\Upsilon}^i(t) + J_1\bigr) = 2 + 2\sin(e_i t).
\]
Consequently,
\begin{equation}\label{e:DiagDelay7}
\nu_{\tau,-J_1}(\hat{\Upsilon}^i) =
\begin{cases}
2, & \text{if } e_i = \dfrac{(4j+3)\pi}{2\tau} \text{ for some } j \in \mathbb{Z}, \\[8pt]
0, & \text{otherwise}.
\end{cases}
\end{equation}
Combining this with \eqref{e:DiagDelay2} and \eqref{e:DiagDelay4}, we obtain \eqref{e:DiagPindex1}.

 For the proof of \eqref{e:DiagPindex2}, define the path $\xi^{(k)} \in \mathcal{P}_\tau(2k)$ by
\begin{equation}\label{e:Xik}
\xi^{(k)}(t) = \exp\!\left( J_k \frac{\pi t}{2\tau} \right) \qquad \text{for all } t \in [0,\tau].
\end{equation}
Then $\xi^{(k)}(\tau) = J_k = (-J_k)^{-1}$ (since $J_k^2 = -I_{2k}$). It is easily verified that
$\xi^{(k)} = (\xi^{(1)})^{\diamond k}$, i.e., the $k$-fold $\diamond$-product of $\xi^{(1)}$.

Because $-J_n = (-J_1)^{\diamond n}$ and both $J_k$ and $\xi^{(k)}(t)$ are orthogonal symplectic matrices,
we may apply \eqref{e:UpsilonProduct} together with the additivity property \eqref{e:sympAddDong1} to obtain
\begin{align}\label{e:DiagDelay13}
i_{\tau,-J_n}(\hat\Upsilon)
&= \Biggl[ \sum_{k=1}^{n} i_\tau\bigl((\hat\Upsilon^k J_1) \ast \xi^{(1)}\bigr) - \frac{n \Delta(\xi^{(1)})}{\pi} \Biggr] \nonumber \\
&= \sum_{k=1}^{n} i_\tau\bigl((\hat\Upsilon^k J_1) \ast \xi^{(1)}\bigr) + \left[ -\frac{n}{2} \right],
\end{align}
where we have used the identities $\Delta(\xi^{(n)}) = n \Delta(\xi^{(1)}) = \frac{\pi n}{2}$ and
$[j + x] = j + [x]$ for all $j \in \mathbb{Z}$. 

Note that $\xi^{(1)}(t) = R\!\left(\frac{\pi t}{2\tau}\right)$ and
$\hat{\Upsilon}^k(t) J_1 = R(e_k t) R(\pi/2) = R\!\left(e_k t + \frac{\pi}{2}\right)$
since $R(\pi/2) = J_1$.
From \eqref{e:pathComp} we obtain
\[
(\hat{\Upsilon}^k J_1) \ast \xi^{(1)}(t) =
\begin{cases}
R\!\left(\dfrac{\pi t}{\tau}\right), & \text{if } t \in [0,\tau/2], \\[8pt]
R\!\left(e_k (2t - \tau) + \dfrac{\pi}{2}\right), & \text{if } t \in [\tau/2,\tau].
\end{cases}
\]
Clearly, the two paths
\[
[0,\tau] \ni t \longmapsto (\hat{\Upsilon}^k J_1) \ast \xi^{(1)}(t)
\qquad\text{and}\qquad
[0,\tau] \ni t \longmapsto \gamma_{c_k}(t) := R\!\left(e_k t + \frac{\pi t}{2\tau}\right),
\]
where $c_k := e_k + \frac{\pi}{2\tau}$, are homotopic in $\mathcal{P}_\tau(2)$ with fixed endpoints.
Hence, by the homotopy invariance \eqref{e:homotopyInvari},
\begin{equation}\label{e:DiagDelay14}
i_\tau\!\bigl((\hat{\Upsilon}^k J_1) \ast \xi^{(1)}\bigr) = i_\tau(\gamma_{c_k}).
\end{equation}
Now take $c = c_k$ and $\theta = 0$ in Lemma~\ref{lem:Dong06}(ii). This yields
\begin{equation}\label{e:DiagDelay15}
i_\tau(\gamma_{c_k}) = 2j_k + 1,
\end{equation}
where $j_k$ is the unique integer such that $c_k \tau \in \bigl(2j_k \pi,\; 2(j_k+1)\pi\bigr]$.
The existence and uniqueness of $j_k$ follow from the partition of $\mathbb{R}$
into the intervals $(2l\pi,\, 2(l + 1)\pi]$ for $l \in \mathbb{Z}$.
Note that
\[
c_k \tau \in \bigl(2j_k\pi,\; 2(j_k+1)\pi\bigr]
\quad\Longleftrightarrow\quad
\frac{(4j_k-1)\pi}{2\tau} < e_k \le \frac{(4j_k+3)\pi}{2\tau}.
\]
Combining \eqref{e:DiagDelay2}, \eqref{e:DiagDelay13}, \eqref{e:DiagDelay14}, and \eqref{e:DiagDelay15}, we obtain
\begin{equation}\label{e:DiagDelay16}
i_{\tau,-J_n}(\Upsilon) = 2j_1 + \dots + 2j_n + n + \left[-\frac{n}{2}\right]
                        = 2j_1 + \dots + 2j_n + \left[\frac{n}{2}\right].
\end{equation}

Next, we prove
\begin{equation}\label{e:DiagDelay17}
i_{\tau,-J_n}(\Upsilon) = i_{\tau}^{-J_n}(\Upsilon) + \left[\frac{n}{2}\right].
\end{equation}
Recall that the path $\xi^{(n)} \in \mathcal{P}_\tau(2n)$ defined in \eqref{e:Xik} satisfies
$\xi^{(n)}(\tau) = J_n = (-J_n)^{-1}$.
By Proposition~\ref{prop:twoDef}, it is sufficient to establish
\[
\left[\, i_\tau(\xi^{(n)}) - \frac{\Delta(\xi^{(n)})}{\pi} \,\right] = \left[\frac{n}{2}\right].
\]
Since $\xi^{(n)} = (\xi^{(1)})^{\diamond n}$, $-J_n = (-J_1)^{\diamond n}$, and
$\Delta(\xi^{(n)}) = n \Delta(\xi^{(1)}) = \frac{\pi n}{2}$, the additivity formula \eqref{e:sympAdd} gives
\[
i_\tau(\xi^{(n)}) - \frac{\Delta(\xi^{(n)})}{\pi}
= n i_\tau(\xi^{(1)})-\frac{n}{2}.
\]
Observe that $\xi^{(1)}(t) = R\!\left(\frac{\pi t}{2\tau}\right) = \gamma_c(t)$ with $c = \frac{\pi}{2\tau}$.
Setting $\theta = 0$ in Lemma~\ref{lem:Dong06}(ii) yields $i_\tau(\xi^{(1)}) = 1$; consequently,
\[
i_\tau(\xi^{(n)}) - \frac{\Delta(\xi^{(n)})}{\pi} = \frac{n}{2}.
\]
Taking the integer part of both sides confirms the required equality.
Finally, combining \eqref{e:DiagDelay16} and \eqref{e:DiagDelay17} we obtain \eqref{e:DiagPindex2}.

It remains to  prove \eqref{e:DiagPindex3} and \eqref{e:DiagPindex4}.
Let $\hat{\Upsilon}(t) = \hat{\Upsilon}^1(t) \diamond \cdots \diamond \hat{\Upsilon}^n(t)$ be the decomposition given in \eqref{e:UpsilonProduct}.
By Lemma~\ref{lem:Dong4.2}, Lemma~\ref{lem:Dong06}(iii), and \eqref{e:DiagDelay2}, we obtain
\begin{equation}\label{e:positive-negativeA4}
i_{\tau,-J_n}(\Upsilon) =i_{\tau,-J_n}(\hat{\Upsilon})= [n/2] + \sum_{0 \le t < \tau} \nu_{t,-J_n}\bigl(\hat{\Upsilon}|_{[0,t]}\bigr)
\end{equation}
if $B$ is positive definite (and so $e_l>0$
for $l=1,\cdots,n$), and
\begin{equation}\label{e:positive-negativeB4}
i_{\tau,-J_n}(\Upsilon) =i_{\tau,-J_n}(\hat{\Upsilon})= [n/2] - \sum_{0 < t \le \tau} \nu_{t,-J_n}\bigl(\hat{\Upsilon}|_{[0,t]}\bigr)
\end{equation}
if $B$ is negative definite (and so $e_l<0$
for $l=1,\cdots,n$).
Since  for every $t \in [0,\tau]$,
\[
\operatorname{Ker}\!\bigl(\hat{\Upsilon}(t) + J_n\bigr) =
\Bigl\{(x_1,\dots,x_n,y_1,\dots,y_n)^\top \in \mathbb{R}^{2n} \Bigm|
(x_k, y_k)^\top \in \operatorname{Ker}\!\bigl(\hat{\Upsilon}^k(t) + J_1\bigr),\; k=1,\dots,n \Bigr\},
\]
from \eqref{e:fundSoluI}, we deduce that for $t > 0$,
\begin{equation}\label{e:Delay6Auto6Xi3}
\begin{cases}
\det\!\bigl(\hat\Upsilon(t) + J_n\bigr) = 0
\Longleftrightarrow\;
e_l t \in 2\pi\mathbb{Z} + \frac{3\pi}{2} \text{ for some } l, \\[6pt]
\dim\operatorname{Ker}\!\bigl(\hat\Upsilon(t) + J_n\bigr) = 2\,\#\bigl\{k \mid \sin(e_k t) = -1\bigr\}.
\end{cases}
\end{equation}
Note that $\nu_{0,-J_n}\bigl(\hat{\Upsilon}|_{[0,0]}\bigr) = \dim\operatorname{Ker}(I_{2n} + J_n) = 0$. 
Substituting the description~(\ref{e:Delay6Auto6Xi3}) into \eqref{e:positive-negativeA4} and \eqref{e:positive-negativeB4} yields \eqref{e:DiagPindex3} and \eqref{e:DiagPindex4} directly.
\end{proof}

\begin{proof}[\bf Proof of Theorem~\ref{th:PIndex2}]
The desired results can be obtained by a suitable adaptation of the proof of Theorem~\ref{th:PIndex1}. All the arguments before equation (\ref{e:DiagDelay2}) continue to hold.

			As in the proofs of (\ref{e:DiagDelay2}), we have			
			\begin{equation}\label{e:DiagDelay2ADD}
				i_{\tau,J_n}(\Upsilon)=i_{\tau,J_n}(\hat\Upsilon)
				\quad\text{and}\quad
				\nu_{\tau,J_n}(\Upsilon)
				=\nu_{\tau,
					J_n}(\hat\Upsilon).
			\end{equation}	
			
			For each $j=1,\cdots,n$, let $\hat\Upsilon^j$ be  the fundamental matrix solution of
			$\dot{z}=J_1\operatorname{diag}(e_j, e_j)z$ with $\hat\Upsilon^j(0)=I_2$; that is,
			$\hat\Upsilon^j(t)=R(e_jt)$ by
			(\ref{e:fundSoluI}).
			Following an argument similar to that for (\ref{e:DiagDelay4}) and (\ref{e:DiagDelay7}), we obtain
			$$
			\nu_{\tau, J_n}(\hat{\Upsilon}) = \sum_{i=1}^{n} \nu_{\tau, J_1}(\hat{\Upsilon}^i),\qquad
			\nu_{\tau,J_1}(\hat{\Upsilon}^i) =
			\begin{cases}
				2, & \text{if } e_i = \dfrac{(4j+1)\pi}{2\tau} \text{ for some } j \in \mathbb{Z}, \\[8pt]
				0, & \text{otherwise},
			\end{cases}
			$$
			and hence (\ref{e:DiagPindex1ADD}).
			
			In order to prove (\ref{e:DiagPindex2ADD}), we define the path $\zeta^{(k)} \in \mathcal{P}_\tau(2k)$ by
			\begin{equation}\label{e:XikADD}
				\zeta^{(k)}(t) = \exp\!\left(-J_k \frac{\pi t}{2\tau} \right) \qquad \text{for all } t \in [0,\tau],
			\end{equation}
			which satisfies  $\zeta^{(k)}(\tau) = -J_k$
			and $\zeta^{(k)} = (\zeta^{(1)})^{\diamond k}$ (the $k$-fold $\diamond$-product of $\zeta^{(1)}$).
			By an argument similar to that  in the proof of (\ref{e:DiagDelay13}), we obtain
			\begin{equation}\label{e:DiagDelay13ADD}
				i_{\tau,J_n}(\hat\Upsilon)
				= \sum_{k=1}^{n} i_\tau\bigl((\hat\Upsilon^k J_1^{-1}) \ast \zeta^{(1)}\bigr) + \left[ \frac{n}{2} \right].
			\end{equation}

			Since $\zeta^{(1)}(t) = R\!\left(-\frac{\pi t}{2\tau}\right)$ and
			$\hat{\Upsilon}^k(t) J_1^{-1} = R(e_k t) R(-\pi/2) = R\!\left(e_k t - \frac{\pi}{2}\right)$, it follows from
			(\ref{e:pathComp}) that
			\[
			(\hat{\Upsilon}^k J_1^{-1}) \ast \zeta^{(1)}(t) =
			\begin{cases}
				R\!\left(-\dfrac{\pi t}{\tau}\right), & \text{if } t \in [0,\tau/2], \\[8pt]
				R\!\left(e_k (2t - \tau) - \dfrac{\pi}{2}\right), & \text{if } t \in [\tau/2,\tau].
			\end{cases}
			\]
			Therefore, for $d_k := e_k - \frac{\pi}{2\tau}$, the two paths
			\[
			[0,\tau] \ni t \longmapsto (\hat{\Upsilon}^k J_1^{-1}) \ast \zeta^{(1)}(t)
			\qquad\text{and}\qquad
			[0,\tau] \ni t \longmapsto \gamma_{d_k}(t) := R\!\left(e_k t - \frac{\pi t}{2\tau}\right),
			\]
			are homotopic in $\mathcal{P}_\tau(2)$ with fixed endpoints, and so
			\begin{equation}\label{e:DiagDelay14ADD}
				i_\tau\!\bigl((\hat{\Upsilon}^k J_1^{-1}) \ast \zeta^{(1)}\bigr) = i_\tau(\gamma_{d_k}).
			\end{equation}
			Taking  $c = d_k$ and $\theta = 0$ in Lemma~\ref{lem:Dong06}(ii), we arrive at
			\begin{equation}\label{e:DiagDelay15ADD}
				i_\tau(\gamma_{d_k}) = 2j_k + 1,
			\end{equation}
			where $j_k$ is the unique integer such that $d_k \tau \in \bigl(2j_k \pi,\; 2(j_k+1)\pi\bigr]$.
			Clearly,
			\[
			d_k \tau \in \bigl(2j_k\pi,\; 2(j_k+1)\pi\bigr]
			\quad\Longleftrightarrow\quad
			\frac{(4j_k+1)\pi}{2\tau} < e_k \le \frac{(4j_k+5)\pi}{2\tau}.
			\]
			From (\ref{e:DiagDelay13ADD}), (\ref{e:DiagDelay15ADD}) and
			(\ref{e:DiagDelay15ADD}) we derive
			\begin{equation}\label{e:DiagDelay16ADD}
				i_{\tau,J_n}(\Upsilon) = 2j_1 + \dots + 2j_n + n + \left[\frac{n}{2}\right]
				= 2j_1 + \dots + 2j_n + \left[\frac{3n}{2}\right].
			\end{equation}

			Next, by  the proof of Lemma~\ref{lem:Dong06}
			we see
			$$
			i_{\tau}(\zeta^{(n)})-\Delta({\zeta^{(n)}})/\pi=ni_{\tau}(\zeta^{(1)})-n\Delta({\zeta^{(1)}})/\pi=-\frac{n}{2},
			$$
			and therefore by Proposition~\ref{prop:twoDef} we obtain
			\begin{align*}
				i_{\tau,J_n}(\Upsilon) = i_{\tau}^{J_n}(\Upsilon) +
				\left[\, i_\tau(\zeta^{(n)}) - \frac{\Delta(\zeta^{(n)})}{\pi} \,\right]
				=i_{\tau}^{J_n}(\Upsilon)+
				\left[-\frac{n}{2}\right].
			\end{align*}
			This and (\ref{e:DiagDelay16ADD}) yield
			(\ref{e:DiagPindex2ADD}).

			Let $\hat{\Upsilon}(t) = \hat{\Upsilon}^1(t) \diamond \cdots \diamond \hat{\Upsilon}^n(t)$ be as in  \eqref{e:UpsilonProduct}.
			Then Lemma~\ref{lem:Dong4.2}, Lemma~\ref{lem:Dong06}(iv) and \eqref{e:DiagDelay2} lead to
			\begin{equation}\label{e:positive-negativeA4ADD}
				i_{\tau,J_n}(\Upsilon) =
				i_{\tau,J_n}(\hat{\Upsilon})= [-n/2] + \sum_{0 \le t < \tau} \nu_{t,J_n}\bigl(\hat{\Upsilon}|_{[0,t]}\bigr)
			\end{equation}
			if $B$ is positive definite, and
			\begin{equation}\label{e:positive-negativeB4ADD}
				i_{\tau,J_n}(\Upsilon) =
				i_{\tau,J_n}(\hat{\Upsilon})= [-n/2] - \sum_{0 < t \le \tau} \nu_{t,J_n}\bigl(\hat{\Upsilon}|_{[0,t]}\bigr)
			\end{equation}
			if $B$ is negative definite.
			
			As in the final part of the proof of Theorem~\ref{th:PIndex1},
			for every $t \in [0,\tau]$, we have that
			\[
			\operatorname{Ker}\!\bigl(\hat{\Upsilon}(t)- J_n\bigr) =
			\Bigl\{(x_1,\dots,x_n,y_1,\dots,y_n)^\top \in \mathbb{R}^{2n} \Bigm|
			(x_k, y_k)^\top \in \operatorname{Ker}\!\bigl(\hat{\Upsilon}^k(t) - J_1\bigr),\; k=1,\dots,n \Bigr\},
			\]
			and
			\begin{equation}\label{e:Delay6Auto6Xi3ADD}
				\begin{cases}
					\det\!\bigl(\Upsilon(t) - J_n\bigr) = 0
					\Longleftrightarrow\;
					e_l t \in 2\pi\mathbb{Z} + \frac{\pi}{2} \text{ for some } l, \\[6pt]
					\dim\operatorname{Ker}\!\bigl(\Upsilon(t) - J_n\bigr) = 2\,\#\bigl\{k \mid \sin(e_k t) = 1\bigr\}.
				\end{cases}
			\end{equation}
			Since $\nu_{0,J_n}\bigl(\hat{\Upsilon}|_{[0,0]}\bigr) = \dim\operatorname{Ker}(I_{2n} - J_n) = 0$,
			using arguments similar to those for (\ref{e:DiagPindex3}) and (\ref{e:DiagPindex4}), the desired equalities (\ref{e:DiagPindex3ADD}) and (\ref{e:DiagPindex4ADD}) follow directly from (\ref{e:positive-negativeA4ADD}), (\ref{e:positive-negativeB4ADD}), and (\ref{e:Delay6Auto6Xi3ADD}).
		\end{proof}

\subsection{Proof of Theorem~\ref{th:PIndex3} }\label{sec:PIndex3}

	
Since $B$ is a real symmetric matrix,
there exists an orthogonal matrix $P$ such that
\begin{equation}\label{e:Add10-}
P^{-1}BP={\rm diag}(e_1,\cdots,  e_n).
\end{equation}
A straightforward verification shows that the
	orthogonal  matrix
$Q:=\left( \begin{array} {cc}
	P & 0 \\
	0 &P\end{array} \right)$
	commutes with $J_n$  (hence is symplectic) and satisfies
\begin{align}\label{e:Add10}
J_n\left( \begin{array} { c c c c }
2{\rm diag}(e_1,\cdots, e_n) & -{\rm diag}(e_1,\cdots, e_n)  \\
-{\rm diag}(e_1,\cdots, e_n) &2{\rm diag}(e_1,\cdots, e_n)
 \end{array} \right)
  &=J_nQ^{-1}\left( \begin{array} { c c }
2B & -B  \\
-B &2B
 \end{array} \right)Q\nonumber\\
 &=Q^{-1}J_n\left( \begin{array} { c c }
2B & -B  \\
-B &2B
 \end{array} \right)Q.
\end{align}
 Denote by $\hat\Upsilon$  the fundamental matrix solution of
 $$
 \dot{z}(t)=J_n\left( \begin{array} { c c c c }
2{\rm diag}(e_1,\cdots, e_n) & -{\rm diag}(e_1,\cdots, e_n)  \\
-{\rm diag}(e_1,\cdots, e_n) &2{\rm diag}(e_1,\cdots, e_n)
 \end{array} \right)z(t)
 $$
 with $\hat\Upsilon(0)=I_{2n}$. 
It follows from (\ref{e:Add10}) that
 \begin{equation}\label{e:Add11}
Q^{-1}\Upsilon(t) Q=\hat\Upsilon(t).
\end{equation}
A straightforward calculation yields
$$
Q^{-1}M_{3,n}Q=\left( \begin{array} { c c c c c }
P^{-1} & 0\\
0 & P^{-1}\end{array} \right)\left( \begin{array} { c c c c c }
0 & I _ { n }\\
 -I _ { n } & I_n\end{array} \right)\left( \begin{array} { c c c c c }
P & 0\\
0 & P\end{array} \right)=M_{3,n}.
$$
This, together with Lemma~\ref{lem:Dong06+}, implies
\begin{equation}\label{e:Add12}
 i_{\tau, M_{3,n}}(\Upsilon)=i_{\tau, M_{3,n}}(\hat\Upsilon)
\quad\text{and}\quad
\nu_{\tau, M_{3,n}}(\Upsilon)
=\nu_{\tau, M_{3,n}}(\hat\Upsilon).
\end{equation}

For each $j=1,\cdots,n$, let $\hat\Upsilon^j$ be  the fundamental matrix solution of
$$
\dot{z}=J_1\left( \begin{array} { c c c c c }
2e_j & -e_j\\
 -e_j& 2e_j\end{array} \right)z
$$
with $\hat\Upsilon^j(0)=I_2$, i.e.,
\begin{align}\label{e:Add13}
\hat\Upsilon^j(t)&=\exp\left(J_1\left( \begin{array} { c c c c }
2e_j & -e_j  \\
-e_j &2e_j
 \end{array} \right)t\right)=\exp\left(J_1\left( \begin{array} { c c c c }
2e_j & -e_j  \\
-e_j &2e_j
 \end{array} \right)t\right)\nonumber \\
&=\exp\left(2e_jtJ_1\left( \begin{array} { c c c c }
1 & 0  \\
0 &1
 \end{array} \right)+e_jtJ_1\left(\begin{array} { c c c c }
0 & -1  \\
-1 &0
 \end{array} \right)\right).
 \end{align}

 Recall that  if two  anti-commutative matrices $X, Y$ of order $N$ (i.e., $XY=-YX$) satisfy
 $X^2=aI_N$ and $Y^2=bI_N$ for some  numbers $a$ and $b$, then
 $(X+Y)^{2k}=(a+b)^kI_N$ ($k\in\mathbb{N}$) and $(X+Y)^{2k+1}=(a+b)^{k}(X+Y)$ ($k\in\mathbb{N}\cup\{0\}$).
 It follows that
 \begin{eqnarray}\label{e:Add14}
 \exp(X+Y)=\sum_{k=0}\frac{1}{(2k)!}(a+b)^kI_N+\sum^{\infty}_{k=0}\frac{1}{(2k+1)!}(a+b)^{k}(X+Y).
 \end{eqnarray}
 %

It is easy to show that
$$ J_1\left( \begin{array} { c c c c }
1 & 0  \\
0 &1
 \end{array} \right)J_1\left(\begin{array} { c c c c }
0 & -1  \\
-1 &0
 \end{array} \right)=-\left(\begin{array} { c c c c }
0 & 1  \\
1 &0
 \end{array} \right)=-J_1\left(\begin{array} { c c c c }
0 & -1  \\
-1 &0
 \end{array} \right)J_1\left( \begin{array} { c c c c }
1 & 0  \\
0 &1
 \end{array} \right),
$$
 and that
 $$
 \left(2e_jtJ_1\left( \begin{array} { c c c c }
1 & 0  \\
0 &1
 \end{array} \right)\right)^2=-4e_j^2t^2I_2\quad\text{and}\quad
 \left(e_jtJ_1\left(\begin{array} { c c c c }
0 & -1  \\
-1 &0
 \end{array} \right)\right)^2=e_j^2t^2I_2.
 $$
Substituting  these into (\ref{e:Add14}), we obtain
\begin{align}\label{e:Add15}
\hat\Upsilon^j(t)&=\sum_{k=0}\frac{1}{(2k)!}(-3e_j^2t^2)^kI_2\nonumber\\+ &\sum^{\infty}_{k=0}\frac{1}{(2k+1)!}(-3e_j^2t^2)^{k}
 \left(2e_jtJ_1\left( \begin{array}{cc}
1 & 0  \\
0 &1
 \end{array} \right)
  +e_jtJ_1\left(\begin{array}{cc}
0 & -1  \\
-1 &0
 \end{array} \right)\right) \nonumber\\
 &=\cos(\sqrt{3}e_jt)I_2+\frac{\sin(\sqrt{3}e_jt)}{\sqrt{3}}J_1\left(\begin{array} { c c }
2 & -1  \\
-1 &2
 \end{array} \right)
 \end{align}


As before, since  $J_{n}=(J_1)^{\diamond n}$ and
\begin{eqnarray*}\label{e:Add15+}
\left( \begin{array} { c c c c }
2{\rm diag}(e_1,\cdots, e_n) & -{\rm diag}(e_1,\cdots, e_n)  \\
-{\rm diag}(e_1,\cdots, e_n) &2{\rm diag}(e_1,\cdots, e_n)
 \end{array} \right)
=\left( \begin{array} { c c c c }
2e_1 & -e_1  \\
-e_1 &2e_1
 \end{array} \right)\diamond\cdots\diamond
 \left( \begin{array} { c c c c }
2e_n & -e_n  \\
-e_n &2e_n
 \end{array} \right),
 \end{eqnarray*}
we derive from these and  Lemma~\ref{lem:diamond-product}(i) that
\begin{eqnarray*}
J_n\left( \begin{array} { c c c c }
2{\rm diag}(e_1,\cdots, e_n) & -{\rm diag}(e_1,\cdots, e_n)  \\
-{\rm diag}(e_1,\cdots, e_n) &2{\rm diag}(e_1,\cdots, e_n)
 \end{array} \right)
=J_1\left( \begin{array} { c c c c }
2e_1 & -e_1  \\
-e_1 &2e_1
 \end{array} \right)\diamond\cdots\diamond
 J_1\left( \begin{array} { c c c c }
2e_n & -e_n  \\
-e_n &2e_n
 \end{array} \right),
 \end{eqnarray*}
 and therefore, by Lemma~\ref{lem:diamond-product}(v),
 \begin{equation}\label{e:Add16}
\hat\Upsilon(t)=\hat\Upsilon^1(t)\diamond\cdots\diamond\hat\Upsilon^n(t).
\end{equation}
 Equation (\ref{e:Add16}) and Lemma~\ref{lem:diamond-product}(iv) imply
 \begin{equation}\label{e:Add17}
 \hat\Upsilon(t)-M_{3,n}=(\hat\Upsilon^1(t)-M_{3,1})\diamond\cdots\diamond(\hat\Upsilon^n(t)-M_{3,1})
 \end{equation}
 because $M_{3,n}=(M_{3,1})^{\diamond n}$ (the $n$-fold $\diamond$ product of $M_{3,1}$, defined by (\ref{e:diamond-product})),
 where $M_{3,1}=\left( \begin{array} {cc}
0 & 1\\
 -1& 1\end{array} \right)$.


%
%
%


\textbf{(I)[{Proof of} (\ref{e:Add2})]}.\quad
Since $J_1\left(\begin{array} { c c c c }
2 & -1  \\
-1 &2
 \end{array} \right)=\left(\begin{array} { c c c c }
1 & -2  \\
2 &-1
 \end{array} \right)$, by (\ref{e:Add15}), we have
  $$
  \begin{aligned}
\hat\Upsilon^j(t)-M_{3,1}&=\cos(\sqrt{3}e_jt)I_2+\frac{\sin(\sqrt{3}e_jt)}{\sqrt{3}}J_1
\left(\begin{array} { c c }
2 & -1  \\
-1 &2
 \end{array} \right)-M_{3,1}\nonumber\\
 &=\left(\begin{array} { c c }
\cos(\sqrt{3}e_jt)+ \frac{1}{\sqrt{3}}\sin(\sqrt{3}e_jt)& -1-\frac{2}{\sqrt{3}}\sin(\sqrt{3}e_jt)  \\
1+\frac{2}{\sqrt{3}}\sin(\sqrt{3}e_jt) &-1-\frac{1}{\sqrt{3}}\sin(\sqrt{3}e_jt)+\cos(\sqrt{3}e_jt).
 \end{array} \right).
 \end{aligned}
 $$
It is straightforward to compute the  determinant
\begin{align}\label{e:Add19}
\det\bigl(\hat\Upsilon^j(t)-M_{3,1}\bigr)
&= -\cos(\sqrt{3}e_j t) + \sqrt{3}\sin(\sqrt{3}e_j t) + 2 \nonumber \\
&= 2\sin\Bigl(\sqrt{3}e_j t-\frac{\pi}{6}\Bigr) + 2.
\end{align}
 Hence we obtain
    \begin{eqnarray}\label{e:Add20}
\det(\hat\Upsilon^j(t)-M_{3,1})=0\;\Longleftrightarrow\;
\sqrt{3}e_jt\in \frac{2\pi}{3}+ (2\mathbb{Z}+1)\pi.
 \end{eqnarray}
 For such $t$,  the difference $\hat\Upsilon^j(t)-M_{3,1}$ is the zero matrix:
   \begin{eqnarray}\label{e:Add21}
\hat\Upsilon^j(t)-M_{3,1}=\left( \begin{array} { c c c c c }
0 & 0\\
0 & 0\end{array} \right),
  \end{eqnarray}
 which implies
    \begin{eqnarray}\label{e:Add22}
{\rm Ker}(\hat\Upsilon^j(t)-M_{3,1})=\mathbb{R}^2.
  \end{eqnarray}
From (\ref{e:Add17}) and the second equality in (\ref{e:Add12}), we derive
\begin{align}\label{e:Add23}
\nu_{\tau, M_{3,n}}(\Upsilon)
&= \nu_{\tau, M_{3,n}}(\hat\Upsilon) \nonumber \\
&= \sum^n_{j=1} \nu_{\tau, M_{3,1}}(\hat\Upsilon^j) \nonumber \\
&= 2\,\sharp\left\{j\;\Big|\;\sqrt{3}e_j\tau \in \frac{2\pi}{3} + (2\mathbb{Z}+1)\pi \right\}.
\end{align}
We have thus verified (\ref{e:Add2}).

\textbf{(II)[{Proof of} (\ref{e:Add3})}.\quad
 For $0\le t\le\tau$, we define
$$
\eta(t)=R\left(\frac{\pi t}{2\tau}\right)\quad\text{and}\quad
\zeta(t)=\left( \begin{array} { c c  }
\frac{t}{\tau} & -1 \\ 1 & 0 \end{array} \right).
$$
Then $\eta(0)=I_2$, $\eta(\tau)=J_1=\zeta(0)$,  $\zeta(\tau)=M_{3,1}^{-1}$, and
$\xi=\zeta\ast\eta\in\mathcal{P}_\tau(2)$ satisfies $\xi(\tau)=\zeta(\tau)=M_{3,1}^{-1}$.
Since $\xi^{\diamond n}\in\mathcal{P}_\tau(2n)$ satisfies $\xi^{\diamond n}(\tau)=\zeta^{\diamond n}(\tau)=M_{3,n}^{-1}$, by (\ref{e:dongIndex}) we have
\begin{align}\label{e:Add26}
i_{\tau,M_{3,n}}(\hat\Upsilon)&=\left[i_\tau((\hat\Upsilon (M_{3,n})^{-1})\ast\xi^{\diamond n})-\frac{\Delta({\xi^{\diamond n}})}{\pi}\right],\\
i_{\tau,M_{3,1}}(\Upsilon^j)&=\left[i_\tau((\hat\Upsilon^j (M_{3,1})^{-1})\ast\xi)-\frac{\Delta({\xi})}{\pi}\right],\quad j=1,\dots,n. \label{e:Add27}
\end{align}
Observe that (\ref{e:Add16}) and (\ref{e:sympAdd}) imply
\begin{align}\label{e:Add28}
i_\tau((\hat\Upsilon (M_{3,n})^{-1})\ast\xi^{\diamond n})&=\sum^n_{j=1}i_\tau((\hat\Upsilon^j (M_{3,1})^{-1})\ast\xi)\nonumber\\
&{=}\sum^n_{j=1}i_{\tau,M_{3,1}}(\hat\Upsilon^j)-n\left[-\frac{\Delta({\xi})}{\pi}\right]
\end{align}
by (\ref{e:Add27}). It follows from this and (\ref{e:Add26}) that
\begin{eqnarray}\label{e:Add29}
i_{\tau,M_{3,n}}(\hat\Upsilon)=
\sum^n_{j=1}i_{\tau,M_{3,1}}(\hat\Upsilon^j)+ \left[-n\left[-\frac{\Delta({\xi})}{\pi}\right]-\frac{\Delta({\xi^{\diamond n}})}{\pi}\right].
\end{eqnarray}

Let us compute $\Delta({\xi})/\pi$ and $\Delta({\xi^{\diamond n}})/\pi$. Notice that
\begin{eqnarray}\label{e:Add30}
\Delta({\xi})=\Delta({\eta})+\Delta({\zeta})\quad\text{and}\quad
\Delta({\xi}^{\diamond n})=\Delta({\eta}^{\diamond n})+\Delta({\zeta}^{\diamond n}).
\end{eqnarray}
It is easy to see that
\begin{eqnarray}\label{e:Add31}
\Delta({\eta})=\frac{\pi}{2}\quad\text{and}\quad
\Delta({\eta}^{\diamond n})=\frac{n\pi}{2}.
\end{eqnarray}
To compute $\Delta({\zeta})$ and $\Delta({\zeta}^{\diamond n})$,
note that $\zeta(t)=\beta_1(\tau-t)$ for $t\in [0,\tau]$. Hence,
\begin{equation}\label{e:Add33}
\Delta(\zeta)=\arctan 2-\frac{\pi}{2}\quad\text{and}\quad \Delta(\zeta^{\diamond n})=n\arctan 2-\frac{n\pi}{2}.
\end{equation}
Combing this with (\ref{e:Add30}), (\ref{e:Add31}), and  (\ref{e:Add33}),       we obtain
\begin{equation}\label{e:Add34}
\Delta({\xi})=\arctan 2\quad\text{and}\quad
\Delta({\xi}^{\diamond n})=n\arctan 2.
\end{equation}

By substituting these into (\ref{e:Add29}), we derive the desired equation
\begin{align}\label{e:Add35}
i_{\tau,M_{3,n}}(\hat\Upsilon)&=
\sum^n_{j=1}i_{\tau,M_{3,1}}(\hat\Upsilon^j)+ \left[-n\left[-\frac{\arctan 2}{\pi}\right]-\frac{
n\arctan 2}{\pi}\right]\nonumber\\
&=\sum^n_{j=1}i_{\tau,M_{3,1}}(\hat\Upsilon^j)+ \left[n-\frac{
n\arctan 2}{\pi}\right]
\end{align}
since $\arctan 2\approx 1.107$ radians implies $[-\arctan 2/\pi]=-1$.
\begin{itemize}
\item[$\bullet$] If $e_j>0$, i.e., $\left( \begin{array} { c c c c }
2e_j & -e_j  \\
-e_j &2e_j
 \end{array} \right)$ is positive  definite, by (\ref{e:Add4}) and (\ref{e:Add23+5}) we have
\begin{align}\label{e:Add37}
i_{\tau, M_{3,1}}(\hat\Upsilon^j)
&=i_{\tau, M_{3,1}}(\xi_{2,\tau})+2\sharp\left\{0<t<\tau\,\Big|\, \sqrt{3}e_jt\in \frac{2\pi}{3}+ (2\mathbb{Z}+1)\pi\right\}\nonumber\\
&=2\sharp\left\{0<t<\tau\,\Big|\, \sqrt{3}e_jt\in \frac{2\pi}{3}+ (2\mathbb{Z}+1)\pi\right\}.
\end{align}
\item[$\bullet$] If $e_j<0$, i.e., $\left( \begin{array} { c c c c }
2e_j & -e_j  \\
-e_j &2e_j
 \end{array} \right)$ is negative definite,  then  (\ref{e:Add5}) and (\ref{e:Add23+5}) give
\begin{align}\label{e:Add38}
i_{\tau, M_{3,1}}(\hat\Upsilon^j)
&=i_{\tau, M_{3,1}}(\xi_{2,\tau})
-2\sharp\left\{0<t\le\tau\,\Big|\, \sqrt{3}e_jt\in \frac{2\pi}{3}+ (2\mathbb{Z}+1)\pi\right\}\nonumber\\
&=-2\sharp\left\{0<t\le\tau\,\Big|\, \sqrt{3}e_jt\in \frac{2\pi}{3}+ (2\mathbb{Z}+1)\pi\right\}.
\end{align}
\item[$\bullet$] If $e_j=0$, then $\hat\Upsilon^j(t)=I_2=\xi_{2,\tau}(t)$ for $0\le t\le\tau$, and therefore
 $i_{\tau, M_{3,1}}(\hat\Upsilon^j)=0$ by (\ref{e:Add23+5}).
 \end{itemize}
Combining this with (\ref{e:Add37}) and (\ref{e:Add38}),
the desired formula (\ref{e:Add3}) follows from (\ref{e:Add35}) and (\ref{e:Add12}).

\textbf{(III)[{Proofs of} (\ref{e:Add4}) and (\ref{e:Add5})]}.\quad
\textsf{Step~1}.\quad Let us prove
\begin{eqnarray}\label{e:Add23+7}
i_{\tau,M_{3,n}}(\xi_{2n,\tau})=\left[n-\frac{n\arctan 2}{\pi}\right].
\end{eqnarray}
 For $0\le t\le\tau$, we define
$$
\beta_1(t)=\left( \begin{array} { c c  }
\frac{\tau-t}{\tau} & -1 \\ 1 & 0 \end{array} \right),\quad
\beta_2(t)=R\left(\frac{\pi}{2}\left(1+\frac{t}{\tau}\right)\right),\quad
\beta_3(t)=\left( \begin{array} { c c  }
-\left(1+\frac{t}{\tau}\right) & 0 \\ 0 &
-\left(1+\frac{t}{\tau}\right) \end{array} \right).
$$
Then $\beta_j(t)\in{\rm Sp}(2)^-$ for $0\le t\le\tau$ and $j=1,2,3$; moreover,
$$
  \begin{aligned}
&\beta_1(0)=M_{3,1}^{-1},\quad
\beta_1(\tau)=J_1,\quad \beta_2(0)=R(\pi)=\left( \begin{array} { c c  }
-1 & 0 \\
0&-1 \end{array} \right),\\
&\beta_3(0)=R(\tau), \quad \beta_3(\tau)=D(-2)=M_1^{-}.
 \end{aligned}
  $$
Define $\beta=\beta_3\ast\beta_2\ast\beta_1$. Then $\beta(0)=M_{3,1}^{-1}$, $\beta(\tau)=M_1^{-}$,
and $\beta(t)\in{\rm Sp}(2)^-$ for $0\le t\le\tau$. By Lemma~\ref{lem:Dong06}(i),
\begin{align}
i_{\tau, M_{3,1}}(\xi_{2,\tau})
&= \left[ \frac{\Delta(\beta)}{\pi} +
\frac{\Delta(\xi_{2,\tau} M_{3,1}^{-1})}{\pi} \right] \nonumber \\
&= \left[ \frac{\Delta(\beta)}{\pi} \right] \nonumber \\
&= \left[ \frac{\Delta(\beta_1)}{\pi} + \frac{\Delta(\beta_2)}{\pi} + \frac{\Delta(\beta_3)}{\pi} \right],
\label{e:Add23+1}
\end{align}
because $\xi_{2,\tau} M^{-1}_{3,1}(t)\equiv M^{-1}_{3,1}$ for all $t\in [0,\tau]$.

Note that $\mathfrak{u}(\beta_2(t))=\exp\left(\sqrt{-1}\frac{\pi}{2}\left(1+\frac{t}{\tau}\right)\right)$. We obtain
\begin{eqnarray}\label{e:Add23+2}
\Delta(\beta_2)=\frac{\pi}{2}.
\end{eqnarray}
To compute $\Delta({\beta_1})$, a straightforward calculation yields
$$
  \begin{aligned}
\beta_1(t)\beta_1(t)^\top&=\frac{\tau}{\sqrt{(\tau-t)^2+4\tau^2}}\left( \begin{array} { c c  }
2&-\frac{\tau-t}{\tau}\\
-\frac{\tau-t}{\tau} & \frac{(\tau-t)^2}{\tau^2}+2 \end{array} \right),\\
\left(\sqrt{\beta_1(t)\beta_1(t)^\top}\right)^{-1}\beta_1(t)&=\frac{\tau}{\sqrt{(\tau-t)^2+4\tau^2}}\left( \begin{array} { c c  }
\frac{\tau-t}{\tau}&-2\\
2&\frac{\tau-t}{\tau} \end{array} \right).
  \end{aligned}
  $$
It follows that
\begin{eqnarray}\label{e:Add23+3}
\mathfrak{u}(\beta_1(t))=\frac{\tau-t}{\sqrt{(\tau-t)^2+4\tau^2}}+\frac{2\tau\sqrt{-1}}{(\tau-t)^2+4\tau^2}
=\exp\left(\sqrt{-1}\Delta_{\beta_1}(t)\right),
\end{eqnarray}
 where the continuous function $\Delta_{\beta_1}:[0, \tau]\to\mathbb{R}$ is defined by
 $$
   \Delta_{\beta_1}(t)= \begin{cases}
 \arctan\left(\frac{2\tau}{\tau-t}\right)\quad&\hbox{if $t\in[0, \tau)$},\\
\lim_{t\to \tau-}\arctan\left(\frac{2\tau}{\tau-t}\right)\quad&\hbox{if $t=\tau$}.
\end{cases}
$$
  Hence,
\begin{eqnarray}\label{e:Add23+4}
\Delta(\beta_1)=\frac{\pi}{2}-\arctan 2.
\end{eqnarray}

It is easy to see that  $\Delta(\beta_3)=0$ because
$$
\left(\sqrt{\beta_3(t)\beta_3(t)^\top}\right)^{-1}\beta_3(t)=\left( \begin{array} { c c  }
-1&0\\
0&-1\end{array} \right)=R(\pi).
$$
 Noting that $\arctan 2\approx 1.107$ radians and combining this with (\ref{e:Add23+2}), (\ref{e:Add23+4}),
we derive from  (\ref{e:Add23+1}) that
\begin{equation}\label{e:Add23+5}
i_{\tau,M_{3,1}}(\xi_{2,\tau})= \left[1-\frac{\arctan 2}{\pi}\right]=0.
\end{equation}

Since $M_{3,n}=M_{3,1}^{\diamond n}$, $\beta^{\diamond n}=\beta_3^{\diamond n}\ast\beta_2^{\diamond n}\ast\beta_1^{\diamond n}$,
 $\beta^{\diamond n}(\tau)=(M_1^{-1})^{\diamond n}=M_n^-$,
and $\beta(t)\in{\rm Sp}(2n)^-$ for $0\le t\le\tau$,
 we have  $\Delta(\xi_{2n,\tau} M^{-1}_{3,n})=0$ (as in the proof of (\ref{e:Add23+1})) and obtain
\begin{align}\label{e:Add23+6}
i_{\tau,M_{3,n}}(\xi_{2n,\tau})&= \left[ \frac{\Delta(\beta^{\diamond n})}{\pi} + \frac{\Delta(\xi_{2n,\tau} M^{-1}_{3,n})}{\pi} \right]\nonumber\\
&=\left[ \frac{\Delta(\beta_1^{\diamond n})}{\pi} + \frac{\Delta(\beta_2^{\diamond n})}{\pi}+\frac{\Delta(\beta_3^{\diamond n})}{\pi} \right].
\end{align}
For $j=1,2,3$, Lemma~\ref{lem:diamond-product}(iii) implies
$$
\mathfrak{u}(\beta_j^{\diamond n}(t))=\operatorname{Diag}(\underbrace{\mathfrak{u}(\beta_j(t)), \dots, \mathfrak{u}(\beta_j(t))}_{n \text{ times}})
$$
and thus $\Delta(\beta_j^{\diamond n})=n\Delta(\beta_j)$.
Combing this with  (\ref{e:Add23+2}) and (\ref{e:Add23+4}), equation \eqref{e:Add23+6} yields (\ref{e:Add23+7}).

\textsf{Step~2}(completing proofs of (\ref{e:Add4}) and (\ref{e:Add5})).\quad
From (\ref{e:Add10-}) and
$$
Q^{-1}\left( \begin{array} { c c c c }
2B & -B  \\
-B &2B
 \end{array} \right)Q=\left( \begin{array} { c c c c }
2{\rm diag}(e_1,\cdots, e_n) & -{\rm diag}(e_1,\cdots, e_n)  \\
-{\rm diag}(e_1,\cdots, e_n) &2{\rm diag}(e_1,\cdots, e_n)
 \end{array} \right),
 $$
we derive
$$
\begin{aligned}
&\text{$B$ is positive (resp. negative) definite}\\
&\Longleftrightarrow\;e_j>0\;\hbox{for $j=1,\cdots,n$}\;
\hbox{(resp. $e_j<0$\;\text{for $j=1,\cdots,n$})}\\
&\Longleftrightarrow\;
\left( \begin{array} { c c c c }
2e_j & -e_j  \\
-e_j &2e_j
 \end{array} \right) \text{($j=1,\cdots,n$) are positive (resp. negative) definite}\\
 &\Longleftrightarrow\;Q^{-1}\left( \begin{array} { c c c c }
2B & -B  \\
-B &2B
 \end{array} \right)Q\;\text{is positive (resp. negative) definite}\\
  &\Longleftrightarrow\;\left( \begin{array} { c c c c }
2B & -B  \\
-B &2B
 \end{array} \right)\;\text{is positive (resp. negative) definite}.
\end{aligned}
$$
Hence, by Lemma~\ref{lem:Dong4.2}, the first equality in (\ref{e:Add12}),
and (\ref{e:Add23+7}), we have:
\begin{equation}\label{e:Add24}
i_{\tau, M_{3,n}}(\Upsilon)=\left[n-\frac{n\arctan 2}{\pi}\right]
+\sum_{0\le t<\tau}\nu_{t,M_{3,n}}\bigl(\hat\Upsilon|_{[0,t]}\bigr)
\end{equation}
 if $B$ is positive definite,
 and
\begin{equation}\label{e:Add25}
i_{\tau,M_{3,n}}(\Upsilon)=\left[n-\frac{n\arctan 2}{\pi}\right]-\sum_{0<t\le\tau}\nu_{t,M_{3,n}}\bigl(\hat\Upsilon|_{[0,t]}\bigr)
\end{equation}
if $B$ is negative definite.
 By (\ref{e:Add19}), $\det(\hat\Upsilon^j(0)-M_{3,1})=-1+2=1$, and thus
 \begin{eqnarray*}
\nu_{0,M_{3,n}}\bigl(\hat\Upsilon|_{[0,0]}\bigr)=\sum^n_{j=1}\nu_{0,M_{3,1}}\bigl(\hat\Upsilon^j|_{[0,0]}\bigr)=0.
 \end{eqnarray*}
For $t>0$, combining this with (\ref{e:Add23}), we obtain
\begin{eqnarray*}
\nu_{t,M_{3,n}}\bigl(\hat\Upsilon|_{[0,t]}\bigr)= 2\sharp\left\{j\,\Big|\, \sqrt{3}e_jt\in \frac{2\pi}{3}+ (2\mathbb{Z}+1)\pi\right\}.
 \end{eqnarray*}
Therefore, equations (\ref{e:Add4}) and (\ref{e:Add5}) follow from  (\ref{e:Add24}) and (\ref{e:Add25}), respectively.






\subsection{Proof of Theorem~\ref{th:PIndex4}}\label{sec:PIndex4}

\begin{lemma}\label{lem:PIndex3}
Let real numbers $a$ and $b$ satisfy $a^2+b^2=1$, and define the matrix $\Omega$ as:
$$
\Omega:=\left( \begin{array} { c c  } A & -B-J_1 \\ B+J_1 & A \end{array} \right),
$$
where  $A={\rm diag}(a,a)$,  $B={\rm diag}(b,b)$, and $J_1$ is defined by (\ref{e:standcompl}).
Then the kernel of $\Omega$ is given by
$$
  {\rm Ker}(\Omega)=
    \begin{cases}
0\quad&\hbox{if $a=0$},\\
0\quad&\hbox{if  $b\ne 0$},\\
\{(U^\top, V^\top)^\top\,|\, U, V\in\mathbb{R}^2,\;U=J_1AV\}\quad&\hbox{if $b=0$}.
\end{cases}
$$
\end{lemma}
\begin{proof}[\bf Proof]
Let $U, V\in\mathbb{R}^2$ be such that $(U^\top, V^\top)^\top\in{\rm Ker}(\Omega)$. Then
$$
AU=(B+J_1)V\quad\text{and}\quad
AV=-(B+J_1)U
$$

If $a=0$, then $A=0$, and the equations becomes $J_1V=-bV$ and $J_1U=bU$. Since all eigenvalues of $J_1$ are $\pm\sqrt{-1}$ (non-real),  so the only real solution is $U=V=0$.

If $b\ne 0$, because $AJ_1=J_1A$ and $AB=BA$, by Lemma~3 in \cite[page~38]{Long02},
$$
\det\Omega=\det(A^2+(B+J_1)^2)=4b^2\ne 0.
$$
Thus  $\Omega$ is invertible, which implies $U=V=0$.

If $b=0$, then $a=\pm 1$ and $B=0$. The equations become:
 $AU=J_1V$ and $J_1U=-AV$.
 Multiplying the first by $A$ on the left and using $A^{-1}=A$ (since $A^2=I_2$) gives
$U=AJ_1V$. Since $AJ_1=J_1A$, we finally obtain $U=J_1AV$.
\end{proof}

%
%

\begin{proof}[\bf Proof of Theorem~\ref{th:PIndex4}]
Let $E=\operatorname{diag}(e_1,\ldots,  e_n)$.
Denote by $\hat\Upsilon$  the fundamental matrix solution of
$\dot{z}(t)=J_n{\rm diag}(E, E)z(t)$
with $\hat\Upsilon(0)=I_{2n}$; that is,
$\hat\Upsilon(t)=\exp(J_n{\rm diag}(E,E)t)$.
Since $J_{2n} = J_n \diamond J_n$ and
$\text{diag}(E, E, E, E) = \text{diag}(E, E) \diamond \text{diag}(E, E)$,
Lemma~\ref{lem:diamond-product}(i) implies:
$$
J_{2n} \text{diag}(E, E, E, E) = \left( J_n \text{diag}(E, E) \right) \diamond \left( J_n \text{diag}(E, E) \right).
$$
Consequently, $\hat\Upsilon(t) \diamond \hat\Upsilon(t)$ is
the fundamental matrix solution of
$$
\dot{Z}(t)=J_{2n} \text{diag}(E, E, E, E)Z(t)
$$
with $\hat\Upsilon(0) \diamond \hat\Upsilon(0)=I_{4n}$.

Let $Q$ be a  real orthogonal and symplectic matrix as in
(\ref{e:DiagDelayA.1-}). Then $\Xi:=\operatorname{diag}(Q, Q)$
is also a real orthogonal and symplectic matrix and satisfies
\begin{equation*}
	\Xi^{-1}\left( \begin{array} {cccc}
B& -C &0&0\\
C&B&0&0\\
0&0&B&-C\\
0&0&C&B
 \end{array} \right)\Xi
		= \text{diag}(E, E, E, E),
\end{equation*}
and therefore
\begin{equation*}\label{e:DiagDelay}
	\Xi^{-1}J_{2n}\left( \begin{array} {cccc}
B& -C &0&0\\
C&B&0&0\\
0&0&B&-C\\
0&0&C&B
 \end{array} \right)\Xi
		= J_{2n}\text{diag}(E, E, E, E).
\end{equation*}
It follows from this, (v) and (i) in  Lemma~\ref{lem:diamond-product} that
\begin{equation}\label{e:DiagDelay1+}
\Upsilon(t) = \Xi\left( \hat\Upsilon(t) \diamond \hat\Upsilon(t) \right) \Xi^{-1}.
\end{equation}

\textbf{(I)[Proof of (\ref{e:DiagPindex1*})]}.\quad
Since   $QJ_n=J_nQ$, it is straightforward to verify that
  \begin{equation}\label{e:DiagDelay1++}
M_n = \Xi^{-1} M_n \Xi.
\end{equation}
Combining this identity  with equation (\ref{e:DiagDelay1+}) yields
 \begin{align}\label{e:DiagDelay1+2}
 \Upsilon(t)&=\Xi(\hat\Upsilon(t)\diamond\hat\Upsilon(t))
 \Xi^{-1},\\
 \Upsilon(t)-M_n&=\Xi(\hat\Upsilon(t)\diamond\hat\Upsilon(t)-M_n)
 \Xi^{-1},\label{e:DiagDelay1+3}\\
 \Upsilon(t)-I_{4n}&=\Xi
 (\hat\Upsilon(t)\diamond\hat\Upsilon(t)-I_{4n})\Xi^{-1}.\label{e:DiagDelay1+4}
 \end{align}

 For a square matrix $X$ of order $m$, recall the following definitions from matrix analysis:
 $$
 \begin{aligned}
\cos X&=\frac{1}{2}\left(e^{iX}+ e^{-iX}\right)=\sum^\infty_{l=0}(-1)^l\frac{X^{2l}}{(2l)!},\\
 \sin X&=\frac{1}{2i}\left(e^{iX}- e^{-iX}\right)=\sum^\infty_{l=0}(-1)^l\frac{X^{2l+1}}{(2l+1)!},
 \end{aligned}
 $$
  where $i = \sqrt{-1}$.  Then $\cos^2X+\sin^2 X=I_m$, and for any
   commuting matrices $X, Y\in\mathbb{C}^{m\times m}$ (i.e., $XY=YX$), we have
  $e^Xe^Y=e^{X+Y}$.
  Using the above definitions and relations, it is straightforward to verify that
   $$
 \hat{\Upsilon}(t)=\exp\left(tJ_n\left( \begin{array} { c c c c }
E & 0  \\
0 &E
 \end{array} \right)\right)=
\left( \begin{array} { c c }
\cos(Et) & -\sin(Et)  \\
\sin(Et) & \cos(Et)
 \end{array} \right).
$$
Therefore, we obtain
$$
\hat\Upsilon(t)\diamond\hat\Upsilon(t)=\left( \begin{array} {cccc}
\cos(Et)& 0 &-\sin(Et)&0\\
0&\cos(Et)&0&-\sin(Et)\\
\sin(Et)&0&\cos(Et)&0\\
0&\sin(Et)&0&\cos(Et)
 \end{array} \right).
$$
Note that
$$
\hat\Upsilon(t)\diamond\hat\Upsilon(t)-M_n=\left( \begin{array} { c c }
\operatorname{diag}(\cos(Et), \cos(Et)) & -\operatorname{diag}(\sin(Et), \sin(Et))-J_n  \\
\operatorname{diag}(\sin(Et), \sin(Et))+J_n & \operatorname{diag}(\cos(Et), \cos(Et))
 \end{array} \right)
$$
and that matrices ${\rm diag}(\cos(Et), \cos(Et))$ and ${\rm diag}(\sin(Et), \sin(Et))+J_n$ commute.
By Lemma~3 in \cite[page~38]{Long02} and the identity  $(\sin(Et))^2+(\cos(Et))^2=I_n$, we  obtain
 \begin{align}\label{e:Delay6Auto6Xi2}
 \det(\Upsilon(t)-M_n)&=\det(\hat\Upsilon(t)\diamond\hat\Upsilon(t)-M_n)
 \nonumber\\
 &=\det\left[({\rm diag}(\cos(Et), \cos(Et)))^2+
 ({\rm diag}(\sin(Et), \sin(Et))+J_n)^2\right]\nonumber\\
 &=\det(2J_n{\rm diag}(\sin(Et), \sin(Et)))\nonumber\\
 &=4^n\prod^n_{l=1}\sin^2(e_lt).
 \end{align}
   Hence (\ref{e:Delay6Auto6Xi2}) implies
\begin{equation}\label{e:Delay6Auto6Xi3*}
\det(\Upsilon(t)-M_n)=0\;\Longleftrightarrow\; e_lt\in \pi\mathbb{Z}\;\hbox{for some $l\in\{1,\ldots,n\}$}.
\end{equation}

To compute ${\rm Ker}(\hat\Upsilon(t)\diamond\hat\Upsilon(t)-M_n)$, we define for $j=1,\ldots,n$,
\begin{equation}\label{e:Delay6Auto6Xi4}
A_j(t)=\operatorname{diag}(\cos(e_jt), \cos(e_jt))\quad\text{and}\quad
B_j(t)=\operatorname{diag}(\sin(e_jt), \sin(e_jt)).
\end{equation}
Then $\operatorname{diag}(\cos(Et), \cos(Et))=A_1(t)\diamond\cdots\diamond A_n(t)$, and
by Lemma~\ref{lem:diamond-product} together with $J_n=(J_1)^{n\diamond}$, we have
\begin{eqnarray*}
\operatorname{diag}(\sin(Et), \sin(Et))+J_n=B_1(t)\diamond\cdots\diamond B_n(t)+J_n
=(B_1(t)+J_1)\diamond\cdots\diamond (B_n(t)+J_1).
\end{eqnarray*}
It follows that a vector $(u_1,\ldots,u_{2n}, v_1,\ldots,v_{2n})^\top$ belongs to ${\rm Ker}(\hat\Upsilon(t)\diamond\hat\Upsilon(t)-M_n)$ if
and only if it satisfies
\begin{eqnarray*}
(A_1(t)\diamond\cdots\diamond A_n(t))(u_1,\ldots,u_{2n})^\top-
((B_1(t)+J_1)\diamond\cdots\diamond (B_n(t)+J_1))(v_1,\ldots,v_{2n})^\top=0,\\
((B_1(t)+J_1)\diamond\cdots\diamond (B_n(t)+J_1))(u_1,\ldots,u_{2n})^\top+
(A_1(t)\diamond\cdots\diamond A_n(t))(v_1,\ldots,v_{2n})^\top=0.
\end{eqnarray*}
By Lemma~\ref{lem:diamond-product}(vi),  the latter two conditions are equivalent respectively to
$$
\begin{aligned}
&(\cos(e_1t)u_1,\ldots, \cos(e_nt)u_n, \cos(e_1t)u_{n+1},\ldots, \cos(e_nt)u_{2n})^\top\\
&-(\sin(e_1t)v_1-v_{n+1},\ldots, \sin(e_nt)v_n-v_{2n}, \sin(e_1t)v_{n+1}+v_1,\ldots, \sin(e_nt)v_{2n}+v_n)^\top                                                                          =0,\\
&(\sin(e_1t)u_1-u_{n+1},\ldots, \sin(e_nt)u_n-u_{2n}, \sin(e_1t)u_{n+1}+u_1,\ldots, \sin(e_nt)u_{2n}+u_n)^\top \\                                                                         &+  (\cos(e_1t)v_1,\ldots, \cos(e_nt)v_n, \cos(e_1t)v_{n+1},\ldots, \cos(e_nt)v_{2n})^\top =0.
\end{aligned}
$$
One can readily verify that these two conditions are equivalent to the system
$$
\begin{pmatrix}
A_j(t) & -B_j(t) - J_1 \\
B_j(t) + J_1 & A_j(t)
\end{pmatrix}
\begin{pmatrix} u_j \\ u_{n+j} \\ v_j \\ v_{n+j} \end{pmatrix} = 0, \quad j = 1, \dots, n.
$$
Combining this observation with Lemma~\ref{lem:PIndex3}, we obtain the following claim.

\begin{claim}\label{cl:5.2}
Assume that $t > 0$ satisfies $\det(\Upsilon(t)-M_n)=0$. Then there exists a partition of
 $\{1, \dots, n\}$ into two subsets $\{i_1, \dots, i_k\}$ and $\{i_{k+1}, \dots, i_n\}$
  (where $1 \le k \le n$) such that:
  \begin{description}
 \item[(i)] $\sin(e_{i_s}t)=0$ for $s=1,\ldots,k$, $\sin(e_{i_s}t)\ne 0$ for $s=k+1,\cdots,n$.
 \item[(ii)] The kernel $\ker\left( \hat\Upsilon(t) \diamond \hat\Upsilon(t) - M_n \right)$ consists of all vectors
$(u_1,\ldots,u_{2n}, v_1,\ldots,v_{2n})^\top\in\mathbb{R}^{4n}$ satisfying the following:
\begin{align*}
&u_j = u_{n+j} = v_j = v_{n+j} = 0, \qquad j = i_{k+1}, \dots, i_n,\\
&u_j = -\cos(e_j t) v_{n+j}\;\text{ and }\;u_{n+j} = \cos(e_j t) v_j, \qquad j = i_1, \dots, i_k.
\end{align*}
In particular, $\dim\ker(\hat\Upsilon(t)\diamond\hat\Upsilon(t)-M_n)=2k$.
\end{description}
\end{claim}
Combining this result with equation (\ref{e:DiagDelay1+}), we conclude that (\ref{e:DiagPindex1*}) holds.\\

\textbf{(II)[Proof of (\ref{e:DiagPindex2*}) for $n=1$]}.\quad
For any integer $n>0$, note that both $M_n$ and $\Xi$ are orthogonal symplectic matrices.
Using Lemma~\ref{lem:Dong06+} together with  (\ref{e:DiagDelay1+}) and (\ref{e:DiagDelay1++}),
we obtain
\begin{eqnarray}\label{e:RobbinSaCZ+RobbinSa7++}
i_{\tau, M_n}(\Upsilon)=i_{\tau, M_n}(\hat\Upsilon\diamond\hat\Upsilon).
\end{eqnarray}
Since $M_n^{-1}M_n=I_{4n}$, it follows from (\ref{e:DiagDelay1+3}) and (\ref{e:Delay6Auto6Xi2}) that
\begin{align*}
D\left((\hat\Upsilon(t)\diamond\hat\Upsilon(t))M_n^{-1}\right)&=(-1)^{2n-1}\det\left((\hat\Upsilon(t)\diamond\hat\Upsilon(t))M_n^{-1}-I_{4n}\right)
\\
&=-\det\left(\hat\Upsilon(t)\diamond\hat\Upsilon(t)-M_n\right)\det M_n^{-1}\\
&=-\det\left(\Upsilon(t)-M_n\right)\\
&=-4^n\prod^n_{l=1}\sin^2(e_lt).
\end{align*}
Consequently,
\begin{equation}\label{e:RobbinSaCZ+RobbinSa7++++}
(\hat\Upsilon(t)\diamond\hat\Upsilon(t))M_n^{-1}\in\operatorname{Sp}(4n)^+ \quad
\text{if } \sin(e_lt)\ne 0\text{ for }l=1,\ldots,n.
\end{equation}
This can also be seen from the explicit expression
\begin{align*}
(\hat\Upsilon(t)\diamond\hat\Upsilon(t))M_n^{-1}&=\left( \begin{array} { c c }
\operatorname{diag}(\cos(Et), \cos(Et)) & -\operatorname{diag}(\sin(Et), \sin(Et))  \\
\operatorname{diag}(\sin(Et), \sin(Et))& \operatorname{diag}(\cos(Et), \cos(Et))
 \end{array} \right)\left( \begin{array} { c c }
0 & J_n  \\
-J_n & 0
 \end{array} \right)\nonumber\\
 &=\left( \begin{array} { c c }
J_n\operatorname{diag}(\sin(Et), \sin(Et)) & J_n\operatorname{diag}(\cos(Et), \cos(Et))  \\
-J_n\operatorname{diag}(\cos(Et), \cos(Et))& J_n\operatorname{diag}(\sin(Et), \sin(Et))
 \end{array} \right).
\end{align*}
By the definition of the map $\mathfrak{u}$ in (\ref{e:LieIsom}) (with $i=\sqrt{-1}$), we have
  $$
\mathfrak{u}\left((\hat\Upsilon(t)\diamond\hat\Upsilon(t))M_n^{-1}\right)=
-J_ni\operatorname{diag}\left(e^{ie_1t}, \ldots, e^{ie_nt},
e^{ie_1t},\ldots, e^{ie_nt}\right)
 $$
 and therefore
  $$
  \det\mathfrak{u}\left((\hat\Upsilon(t)\diamond\hat\Upsilon(t))M_n^{-1}\right)=
\exp\left(\frac{n\pi}{2}i+ 2i(e_1+\cdots+e_n)t)\right).
$$
Hence
 \begin{equation}\label{e:RobbinSaCZ+RobbinSa7+++++}
 \Delta\left((\hat\Upsilon\diamond\hat\Upsilon)M_n^{-1}\right)=2(e_1+\cdots+e_n)\tau.
 \end{equation}
In view of (\ref{e:RobbinSaCZ+RobbinSa7++}), it suffices to prove
the following:

\begin{claim}\label{cl:5.3}
For each integer $m$, there holds
\begin{equation}\label{e:Delay6Auto7G++*}
 i_{\tau, M_1}(\hat\Upsilon\diamond\hat\Upsilon)=
 \begin{cases}
 2m+1\quad&\hbox{if $m< e_1\tau< (m+1)\pi$},\\
2m+1\quad&\hbox{if $e_1\tau=(m+1)\pi$},
\end{cases}
 \end{equation}
where
 \begin{equation}\label{e:Delay6Auto7G.0}
\hat\Upsilon(t)\diamond\hat\Upsilon(t)=\left( \begin{array} { c c }
\operatorname{diag}(\cos(e_1t), \cos(e_1t)) & -\operatorname{diag}(\sin(e_1t), \sin(e_1t))  \\
\operatorname{diag}(\sin(e_1t), \sin(e_1t)) & \operatorname{diag}(\cos(e_1t), \cos(e_1t))
 \end{array} \right).
\end{equation}
\end{claim}
\begin{proof}[\bf Proof]
We prove (\ref{e:Delay6Auto7G++*}) in three steps.

\textbf{Step 1 ({Proof of (\ref{e:Delay6Auto7G++*})
for $e_1\tau\in (m, m+1)$ with $m=2k$})}.\quad
Note that
$$
(\hat\Upsilon(\tau)\diamond\hat\Upsilon(\tau))M_1^{-1}=
\left( \begin{array} { c c c c }
0& -\sin(e_1\tau) & 0 & -\cos(e_1\tau)\\
\sin(e_1\tau)& 0 & \cos(e_1\tau) & 0 \\
0& \cos(e_1\tau)& 0 & -\sin(e_1\tau)\\
-\cos(e_1\tau) & 0 &\sin(e_1\tau) & 0
 \end{array} \right),
$$
which, by (\ref{e:RobbinSaCZ+RobbinSa7++++}), belongs to  ${\rm Sp}(4)^+$
if and only if $e_1\tau\notin \pi\mathbb{Z}$.
To compute $i_{\tau, M_1}(\hat\Upsilon\diamond\hat\Upsilon)$ with Lemma~\ref{lem:Dong06}(i)
and (\ref{e:RobbinSaCZ+RobbinSa7+++++}), we  construct a path $\eta\in C([0, \tau], \operatorname{Sp}(4)^+)$
 such that
 $$
\eta(0)=(\hat\Upsilon(\tau)\diamond\hat\Upsilon(\tau))M_1^{-1}\quad\text{and}\quad
\eta(\tau)=M_2^+=\operatorname{diag}(2, 2, 1/2, 1/2).
$$
For this purpose, define for $0\le t\le \tau$ the function $\chi(t):=1+\frac{t}{\tau}$ and
\begin{equation}\label{e:Delay6Auto7G.1}
\zeta_2(t):=\left( \begin{array} {cc}
\chi(t)R\left(\frac{(4k+1)\pi}{2}(1-\frac{t}{\tau})\right) & 0 \\
0 & \frac{1}{\chi(t)}R\left(\frac{(4k+1)\pi}{2}(1-\frac{t}{\tau})\right)
 \end{array} \right).
\end{equation}
Then
$$
\zeta_2(0)=\left( \begin{array} {cc}
J_1& 0 \\
0 & J_1 \end{array} \right),\quad \zeta_2(\tau)=M_2^+,
$$
 and a straightforward computation shows that each $\zeta_2(t)$ is a symplectic matrix.
 Observe that
$$
\sigma\left(R\left(\frac{(4k+1)\pi}{2}\bigl(1-\frac{t}{\tau}\bigr)\right)\right)=
\left\{\exp\left(\pm i\frac{(4k+1)\pi}{2}\bigl(1-\frac{t}{\tau}\bigr)\right) \right\}\subset\{z\in\mathbb{C}\,\mid\,|z|=1\},
$$
and
\begin{align*}
\det(\zeta_2(t)-I_4)&=
\det\left(\chi(t)R\left(\frac{(4k+1)\pi}{2}\bigl(1-\frac{t}{\tau}\bigr)\right)-
I_2\right)\\
&\quad\times\det\left(\frac{1}{\chi(t)}R\left(\frac{(4k+1)\pi}{2}\bigl(1-
\frac{t}{\tau}\bigr)\right)-I_2\right).
\end{align*}
If $0<t\le \tau$, then $\chi(t)>1$ and consequently  $\det(\zeta_2(t)-I_4)\ne 0$. Moreover,
$\det(\zeta_2(0)-I_4)=4>0$. Hence $\zeta_2(t)\in \operatorname{Sp}(4)^+$  for all $0\le t\le \tau$.  Since
$R(-\theta)=R(\theta)^{-1}=R(\theta)^\top$ implies
$$
\zeta_2(t)^\top\zeta_2(t)=\zeta_2(t)\zeta_2(t)^\top=
\left(
 \begin{array} {cc}
(\chi(t))^2I_2 &0 \\
0 & \frac{1}{(\chi(t))^2}I_2\end{array} \right),
$$
we have
$$
\mathfrak{u}(\zeta_2(t))=\left(\sqrt{\zeta_2(t)\zeta_2(t)^\top}\right)^{-1}\zeta_2(t)=
\left(
 \begin{array} {cc}
R\left(\frac{(4k+1)\pi}{2}(1-\frac{t}{\tau})\right) &0 \\
0 & R\left(\frac{(4k+1)\pi}{2}(1-\frac{t}{\tau})\right)\end{array} \right)
$$
and therefore $\det \mathfrak{u}(\zeta_2(t))\equiv 1$. It follows that
\begin{equation}\label{e:Delay6Auto7G.2}
\Delta(\zeta_2)=0.
\end{equation}

Since the isomorphism $\mathfrak{u}$ in (\ref{e:LieIsom}) maps
$\zeta_2(0)$ and $(\hat\Upsilon(\tau)\diamond\hat\Upsilon(\tau))M_1^{-1}$ to unitary matrices
$$
J_1=\left(
 \begin{array} {cc}
0 &ie^{i\frac{(4k+1)\pi}{2}} \\
-ie^{i\frac{(4k+1)\pi}{2}} & 0\end{array} \right)
\quad\text{and}\quad
\left( \begin{array} {cc}
0 &ie^{ie_1\tau} \\
-ie^{ie_1\tau} & 0\end{array} \right),
$$
 respectively, there exists a path of unitary matrices from the latter to $J_1$,
$$
[0,\tau]\ni t\mapsto \delta(t)=\left(\begin{array} {cc}
0 &ie^{ia(\lambda,t)} \\
-ie^{ia(\lambda,t)} & 0\end{array} \right),
$$
 where $a(\tau,t)=(1-\frac{t}{\tau})e_1\tau+ \frac{(4k+1)\pi}{2}\frac{t}{\tau}$. Clearly, this path $\delta$
 yields a path of symplectic and orthogonal matrices from $(\hat\Upsilon(\tau)\diamond\hat\Upsilon(\tau))M_1^{-1}$
 to $\zeta_2(0)$,
$$
[0, \tau]\ni t\mapsto\zeta_1(t):=\left( \begin{array} { c c c c }
0& -\sin(a(\tau,t)) & 0 & -\cos(a(\tau,t))\\
\sin(a(\tau,t)) & 0 & \cos(a(\tau,t))&0  \\
0& \cos(a(\tau,t))& 0 & -\sin(a(\tau,t))\\
-\cos(a(\tau,t)) & 0 &\sin(a(\tau,t)) & 0
 \end{array} \right).
$$
It follows from Lemma~3 in \cite[page~38]{Long02} that
\begin{align*}
\det(\zeta_1(t)-I_4)&=\det
\left( \begin{array} { c c c c }
-1& -\sin(a(\tau,t)) & 0 & -\cos(a(\tau,t))\\
\sin(a(\tau,t)) & -1 & \cos(a(\tau,t))&0  \\
0& \cos(a(\tau,t))& -1 & -\sin(a(\tau,t))\\
-\cos(a(\tau,t)) & 0 &\sin(a(\tau,t)) & -1
 \end{array} \right).\\
&=4\left(\sin\left((1-\frac{t}{\tau})e_1\tau+ \frac{(4k+1)\pi}{2}\frac{t}{\tau}\right)\right)^2.
\end{align*}
By assumption, 
$e_1\tau\in (2k\pi, (2k+1)\pi)$. Then for all $t\in [0, \tau]$, 
$$
(1-\frac{t}{\tau})e_1\tau+ \frac{(4k+1)\pi}{2}\frac{t}{\tau}\in
(2k\pi, (2k+1)\pi)
$$
and therefore $D(\zeta_1(t))=(-1)\det(\zeta_1(t)-I_4)<0$, i.e., $\zeta_1(t)\in \operatorname{Sp}(4)^+$.
Since $\mathfrak{u}(\zeta_1(t))=\delta(t)$, we obtain
$\det \mathfrak{u}(\zeta_1(t))=e^{i2a(\tau, t))}$ and consequently
\begin{equation}\label{e:Delay6Auto7G.3}
\Delta(\zeta_1)=(4k+1)\pi-2e_1\tau.
\end{equation}

Now  define $\eta=\zeta_2\ast\zeta_1$ as in (\ref{e:pathComp}).
Then
$$
\eta(0)=(\hat\Upsilon(\tau)\diamond\hat\Upsilon(\tau))M_1^{-1}=(\hat\Upsilon(\tau)\diamond\hat\Upsilon(\tau))M_1,
 \quad \eta(\tau)=M^+_2,
 $$
 and  $\eta(t)\in \operatorname{Sp}(4)^+$ for all $t\in [0, \tau]$.
By (\ref{e:Delay6Auto7G.2}) and (\ref{e:Delay6Auto7G.3}),
  $$
  \Delta(\eta)=\Delta(\zeta_1)+\Delta(\zeta_2)=(4k+1)\pi-2e_1\tau
  =(2m+1)\pi-2e_1\tau.
  $$
Combining this with  (\ref{e:RobbinSaCZ+RobbinSa7+++++})  and Lemma~\ref{lem:Dong06}(i), we obtain
 \begin{equation}\label{e:Delay6Auto6G.4}
i_{\tau, M_1}(\hat\Upsilon\diamond\hat\Upsilon)=
\left[\frac{\Delta( \eta) } { \pi } + \frac { 2e_1\tau} { \pi } \right]
=\left[\frac{(2m+1)\pi-2e_1\tau}{\pi}+\frac{2e_1\tau}{\pi}\right]=2m+1.
\end{equation}

\textbf{Step 2 ({Proof of (\ref{e:Delay6Auto7G++*})
for $e_1\tau\in (m, m+1)$ with $m=2k+1$})}.\quad
Let us redefine $\zeta_2(t)$ in (\ref{e:Delay6Auto7G.1}) as
\begin{equation}\label{e:Delay6Auto7G.1*}
\zeta_2(t) := \begin{pmatrix}
\chi(t) R\left(\frac{(4k+3)\pi}{2}\left(1-\frac{t}{\tau}\right)\right) & 0 \\
0 & \frac{1}{\chi(t)} R\left(\frac{(4k+3)\pi}{2}\left(1-\frac{t}{\tau}\right)\right)
\end{pmatrix},
\end{equation}
where $\chi(t)=1+\frac{t}{\tau}$ is as in (\ref{e:Delay6Auto7G.1}). Then
$$
\zeta_2(0) = \begin{pmatrix} -J_1 & 0 \\ 0 & -J_1 \end{pmatrix},
\quad\zeta_2(\tau) = M_2^+,
$$
and each $\zeta_2(t)$ is a symplectic matrix belonging to $\operatorname{Sp}(4)^+$. 
As above, we have
\begin{equation}\label{e:Delay6Auto7G.2*}
\Delta(\zeta_2) = 0.
\end{equation}
Define a new path of matrices:
$$
[0, \tau] \ni t \mapsto \zeta_1(t) := \begin{pmatrix}
0 & -\sin(b(\tau,t)) & 0 & -\cos(b(\tau,t)) \\
\sin(b(\tau,t)) & 0 & \cos(b(\tau,t)) & 0 \\
0 & \cos(b(\tau,t)) & 0 & -\sin(b(\tau,t)) \\
-\cos(b(\tau,t)) & 0 & \sin(b(\tau,t)) & 0
\end{pmatrix},
$$
where
$$
b(\tau,t) = \left(1-\frac{t}{\tau}\right)e_1\tau + \frac{(4k+3)\pi}{2} \frac{t}{\tau} .
$$
This gives a path of symplectic and orthogonal matrices connecting $ (\hat{\Upsilon}(\tau) \diamond \hat{\Upsilon}(\tau)) M_1^{-1} $ to
$$
\zeta_2(0) = \begin{pmatrix}
R\left(\frac{(4k+3)\pi}{2}\right) & 0 \\
0 & R\left(\frac{(4k+3)\pi}{2}\right)
\end{pmatrix} = \begin{pmatrix}
-J_1 & 0 \\
0 & -J_1
\end{pmatrix},
$$
with $ J_1 = R(\pi/2) $ being the standard rotation matrix.

Since $ e_1\tau \in ((2k+1)\pi, (2k+2)\pi)$, for all $ t \in [0, \tau] $ we have
$$
\left(1 - \frac{t}{\tau}\right)e_1\tau + \frac{(4k+3)\pi}{2} \frac{t}{\tau} \in ((2k+1)\pi, 2(k+1)\pi)
$$
and consequently $\zeta_1(t)\in \operatorname{Sp}(4)^+ $.
Following the same computation as before, we obtain
$$
\mathfrak{u}(\zeta_1(t)) = \begin{pmatrix}
0 & -e^{ib(\tau,t)} \\
e^{ib(\tau,t)} & 0
\end{pmatrix},
$$
and therefore
\begin{equation}\label{e:Delay6Auto7G.3*}
\Delta(\zeta_1) = 2b(\tau, \tau) - 2b(\tau, 0) = (4k+3)\pi - 2e_1\tau.
\end{equation}
Defining $ \eta = \zeta_2 \ast \zeta_1 $, as in (\ref{e:Delay6Auto6G.4}), it follows from (\ref{e:Delay6Auto7G.2*}) and (\ref{e:Delay6Auto7G.3*}) that
 $$
  \Delta(\eta)=\Delta(\zeta_1)+\Delta(\zeta_2)=(4k+3)\pi-2e_1\tau
  =(2m+1)\pi-2e_1\tau,
  $$
  and therefore
$$
   i_{\tau, M_1}(\hat\Upsilon\diamond\hat\Upsilon)=
    \left[\frac{(2m+1)\pi-2e_1\tau}{\pi}+\frac{2e_1\tau}{\pi}\right]=2m+1
    $$
by Lemma~\ref{lem:Dong06}(i). Combining this result with (\ref{e:Delay6Auto6G.4}), we finally obtain (\ref{e:Delay6Auto7G++*}).


\textbf{Step 3 ({Proof of (\ref{e:Delay6Auto7G++*})
for $e_1\tau=m+1$)}}.\quad
By (\ref{e:dongIndex}),
\begin{equation}\label{e:Delay6Auto7G.3**}
 i_{\tau,M_1}(\hat\Upsilon\diamond\hat\Upsilon)=
 [i_\tau(((\hat\Upsilon\diamond\hat\Upsilon)
  M_1^{-1})\ast\xi)-\Delta({\xi})/\pi]
  \end{equation}
 where $\hat\Upsilon\diamond\hat\Upsilon(t)=
 \hat\Upsilon(t)\diamond\hat\Upsilon(t)$ is given by
 (\ref{e:Delay6Auto7G.0}), and $\xi\in\mathcal{P}_\tau(4)$
 is any element satisfying $\xi(\tau)=M_1^{-1}=\left( \begin{array} { c c }
0 & J_1  \\
-J_1 & 0
 \end{array} \right)$, for example, we can take $\xi=\xi_2\ast\xi_1$, where
 for $0\le t\le\tau$,
 \begin{equation}\label{e:Delay6Auto7G.3***}
 \xi_1(t)=
 \exp(-J_2\frac{\pi t}{2\tau}),\quad
 \xi_2(t)=\begin{pmatrix}
0 & \exp(J_1\frac{\pi t}{2\tau}) \\
-\exp(J_1\frac{\pi t}{2\tau}) & 0
\end{pmatrix}.
   \end{equation}
 Since
 $\mathfrak{u}(\xi_1(t))=\exp(-i\frac{\pi t}{2\tau})I_2$ and
  $\mathfrak{u}(\xi_2(t))=i\left( \begin{array} { c c }
-\cos(\frac{\pi t}{2\tau}) & \sin(\frac{\pi t}{2\tau})  \\
-\sin(\frac{\pi t}{2\tau}) & -\cos(\frac{\pi t}{2\tau})
 \end{array} \right)$ imply, respectively,
 $\det\mathfrak{u}(\xi_1(t))=\exp(-i\frac{\pi t}{\tau})$
 and $\det\mathfrak{u}(\xi_2(t))=-1$,
 we obtain $\Delta(\xi_1)=-\pi$, $\Delta({\xi}_2)=0$, and so
 \begin{equation}\label{e:Delay6Auto7G.3****}
 \Delta({\xi})=\Delta({\xi}_1)+\Delta({\xi}_2)=-\pi.
 \end{equation}

 Note that
  $\nu_{\tau}(\hat\Upsilon\diamond\hat\Upsilon)=
 \dim{\rm Ker}((\hat\Upsilon\diamond\hat\Upsilon)
  M_1^{-1})\ast\xi(\tau)-I_4)=
  \dim{\rm Ker}((\hat\Upsilon(\tau)\diamond\hat\Upsilon(\tau))-M_1)=2$.
  We cannot use the methods in Steps~1 and 2.

 When $e_1\tau\in (m, m+1)$ with $m=2k$,
 let $\eta\in C([0, \tau], \operatorname{Sp}(4)^+)$
 be the path from
 $\eta(0)=(\hat\Upsilon(\tau)\diamond\hat\Upsilon(\tau))M_1^{-1}$
 to $\eta(\tau)=M_2^+=\operatorname{diag}(2, 2, 1/2, 1/2)$ constructed
 in Step~1.
For the above $\xi$, by the proof of Lemma~\ref{lem:Dong06}(i),
(replacing $\beta$ with $\eta$ in the present case), we have
  \begin{align}\label{e:Delay6Auto7G.3****A}
  i_\tau(((\hat\Upsilon\diamond\hat\Upsilon)
  M_1^{-1})\ast\xi)&= \frac { \Delta( \eta ) } { \pi } + \frac { \Delta\left(
  (\hat\Upsilon\diamond\hat\Upsilon)
  M_1^{-1} \right) } { \pi } +
 \frac { \Delta( \xi ) } { \pi }\nonumber\\
 &=(2m+1)-\frac{2e_1\tau}{\pi}+ \frac{2e_1\tau}{\pi}-1=2m
 \end{align}
 by (\ref{e:RobbinSaCZ+RobbinSa7+++++}) and (\ref{e:Delay6Auto7G.3****}).

Similarly, when $e_1\tau\in (m, m+1)$ with $m=2k+1$,
 for the path $\eta\in C([0, \tau], \operatorname{Sp}(4)^+)$
 constructed in Step~2, we have
 \begin{align}\label{e:Delay6Auto7G.3****B}
  i_\tau(((\hat\Upsilon\diamond\hat\Upsilon)
  M_1^{-1})\ast\xi)= \frac { \Delta( \eta ) } { \pi } + \frac { \Delta\left(
  (\hat\Upsilon\diamond\hat\Upsilon)
  M_1^{-1} \right) } { \pi } +
 \frac { \Delta( \xi ) } { \pi }=2m.
 \end{align}

Our goal is to compute $i_\tau(((\hat\Upsilon\diamond\hat\Upsilon)
  M_1^{-1})\ast\xi)$ for $e_1\tau=m+1$.
  In the present case
  $$
  \nu_\tau(((\hat\Upsilon\diamond\hat\Upsilon)
  M_1^{-1})\ast\xi)=\dim{\rm Ker}((\hat\Upsilon\diamond\hat\Upsilon)
  M_1^{-1}(\tau)-I_4)=\dim{\rm Ker}((\hat\Upsilon\diamond\hat\Upsilon)
 (\tau)- M_1)=2
 $$
 by Claim~\ref{cl:5.2}.  For $\epsilon\in [0, \pi/\tau)$, consider paths
  $\rho_{\pm\epsilon}: [0, \tau]\to\operatorname{Sp}(4)$ defined by
$$
\rho_{\pm\epsilon}(t)=
\left( \begin{array} { c c c c }
0& -\sin((e_1\pm\epsilon)t) & 0 & -\cos((e_1\pm\epsilon)t)\\
\sin((e_1\pm\epsilon)t)& 0 & \cos((e_1\pm\epsilon)t) & 0 \\
0& \cos((e_1\pm\epsilon)t)& 0 & -\sin((e_1\pm\epsilon)t)\\
-\cos((e_1\pm\epsilon)t) & 0 &\sin((e_1\pm\epsilon)t) & 0
 \end{array} \right).
$$
Then $\rho_0=(\hat\Upsilon\diamond\hat\Upsilon)M_1^{-1}$,
and (\ref{e:Delay6Auto7G.3****A}) and (\ref{e:Delay6Auto7G.3****B})
lead to
$$
i_\tau(\rho_\epsilon\ast\xi)=2m+2\quad\hbox{and}\quad
i_\tau(\rho_{-\epsilon}\ast\xi)=2m
  $$
 if $\epsilon>0$ is small enough.
It follows from Corollary~\ref{cor:Long142A}(B) that
$i_\tau(((\hat\Upsilon\diamond\hat\Upsilon)M_1^{-1})\ast\xi)= 2m$.
This, together with (\ref{e:Delay6Auto7G.3**}) and (\ref{e:Delay6Auto7G.3****})
, yields
for $e_1\tau=m+1$,
\begin{align}\label{e:Delay6Auto7G.3****C}
   i_{\tau,M_1}(\hat\Upsilon\diamond\hat\Upsilon)=
 [i_\tau(((\hat\Upsilon\diamond\hat\Upsilon)
  M_1^{-1})\ast\xi)-\Delta({\xi})/\pi]=2m+1.
 \end{align}
\end{proof}

%

\textbf{(III)[Proof of (\ref{e:DiagPindex2*}) for $n>1$]}.\quad
The standard symplectic space $\mathbb{R}^{4n}(u_1,\ldots, u_{2n}; v_1,\ldots,v_{2n})$ admits a natural
 symplectic direct sum decomposition:
$$
\mathbb{R}^{4n} = \bigoplus^n_{j=1} \mathbb{R}^4(u_j, u_{j+n}; v_{j+2n}, v_{j+3n}),
$$
where each subspace $\mathbb{R}^4(u_j, u_{j+n}; v_{j+2n}, v_{j+3n})$ consists of vectors in $\mathbb{R}^{4n}$
whose only nonzero coordinates are the $j$-th, $(j+n)$-th, $(j+2n)$-th, and $(j+3n)$-th entries.
Moreover, these subspaces are pairwise symplectically orthogonal.
For each $j=1,\cdots,n$, the subspace $\mathbb{R}^{4}(u_j, u_{j+n}; v_{j+2n}, v_{j+3n})$
is the image of the  injective symplectic homomorphism
$$
\digamma_j:\mathbb{R}^4(x_1,x_2;y_1,y_2)\to
\mathbb{R}^{4n}(u_1,\ldots, u_{2n}; v_1,\ldots,v_{2n})
$$
 defined by
 $$
 \digamma_j(x_1,x_2;y_1,y_2)=({\bf U}^{j}(x_1,x_2;y_1,y_2)^\top, {\bf V}^{j}(x_1,x_2;y_1,y_2)^\top)^\top,
 $$
 where
 \begin{align*}
 {\bf U}^{j}(x_1,x_2;y_1,y_2)&=(U^{j}_1(x_1,x_2;y_1,y_2),\ldots, U^{j}_n(x_1,x_2;y_1,y_2)^\top,\\ 
 {\bf V}^{j}(x_1,x_2;y_1,y_2)&=(V^{j}_1(x_1,x_2;y_1,y_2),\ldots, V^{j}_n(x_1,x_2;y_1,y_2))^\top
 \end{align*}
 are given, for $k=1,\ldots, k$, by
$$
{U}^{j}_k(x_1,x_2;y_1,y_2)  =
    \begin{cases}
x_1\quad&\hbox{if $k=j$},\\
x_2\quad&\hbox{if $k=n+j$},\\
0\quad&\hbox{otherwise}
\end{cases}\qquad
{V}^{j}_k(x_1,x_2;y_1,y_2)  =
    \begin{cases}{ll}
y_1\quad&\hbox{if $k=j$},\\
y_2\quad&\hbox{if $k=n+j$},\\
0\quad&\hbox{otherwise}.
\end{cases}
$$
 Thus, $\mathbb{R}^{4}(x_1, x_2; y_1, y_2)$ is  symplectomorphic to
 $\mathbb{R}^{4}(u_j, u_{j+n}; v_{j+2n}, v_{j+3n})$ via $\digamma_j$.
  All these $\digamma_j$  induce an injective group homomorphism
$$
\hat\digamma: (\operatorname{Sp}(4))^n\to \operatorname{Sp}(4n)
$$
defined by
$$
\hat\digamma(P_1,\ldots,P_n)(u_1,\ldots,u_{2n}, v_1,\ldots,v_{2n})=\sum^n_{j=1}
\digamma_j\left(P_j(u_j, u_{n+j};v_j, v_{n+j})^T\right).
$$
One can directly verify the following identities:
\begin{align}\label{e:RobbinSaCZ+RobbinSa74+}
\hat\digamma(M_1,\ldots,M_1)&=M_n\quad\text{and}\quad \hat\digamma(J_2,\ldots, J_2)=J_{2n},\\
\mathfrak{u}\left(\hat\digamma(P_1,\ldots,P_n)\right)&=\hat\digamma(\mathfrak{u}(P_1),\ldots, \mathfrak{u}(P_n)),\label{e:RobbinSaCZ+RobbinSa75+}\\
\det\hat\digamma(P_1,\ldots,P_n)&=\det P_1\cdots\det P_n.
\label{e:RobbinSaCZ+RobbinSa76+}
\end{align}

For each $ j = 1, \dots, n $, define $\hat{\Upsilon}^j(t) = R(e_j t) = \exp(e_j t J_1)$, where
$J_1$ is defined by (\ref{e:standcompl}). 
Then the fundamental matrix solution of
$$
\dot{z}(t) = J_2 \operatorname{diag}(e_j, e_j, e_j, e_j) z(t)
$$
with initial condition $\hat{\Upsilon}^j(0) \diamond \hat{\Upsilon}^j(0) = I_4$ is given by
\begin{equation}\label{e:RobbinSaCZ+RobbinSa77+}
\hat\Upsilon^j(t)\diamond\hat\Upsilon^j(t)=\left( \begin{array} { c c }
\operatorname{diag}(\cos(e_jt), \cos(e_jt)) & -\operatorname{diag}(\sin(e_jt), \sin(e_jt))  \\
\operatorname{diag}(\sin(e_jt), \sin(e_jt))& \operatorname{diag}(\cos(e_jt), \cos(e_jt))
 \end{array} \right).
 \end{equation}
Moreover, the following identities hold:
\begin{align*}
 \hat\Upsilon(t)\diamond\hat\Upsilon(t)&=\hat\digamma\bigl(\hat\Upsilon^1(t)\diamond\hat\Upsilon^1(t),\ldots,
  \hat\Upsilon^n(t)\diamond\hat\Upsilon^n(t)\bigr),\\
 (\hat\Upsilon(t)\diamond\hat\Upsilon(t))M_n^{-1}&=
  \hat\digamma\bigl((\hat\Upsilon^1(t)\diamond\hat\Upsilon^1(t))M_1^{-1},\ldots,
  (\hat\Upsilon^n(t)\diamond\hat\Upsilon^n(t))M_1^{-1}\bigr).
  \end{align*}

\textbf{Case~1} ($e_j\tau\ne (m_j+1)\pi$ for $j=1,\cdots,n$).\quad
Our goal is to construct a path $\eta \in C([0, \tau], \mathrm{Sp}(4n))$ such that:
\begin{itemize}
  \item $\eta(t) \in \mathrm{Sp}(4n)^\ast$ for all $t \in [0, \tau]$,
  \item $\eta(0) = (\hat{\Upsilon}(\tau) \diamond \hat{\Upsilon}(\tau)) M_n^{-1}$,
  \item $\eta(\tau) = M_{2n}^+ = D(2)^{\diamond 2n} = \operatorname{diag}(2 I_n, \frac{1}{2} I_n)$.
\end{itemize}
To achieve this, for each $j = 1, \dots, n$, we define paths $\zeta^j_1, \zeta^j_2 \colon [0, \tau] \to \mathrm{Sp}(4)$ as follows:
\begin{itemize}
  \item If $2k_j\pi< e_j \tau < (2k_j + 1)\pi$:
  \begin{align*}
    \zeta^j_2(t) &= \begin{pmatrix}
      \chi(t) R\left(\frac{(2k_j - 1)\pi}{2}\left(1 - \frac{t}{\tau}\right)\right) & 0 \\
      0 & \frac{1}{\chi(t)} R\left(\frac{(2k_j - 1)\pi}{2}\left(1 - \frac{t}{\tau}\right)\right)
    \end{pmatrix}, \\
    \zeta^j_1(t) &= \begin{pmatrix}
      0 & -\sin(a_j(\tau, t)) & 0 & -\cos(a_j(\tau, t)) \\
      \sin(a_j(\tau, t)) & 0 & \cos(a_j(\tau, t)) & 0 \\
      0 & \cos(a_j(\tau, t)) & 0 & -\sin(a_j(\tau, t)) \\
      -\cos(a_j(\tau, t)) & 0 & \sin(a_j(\tau, t)) & 0
    \end{pmatrix},
  \end{align*}
  where $a_j(\tau, t) = \left(1 - \frac{t}{\tau}\right) e_j \tau + \frac{(4k_j + 1)\pi}{2} \frac{t}{\tau}$.

  \item If $(2k_j + 1)\pi< e_j \tau < (2k_j + 2)\pi$:
  \begin{align*}
    \zeta^j_2(t) &= \begin{pmatrix}
      \chi(t) R\left(\frac{(2k_j + 1)\pi}{2}\left(1 - \frac{t}{\tau}\right)\right) & 0 \\
      0 & \frac{1}{\chi(t)} R\left(\frac{(2k_j + 1)\pi}{2}\left(1 - \frac{t}{\tau}\right)\right)
    \end{pmatrix}, \\
    \zeta^j_1(t) &= \begin{pmatrix}
      0 & -\sin(b_j(\tau, t)) & 0 & -\cos(b_j(\tau, t)) \\
      \sin(b_j(\tau, t)) & 0 & \cos(b_j(\tau, t)) & 0 \\
      0 & \cos(b_j(\tau, t)) & 0 & -\sin(b_j(\tau, t)) \\
      -\cos(b_j(\tau, t)) & 0 & \sin(b_j(\tau, t)) & 0
    \end{pmatrix},
  \end{align*}
  where $b_j(\tau, t) = \left(1 - \frac{t}{\tau}\right) e_j \tau + \frac{(4k_j + 3)\pi}{2} \frac{t}{\tau}$.
\end{itemize}
The desired path $\eta$ is then defined by
$$
\eta(t)=\hat\digamma\bigl((\zeta^1_2\ast\zeta^1_1)(t),\ldots, (\zeta^n_2\ast\zeta^n_1)(t)\bigr),
$$
where $\zeta^j_2\ast\zeta^j_1$ ($j=1,\ldots,n$) are defined by (\ref{e:pathComp}).
From (\ref{e:RobbinSaCZ+RobbinSa75+}) and (\ref{e:RobbinSaCZ+RobbinSa76+}), we deduce
$$
\Delta(\eta) = \sum_{j=1}^n \Delta(\zeta^j_2 \ast \zeta^j_1) =\sum_{j=1}^n\Delta(\zeta^j_2)+\sum_{j=1}^n\Delta(\zeta^j_1)=\sum_{j=1}^n\Delta(\zeta^j_1)
$$
since $\Delta(\zeta^j_2) = 0$ for all $j$.

From (\ref{e:RobbinSaCZ+RobbinSa7+++++}), it follows that
$$\mathfrak{u}\!\left((\hat\Upsilon(t)\diamond\hat\Upsilon(t))M_n^{-1}\right)=
-J_n i \operatorname{diag}\!\left(e^{i e_1 t}, \dots, e^{i e_n t}, e^{i e_1 t}, \dots, e^{i e_n t}\right),
$$
 where $i=\sqrt{-1}$, and therefore
 $$
 \det\mathfrak{u}\!\left((\hat\Upsilon(t)\diamond\hat\Upsilon(t))M_n^{-1}\right)=
\exp\!\left(\frac{n\pi}{2}i + 2i(e_1 + \dots + e_n)t\right).
$$
 Then $\Delta\left((\hat\Upsilon\diamond\hat\Upsilon)M_n^{-1}\right)=2(e_1+\cdots+e_n)\tau$.
 By Lemma~\ref{lem:Dong06}(i), we conclude
$$
i_{\tau, M_n}(\hat\Upsilon \diamond \hat\Upsilon)=
\left[ \frac{\Delta(\eta)}{\pi} + \frac{2(e_1 + \dots + e_n)\tau}{\pi} \right]
= \left[ \sum_{j=1}^{n} \frac{\Delta(\zeta_1^j)}{\pi} + \frac{2(e_1 + \dots + e_n)\tau}{\pi} \right].
$$
By (\ref{e:Delay6Auto7G.3}) and (\ref{e:Delay6Auto7G.3*}),
$\Delta(\zeta_1^j)=(2m_j+1)\pi-2e_j\tau$
if  $m_j\pi\le e_j\tau< (m_j+1)\pi$, $j=1,\cdots,n$.
Hence $\sum_{j=1}^{n} \frac{\Delta(\zeta_1^j)}{\pi}=
\sum_{j=1}^{n}(2m_j+1)-\frac{2\tau}{\pi}\sum_{j=1}^ne_j$ and so
$$
i_{\tau, M_n}(\hat\Upsilon \diamond \hat\Upsilon)=
2\sum_{j=1}^{n}m_j+ n.
$$

\textbf{Case~2} ($\{j\,|\,e_j\tau=(m_j+1)\pi\}=\{j_1,\cdots,j_q\}$ for $q\ge 1$).\quad In this case
$\nu_{\tau, M_n}(\hat\Upsilon \diamond \hat\Upsilon)=2q$.
For $\epsilon\in [0, \pi/\tau)$ and each $j=1,\cdots,n$,
we define paths $\phi^j_{\pm\epsilon}: [0, \tau]\to\operatorname{Sp}(4)$ by
\begin{equation*}
\phi^j_{\pm\epsilon}(t)=\left( \begin{array} { c c }
\operatorname{diag}(\cos((e_j\pm\epsilon)t), \cos((e_j\pm\epsilon)t) & -\operatorname{diag}(\sin((e_j\pm\epsilon)t), \sin((e_j\pm\epsilon)t))  \\
\operatorname{diag}(\sin((e_j\pm\epsilon)t), \sin((e_j\pm\epsilon)t))& \operatorname{diag}(\cos((e_j\pm\epsilon)t), \cos((e_j\pm\epsilon)t))
 \end{array} \right).
 \end{equation*}
Then $\phi^j_{0}(t)$ is equal to $\hat\Upsilon^1(t)\diamond\hat\Upsilon^1(t)$
in (\ref{e:RobbinSaCZ+RobbinSa77+}).
For $\epsilon>0$ small enough, define paths $\sigma_{\pm\epsilon}:[0, \tau]\to{\rm Sp}(4n)$ by
 $$
 \sigma_{\pm\epsilon}(t)=\hat\digamma\bigl(\theta^1(t),\ldots, \theta^1(t)\bigr),
 $$
 where $\theta^j(t)=\rho^j_{\pm\epsilon}$ for $j=j_1,\cdots,j_q$, and
 $\theta^j(t)=(\hat\Upsilon^j(t)\diamond\hat\Upsilon^j(t))$
 in (\ref{e:RobbinSaCZ+RobbinSa77+}) for $j\in\{1,\cdots,n\}\setminus\{j_1,\cdots,j_q\}$.
 By Case~1 we obtain
\begin{align*}
i_{\tau, M_n}(\sigma_{\epsilon})&=
2\sum_{j\in\{1,\cdots,n\}\setminus\{j_1,\cdots,j_q\}}m_j+ 2\sum^q_{\ell=1} (m_{j_\ell}+1)+ n=2\sum_{j=1}^{n}m_j+ n+ 2q,\\
i_{\tau, M_n}(\sigma_{-\epsilon})&=2\sum_{j=1}^{n}m_j+ n.
\end{align*}
Since $\sigma_0=\hat\Upsilon \diamond \hat\Upsilon$,
$\nu_{\tau, M_n}(\sigma_0)=2q$. From Corollary~\ref{cor:Long142B}(B) we derive
$$
i_{\tau, M_n}(\hat\Upsilon \diamond \hat\Upsilon)=i_{\tau, M_n}(\sigma_0)=
i_{\tau, M_n}(\sigma_{-\epsilon})=
2\sum_{j=1}^{n}m_j+ n.
$$

%
%
%

\textbf{(IV)[Proof for (\ref{e:positive-negativeAG}) and (\ref{e:positive-negativeBG})]}.\quad
Since $\dim{\rm Ker}(I_{4n}-M_n)=2n$, as in the proof of Theorem~\ref{th:PIndex1}, from (\ref{e:RobbinSaCZ+RobbinSa7++}), (\ref{e:DiagPindex1*}), (\ref{e:positive-negativeA}), and (\ref{e:positive-negativeB})
we  derive
 \begin{align*}
i_{\tau,M_n}(\Upsilon)&=i_{\tau, M_n}(\hat\Upsilon\diamond\hat\Upsilon)\\
&=i_{\tau,M_n}(\xi_{4n,\tau})
+\sum_{0\le t<\tau}\nu_{t,M_n}(\hat\Upsilon\diamond\hat\Upsilon|_{[0,t]})\\
&=i_{\tau,M_n}(\xi_{4n,\tau})+2n
+2\sum_{0<t<\tau}\sharp\left\{k\mid te_k\in\pi\mathbb{Z}\right\}
\end{align*}
if $B$ is positive definite, and
\begin{align*}
i_{\tau,M}(\Upsilon)&=i_{\tau,M}(\xi_{2n,\tau})
-\sum_{0<t\le\tau}\nu_{t,M}(\hat\Upsilon\diamond\hat\Upsilon|_{[0,t]})\\
&=i_{\tau,M}(\xi_{2n,\tau})-2\sum_{0<t\le\tau}\sharp\left\{k\mid te_k\in\pi\mathbb{Z}\right\}
\end{align*}
if $B$ is negative definite. (\ref{e:positive-negativeAG}) and (\ref{e:positive-negativeBG}) are proved.
\end{proof}

\subsection{Proof of Theorem~\ref{th:PIndex5}}\label{sec:PIndex5}

\begin{lemma}\label{lem:distri1}
For reals $a,b$,  let $\gamma_{a,b}$ be  the fundamental matrix solution of
			$\dot{z}=J_1\operatorname{diag}(a, b)z$ with $\gamma_{a,b}(0)=I_2$.
\begin{description}
\item[(i)] If $ab>0$, let $\omega=\sqrt{ab}$. Then
$$
\gamma_{a,b}=\begin{pmatrix}
\cos(\omega t) & -\frac{b}{\omega}\sin(\omega t) \\
\frac{a}{\omega}\sin(\omega t) & \cos(\omega t)
\end{pmatrix}.
$$

\item[(ii)]If $ab<0$, let $\alpha=\sqrt{-ab}$. Then
$$
\gamma_{a,b}=\begin{pmatrix}
\cosh(\alpha t) & -\frac{b}{\alpha}\sinh(\alpha t) \\
\frac{a}{\alpha}\sinh(\alpha t) & \cosh(\alpha t)
\end{pmatrix}.
$$
\item[(iii)] If $a=0$ and $b\ne 0$, then $\gamma_{a,b}=\begin{pmatrix}
1 & -bt\\
0 & 1
\end{pmatrix}$.

\item[(iv)] If $a\ne 0$ and $b= 0$, then $\gamma_{a,b}=\begin{pmatrix}
1 & 0\\
at & 1
\end{pmatrix}$.
\item[(v)] If $a=0$ and $b= 0$, then $\gamma_{a,b}=I_2$.
\end{description}
\end{lemma}

\begin{lemma}\label{lem:distri2}
For reals $a,b$,  let $\gamma_{a,b}$ be given by Lemma~\ref{lem:distri1}.
\begin{description}
\item[(I)] Suppose $ab>0$. Then
$$
 \nu_{\tau,-I_2}(\gamma_{a,b})=
 \begin{cases}
 2\quad&\hbox{if  $\sqrt{ab}\tau\in (2\mathbb{N}-1)\pi$},\\
0\quad&\hbox{otherwise}.
\end{cases}
 $$
 In particular, $\nu_{\tau,-I_2}(\gamma_{a,b})=0$ if $\sqrt{ab}\tau<\pi$.
 Moreover (recalling that $\lceil A \rceil$ denotes the least integer greater than or equal to a real $A$),

\begin{itemize}
\item[\rm (I.1)] if $a>0$,
$i_{\tau,-I_2}(\gamma_{a,b})$ is the unique even integer lying in
 $[\frac{\sqrt{ab}\tau}{\pi}-1,\; \frac{\sqrt{ab}\tau}{\pi}+1)$, i.e.,
 $$
 i_{\tau,-I_2}(\gamma_{a,b})=2\lceil \frac{\sqrt{ab}\tau}{2\pi}-\frac{1}{2} \rceil=2
 \min\left\{k\in\mathbb{Z}\,|\,k\ge \frac{\sqrt{ab}\tau}{2\pi}-\frac{1}{2}\right\},
 $$
and hence $i_{\tau,-I_2}(\gamma_{a,b})=0$ if $\sqrt{ab}\tau<\pi$;
\item[\rm (I.2)] if $a<0$, $i_{\tau,-I_2}(\gamma_{a,b})$ is the unique even integer lying in
 $[\frac{-\sqrt{ab}\tau}{\pi}-1,\; \frac{-\sqrt{ab}\tau}{\pi}+1)$, i.e.,
 $$
 i_{\tau,-I_2}(\gamma_{a,b})=2\lceil \frac{-\sqrt{ab}\tau}{2\pi}-\frac{1}{2} \rceil=2
 \min\left\{k\in\mathbb{Z}\,|\,k\ge \frac{-\sqrt{ab}\tau}{2\pi}-\frac{1}{2}\right\},
 $$
 and hence $i_{\tau,-I_2}(\gamma_{a,b})=0$ if $\sqrt{ab}\tau<\pi$.
\end{itemize}
\item[(II)] If $ab\le 0$, then $i_{\tau,-I_2}(\gamma_{a,b})=0$ and
$\nu_{\tau,-I_2}(\gamma_{a,b})=0$.
\end{description}

\end{lemma}
\begin{proof}[\bf Proof]
(I) Because $\det(\gamma_{a,b}(\tau)+I_2)=2+2\cos(\omega\tau)$, we have
 $\nu_{\tau,-I_2}(\gamma_{a,b})=\dim{\rm Ker}(\gamma_{a,b}(\tau)+I_2)=2$
if $\omega\tau\in (2\mathbb{Z}+1)\pi$, and $0$ otherwise.

 In order to compute $i_{\tau,-I_2}(\gamma_{a,b})$, note that
 for any  $\gamma\in\mathcal{P}_\tau(2)$,  Proposition~\ref{prop:twoDef} yields
  \begin{equation}\label{e:Distri-indexrelation1}
 i_{\tau,-I_2}(\gamma)=i^{-I_2}_\tau(\gamma)+ [i_\tau(\xi)-\Delta({\xi})/\pi],
\end{equation}
where $\xi\in\mathcal{P}_\tau(2)$ is any path satisfying $\xi(\tau)=-I_2$. We can take
$\xi(t)=R(\frac{\pi}{\tau}t)$. By Lemma~\ref{lem:Dong06}(ii),
$\Delta({\xi})=\pi$ and $i_\tau(\xi)=1$. Thus, for any  $\gamma\in\mathcal{P}_\tau(2)$,
 \begin{equation}\label{e:Distri-indexrelation2}
 i_{\tau,-I_2}(\gamma)=i^{-I_2}_\tau(\gamma).
\end{equation}

Assume that $a>0$. Since $\frac{a}{\omega}>0$ and $\frac{a}{\omega} \frac{b}{\omega}=1$, for each $0\le s\le 1$ we define
$$
\gamma^s_{a,b}=\begin{pmatrix}
\cos(\omega t) & -(\frac{a}{\omega}(1-s)+s)^{-1}\sin(\omega t) \\
(\frac{a}{\omega}(1-s)+s)\sin(\omega t) & \cos(\omega t)
\end{pmatrix},\quad 0\le t\le\tau.
$$
Thus all these paths belong to $\mathcal{P}_\tau(2)$, and $\det(\gamma^s_{a,b}(\tau)+I_2)=2+2\cos(\omega\tau)$ for all $s\in [0,1]$.
Note that $\dim\operatorname{Ker}(\gamma^s_{a,b}(\tau)+I_2)=2$ if $\omega\tau\in (2\mathbb{Z}+1)\pi$, and $0$ otherwise.
It follows from \cite[Theorem 4.7]{Liu06} that $i^{-I_2}_\tau(\gamma_{a,b})=i^{-I_2}_\tau(\gamma^0_{a,b})=i^{-I_2}_\tau(\gamma^1_{a,b})$.
Together with  (\ref{e:Distri-indexrelation2}), this gives
 \begin{equation}\label{e:Distri-indexrelation3}
 i_{\tau,-I_2}(\gamma_{a,b})=i_{\tau,-I_2}(\gamma^1_{a,b}).
\end{equation}
Observe that $\gamma^1_{a,b}$ coincides with $\gamma_\omega$ appearing  in Lemma~\ref{lem:Dong06}(ii).
Hence, by the part (ii) of that lemma, $i_{\tau,-I_2}(\gamma^1_{a,b})=i_{\tau, R(\pi)}(\gamma^1_{a,b})=2k$,
where $k$ is the unique integer with $\pi+2(k-1)\pi<\sqrt{ab}\tau\le \pi+2k\pi$.
Combining this with (\ref{e:Distri-indexrelation3}) we conclude that
 $i_{\tau,-I_2}(\gamma_{a,b})$ is the unique even
integer  lying in $[\frac{\sqrt{ab}\tau}{\pi}-1, \frac{\sqrt{ab}\tau}{\pi}+1)$.

Suppose $a<0$. For each $0\le s\le 1$ we define
$$
\hat{\gamma}^s_{a,b}=\begin{pmatrix}
\cos(\omega t) & -(\frac{a}{\omega}(1-s)-s)^{-1}\sin(\omega t) \\
(\frac{a}{\omega}(1-s)-s)\sin(\omega t) & \cos(\omega t)
\end{pmatrix},\quad 0\le t\le\tau.
$$
Arguing as before, we deduce that $i^{-I_2}_\tau(\gamma_{a,b})=i^{-I_2}_\tau(\hat{\gamma}^0_{a,b})=i^{-I_2}_\tau(\hat{\gamma}^1_{a,b})$,
which yields
 \begin{equation}\label{e:Distri-indexrelation4}
 i_{\tau,-I_2}(\gamma_{a,b})=i_{\tau,-I_2}(\hat{\gamma}^1_{a,b})=2k,
\end{equation}
where $2k$ is the unique even integer lying
 in $[\frac{-\sqrt{ab}\tau}{\pi}-1, \frac{-\sqrt{ab}\tau}{\pi}+1)$.

(II) [Case $ab<0$]. Note that  $\det(\gamma_{a,b}(t)+I_2)=2+2\cosh(\alpha t)>0$ for all $t$ (because $\cosh(\alpha t)=(e^{\alpha t}+e^{-\alpha t})/2\ge 1$).
For $(s,t)\in [0, 1]\times [0,\tau]$, define
$$
\delta(s,t)=\begin{pmatrix}
\cosh(\alpha st) & -\frac{b}{\alpha}\sinh(\alpha st) \\
\frac{a}{\alpha}\sinh(\alpha st) & \cosh(\alpha st)
\end{pmatrix}.
$$
Then $\det(\delta(s,t)+I_2)=2+2\cosh(\alpha st)>0$ for all $t$,
$\delta(1,\cdot)=\gamma_{a,b}$, and
$\delta_0(t):=\delta(0, t)\equiv I_2$.
By \cite[Theorem~4.7]{Liu06} and (\ref{e:Distri-indexrelation2}), we conclude
\begin{equation}\label{e:Distri-indexrelation5}
 i_{\tau,-I_2}(\gamma_{a,b})=i_{\tau}^{-I_2}(\delta_0)=i_{\tau,-I_2}(\delta_0).
\end{equation}
In view of the definition in (\ref{e:LiuTang-index}), $i_{\tau}^{-I_2}(\delta_0)=
i_{\tau}((-\delta_0)\ast\xi)-i_\tau({\xi})$, where
$\xi(t)=R(\frac{\pi}{\tau}t)$. Note that the paths $(-\delta_0)\ast\xi$ and $\xi$
are homotopic in $\mathcal{P}_\tau(2)$ with fixed endpoints. Then $i_{\tau}((-\delta_0)\ast\xi)=i_\tau({\xi})$, and hence
\begin{equation}\label{e:Distri-indexrelation6}
 i_{\tau,-I_2}(\gamma_{a,b})=0.
\end{equation}

(II)[Case $ab=0$]. Since  $\det(\gamma_{a,b}(t)+I_2)=4>0$ for all $t$,
 we define for $(s,t)\in [0, 1]\times [0,\tau]$,
$$
\delta(s,t)=\begin{pmatrix}
1 & -bst \\
1 & 0
\end{pmatrix}.
$$
Then $\det(\delta(s,t)+I_2)=4>0$ for all $t$. Arguing as in (ii),
we obtain $i_{\tau,-I_2}(\gamma_{a,b})=0$.

Similarly, one shows that $i_{\tau,-I_2}(\gamma_{a,b})=0$ in Cases (iv) and (v) as well.
\end{proof}

\begin{proof}[\bf Proof of Theorem~\ref{th:PIndex5}]
Let $\hat\Upsilon$ is the fundamental matrix solution of
$$
\dot{z}(t)=J_n{\rm diag}(a_1,\cdots,  a_n, b_1,\cdots,  b_n)z(t)
$$
with $\hat\Upsilon(0)=I_{2n}$; that is,  $\hat\Upsilon(t)=\exp(J_n{\rm diag}(a_1,\cdots,  a_n, b_1,\cdots,  b_n)t)$.
Since the orthogonal symplectic matrix $Q:=\left( \begin{array} { c c c c }
	\Xi & 0 \\
	0 &\Xi
\end{array} \right)$ satisfies
\begin{equation*}
	J_n{\rm diag}(a_1,\cdots,
	a_n, b_1,\cdots, b_n)=J_nQ^{-1}\left( \begin{array} { c c c c }
		A & 0  \\
		0 &B
	\end{array} \right)Q=Q^{-1}J_n\left( \begin{array} { c c c c }
		A & 0  \\
		0 &B
	\end{array} \right)Q,
\end{equation*}
we have $Q^{-1}\Upsilon(t) Q=\hat\Upsilon(t)$, and therefore
 (\ref{e:DiagDelay1}) and Lemma~\ref{lem:Dong06+} lead to
\begin{equation}\label{e:DistriDiagDelay2}
 i_{\tau,-I_{2n}}(\Upsilon)=i_{\tau,-I_{2n}}(\hat\Upsilon)
\quad\text{and}\quad
\nu_{\tau,-I_{2n}}(\Upsilon)
=\nu_{\tau,-I_{2n}}(\hat\Upsilon).
\end{equation}

For each $j=1,\cdots,n$, let $\hat\Upsilon^j$ be  the fundamental matrix solution of
$\dot{z}=J_1\operatorname{diag}(a_j, b_j)z$ with $\hat\Upsilon^j(0)=I_2$; that is,
it is precisely $\gamma_{a_j,b_j}$ appearing in Lemma~\ref{lem:distri2}.
As before, since  
\begin{eqnarray*}
J_n{\rm diag}(e_1,\cdots,
 e_n, e_1,\cdots,
 e_n)
=J_1\left( \begin{array} { c c c c }
e_1 & 0  \\
0 &e_1
 \end{array} \right)\diamond\cdots\diamond
 J_1\left( \begin{array} { c c c c }
e_n & 0  \\
0 &e_n
 \end{array} \right).
 \end{eqnarray*}
 Consequently, by Lemma~\ref{lem:diamond-product}(v) we obtain
 \begin{equation}\label{e:DistriUpsilonProduct}
\hat\Upsilon(t)=\gamma_{a_1,b_1}(t)\diamond\cdots\diamond\gamma_{a_n,b_n}(t).
\end{equation}
\eqref{e:UpsilonProduct} together with Lemma~\ref{lem:diamond-product}(iv) yields
 $$
 \hat\Upsilon(\tau)-(-I_{2n})=(\gamma_{a_1,b_1}(\tau)+ I_2)\diamond\cdots\diamond(\gamma_{a_n,b_n}(\tau)+I_2).
 $$
 From this, it  easily follows that 
 $$
 {\rm Ker}(\hat\Upsilon(\tau)+I_{2n})=\{(x_1,\cdots,x_n,y_1,\cdots,y_n)^\top\in\mathbb{R}^{2n}\,|\,
 (x_i, y_i)^\top\in {\rm Ker}(\gamma_{a_i,b_i}(\tau)+I_2),\;i=1,\cdots,n\}
 $$
 and therefore
 \[
\dim \operatorname{Ker}\!\bigl(\hat{\Upsilon}(\tau) + I_{2n}\bigr)
= \sum_{i=1}^{n} \dim \operatorname{Ker}\!\bigl(\gamma_{a_i,b_i}(\tau) + I_2\bigr),
\]
which implies
 \begin{equation}\label{e:DistriDiagDelay4}
 \nu_{\tau, -I_{2n}}(\hat{\Upsilon}) = \sum_{i=1}^{n} \nu_{\tau, -I_2}(\gamma_{a_i,b_i}).
 \end{equation}
The desired equality (\ref{e:DistriDiagPindex1}) follows from this, the second equality in (\ref{e:DistriDiagDelay2}) and Lemma~\ref{lem:distri2}(I).

In order to prove  (\ref{e:DistriDiagPindex2}), let us define the path $\xi^{(k)} \in \mathcal{P}_\tau(2k)$ by
\begin{equation}\label{e:DistriXik}
\xi^{(k)}(t) = \exp\!\left( J_k \frac{\pi t}{\tau} \right) \qquad \text{for all } t \in [0,\tau].
\end{equation}
Then $\xi^{(k)}(\tau) = -I_{2k}$, and $\xi^{(k)} = (\xi^{(1)})^{\diamond k}$ (the $k$-fold $\diamond$-product of $\xi^{(1)}$).
Since $\Delta(\xi^{(n)}) = n \Delta(\xi^{(1)}) = n\pi$, applying \eqref{e:UpsilonProduct} together with the additivity property \eqref{e:sympAddDong1} we obtain
\begin{align*}\label{e:DistriDiagDelay13}
i_{\tau,-I_{2n}}(\hat\Upsilon)
&= \Biggl[ \sum_{k=1}^{n} i_\tau\bigl((-\hat\Upsilon^k ) \ast \xi^{(1)}\bigr) - \frac{n \Delta(\xi^{(1)})}{\pi} \Biggr] \nonumber \\
&= \sum_{k=1}^{n} \left(i_\tau\bigl((-\gamma_{a_k,b_k}) \ast \xi^{(1)}\bigr) -\frac{ \Delta(\xi^{(1)})}{\pi}\right)\nonumber \\
&= \sum_{k=1}^{n}i_{\tau,-I_2}(\gamma_{a_k,b_k})\\
&=\sum_{a_kb_k>0\& a_k>0}2\lceil \frac{\sqrt{a_kb_k}\tau}{2\pi}-\frac{1}{2} \rceil+
\sum_{a_kb_k>0\& a_k<0}
2\lceil \frac{-\sqrt{a_kb_k}\tau}{2\pi}-\frac{1}{2} \rceil,
\end{align*}
where the final equality comes from Lemma~\ref{lem:distri2}(I).
This and the first equality in (\ref{e:DistriDiagDelay2}) yield (\ref{e:DistriDiagPindex2}).

Now we are in a position  to  prove \eqref{e:DistriDiagPindex3} and \eqref{e:DistriDiagPindex4}.

Firstly, let us compute $i_{\tau,-I_{2n}}(\xi_{2n,\tau})$ for $\xi_{2n,\tau}$ appearing in
Lemma~\ref{lem:Dong06}(ii). Note that $\xi_{2n,\tau}=(\xi_{2,\tau})^{\diamond n}$ (the $n$-fold $\diamond$-product of $\xi_{2,\tau}$).
For $\xi^{(k)}$ in (\ref{e:DistriXik}), arguing as above we obtain
\begin{align*}
i_{\tau,-I_{2n}}(\xi_{2n,\tau})
&= \Biggl[ {n} i_\tau\bigl((-\xi_{2,\tau} ) \ast \xi^{(1)}\bigr) - \frac{n \Delta(\xi^{(1)})}{\pi} \Biggr] \nonumber \\
&= {n} \left(i_\tau\bigl((-\xi_{2,\tau}) \ast \xi^{(1)}\bigr) -\frac{ \Delta(\xi^{(1)})}{\pi}\right)\nonumber \\
&= {n}i_{\tau,-I_2}(\xi_{2,\tau}).
\end{align*}
 Observe that $\xi_{2,\tau}=\gamma_{a,b}$ with $a=b=0$ by Lemma~\ref{lem:distri1}(v).
Lemma~\ref{lem:distri2}(II) gives $i_{\tau,-I_2}(\xi_{2,\tau})=0$ and hence $i_{\tau,-I_{2n}}(\xi_{2n,\tau})=0$.

By the equality established above and Lemma~\ref{lem:Dong4.2}, we obtain
\begin{equation}\label{e:Distripositive-negativeA4}
i_{\tau,-I_{2n}}(\Upsilon) =i_{\tau,-I_{2n}}(\hat{\Upsilon})= \sum_{0 \le t < \tau} \nu_{t,-I_{2n}}\bigl(\hat{\Upsilon}|_{[0,t]}\bigr)
\end{equation}
if $A$ and $B$ are positive definite (and so $a_l>0$ and $b_l>0$
for $l=1,\cdots,n$), and
\begin{equation}\label{e:Distripositive-negativeB4}
i_{\tau,-I_{2n}}(\Upsilon) =i_{\tau,-I_{2n}}(\hat{\Upsilon})=  - \sum_{0 < t \le \tau} \nu_{t,-I_{2n}}\bigl(\hat{\Upsilon}|_{[0,t]}\bigr)
\end{equation}
if $A$ and $B$ are negative definite (and so $a_l<0$ and $b_l<0$
for $l=1,\cdots,n$).

Arguing as in (\ref{e:DistriDiagDelay4}),  (\ref{e:DistriUpsilonProduct})  implies
$$
\nu_{t,-I_{2n}}\bigl(\hat{\Upsilon}|_{[0,t]}\bigr)=\sum^n_{i=1}\nu_{t,-I_{2}}\bigl(\gamma_{a_i,b_i}|_{[0,t]}\bigr)
$$
Note that $\nu_{0,-I_{2n}}\bigl(\hat{\Upsilon}|_{[0,0]}\bigr) = \dim\operatorname{Ker}(I_{2n} + I_{2n}) = 0$.
(\ref{e:DistriDiagPindex3}) [resp. (\ref{e:DistriDiagPindex4})]
follows from Lemma~\ref{lem:distri2}
(\ref{e:Distripositive-negativeA4}) [resp. (\ref{e:Distripositive-negativeB4})] immediately.
 \end{proof}
\vspace{5mm}


\noindent{\bf  Funding}\quad
This work was supported by National Natural Science Foundation of China (Grant No. 12371108).\vspace{2mm}

\noindent{\bf  Author Contribution}\quad
The author confirms sole responsibility for the following: study conception and design, data collection, analysis and interpretation of results, and manuscript preparation.

\noindent{\bf  Conflict of Interest}\quad
The author declares that there is no conflict of interest regarding the publication of this paper.
\vspace{2mm}

%
%

%
%
%
%
%
%
%
%
%

\renewcommand{\refname}{REFERENCES}

\medskip

\begin{tabular}{l}
	School of Mathematical Sciences, Beijing Normal University\\
	Laboratory of Mathematics and Complex Systems, Ministry of Education\\
	Beijing 100875, The People's Republic of China\\
	E-mail address: gclu@bnu.edu.cn\\
\end{tabular}

\end{document}